\documentclass[11pt]{amsart}
\usepackage{amsmath,amsthm,amssymb,amsfonts, enumitem, fancyhdr, color, comment, graphicx, environ, mathtools, tikz-cd, bm, bbm, mathrsfs, accents, tensor, upgreek, hyperref, stackrel,nicematrix, caption,epigraph, tabularray, pdflscape, pgfplots, listings}
\usepackage[dvipsnames]{xcolor}

\definecolor{codegreen}{rgb}{0,0.6,0}
\definecolor{codegray}{rgb}{0.5,0.5,0.5}
\definecolor{codepurple}{rgb}{0.58,0,0.82}
\definecolor{backcolour}{rgb}{0.95,0.95,0.92}

\lstdefinestyle{mystyle}{
    backgroundcolor=\color{backcolour},   
    commentstyle=\color{codegreen},
    keywordstyle=\color{magenta},
    numberstyle=\tiny\color{codegray},
    stringstyle=\color{codepurple},
    basicstyle=\ttfamily\footnotesize,
    breakatwhitespace=false,         
    breaklines=true,                 
    captionpos=b,                    
    keepspaces=true,                 
    numbers=left,                    
    numbersep=5pt,                  
    showspaces=false,                
    showstringspaces=false,
    showtabs=false,                  
    tabsize=2
}

\DeclareFontFamily{U}{mathx}{}
\DeclareFontShape{U}{mathx}{m}{n}{<-> mathx10}{}
\DeclareSymbolFont{mathx}{U}{mathx}{m}{n}
\DeclareMathAccent{\widehat}{0}{mathx}{"70}
\DeclareMathAccent{\widecheck}{0}{mathx}{"71}

\title[Open book decompositions with page a four-punctured sphere]{Open book decompositions with page a four-punctured sphere}
\newtheorem{theorem}{Theorem}
\newtheorem{definition}[theorem]{Definition}
\newtheorem{proposition}[theorem]{Proposition}
\newtheorem{lemma}[theorem]{Lemma}

\newtheorem{corollary}[theorem]{Corollary}

\theoremstyle{remark}

\newtheorem{remark}[theorem]{Remark}

\newtheorem{question}[theorem]{Question}

\numberwithin{equation}{section}
\numberwithin{theorem}{section}
\numberwithin{figure}{section}
\numberwithin{table}{section}

\renewcommand{\epsilon}{\varepsilon}
\DeclareMathOperator{\Int}{int}

\usepackage{hyphenat}

\usepackage[backend=bibtex,style=alphabetic,maxalphanames=4,maxnames=4]{biblatex}

\renewbibmacro{in:}{}

\DeclareDelimFormat[bib,biblist]{nametitledelim}{\addcomma\space}

\DeclareFieldFormat*{title}{\mkbibitalic{#1}\addcomma}
\DeclareFieldFormat*{journaltitle}{#1}
\DeclareFieldFormat*{volume}{\mkbibbold{#1}}
\DeclareFieldFormat{pages}{#1}
\DeclareFieldFormat[misc]{date}{preprint {#1}}
\DeclareFieldFormat{mr}{%
  MR\addcolon\space
  \ifhyperref
    {\href{http://www.ams.org/mathscinet-getitem?mr=MR#1}{\nolinkurl{#1}}}
    {\nolinkurl{#1}}}
    
\AtEveryBibitem{
  \clearfield{url}
  \clearfield{number}
  \clearfield{doi}
  \clearfield{issn}
  \clearfield{isbn}
  \clearfield{eprintclass}
}
\AtEveryBibitem{\ifentrytype{book}{\clearfield{pages}}{}}

\bibliography{link}

\author{Harahm Park}
\address{University of California, Los Angeles, Los Angeles, CA 90095}
\email{harahmpark@math.ucla.edu} 
\date{\today}
\thanks{The author
was supported by NSF grant DMS-2136090}
\pgfplotsset{compat=1.18}

\begin{document}
\DefTblrTemplate{firsthead, middlehead,lasthead}{default}{}
\setcounter{tocdepth}{1}
   
    \begin{abstract}
     In this paper, we study contact structures supported by open book decompositions whose pages are four-punctured spheres. The paper is split into two parts. In the first part, we find infinitely many overtwisted, right-veering monodromies on the four-punctured sphere. This is done using the techniques developed by Ito--Kawamuro in the papers \cite{ItoKawamuro, ItoKawamuroOvertwisted}. Although most of the monodromies that we show are overtwisted are pseudo-Anosov, we are also able to classify precisely which reducible monodromies on the four-punctured sphere are tight. In the second part of the paper, we reprove part of a result of Lekili \cite{Lekili} by classifying which reducible mondromies have non-zero Heegaard Floer invariant. This is done by using the bordered contact invariants of Min--Varvarezos \cite{MinVarvarezos}.
    \end{abstract}
 \maketitle
 
\tableofcontents

\section{Introduction} \label{sec:intro} 
Since Giroux introduced a dictionary between contact structures and open book decompositions in the 2002 ICM proceedings \cite{Giroux}, a core theme in 3-dimensional contact topology has been the identification of properties of contact structures from the data of open books which support them. (See also \cite{BHH} and \cite{LicataVertesi1,LicataVertesi2} for modern treatments of the Giroux correspondence.) It is not straightforward to determine from the data of an open book whether the supported contact structure is tight or has non-vanishing Ozsv\'{a}th--Szab\'{o} contact invariant \cite{OzsvathSzabo}.

In the case where the page of the open book is a disk, annulus, or pair-of-pants, it is straightforward to see that a monodromy either factors into positive Dehn twists or is not right-veering. Accordingly, the supported contact structure is Stein fillable--and hence tight with non-vanishing contact invariant--or overtwisted.

In the case where the page of the open book is a once-punctured torus, by independent work of Honda--Kazez--Mati\'{c} and Baldwin it is known that the supported contact structure is tight if and only if it has non-vanishing contact invariant. This occurs if and only if the monodromy is right-veering \cite{HKMRVI, HKMRVII, HKM, BaldwinTorus}.

In this paper we study the next most approachable case, which is when the page of the open book is a four-times punctured sphere \(\Sigma_{0,4}\). We can split this case up further by applying the Nielsen--Thurston classification to the monodromy of the open book. In the case that the monodromy is periodic, it must be a product of boundary-parallel Dehn twists. This subcase is then the same as that of the disk, annulus, and pair-of-pants. Hence we may restrict our attention to reducible and pseudo-Anosov monodromies. 

This paper is split into two parts. In Part 1, we find infinitely many right-veering, overtwisted monodromies on \(\Sigma_{0,4}\). Given a compact orientable surface \(\Sigma\), let \(\mathrm{Mod}(\Sigma,\partial \Sigma)\) denote the relative mapping class group of \(\Sigma\). In \cite{ItoKawamuroTight} the following theorem is proven. 

\begin{theorem}[Ito--Kawamuro]
\label{thm: Ito-Kawamuro}
If \(\Sigma\) is a compact planar surface with non-empty boundary and \(f\in \mathrm{Mod}(\Sigma,\partial \Sigma)\) is a monodromy whose fractional Dehn twist coefficients {\normalfont(}FDTCs{\normalfont)} are all strictly greater than \(1\), then the corresponding open book supports a tight contact structure. 
\end{theorem}

On a four-punctured sphere \(\Sigma_{0,4}\), any monodromy has integral FDTCs. In addition, any monodromy on \(\Sigma_{0,4}\) which is right-veering and has minimum FDTC zero factors into positive Dehn twists. Thus we may restrict our attention to monodromies whose minimum FDTC is 1. We are able to find two infinite families of such monodromies which are overtwisted using the technique of \textit{open book foliations}, developed and used by Ito and Kawamuro in \cite{ItoKawamuro, ItoKawamuroOperations,ItoKawamuroOvertwisted,ItoKawamuroSelfLinking,ItoKawamuroTight,ItoKawamuroCoverings,ItoKawamuroFDTC,ItoKawamuroQuestion,ItoKawamuroPositive,ItoKawamuroQuasi,twistleft}. 

Although Ito and Kawamuro were the first to use the term ``open book foliation'', the concept goes back to work of Bennequin \cite{Bennequin} distinguishing contact structures on \(S^3\). Later, Birman and Menasco used similar ideas in \cite{BirmanMenascoI, BirmanMenascoII, BirmanMenascoIII, BirmanMenascoIV, BirmanMenascoV, BirmanMenascoVI} to study braids. They used the name \textit{braid foliations}. In her thesis \cite{Pavelescu}, Pevalescu generalized braid foliations to arbitrary contact 3-manifolds. 

The collection of monodromies we can show are overtwisted using this method is a bit difficult to describe; we postpone the precise statement to Theorems \ref{thm:main_one} and \ref{thm:main_two}. Most of the monodromies that we show are overtwisted are pseudo-Anosov. However, we are able to show enough reducible monodromies are overtwisted to classify precisely which reducible monodromies are tight. Let \(f\in \mathrm{Mod}(\Sigma_{0,4},\partial \Sigma_{0,4})\) be a reducible monodromy fixing an essential circle \(\gamma\subset \Sigma_{0,4}\), and let \(a_1,\ldots,a_4\) denote circles parallel to the boundary components of \(\Sigma_{0,4}\). Then, up to isotopy,
\begin{equation}f= \tau_{a_1}^{n_1}\tau_{a_2}^{n_2}\tau_{a_3}^{n_3}\tau_{a_4}^{n_4}\tau_{\gamma}^{n_\gamma},\label{eq:intro_monoddromy}\end{equation}
for some \(n_1,n_2,n_3,n_4, n_{\gamma}\in\mathbb{Z}\). Here \(\tau_{\gamma}\) denotes a positive, i.e., a right-handed Dehn twist about \(\gamma\). (In the case when \(n_\gamma=0\), the mondromy becomes periodic instead of reducible, but the results of Theorem \ref{thm:classify_tight_reducible} and \ref{thm: Lekili} are still applicable.)  
\begin{theorem}\label{thm:classify_tight_reducible}
{\normalfont(}Proof on page {\normalfont\pageref{proof:classify_tight_reducible})} If \(f\) is as in {\normalfont(\ref{eq:intro_monoddromy})}, then \(f\) is tight if and only if 
\begin{enumerate}[label={\normalfont (\roman*)}]
\item \(\min\{n_1,n_2,n_3,n_4\}\geq 2\), or 
\item \(\min\{n_1,n_2,n_3,n_4\} \geq \max\{-n_{\gamma},0\}\). 
\end{enumerate}
\end{theorem}
 In Part 2, we reprove part of a result of Lekili by classifying which reducible monodromies on \(\Sigma_{0,4}\) have non-vanishing Heegaard Floer contact invariant. 

\begin{theorem}[Lekili]
\label{thm: Lekili}
The open book specified by the monodromy {\normalfont(\ref{eq:intro_monoddromy})} has non-vanishing invariant precisely when it is Stein fillable. This occurs if and only if
\[\min \{n_1,n_2,n_3,n_4 \}\geq \max \{-n_\gamma,0\}.\]
\end{theorem}
Lekili proves a slightly more general statement, as explained in Remark \ref{rem:Lekili}. This is done via Baldwin's capping off theorem \cite{BaldwinCap}. We note that Baldwin's capping off result is not valid in full generality, as counterexamples have been found by Min. However, the result holds when the capped-off manifold is a rational homology sphere. Together with the naturality of the contact invariant under Legendrian surgery, this is enough to give a proof of the result in Remark \ref{rem:Lekili}. 

We give an alternate proof of the above theorem using the contact invariants in bordered-sutured Floer homology due to Min--Varvarezos \cite{MinVarvarezos}. Precursors of this invariant include the bordered invariants of \cite{AFHLPV} and the sutured invariant of \cite{HKMsuture}. Bordered-sutured Floer homology is defined by Zarev in \cite{ZarevThesis}.

\subsection*{Conventions}
In the mapping class group, we use function notation without the symbol \(\circ\), so that the elements on the right are applied first. For \(\widehat{\mathit{HF}}\) and the associated bordered-sutured theory, we use \(\mathbb{Z}/2\) coefficients. 
\medskip

\noindent{\em Acknowledgments.}  I thank my advisor Ko Honda for his encouragement and support, and for helpful comments on earlier drafts of this paper. I thank Hyunki Min and Konstantinos Varvarezos for their thorough guidance in writing the second part of this paper. I thank Cagatay Kutluhan for help with the Sage programs \texttt{hf-hat-obd} and \texttt{hf-hat-obd-nice}, which were useful in writing the first part of this paper. Finally, I thank John Baldwin for his online problem list \url{https://floerhomologyproblems.blogspot.com}, which is what got me started working on this problem. 

\part{Overtwisted, right-veering monodromies}
\section{Curves, arcs, and the mapping class group} 
In this section we review some preliminaries about the four-punctured sphere which we will need to argue overtwistedness. We also establish some terminology. In Figure \ref{fig:page}, we give a picture of the four-punctured sphere \(\Sigma_{0,4}\) together with some embedded closed curves. 
\begin{figure}[ht]
    \centering
    \includegraphics[scale = 0.15]{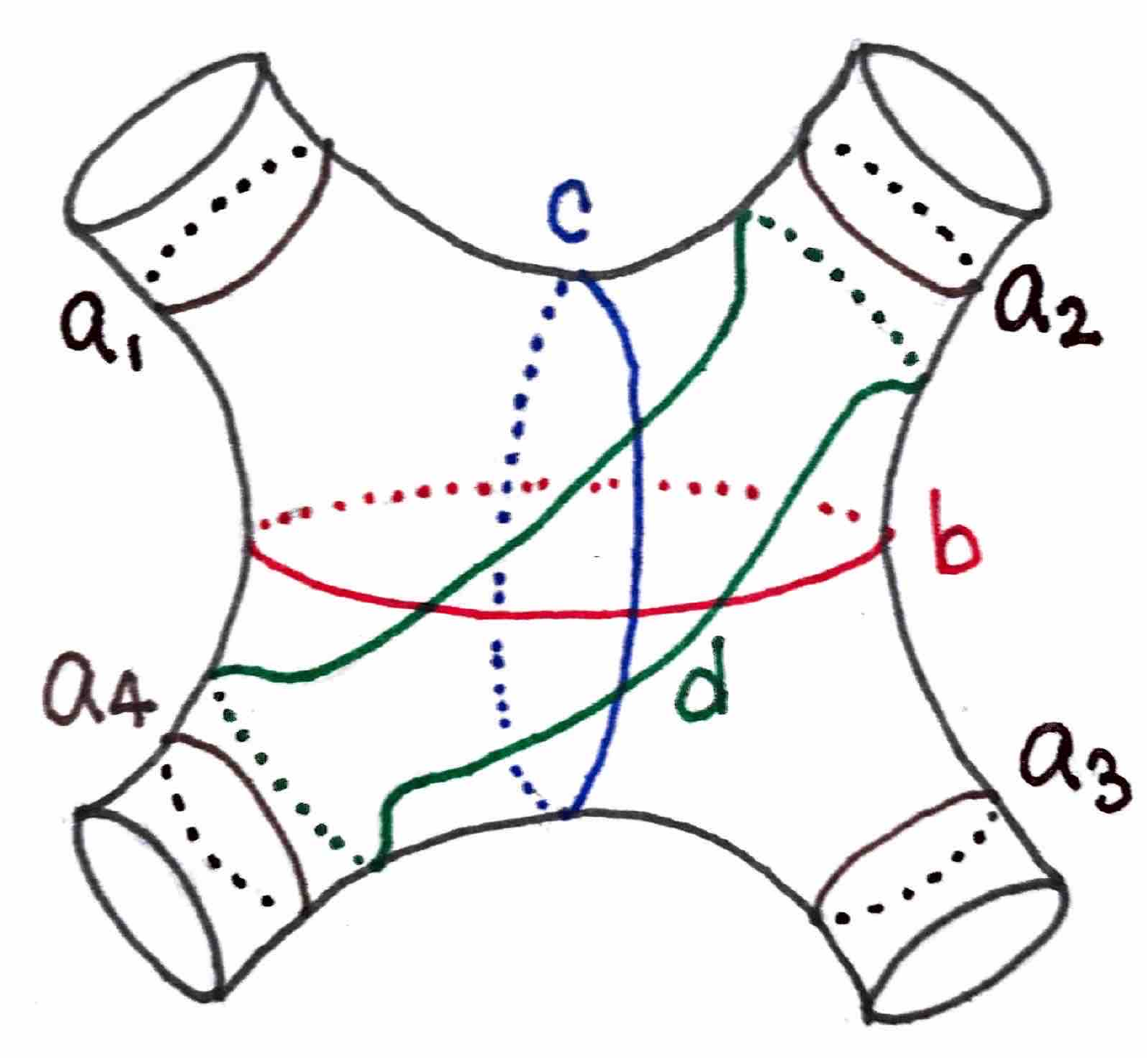}
    \caption{}
    \label{fig:page}
\end{figure}

Note \(\Sigma_{0,4}\) is covered by the plane \(\mathbb{R}^2\) punctured at the integer lattice \(\mathbb{Z}^2\), as shown in Figure \ref{fig:cover}.
\begin{figure}[ht]
    \centering
         \raisebox{-0.5\height}{\includegraphics[scale = 0.15]{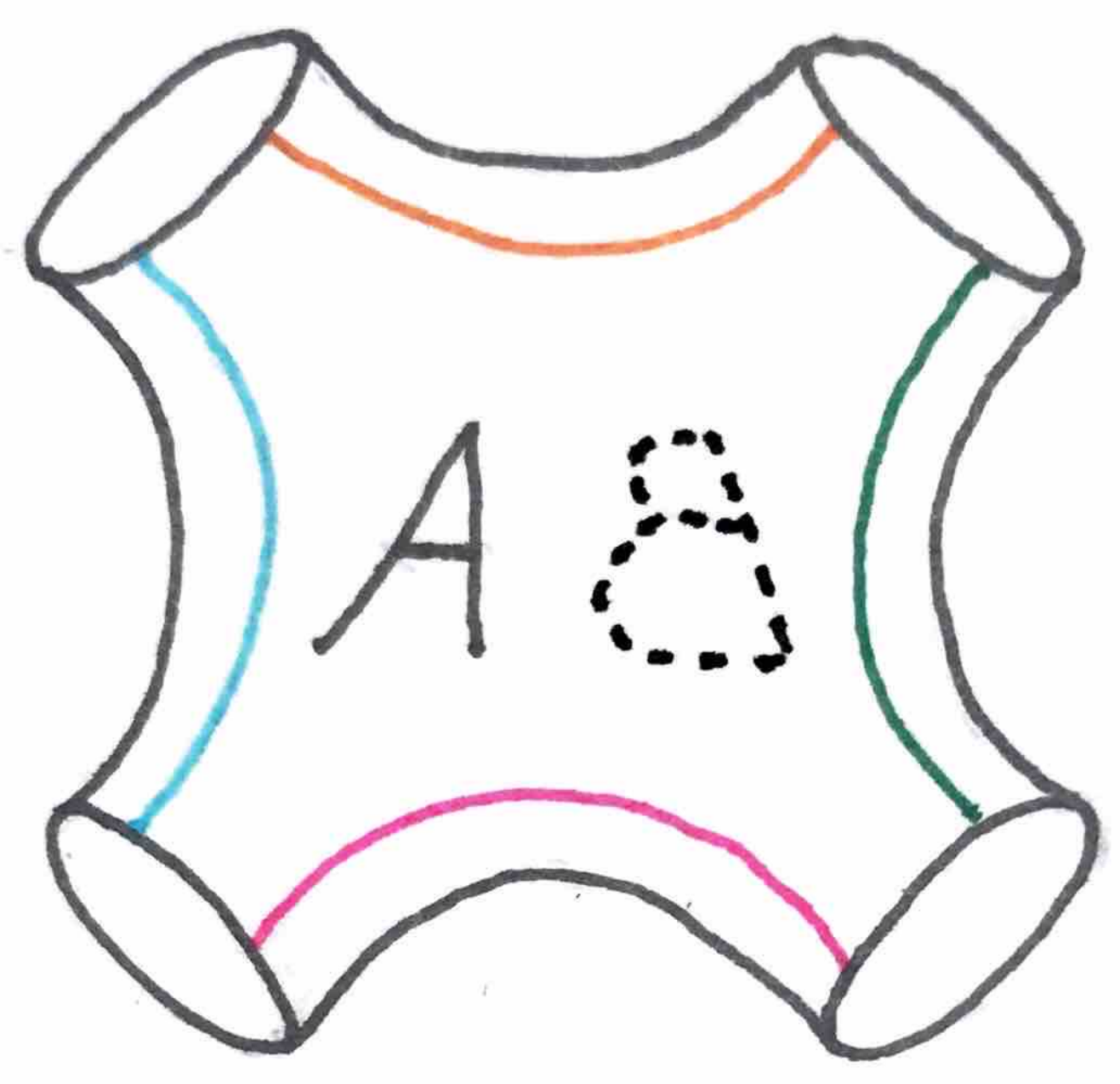}}\qquad  \raisebox{-0.5\height}{\includegraphics[scale = 0.25]{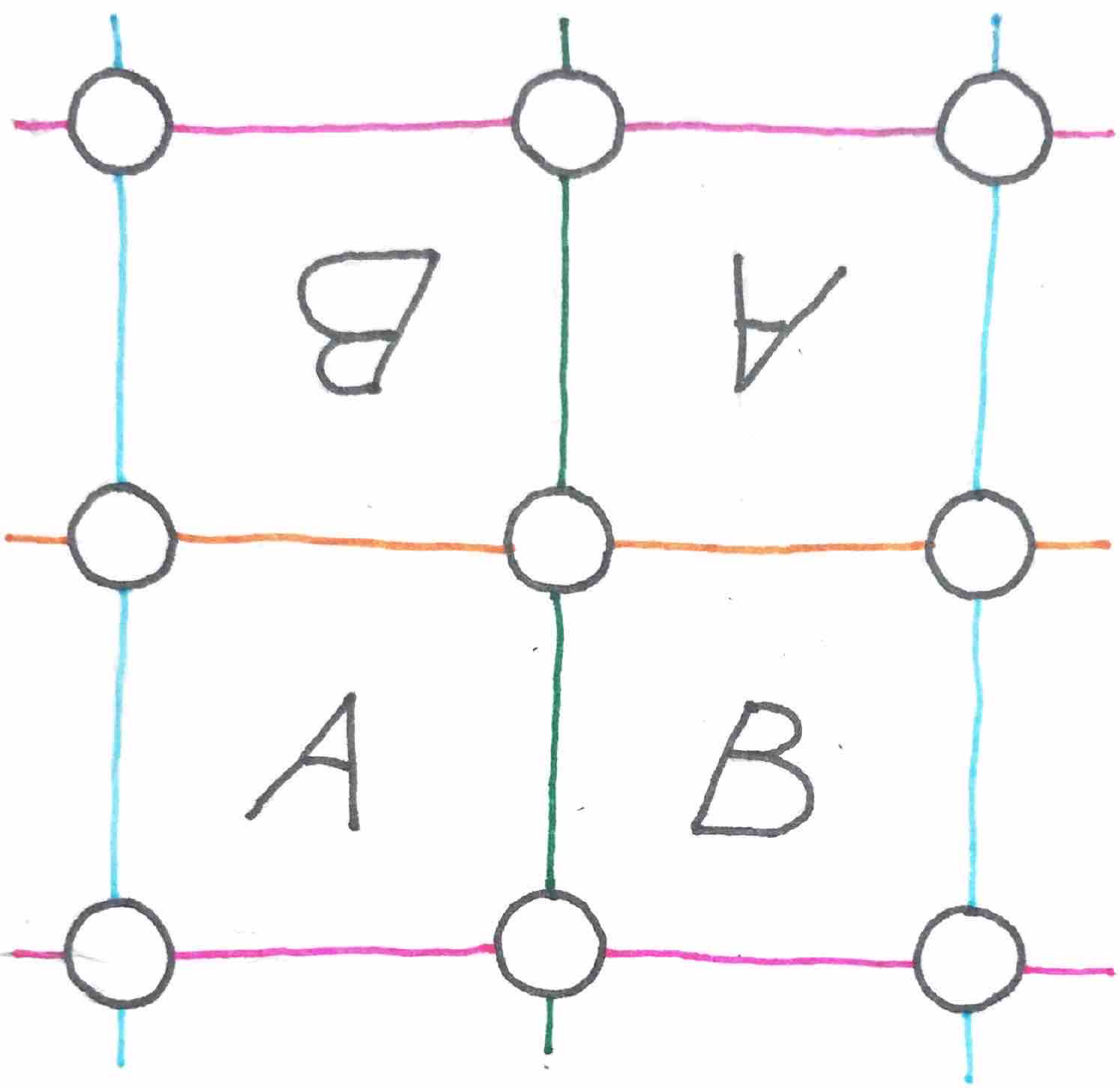}}
    \caption{}
    \label{fig:cover}
\end{figure}
We can assign essential embedded circles in \(\Sigma_{0,4}\) slopes in \(\mathbb{Q}\cup \{\infty\}\) by lifting them to lines in this cover. Here essential means not bounding an embedded disk and not being boundary parallel. We do this in such a way that the curve \(b\) has slope \(0\) and \(c\) has slope \(\infty\) in Figure \ref{fig:page}. Suppose \(s=p/q\) is a slope with \(p,q\in\mathbb{Z}\) relatively prime. By an \textit{arc of slope} \(s\), we mean a properly embedded arc in \(\Sigma_{0,4}\) whose lifts to the cover in Figure \ref{fig:cover} satisfies the following properties:
\begin{enumerate}[label=(\roman*)]
\item  After shrinking the punctures to points in \(\mathbb{Z}^2\), each lift has endpoints which differ by the vector \(\pm(p,q)\) (not a scalar multiple of this vector).
\item After shrinking punctures to points, each lift can be straightened to a line segment in the complement of the punctures.
\end{enumerate}

Using the assignment of embedded circles to slopes, we can think of isotopy classes of embedded circles in \(\Sigma_{0,4}\) as vertices in the \textit{Farey tesselation}; see Figure \ref{fig:Farey}. Two embedded circles are joined by an arc in the Farey tesselation if they can be isotoped to intersect transversely only twice–the minimum number in \(\Sigma_{0,4}\).
\begin{figure}[ht]
    \centering
         \begin{tikzpicture}
            \def\r{2.5}
             \draw [thick] (0,0) circle (\r);
             \filldraw [black] (\r,0) circle (2pt) node[anchor=west]{$\frac{0}{1}$};
             \filldraw [black] (0,\r) circle (2pt) node[anchor=south]{$\frac{1}{1}$};
             \filldraw [black] (-\r,0) circle (2pt) node[anchor=east]{$\frac{1}{0}$};
             \filldraw [black] (0,-\r) circle (2pt) node[anchor=north]{$-\frac{1}{1}\phantom{-}$};
             \filldraw [black] ({(3/5)*\r},{(4/5)*\r}) circle (2pt) node[anchor=south west]{$\frac{1}{2}$};
             \filldraw [black] ({-(3/5)*\r},{(4/5)*\r}) circle (2pt) node[anchor=south east]{$\frac{2}{1}$};
             \filldraw [black] ({-(3/5)*\r},{-(4/5)*\r}) circle (2pt) node[anchor=north east]{$-\frac{2}{1}$};
             \filldraw [black] ({(3/5)*\r},{-(4/5)*\r}) circle (2pt) node[anchor=north west]{$-\frac{1}{2}$};
             \filldraw [black] ({(4/5)*\r},{(3/5)*\r}) circle (2pt) node[anchor=south west]{$\frac{1}{3}$};
             \filldraw [black] ({-(4/5)*\r},{(3/5)*\r}) circle (2pt) node[anchor=south east]{$\frac{3}{1}$};
             \filldraw [black] ({-(4/5)*\r},{-(3/5)*\r}) circle (2pt) node[anchor=north east]{$-\frac{3}{1}$};
             \filldraw [black] ({(4/5)*\r},{-(3/5)*\r}) circle (2pt) node[anchor=north west]{$-\frac{1}{3}$};
             \filldraw [black] ({(5/13)*\r},{(12/13)*\r}) circle (2pt) node[anchor=south]{$\frac{2}{3}$};
             \filldraw [black] ({(-5/13)*\r},{(12/13)*\r}) circle (2pt) node[anchor=south]{$\frac{3}{2}$};
             \filldraw [black] ({(5/13)*\r},{-(12/13)*\r}) circle (2pt) node[anchor=north]{$-\frac{2}{3}$};
             \filldraw [black] ({(-5/13)*\r},{-(12/13)*\r}) circle (2pt) node[anchor=north]{$-\frac{3}{2}\phantom{-}$};
             \filldraw [black] ({(15/17)*\r},{(8/17)*\r}) circle (2pt) node[anchor= west]{$\frac{1}{4}$};
             \filldraw [black] ({(-15/17)*\r},{(8/17)*\r}) circle (2pt) node[anchor= east]{$\frac{4}{1}$};
             \filldraw [black] ({(15/17)*\r},{-(8/17)*\r}) circle (2pt) node[anchor= west]{$-\frac{1}{4}$};
             \filldraw [black] ({(-15/17)*\r},{-(8/17)*\r}) circle (2pt) node[anchor= east]{$-\frac{4}{1}$};
             \draw (-\r,0)--(\r,0);
             \draw (0,\r) arc (180:270:\r);
             \draw (-\r,0) arc (270:360:\r);
             \draw (0,-\r) arc (0:90:\r);
             \draw (\r,0) arc (90:180:\r);
             \draw [domain=0:180, variable = \t, samples=20, trig format=deg, smooth] plot ({\r*(1-(1/4+cos(\t)/4)^2-(sin(\t)/4)^2)/((1/4+cos(\t)/4)^2+(1+sin(\t)/4)^2)},{2*\r*(1/4+cos(\t)/4)/((1/4+cos(\t)/4)^2+(1+sin(\t)/4)^2)});
             \draw [domain=0:180, variable = \t, samples=20, trig format=deg, smooth] plot ({\r*(1-(1/4+cos(\t)/4)^2-(sin(\t)/4)^2)/((1/4+cos(\t)/4)^2+(1+sin(\t)/4)^2)},{-2*\r*(1/4+cos(\t)/4)/((1/4+cos(\t)/4)^2+(1+sin(\t)/4)^2)});
             \draw [domain=0:180, variable = \t, samples=20, trig format=deg, smooth] plot ({-\r*(1-(1/4+cos(\t)/4)^2-(sin(\t)/4)^2)/((1/4+cos(\t)/4)^2+(1+sin(\t)/4)^2)},{2*\r*(1/4+cos(\t)/4)/((1/4+cos(\t)/4)^2+(1+sin(\t)/4)^2)});
             \draw [domain=0:180, variable = \t, samples=20, trig format=deg, smooth] plot ({-\r*(1-(1/4+cos(\t)/4)^2-(sin(\t)/4)^2)/((1/4+cos(\t)/4)^2+(1+sin(\t)/4)^2)},{-2*\r*(1/4+cos(\t)/4)/((1/4+cos(\t)/4)^2+(1+sin(\t)/4)^2)});
             \draw [domain=0:180, variable = \t, samples=20, trig format=deg, smooth] plot ({\r*(1-(3/4+cos(\t)/4)^2-(sin(\t)/4)^2)/((3/4+cos(\t)/4)^2+(1+sin(\t)/4)^2)},{2*\r*(3/4+cos(\t)/4)/((3/4+cos(\t)/4)^2+(1+sin(\t)/4)^2)});
             \draw [domain=0:180, variable = \t, samples=20, trig format=deg, smooth] plot ({-\r*(1-(3/4+cos(\t)/4)^2-(sin(\t)/4)^2)/((3/4+cos(\t)/4)^2+(1+sin(\t)/4)^2)},{2*\r*(3/4+cos(\t)/4)/((3/4+cos(\t)/4)^2+(1+sin(\t)/4)^2)});
             \draw [domain=0:180, variable = \t, samples=20, trig format=deg, smooth] plot ({\r*(1-(3/4+cos(\t)/4)^2-(sin(\t)/4)^2)/((3/4+cos(\t)/4)^2+(1+sin(\t)/4)^2)},{-2*\r*(3/4+cos(\t)/4)/((3/4+cos(\t)/4)^2+(1+sin(\t)/4)^2)});
             \draw [domain=0:180, variable = \t, samples=20, trig format=deg, smooth] plot ({-\r*(1-(3/4+cos(\t)/4)^2-(sin(\t)/4)^2)/((3/4+cos(\t)/4)^2+(1+sin(\t)/4)^2)},{-2*\r*(3/4+cos(\t)/4)/((3/4+cos(\t)/4)^2+(1+sin(\t)/4)^2)});
             \draw [domain=0:180, variable = \t, samples=20, trig format=deg, smooth] plot ({\r*(1-(1/6+cos(\t)/6)^2-(sin(\t)/6)^2)/((1/6+cos(\t)/6)^2+(1+sin(\t)/6)^2)},{2*\r*(1/6+cos(\t)/6)/((1/6+cos(\t)/6)^2+(1+sin(\t)/6)^2)});
             \draw [domain=0:180, variable = \t, samples=20, trig format=deg, smooth] plot ({-\r*(1-(1/6+cos(\t)/6)^2-(sin(\t)/6)^2)/((1/6+cos(\t)/6)^2+(1+sin(\t)/6)^2)},{2*\r*(1/6+cos(\t)/6)/((1/6+cos(\t)/6)^2+(1+sin(\t)/6)^2)});
             \draw [domain=0:180, variable = \t, samples=20, trig format=deg, smooth] plot ({\r*(1-(1/6+cos(\t)/6)^2-(sin(\t)/6)^2)/((1/6+cos(\t)/6)^2+(1+sin(\t)/6)^2)},{-2*\r*(1/6+cos(\t)/6)/((1/6+cos(\t)/6)^2+(1+sin(\t)/6)^2)});
             \draw [domain=0:180, variable = \t, samples=20, trig format=deg, smooth] plot ({-\r*(1-(1/6+cos(\t)/6)^2-(sin(\t)/6)^2)/((1/6+cos(\t)/6)^2+(1+sin(\t)/6)^2)},{-2*\r*(1/6+cos(\t)/6)/((1/6+cos(\t)/6)^2+(1+sin(\t)/6)^2)});
             \draw [domain=0:180, variable = \t, samples=20, trig format=deg, smooth] plot ({\r*(1-(5/12+cos(\t)/12)^2-(sin(\t)/12)^2)/((5/12+cos(\t)/12)^2+(1+sin(\t)/12)^2)},{2*\r*(5/12+cos(\t)/12)/((5/12+cos(\t)/12)^2+(1+sin(\t)/12)^2)});
             \draw [domain=0:180, variable = \t, samples=20, trig format=deg, smooth] plot ({-\r*(1-(5/12+cos(\t)/12)^2-(sin(\t)/12)^2)/((5/12+cos(\t)/12)^2+(1+sin(\t)/12)^2)},{2*\r*(5/12+cos(\t)/12)/((5/12+cos(\t)/12)^2+(1+sin(\t)/12)^2)});
             \draw [domain=0:180, variable = \t, samples=20, trig format=deg, smooth] plot ({\r*(1-(5/12+cos(\t)/12)^2-(sin(\t)/12)^2)/((5/12+cos(\t)/12)^2+(1+sin(\t)/12)^2)},{-2*\r*(5/12+cos(\t)/12)/((5/12+cos(\t)/12)^2+(1+sin(\t)/12)^2)});
             \draw [domain=0:180, variable = \t, samples=20, trig format=deg, smooth] plot ({-\r*(1-(5/12+cos(\t)/12)^2-(sin(\t)/12)^2)/((5/12+cos(\t)/12)^2+(1+sin(\t)/12)^2)},{-2*\r*(5/12+cos(\t)/12)/((5/12+cos(\t)/12)^2+(1+sin(\t)/12)^2)});
             \draw [domain=0:180, variable = \t, samples=20, trig format=deg, smooth] plot ({\r*(1-(7/12+cos(\t)/12)^2-(sin(\t)/12)^2)/((7/12+cos(\t)/12)^2+(1+sin(\t)/12)^2)},{2*\r*(7/12+cos(\t)/12)/((7/12+cos(\t)/12)^2+(1+sin(\t)/12)^2)});
             \draw [domain=0:180, variable = \t, samples=20, trig format=deg, smooth] plot ({-\r*(1-(7/12+cos(\t)/12)^2-(sin(\t)/12)^2)/((7/12+cos(\t)/12)^2+(1+sin(\t)/12)^2)},{2*\r*(7/12+cos(\t)/12)/((7/12+cos(\t)/12)^2+(1+sin(\t)/12)^2)});
             \draw [domain=0:180, variable = \t, samples=20, trig format=deg, smooth] plot ({\r*(1-(7/12+cos(\t)/12)^2-(sin(\t)/12)^2)/((7/12+cos(\t)/12)^2+(1+sin(\t)/12)^2)},{-2*\r*(7/12+cos(\t)/12)/((7/12+cos(\t)/12)^2+(1+sin(\t)/12)^2)});
             \draw [domain=0:180, variable = \t, samples=20, trig format=deg, smooth] plot ({-\r*(1-(7/12+cos(\t)/12)^2-(sin(\t)/12)^2)/((7/12+cos(\t)/12)^2+(1+sin(\t)/12)^2)},{-2*\r*(7/12+cos(\t)/12)/((7/12+cos(\t)/12)^2+(1+sin(\t)/12)^2)});
             \draw [domain=0:180, variable = \t, samples=20, trig format=deg, smooth] plot ({\r*(1-(5/6+cos(\t)/6)^2-(sin(\t)/6)^2)/((5/6+cos(\t)/6)^2+(1+sin(\t)/6)^2)},{2*\r*(5/6+cos(\t)/6)/((5/6+cos(\t)/6)^2+(1+sin(\t)/6)^2)});
             \draw [domain=0:180, variable = \t, samples=20, trig format=deg, smooth] plot ({-\r*(1-(5/6+cos(\t)/6)^2-(sin(\t)/6)^2)/((5/6+cos(\t)/6)^2+(1+sin(\t)/6)^2)},{2*\r*(5/6+cos(\t)/6)/((5/6+cos(\t)/6)^2+(1+sin(\t)/6)^2)});
             \draw [domain=0:180, variable = \t, samples=20, trig format=deg, smooth] plot ({\r*(1-(5/6+cos(\t)/6)^2-(sin(\t)/6)^2)/((5/6+cos(\t)/6)^2+(1+sin(\t)/6)^2)},{-2*\r*(5/6+cos(\t)/6)/((5/6+cos(\t)/6)^2+(1+sin(\t)/6)^2)});
             \draw [domain=0:180, variable = \t, samples=20, trig format=deg, smooth] plot ({-\r*(1-(5/6+cos(\t)/6)^2-(sin(\t)/6)^2)/((5/6+cos(\t)/6)^2+(1+sin(\t)/6)^2)},{-2*\r*(5/6+cos(\t)/6)/((5/6+cos(\t)/6)^2+(1+sin(\t)/6)^2)});
             \draw [domain=0:180, variable = \t, samples=20, trig format=deg, smooth] plot ({\r*(1-(1/8+cos(\t)/8)^2-(sin(\t)/8)^2)/((1/8+cos(\t)/8)^2+(1+sin(\t)/8)^2)},{2*\r*(1/8+cos(\t)/8)/((1/8+cos(\t)/8)^2+(1+sin(\t)/8)^2)});
             \draw [domain=0:180, variable = \t, samples=20, trig format=deg, smooth] plot ({-\r*(1-(1/8+cos(\t)/8)^2-(sin(\t)/8)^2)/((1/8+cos(\t)/8)^2+(1+sin(\t)/8)^2)},{2*\r*(1/8+cos(\t)/8)/((1/8+cos(\t)/8)^2+(1+sin(\t)/8)^2)});
             \draw [domain=0:180, variable = \t, samples=20, trig format=deg, smooth] plot ({\r*(1-(1/8+cos(\t)/8)^2-(sin(\t)/8)^2)/((1/8+cos(\t)/8)^2+(1+sin(\t)/8)^2)},{-2*\r*(1/8+cos(\t)/8)/((1/8+cos(\t)/8)^2+(1+sin(\t)/8)^2)});
             \draw [domain=0:180, variable = \t, samples=20, trig format=deg, smooth] plot ({-\r*(1-(1/8+cos(\t)/8)^2-(sin(\t)/8)^2)/((1/8+cos(\t)/8)^2+(1+sin(\t)/8)^2)},{-2*\r*(1/8+cos(\t)/8)/((1/8+cos(\t)/8)^2+(1+sin(\t)/8)^2)});
             \draw [domain=0:180, variable = \t, samples=20, trig format=deg, smooth] plot ({\r*(1-(7/24+cos(\t)/24)^2-(sin(\t)/24)^2)/((7/24+cos(\t)/24)^2+(1+sin(\t)/24)^2)},{2*\r*(7/24+cos(\t)/24)/((7/23+cos(\t)/23)^2+(1+sin(\t)/24)^2)});
             \draw [domain=0:180, variable = \t, samples=20, trig format=deg, smooth] plot ({-\r*(1-(7/24+cos(\t)/24)^2-(sin(\t)/24)^2)/((7/24+cos(\t)/24)^2+(1+sin(\t)/24)^2)},{2*\r*(7/24+cos(\t)/24)/((7/23+cos(\t)/23)^2+(1+sin(\t)/24)^2)});
             \draw [domain=0:180, variable = \t, samples=20, trig format=deg, smooth] plot ({\r*(1-(7/24+cos(\t)/24)^2-(sin(\t)/24)^2)/((7/24+cos(\t)/24)^2+(1+sin(\t)/24)^2)},{-2*\r*(7/24+cos(\t)/24)/((7/23+cos(\t)/23)^2+(1+sin(\t)/24)^2)});
             \draw [domain=0:180, variable = \t, samples=20, trig format=deg, smooth] plot ({-\r*(1-(7/24+cos(\t)/24)^2-(sin(\t)/24)^2)/((7/24+cos(\t)/24)^2+(1+sin(\t)/24)^2)},{-2*\r*(7/24+cos(\t)/24)/((7/23+cos(\t)/23)^2+(1+sin(\t)/24)^2)});
         \end{tikzpicture}
    \caption{}
    \label{fig:Farey}
\end{figure}

 Recall from \cite[Section 2.2.5]{primer} that the mapping class group of \(\Sigma_{0,4}\) is 
\[\mathrm{Mod}(\Sigma_{0,4}) \cong (\mathbb{Z}/2\times \mathbb{Z}/2) \rtimes \mathrm{PSL}(2,\mathbb{Z}).\]
The projection \(\pi\colon \mathrm{Mod}(\Sigma_{0,4})\to \mathrm{PSL}(2,\mathbb{Z})\) records the action of \(\Sigma_{0,4}\) on simple closed curves. One way to see this action is to view the Farey tesselation as a tiling of the upper half space model of hyperbolic space. We then view \(\mathrm{PSL}(2,\mathbb{Z})\leq \mathrm{PSL}(2,\mathbb{R})\) as a subgroup of orientation-preserving hyperbolic isometries, i.e., M\"{o}bius transformations. The kernel of \(\pi\) is generated by two \textit{hyperelliptic involutions} which permute the boundary components by double transpositions; see Figure \ref{fig:involution}
\begin{figure}[ht]
\centering
\includegraphics[scale=0.15]{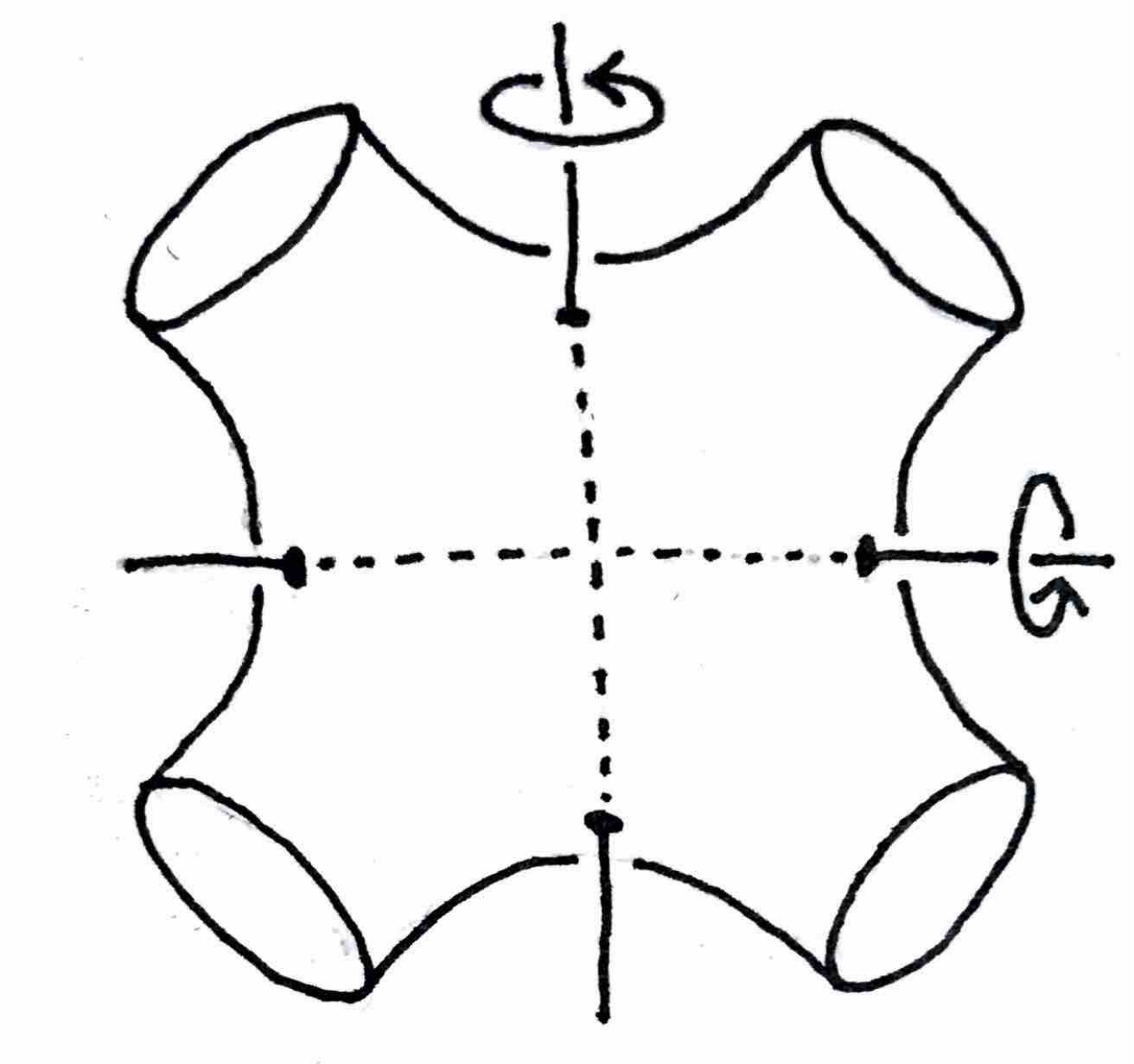}
\caption{}
\label{fig:involution}
\end{figure}

Next, from \cite{Lekili} we have a non-canonical isomorphism of the relative mapping class group
\[\mathrm{Mod}(\Sigma_{0,4},\partial \Sigma_{0,4})\cong \mathbb{Z}^4\times F_2,\]
where \(F_2\) is the free group on two elements. The \(\mathbb{Z}^4\) factor is canonical. However, there is no canonical choice of \(F_2\) factor. There is a natural map \[\mathrm{Mod}(\Sigma_{0,4},\partial \Sigma_{0,4})\to \mathrm{Mod}(\Sigma_{0,4}),\]
whose kernel is the \(\mathbb{Z}^4\) factor above. This factor is generated by boundary parallel Dehn twists. Next, the composition
\[\mathrm{Mod}(\Sigma_{0,4},\partial \Sigma_{0,4})\to \mathrm{Mod}(\Sigma_{0,4})\xrightarrow{\,\pi\,}\mathrm{PSL}(2,\mathbb{Z})\]
is injective on the \(F_2\) factor. We also call the above composition \(\pi\). One can verify that
\[\pi (\tau_b)=\pm\begin{bmatrix}1 & 2\\ 0 & 1\end{bmatrix},\qquad \pi(\tau_c) = \pm\begin{bmatrix}\phantom{-}1 & 0\\ -2 & 1\end{bmatrix}.\]
In addition,
\begin{equation}\pi(F_2)  = \bigg\{{\pm \begin{bmatrix}r & s\\ p& q\end{bmatrix}\in \mathrm{PSL}(2,\mathbb{Z}): r, q \text{ odd},\,p, s \text{ even} }\bigg\}\trianglelefteq \mathrm{PSL}(2,\mathbb{Z}).\label{eq:subgroup}\end{equation}
Note two elements of \(\mathrm{Mod}(\Sigma_{0,4},\partial \Sigma_{0,4})\) are the same if and only if they have the same image under \(\pi\) and have the same FDTCs.
\section{Overtwistedness}
To show overtwistedness, we first choose a good conjugacy class representative of the monodromy \(f\in \mathrm{Mod}(\Sigma_{0,4},\partial \Sigma_{0,4})\).
\begin{proposition}
\label{prop:conjugacy}
Let \(a_i,b,c,d\) be the curves in Figure {\normalfont \ref{fig:page}}. If, up to conjugacy by \(\mathrm{Mod}(\Sigma_{0,4})\), the monodromy \(f\) is not
\begin{equation}\tau_{a_1}^{n_1}\tau_{a_2}^{n_2}\tau_{a_3}^{n_3}\tau_{a_4}^{n_4}\tau_c^{n_c},\qquad n_c\geq 0, \label{eq:forbidden}\end{equation}
then, up to conjugacy, one may pick \(f\) so that
\[\pi(f) =\pm \begin{bmatrix}p'& q'\\ p & q\end{bmatrix}\]
where \(p,q,p',q'\geq 0\) and \(p>q\). 
\end{proposition}
\begin{proof}
If \(f\) is not pseudo-Anosov and is not one of of the above forbidden monodromies, then up to conjugacy it is of the form 
\[\tau_{a_1}^{n_1}\tau_{a_2}^{n_2}\tau_{a_3}^{n_3}\tau_{a_4}^{n_4}\tau_c^{n_c},\qquad n_c<0.\]
Applying \(\pi\) to the above monodromy gives 
\[\pm \begin{bmatrix}\phantom{-}1 & 0\\ -2n_c & 1\end{bmatrix},\]
which is indeed in the desired form. Assume now that \(f\) is pseudo-Anosov. We first claim that up to conjugacy, one can choose \(f\) so that 
\begin{equation}\pi(f) =\pm \begin{bmatrix}p'& q'\\ p & q\end{bmatrix},\qquad p,q,p',q'\geq 0.\label{eq:positive}\end{equation}
We argue as in \cite[Lemma 4.2]{HKM}. Note that \(\pi(f)\) is a hyperbolic element of \(\mathrm{PSL}(2,\mathbb{Z})\). Let \(\lambda_1\) denote the slope of the eigenspace with eigenvalue of absolute value greater than \(1\). Let \(\lambda_2\) be the slope of the other eigenspace. In the Farey tesselation, let \(\mu_1\) be a rational slope on the counterclockwise edge from \(\lambda_2\) to \(\lambda_1\) and \(\mu_2\) a rational slope on the counterclockwise edge from \(\lambda_1\) to \(\lambda_2\), chosen so that there is an edge between \(\mu_1,\mu_2\). (Such a pair exists by considering continued fraction approximations of, say, \(\lambda_1\).) Let \(A\in \mathrm{PSL}(2,\mathbb{Z})\) be the unique element whose columns have slope \(\mu_1,\mu_2\), in order. Then \(A\pi(f)A^{-1}\) has entries all of the same sign. Since the map \(\pi\colon \mathrm{Mod}(\Sigma_{0,4})\to \mathrm{PSL}(2,\mathbb{Z})\) is onto, this completes the intermediate claim.  

Now suppose that \(f\) is chosen as in (\ref{eq:positive}). As \(f\) is not reducible, \(p\) cannot be \(0\), so must be strictly positive. Suppose that \(p\leq q\). We compute 
\[\begin{bmatrix}1 & 1\\ 0 & 1 \end{bmatrix}\pi(f)\begin{bmatrix}1 & 1\\ 0 & 1 \end{bmatrix}^{-1}=\pm  \begin{bmatrix}p+p' & q-p+q'-p'\\ p & q-p\end{bmatrix}.\]
All the entries on the right-hand-side are non-negative except perhaps \(q-p+q'-p'\). Furthermore the diagonal entries are strictly positive, since they are odd. As this matrix has determinant \(1\), the only way for \(q-p+q'-p'\) to be negative would be to have \(p=0\). This puts us back into the reducible case considered above. By iterating the above conjugation, we can eventually arrange so that \(q<p\). 
\end{proof}
\begin{remark}
Note the monodromies in (\ref{eq:forbidden}) are tight if and only if \(n_1,n_2,n_3,n_4\geq 0\). Indeed, if each \(n_i\) is non-negative, then the monodromy factors into positive Dehn twists. On the other hand, if some \(n_i<0\) then the monodromy is not right-veering.
\end{remark}

From here on out, \textit{we assume \(f\in \mathrm{Mod}(\Sigma_{0,4},\partial \Sigma_{0,4})\) is taken as in the above proposition}. That is to say, we assume
\[\pi(f)= \pm \begin{bmatrix}p' & q'\\ p & q\end{bmatrix}\]
where \(p,q,p',q'\geq 0\) and \(p>q\). Recall that \(p\) is even and \(q\) is odd. The fact that the above matrix has determinant \(1\) implies \(p\) and \(q\) share no common factors. 

In general, given relatively prime positive integers \(p>q>0\), one can consider a ``subtractive'' continued fraction expansion for \(-p/q\) of the form 
\[-\frac{p}{q}= r_0 - \dfrac{1}{r_1-\dfrac{1}{r_2-\cdots-\dfrac{1}{r_k}}}\]
where each \(r_j<-1\) is an integer. We will denote the right-hand side above by \([r_0,\ldots,r_k]\). Let 
\[N = \lvert(r_0+2)+(r_1+2)+\cdots +(r_{k-1}+2)+(r_k+1)\rvert.\]
For \(i=0,\ldots,N\), we inductively define rational numbers \(s_i \in \mathbb{Q}\) by first setting \(s_N = -p/q\). Next, for \(1\leq i\leq N\) if \(s_i\) has the continued fraction expansion 
\(s_i = [a_0,\ldots,a_\ell],\) with all \(a_j<-1\) , then we set \(s_{i-1}=[a_0,\ldots,a_\ell +1]\). Note if \(a_\ell = -2\), this is the same as \([a_0,\ldots,a_{\ell-1}+1]\). This process terminates with \(s_0 =-1\). 
\begin{lemma}\label{prop:fractions} Write \(s_i =-p_i/q_i\) in lowest terms with \(p_i,q_i>0\). {\normalfont(}By "lowest terms" we mean \(p_i,q_i\) do not share any common factors.{\normalfont)} In particular, \(p_N =p\) and \(q_N =q\).
\begin{enumerate}[label={\normalfont(\alph*)}]
\item For each \(1\leq i\leq N\), one has \(p_iq_{i-1}-p_{i-1}q_i =1\). Hence
\[\frac{p_i}{q_i}=\frac{p_{i-1}}{q_{i-1}}+\frac{1}{q_{i-1}q_i}.\]
\item We have strict monotonicity \[ -\frac{p}{q}=s_N <s_{N-1}<\cdots < s_1<s_0 =-1.\]
\item For each \(0\leq j\leq k\), define \(x_j,y_j\) recursively by 
\begin{alignat*}{4}
x_0 & = r_0,&\qquad x_1 &=r_0r_1-1,&\qquad x_j = r_jx_{j-1}-x_{j-2},&\quad j&\geq 2,\\
y_0 & = 1,&\qquad y_1 &=r_1,&\qquad y_j = r_jy_{j-1}-y_{j-2},&\quad j&\geq 2.
\end{alignat*}
Then 
\[\frac{x_j}{y_j}=[r_0,\ldots,r_j]\]
and 
\[x_j y_{j-1}-x_{j-1}y_j=-1,\qquad 1\leq j\leq k.\]
Consequently, the fraction \(x_j/y_j\) is in lowest terms.
\end{enumerate}
\end{lemma}
\begin{proof}
We start with (a). If \(s_i=-n\) for some integer \(n\geq 2\) then \(s_{i-1}=1-n\). Hence
\[p_i q_{i-1}-p_{i-1}q_i = n\cdot 1 - (n-1)\cdot 1 =1.\]
We now induct on the length of the continued fraction expansion of \(s_i\). Suppose \(s_i = [a_0,\ldots,a_{\ell}]\) with \(a_j<-1\). By induction, if we write 
\[ -\frac{p'}{q'}=[a_1,\ldots,a_{\ell}],\qquad -\frac{p''}{q''}=[a_1,\ldots,a_{\ell}+1]\]
in lowest terms with \(p',q',p'',q''>0\), then \(p' q'' -p'' q' =1.\) Now 
\[s_i = a_0+\frac{1}{p'/q'}=\frac{a_0p'+q'}{p'},\qquad s_{i-1}=a_0+\frac{1}{p''/q''}=-\frac{a_0p''+q''}{p''}.\]
Note the last fraction in each equation is in lowest terms. Hence 
\[p_i =  -(a_0p'+q'),\qquad q_i = p',\qquad p_{i-1} = -(a_0p''+q''),\qquad q_{i-1} = p''. \]
We now calculate 
\[p_i q_{i-1}-p_{i-1}q_i = -(a_0p'+q')p'' + (a_0p''+q'')p'=q''p'-q'p''=1.\]
The last part of (a) now follows, since
\[\frac{p_{i-1}}{q_{i-1}}+\frac{1}{q_{i-1}q_i}=\frac{p_{i-1}q_i+1}{q_{i-1}q_i}=\frac{p_iq_{i-1}}{q_{i-1}q_i}=\frac{p_i}{q_i}.\]

Note (b) follows from the last equation in (a). For (c),  we first prove by induction that  
\[\frac{x_j}{y_j}=[r_0,\ldots,r_j].\]
Here it will be useful to allow \(r_j\) to be an arbitrary rational number. We first quickly verify 
\begin{align*}\frac{x_0}{y_0}&=\frac{r_0}{1}=[r_0],\\
\frac{x_1}{y_1}&=\frac{r_1r_0-1}{r_1}=r_0-\frac{1}{r_1}=[r_0,r_1],\\
\frac{x_2}{y_2}&=\frac{r_2r_1r_0 - r_2 -r_0}{r_2r_1 -1}=r_0 -\frac{r_2}{r_2r_1-1}=r_0-\frac{1}{r_1-1/r_2}=[r_0,r_1,r_2].\end{align*}
Next, for \(j\geq 3\) we have by induction that
\begin{align*}[r_0,\ldots,r_{j-1},r_j]&=\Big[r_0,\ldots,r_{j-1}-\frac{1}{r_j}\Big]=\frac{(r_{j-1}-1/r_j)x_{j-2}-x_{j-3}}{(r_{j-1}-1/r_j)y_{j-2}-y_{j-3}}=\frac{r_j(r_{j-1}x_{j-2}-x_{j-3})-x_{j-2}}{r_j(r_{j-1}y_{j-2}-y_{j-3})-y_{j-2}}\\&=\frac{r_jx_{j-1}-x_{j-2}}{r_jy_{j-1}-y_{j-2}}=\frac{x_j}{y_j}.\end{align*}
We now prove the second statement in (c) by induction. For the base cases \(j=1,2\) we verify 
\begin{align*}x_1 y_0 - x_0 y_1 &=(r_1r_0-1)\cdot 1 - r_0\cdot r_1 = -1,\\
x_2y_1 -y_1x_2&= (r_2r_1r_0-r_2-r_0)\cdot r_1 - (r_1r_0-1)\cdot (r_2r_1-1) = -1.\end{align*}
For \(j\geq 3\), one has by induction that 
\begin{align*}x_j y_{j-1}-x_{j-1}y_j &= (r_j x_{j-1}-x_{j-2})y_{j-1}-x_{j-1}(r_jy_{j-1}-y_{j-2})\\&= x_{j-1}y_{j-2}-x_{j-2}y_{j-1}=-1.\end{align*}
\end{proof}
In the next lemma, we consider some numerology which will be useful to us later. We continue to use the notation of Lemma \ref{prop:fractions}.
\begin{lemma}
\begin{enumerate}[label={\normalfont(\alph*)}]
\setcounter{enumi}{3}
\item The denominators \(q_0,\ldots,q_N\) are all odd if and only if \(k\) is even and
\[r_1=r_3 = r_5 = \cdots = r_{k-1}=-2.\]
In this situation, \(p=p_N\) has the same parity as
\[r_0 + r_2 + r_4+ \cdots + r_k.\]
\item If the condition in {\normalfont(d)} holds and \(p\) is even, then \(N\) is odd and the parity of \(p_i\) alternates. 
\item Suppose \(r_0\) is odd and \(k\geq 1\). The following conditions are equivalent. 
\begin{itemize}
\item For each \(i\geq \lvert r_0\rvert -1\), one of \(p_i\) or \(q_i\) is even. 
\item \(k\) is odd and 
\[r_2 = r_4 = \cdots = r_{k-1}=-2.\]
\end{itemize}
\item If the condition in {\normalfont(f)} holds, then for \(i\geq \lvert r_0\rvert -1\), the parity of \(p_i\) alternates.
\end{enumerate}

\end{lemma}
\begin{proof}

For (d), the statement is evidently true if \(k=0\), so assume \(k>0\). For each \(\lvert r_0\rvert -1\leq i\leq N\) let \(j(i)\in \{0,\ldots,k-1\}\) and \(t_i<-1\) be so that the continued fraction expansion of \(s_i\) is 
\[s_i = [r_0,\ldots, r_{j(i)},t_i].\]
Note \(t_i> r_{j(i)+1}\) except for the case \(t_N = r_k\). If \(j(i)\geq 1 \) note 
\[q_i = t_i y_{j(i)}-y_{j(i)-1}.\]
If \(j(i) =0\) then \(q_i = t_i\). Now, suppose every \(q_i\) is odd. By considering the indices \(i\) for which \(j(i)=0\), we see that if \(k> 0\), then in fact \(k\geq 2\) and \(r_1=y_1=-2\). Suppose by induction that for some \(\ell\),
\begin{enumerate}[label=(\roman*)]
\item \(k>2\ell-1\),
\item \(r_1 = r_3 = \cdots = r_{2\ell-1}=-2,\)
\item \(y_1,y_3,\ldots,y_{2\ell-1}\) are all even, 
\item \(y_0,y_1,\ldots,y_{2\ell-2}\) are all odd. 
\end{enumerate}
If \(k=2\ell\) then we are finished. If not, arguing by contradiction, suppose that \(a_{2\ell+1}<-2\). Then there is some index \(i\) with \(j(i) = 2\ell\) and \(t_i =-2\). Note \(q_i =-2y_{2\ell}-y_{2\ell -1}\) is even, which is a contradiction. For the same reason, we cannot have \(k=2\ell+1\). To finish the inductive step, note \(y_{2\ell}=a_{2\ell}y_{2\ell-1}-y_{2\ell-2}\) is odd and \(y_{2\ell+1}= -2y_{2\ell}-y_{2\ell-1}\) is even. 

Conversely, suppose \(k\) is even and \(r_1=r_3=\cdots =r_{k-1}=-2\). The same argument above shows \(y_1,y_3,\ldots,y_{k-1}\) are even and \(y_0,y_2,\ldots,y_k\) are odd. By the second statement in (c), note \(x_1,x_3,\ldots,x_{k-1}\) are odd. A similar induction shows \(x_{2\ell}\) has the same parity as \(r_0 + r_2 + \cdots + r_{\ell}\). In particular, this proves the statement about the parity of \(p=x_k\). Now for \(0\leq i< \lvert r_0 \rvert-1\) we have \(q_i =1\) is odd. For \(\lvert r_0\rvert-1 \leq i\leq N\), note \(j(i)\) is always odd, hence \(q_i =t_i y_{j(i)}-y_{j(i)-1}\) is odd. Part (e) follows from part (a) of Lemma \ref{prop:fractions}. 

For part (f), assume \(r_0\) is odd. Note the conditions in (f) are both true if \(k=1\), so we may assume \(k\geq 2\). We keep the notation \(j(i)\) and \(t_i\) as above. Suppose for each \(i\geq \lvert r_0\rvert -1 \), at least one of \(p_i\) or \(q_i\) is even. If it were the case that \(k=2\) or \(r_2<-2\), then there would exist an index \(i\) with \(j(i) =2\) and \(t_i =-2\). Then both
\[p_i =-2x_1 -r_0\qquad\text{and}\qquad q_i = -2y_1 -1 \]
are odd, which is a contradiction. Therefore \(k\geq 3\), \(r_2 = -2\), and consequently \(x_2,y_2\) are both odd. Suppose by induction that for some \(\ell\),
\begin{enumerate}[label=(\roman*)]
\item \(k>2\ell\),
\item \(r_2=r_4=\cdots = r_{2\ell}=-2\),
\item \(x_0,x_2,\ldots,x_{2\ell}\) are all odd, 
\item \(y_0,y_2,\ldots,y_{2\ell}\) are all odd. 
\end{enumerate}
If \(k=2\ell +1\) we are finished. If not, arguing by contradiction, suppose that either \(k=2\ell+2\) or \(a_{2\ell + 2}<-2\). Then there is some index \(i\) with \(j(i) = 2\ell+1\) and \(t_i =-2\). 
Then both
\[p_i =-2x_{2\ell+1} -x_{2\ell}\qquad\text{and}\qquad q_i = -2y_{2\ell+1} -y_{2\ell}\]
are odd, which is a contradiction. Therefore \(k>2\ell +2\), \(r_{2\ell+2}=-2\), and \(x_{2\ell+2},y_{2\ell+2}\) are odd. 

Conversely, suppose \(k\) is odd and \(r_2 =r_4=\cdots=r_{k-1}=-2\). Then necessarily \(x_0,x_2,\ldots,x_{2\ell}\) and \(y_0,y_2,\ldots,y_{2\ell}\) are all odd. By the second statement in (c), if \(i\) is odd, then \(x_i,y_i\) have opposite parity. For \(i\geq \lvert r_0 \rvert-1\), note \(j(i)\) is always even. Hence 
\[p_i =t_ix_{j(i)} -x_{j(i)-1}\qquad\text{and}\qquad q_i = t_iy_{j(i)} -y_{j(i)-1}\]
have opposite parity. In particular one of them is even. Part (g) follows from part (a) of Lemma \ref{prop:fractions}. 
\end{proof}
\begin{remark}
The fractions 
\[ -\frac{p}{q}=s_N <s_{N-1}<\cdots < s_1<s_0 =-1.\]
pick out the shortest counterclockwise path from \(-p/q\) to \(-1\) in the Farey tesselation. 
\end{remark} 
We are now ready to describe the first family of monodromies we are able to show are overtwisted. 
\begin{theorem}\label{thm:main_one}
Suppose \(f\in \mathrm{Mod}(\Sigma_{0,4},\partial \Sigma_{0,4})\) satisfies
\[\pi(f) = \begin{bmatrix}p' & q'\\ p & q\end{bmatrix}\]
with \(p,q,p',q'\geq 0\) and \(p>q\). Additionally, suppose the continued fraction expansion of \(-p/q\) is of the form
\[-\frac{p}{q}=[r_0,-2,r_2,-2,r_4,\ldots, r_{k-2},-2,r_k],\]
where \(r_j<-1\) and \(k\) is even. In addition suppose that there is some \(r_j\) which is not \(-2\). If \(f\) has minimum FDTC \(1\), then \(f\) is overtwisted. 
\end{theorem}

The strategy of the proof of Theorem \ref{thm:main_one} is to exhibit an explicit overtwisted disk via a movie presentation, as described in \cite{ItoKawamuro, ItoKawamuroOvertwisted}. We review the necessary details below. This construction is a modification of the movie presentation in Theorem 4.1 of \cite{ItoKawamuroOvertwisted}.

\begin{proof} We will call the binding component parallel to \(a_i\) by \(B_i\). Let \(n_i\) denote the FDTC of \(f\) at \(B_i\). Up to conjugacy by hyperelliptic involutions, we may assume that \(n_2=1\).  We will now draw the page as in Figure \ref{fig:square_holes}, with square boundary components.  
\begin{figure}[ht]
    \centering
    \includegraphics[scale = 0.2]{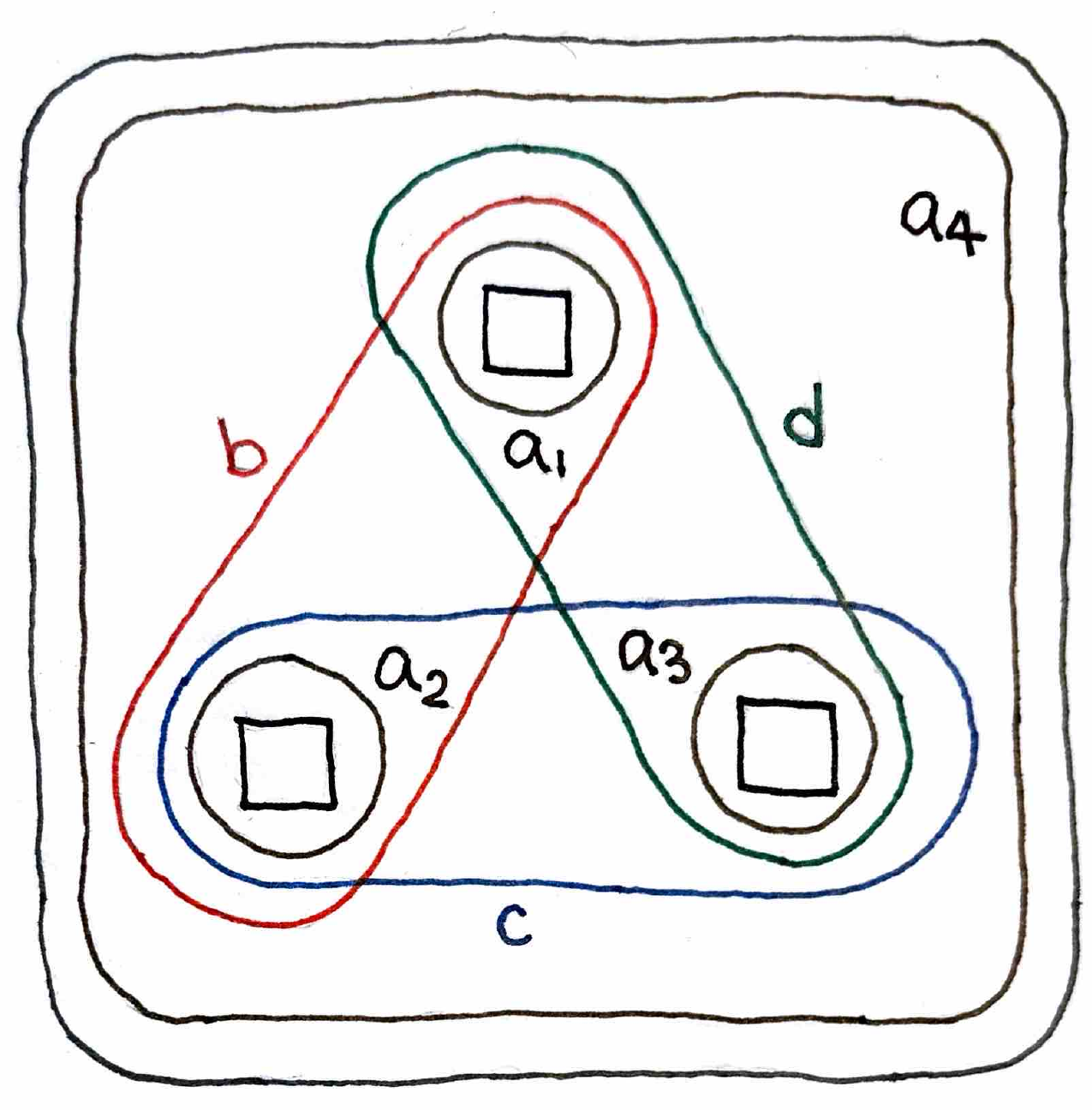}
    \caption{}
    \label{fig:square_holes}
\end{figure}

Denote the pages of the open book by \(\Sigma_{\varphi}\), where \(\varphi \in \mathbb{R}/2\pi\mathbb{Z}\). For \(\varphi \in [0,2\pi)\), identify \(\Sigma_{\varphi}\) with \(\Sigma_0\) by traveling forward along the Reeb vector field from \(\Sigma_0\) to \(\Sigma_{\varphi}\). Figure \ref{fig:page_zero} shows the intersection of the candidate overtwisted disk \(D\) with \(\Sigma_0\). 
\begin{figure}[ht]
\centering
    \includegraphics[scale = 0.2]{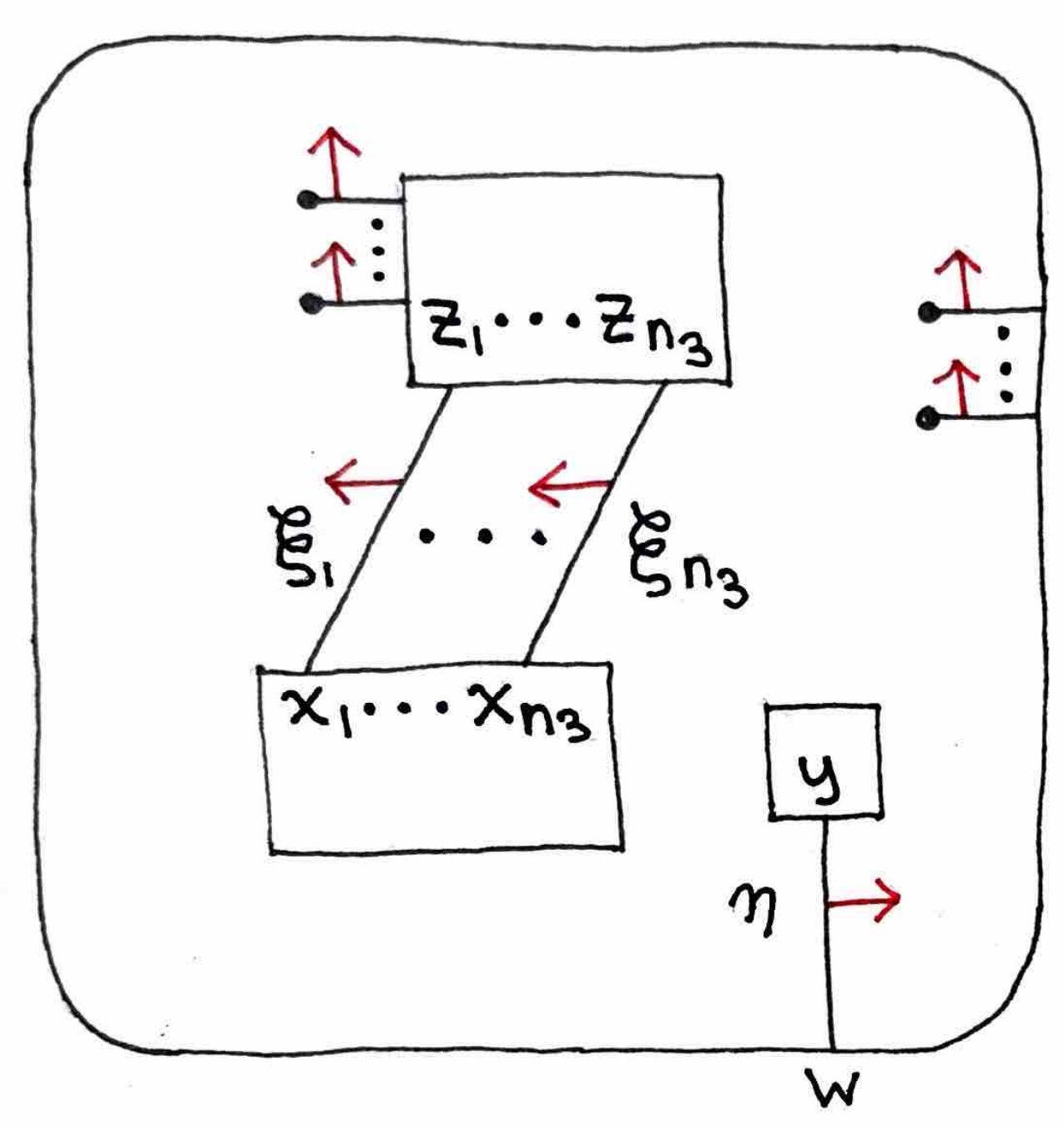}
    \caption{Boundary components have been stretched to fit labels.}
    \label{fig:page_zero}
\end{figure}
This intersection consists of: 
\begin{itemize}
\item properly embedded arcs \(\xi_{1},\ldots,\xi_{n_3}\) from \(B_2\) to \(B_1\) of slope \(0\),
\item a properly embedded arc \(\eta\) from \(B_3\) to \(B_4\) of slope \(0\),
\item several ``hairs'' (or antennae) with one endpoint in the interior of the page and the other endpoint on \(B_1\) or \(B_4\). 
\end{itemize}
The number of hairs needed will become clear during the course of the proof. The red arrows indicate the (co)orientation of \(D\). The endpoints of hairs in the interior of \(\Sigma\) are points on \(\partial D\). We insist that
\begin{itemize}
\item \(\partial D\) intersect all the pages transversely and positively, 
\item  the intersections of \(D\) with the binding \(B=B_1\cup \cdots \cup B_4\) are all transverse, and are locally modeled on Figure \ref{fig:elliptic}.
\end{itemize}
\begin{figure}[ht]
    \centering
    \includegraphics[scale =0.17]{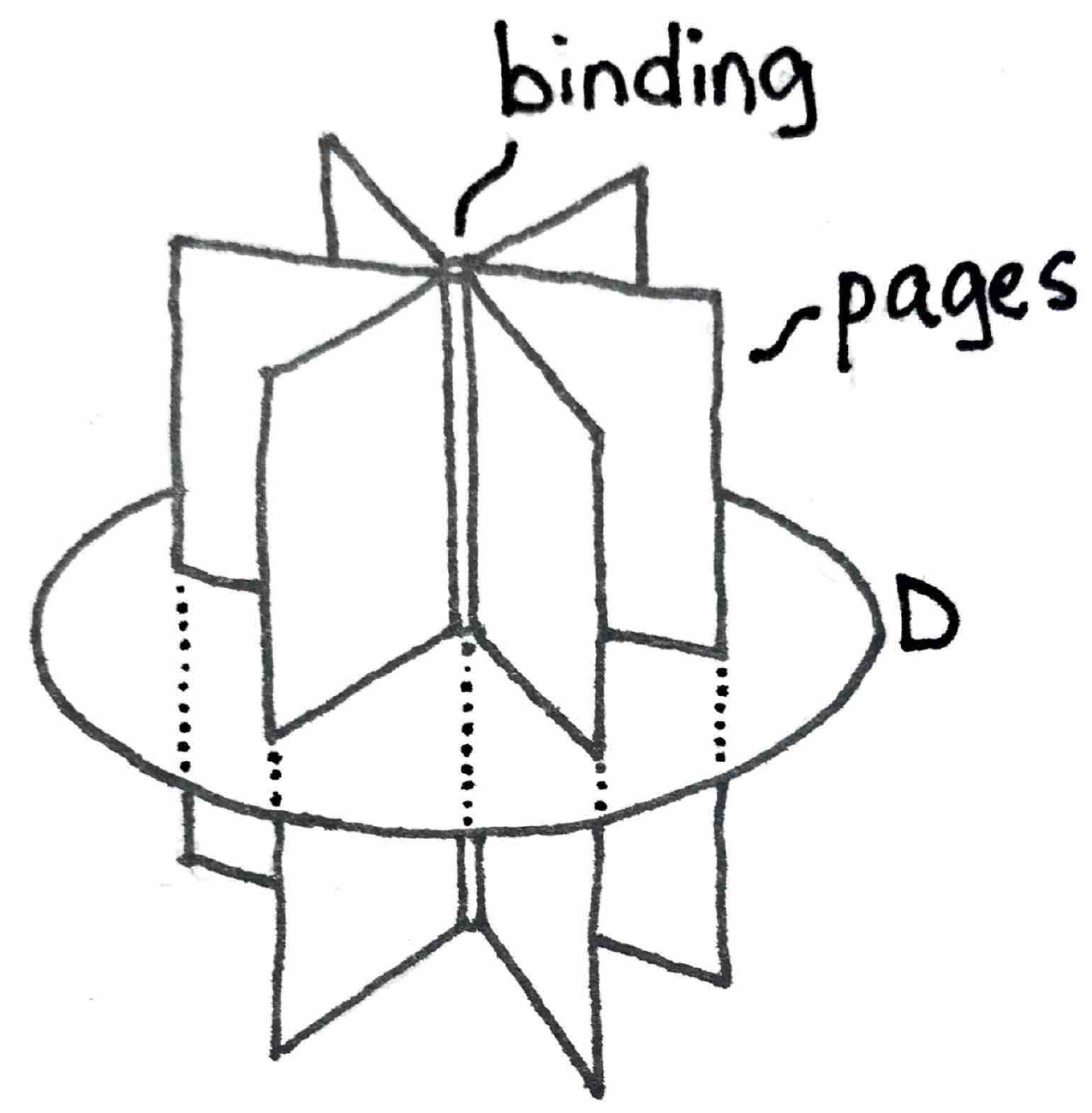}
    \captionof{figure}{A neighborhood of a point in \(D\pitchfork B\)}
    \label{fig:elliptic}
\end{figure}

Since \(B\) and \(D\) are both oriented, we can give each such intersection a sign. Note \(x_{1},\ldots,x_{n_3},y\) have negative sign, whereas \(z_1,\ldots,z_{n_3},w\) have positive sign. The endpoint of each hair on a binding component has positive sign. The open book induces a singular foliation on \(D\), called its \textit{open book foliation} \(\mathcal{F}_{\mathrm{ob}}(D)\). In this foliation, points of \(D\pitchfork B\) correspond to radial (elliptic) singularities.

We now show how the intersection of \(D\) and \(\Sigma_{\varphi}\) changes as we increase \(\varphi\) from \(0\) to \(2\pi\). Between finitely many critical values of \(\varphi\), the picture will change only by an isotopy which is the identity on \(\partial \Sigma\). At these critical values, the page will sweep past a single saddle (hyperbolic) singularity, as in Figure \ref{fig:saddle}.
\begin{figure}[ht]
\centering
\includegraphics[scale =0.2]{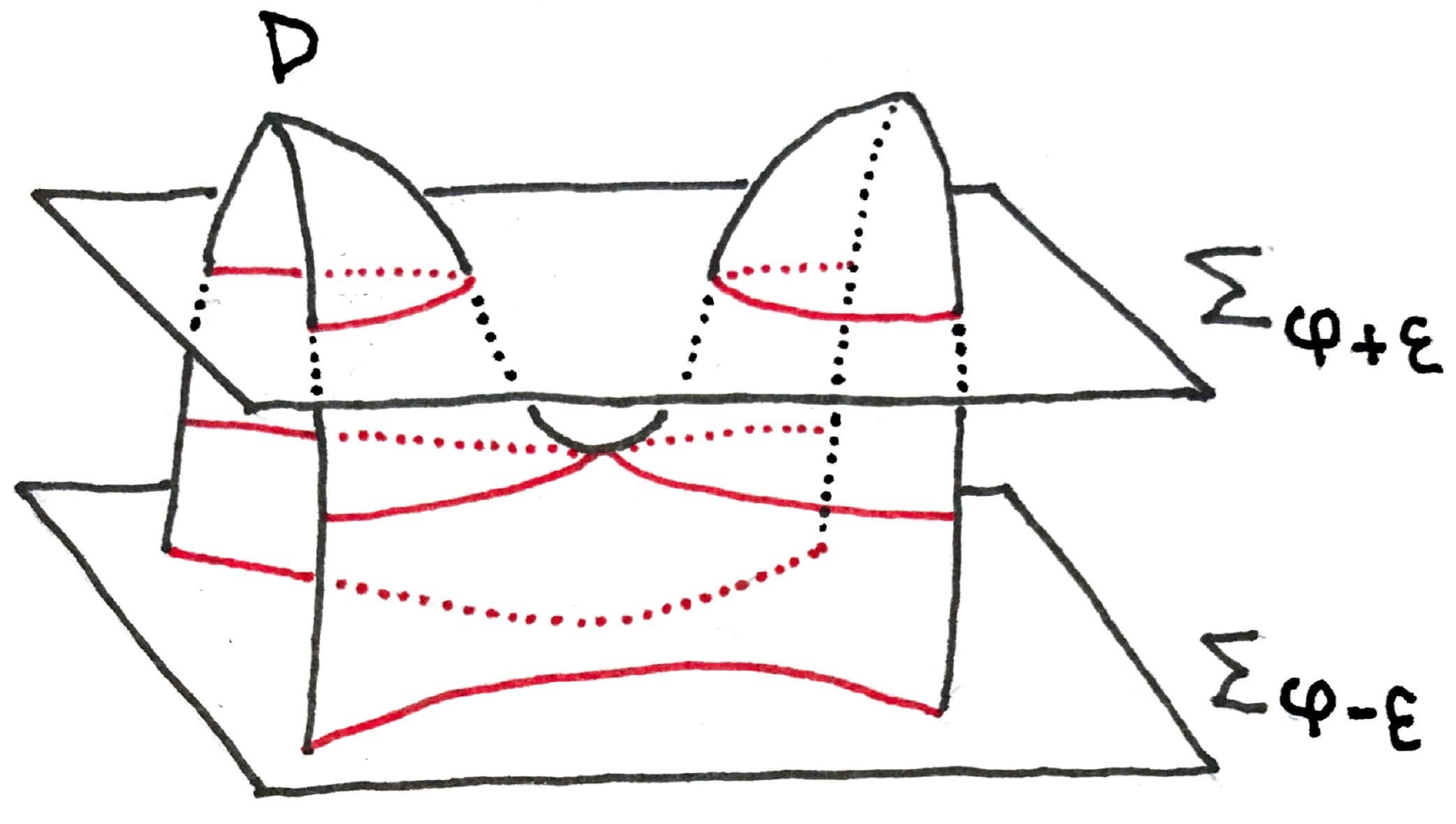}
    \captionof{figure}{A neighborhood of a saddle singularity }
    \label{fig:saddle}
\end{figure}
If \(p\in \Sigma_{\varphi}\cap D\) is such a singularity, then the sign of \(p\) is defined to be positive if the orientations on \(T_pD=T_p\Sigma\) induced by \(D\) and \(\Sigma\) agree, and is defined to be negative otherwise. The formation of such a singularity will be denoted by a dashed line between arcs, as indicated in Figure \ref{fig:dashed}.
\begin{figure}[ht]
    \centering
    \includegraphics[scale =0.2]{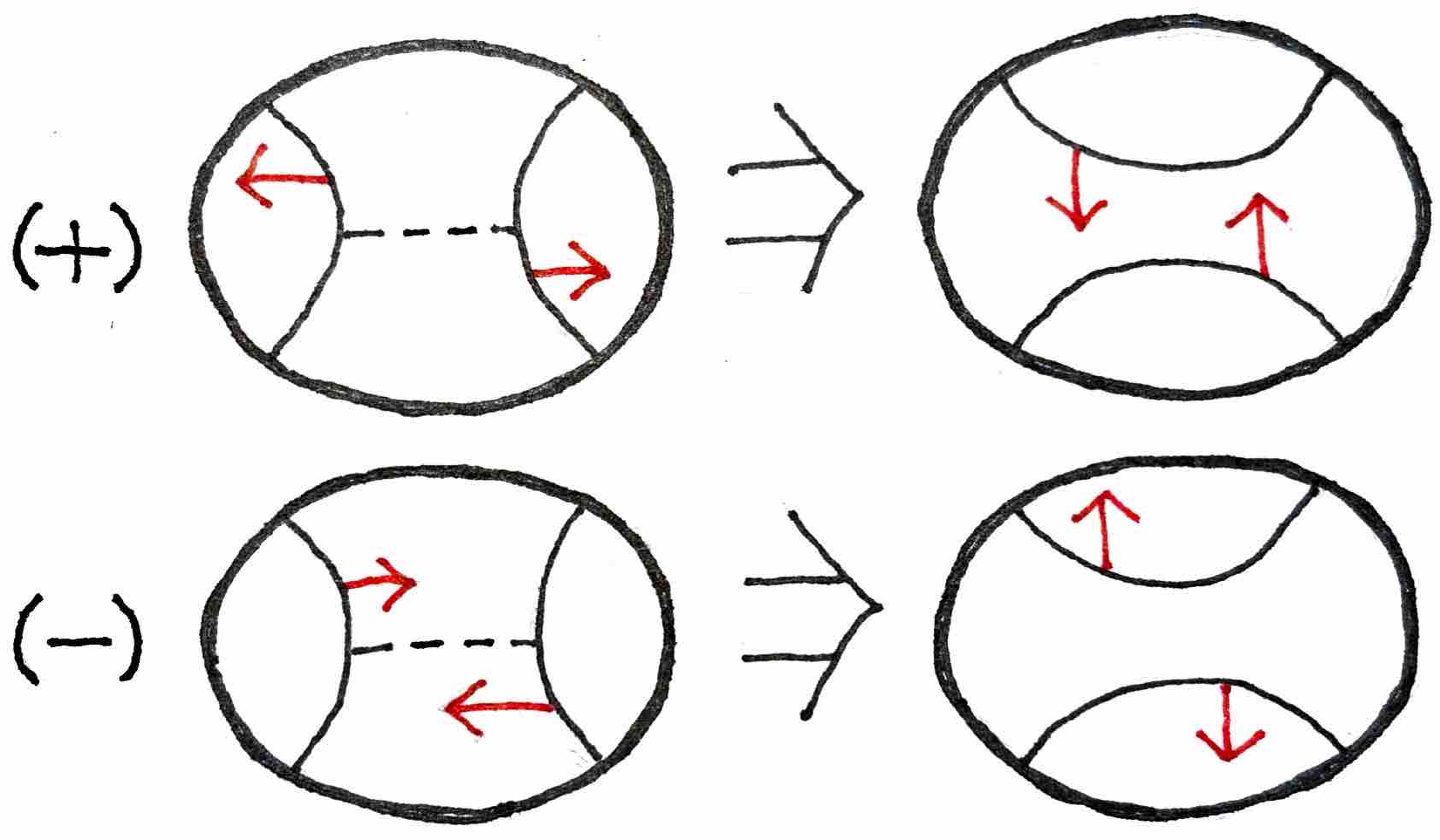}\qquad\qquad
    \caption{}
    \label{fig:dashed}
\end{figure}
Let us say an embedded disk \(D\) is in \textit{general position} if its boundary is a positive transverse knot, and the intersections of \(D\) with the binding and pages are as specified above. 

Away from singular values of \(\varphi\) there will always be \(n_3+1\) properly embedded arcs on \(\Sigma_{\varphi}\), each with an endpoint at one of the negative elliptic singularities \(x_{n_1},\ldots,x_{n_3},y\). \textit{We continue to refer to these arcs as \(\xi_1,\ldots,\xi_{n_3},\eta\).} We emphasize that these arcs will always be labeled according to their \textit{negative} endpoints. At various steps, positive hyperbolic singularities are formed between properly embedded arcs and hairs. Such a move always shifts the properly embedded arc (keeping the negative endpoint fixed) to the right, in the sense of \cite{HKMRVI}. The number of hairs should be chosen so that each hair is used exactly once.

We keep the notation 
\[-\frac{p}{q}=s_N<s_{N-1}<\cdots <s_1<s_0 =-1.\]
with \(s_i = p_i/q_i\) from Lemma \ref{prop:fractions}. Our current assumptions necessitate \(N\geq 3\). The name of the game is to shift the arcs on \(\Sigma_0\) to their image under \(f^{-1}\) using hyperbolic singularities. These fractions describe the slopes that the properly embedded arcs take throughout the process. As consecutive slopes form an integral basis of \(\mathbb{Z}^2\), these arcs do not intersect each other as they are shifted around. 

Since the arcs \(\Sigma_{2\pi-\epsilon}\) will be arranged to match up with the arcs on \(\Sigma_{0}\), the movie presentation below specifies an embedded surface whose boundary is traced out by the endpoints of hairs lying in the interior of the pages. The topological type of the surface is determined entirely by the movie presentation. Indeed, one can picture ``building up'' the surface from the arcs on \(\Sigma_{\varphi}\). After the movie presentation, we will describe how to see that this surface is indeed a disk.

We will describe the general case in text, and illustrate the case 
\[-\frac{p}{q}=-\frac{10}{7}=[-2,-2,-4]\]
with \(n_2 = 1\) and \(n_1=n_3=n_4=2\). Note in this case 
\[s_0 = -1,\quad s_1 = - \frac{4}{3},\quad s_2 = -\frac{7}{5},\qquad s_3 = -\frac{10}{7}.\]
The movie proceeds in two major stages. 
\subsection*{Stage 1}
We form a negative hyperbolic singularity using between \(\xi_1\) and \(\eta\). This is done in such a way that the resulting arcs both have slope \(-1\). See Figure \ref{fig:movie_one}.
\begin{figure}[ht]
    \centering
    \includegraphics[scale =0.35]{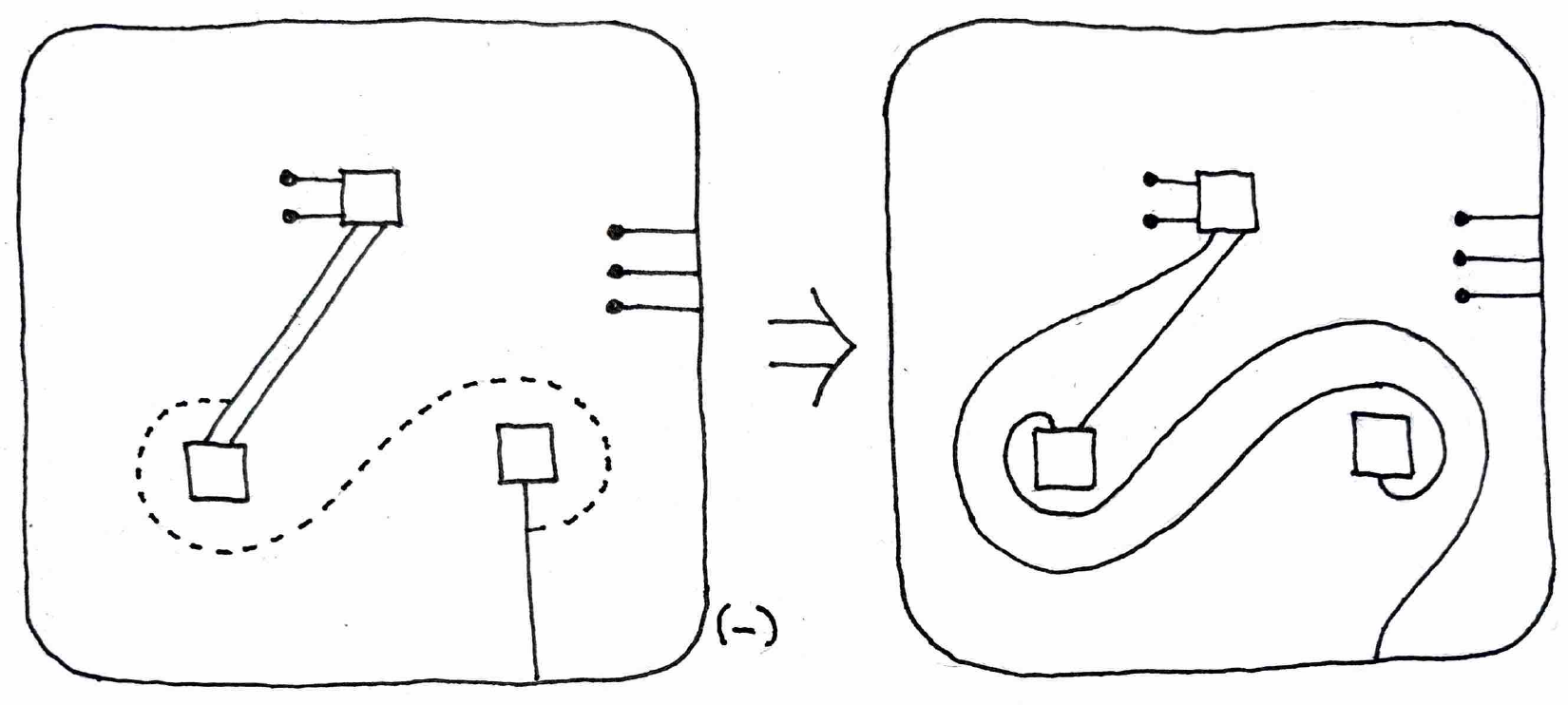}
    \caption{}
    \label{fig:movie_one}
\end{figure}

If \(n_3\geq 2\), use positive hyperbolic singularities with hairs to shift the positive endpoints of \(\xi_{n_3},\ldots,\xi_{2}\) (in this order) to \(B_4\) This is done in such a way that all properly embedded arcs now have slope \(-1\). (If \(s_1=-2\), one can skip this step and directly realize the negative hyperbolic singularities in the next step. However, this would change some subsequent steps, so let us take the current step even if \(s_1 = -2\).) See Figure \ref{fig:movie_two}.

\begin{figure}[ht]
    \centering
    \includegraphics[scale =0.35]{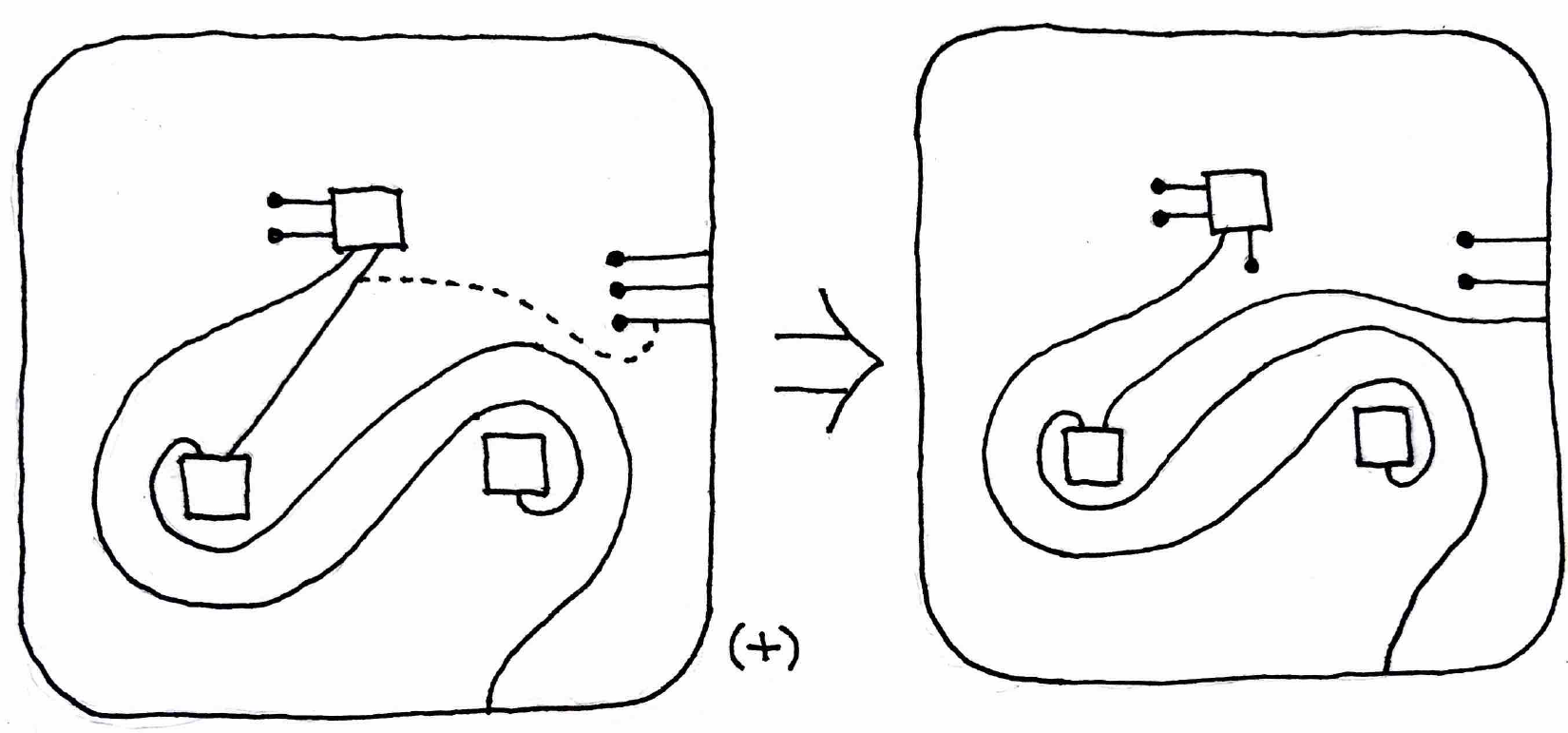}
    \caption{}
    \label{fig:movie_two}
\end{figure}
If \(n_3\geq 2\), form negative hyperbolic singularities between \(\eta\) and \(\xi_2,\ldots,\xi_{n_3}\), in this order. This is done in such a way that all properly embedded arcs, except for \(\xi_1\), have slope \(s_1\). Each additional negative singularity formed with \(\eta\) makes \(\eta\) twist to the left around \(B_3\). If \(n_3 =1\), use a positive hyperbolic singularity with a hair to shift \(\eta\) to have slope \(s_1\). Since \(p_1\) is even and \(q_1\) is odd, in all cases the positive endpoint of \(\eta\) is now on \(B_4\). See Figure \ref{fig:movie_three}.
\begin{figure}[ht]
    \centering
    \includegraphics[scale =0.35]{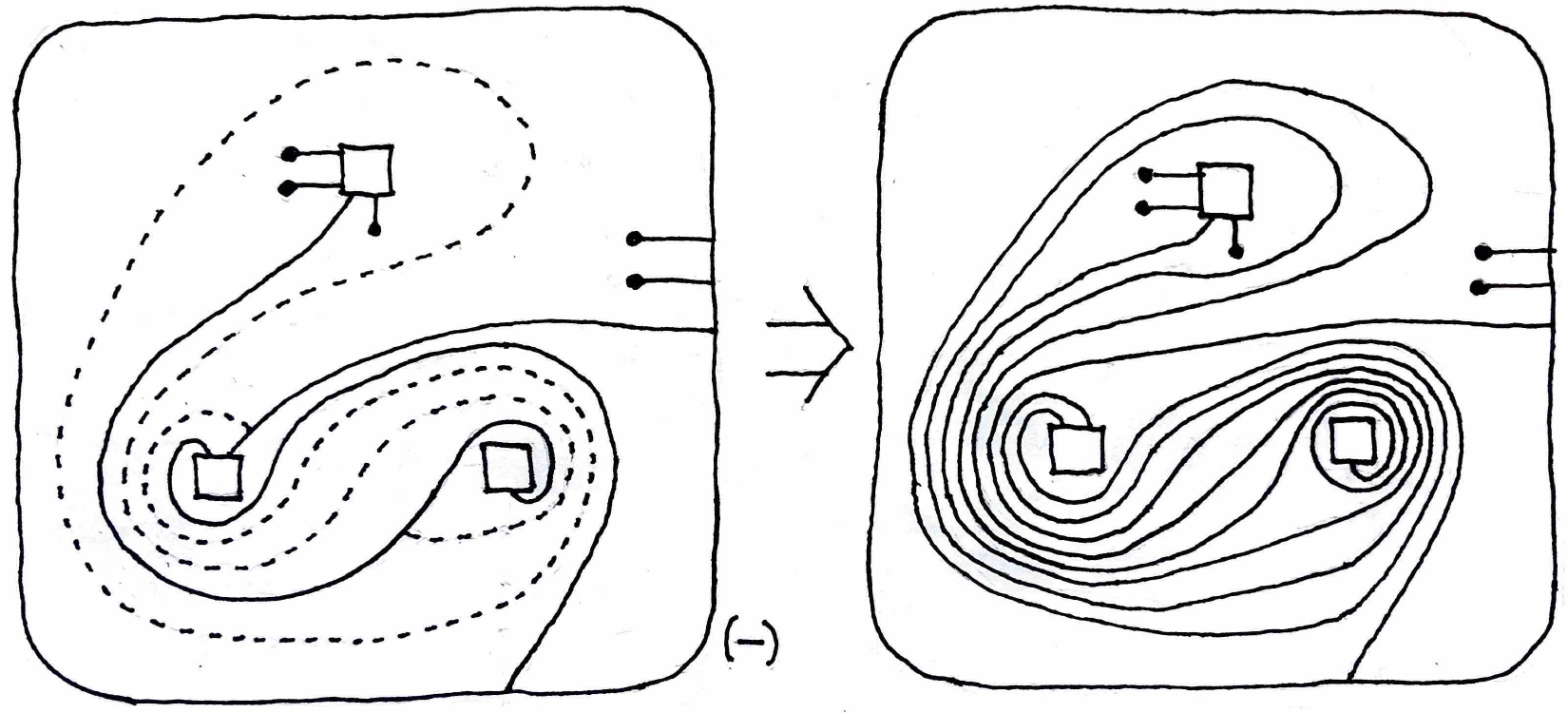}
    \caption{}
    \label{fig:movie_three}
\end{figure}

To conclude Stage 1, use a positive hyperbolic singularity with a hair to shift \(\xi_1\) to have slope \(s_1\). The positive endpoints of \(\xi_1,\ldots,\xi_{n_3}\) are all now on \(B_1\). See Figure \ref{fig:movie_four}.
\begin{figure}[ht]
    \centering
    \includegraphics[scale =0.35]{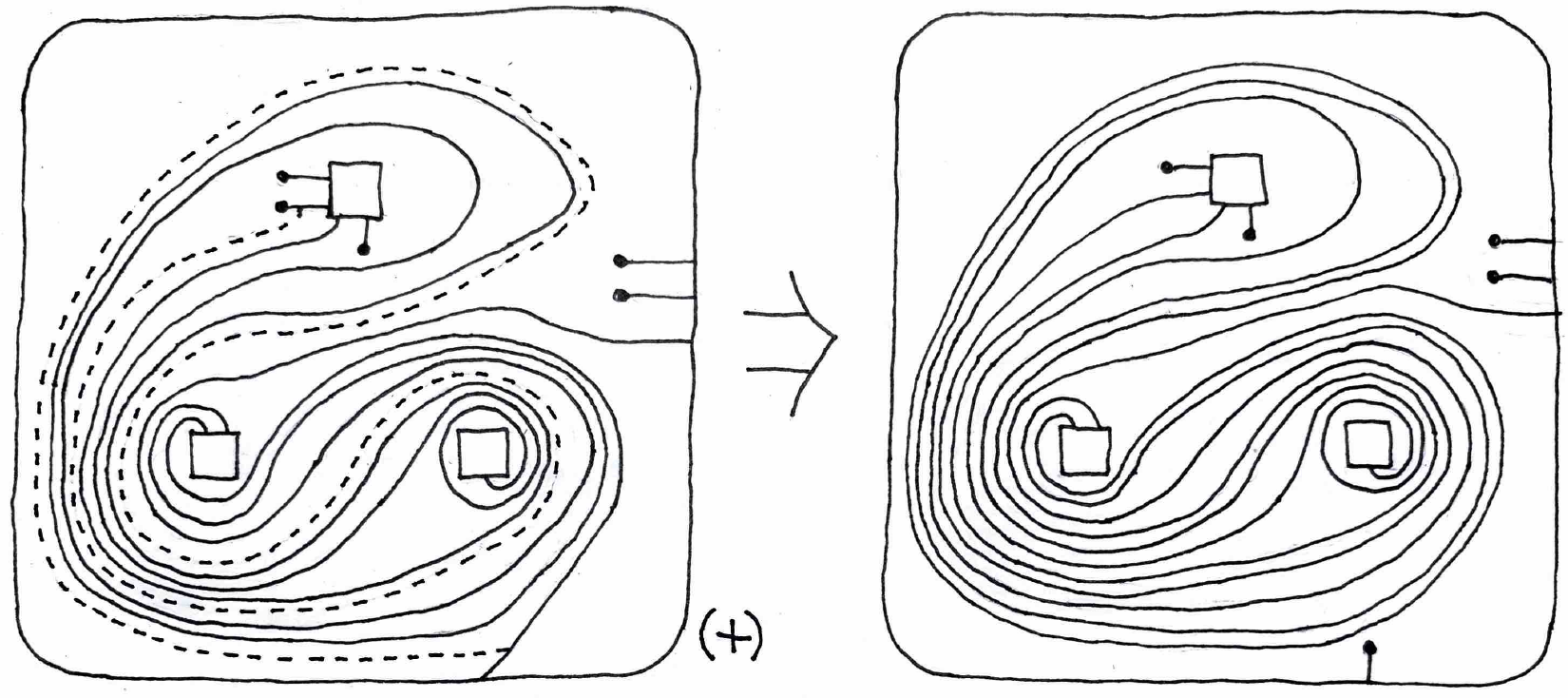}
    \caption{A taffy machine?}
    \label{fig:movie_four}
\end{figure}
We pause to take a look at what the open book foliation on \(D\) looks like at the end of stage 1 in our specific example. See Figure  \ref{fig:foliation_one}.
\begin{figure}[ht]
    \centering
    \includegraphics[scale =1]{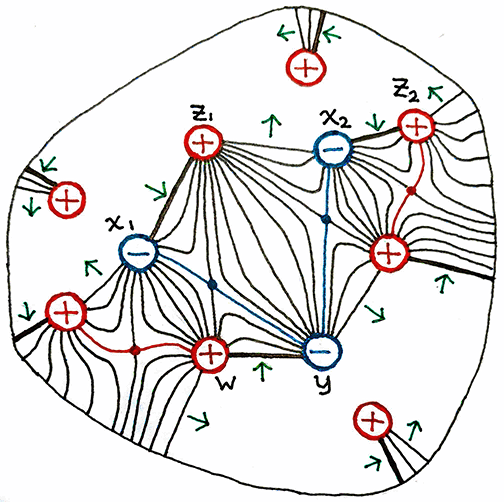}
    \caption{ Elliptic singularities are shown as hollowed circles. Hyperbolic singularities are shown as dots. Singularities are colored red or blue according to whether they are positive or negative, respectively. Arcs which are on \(\Sigma_0\) are bolded. Green arrows indicate the direction of increasing \(\varphi\).}
    \label{fig:foliation_one}
\end{figure}
\subsection*{Stage 2}
In this stage, we repeatedly use positive singularities with hairs to shift the arcs \[\xi_{n_3},\xi_{n_3-1},\ldots,\xi_2,\xi_1,\eta\] (in this order) first all to slope \(s_2\), then all to slope \(s_3\), and so on until all arcs have reached slope \(s_N =-p/q\). As each \(q_i\) is odd, the positive endpoint of a properly embedded arc is always shifted to \(B_1\) or \(B_4\). This prevents the situation where \(\eta\) and some \(\xi_{j}\) run parallel between \(B_2\) and \(B_3\), blocking each other from moving. (One can in fact be slightly greedier in the following manner: one can first shift the \(\xi_j\) to slope \(s_2\), then \(\eta\) to slope \(s_3\), then the \(\xi_j\) to slope \(s_4\), and so on. However, one must slow down at the final few steps as explained presently.)

Before the last few steps, the \(\xi_j\) have realized slope \(s_{N-1}\) and \(\eta\) has realized slope \(s_{N-2}\). Since \(p_{N-2}\) is even and \(p_{N-1}\) is odd,  all the positive endpoints of properly embedded arcs are on \(B_4\). When we move \(\eta\) to slope \(s_{N-1}\), we are therefore free to have \(\eta\) twist left around \(B_1\) however many times we like. This should be done in a manner so that when each \(\xi_j\) is shifted to slope \(s_N\) with positive endpoint \(z_j\), we can realize the FDTC \(n_1\) at \(B_1\). In the last step, we shift \(\eta\) to realize slope \(s_N\) with positive endpoint \(w\), realizing the FDTC \(n_4\) at \(B_4\). We show this process for the present case in Figure \ref{fig:stage_two}
\begin{center}
    \captionsetup{type=figure}
    \begin{longtblr}{
    colspec = {X[c,h]X[l]},
    }
    \includegraphics[scale = 0.3]{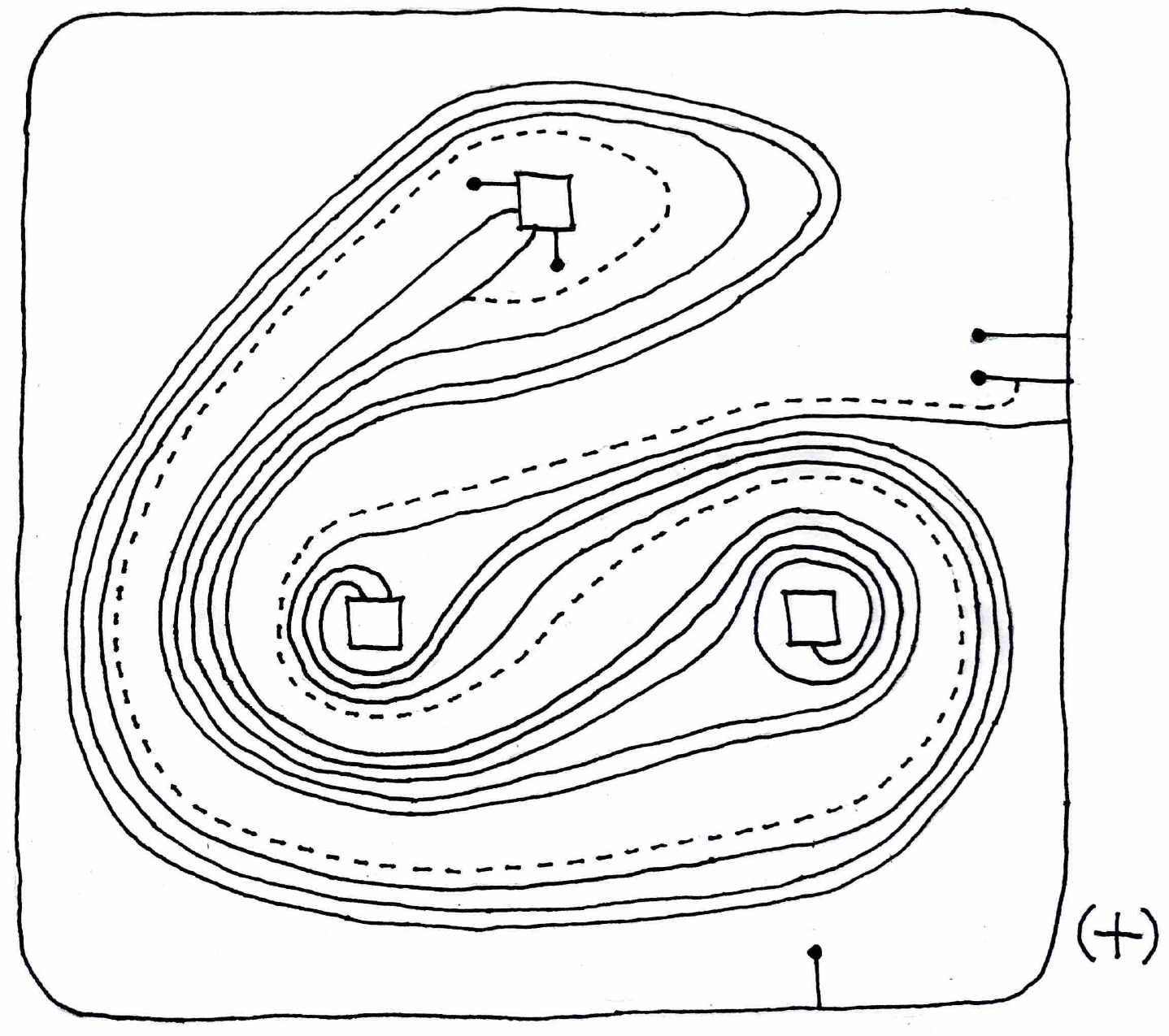} & Shift \(\xi_2\) to have slope \(-7/5\). \\
    \includegraphics[scale = 0.3]{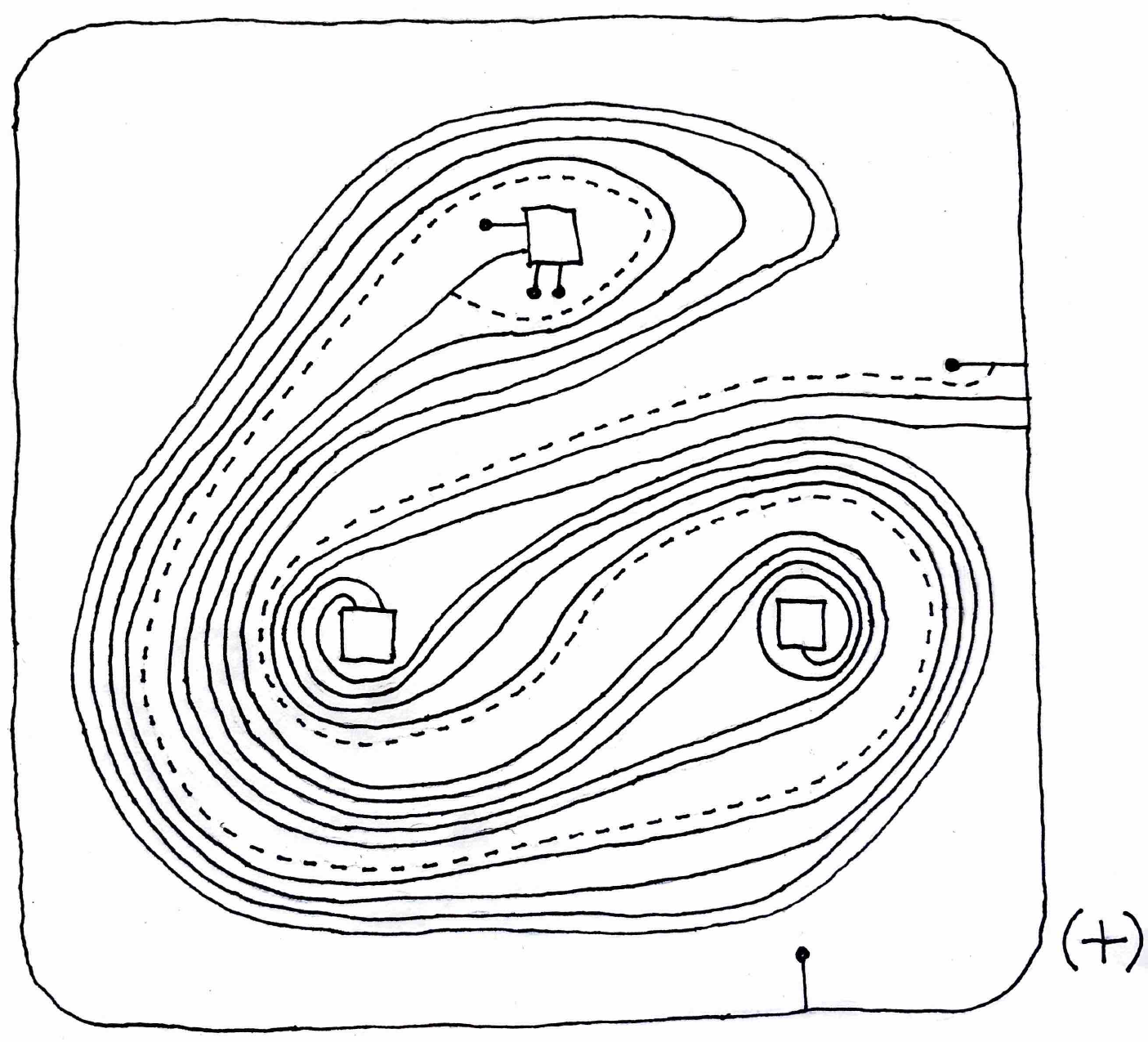} & Shift \(\xi_1\) to have slope \(-7/5\).\\
    \includegraphics[scale = 0.3]{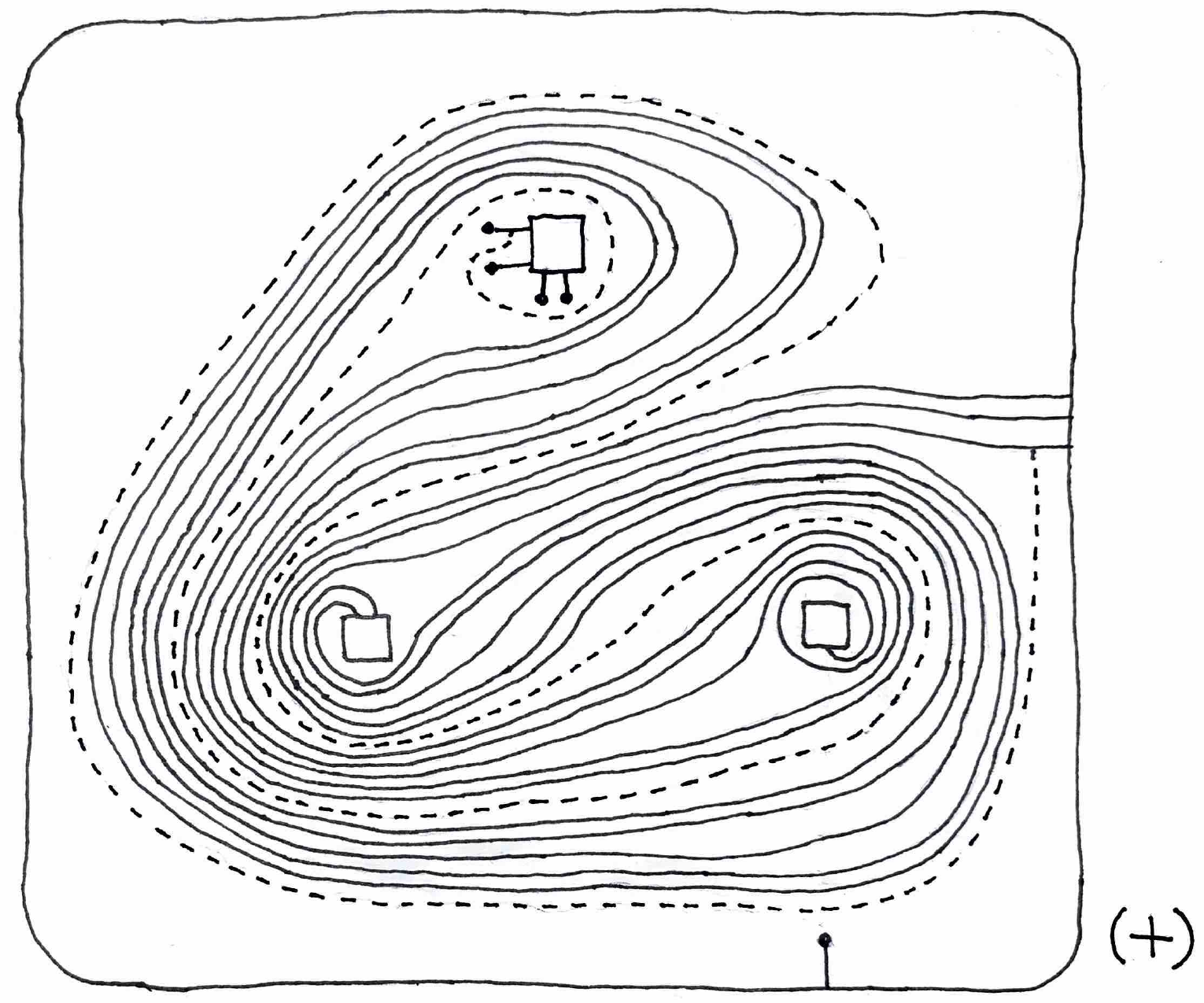} & Shift \(\eta\) to have slope \(-7/5\), twisting left around \(B_1\).\\
    \includegraphics[scale = 0.3]{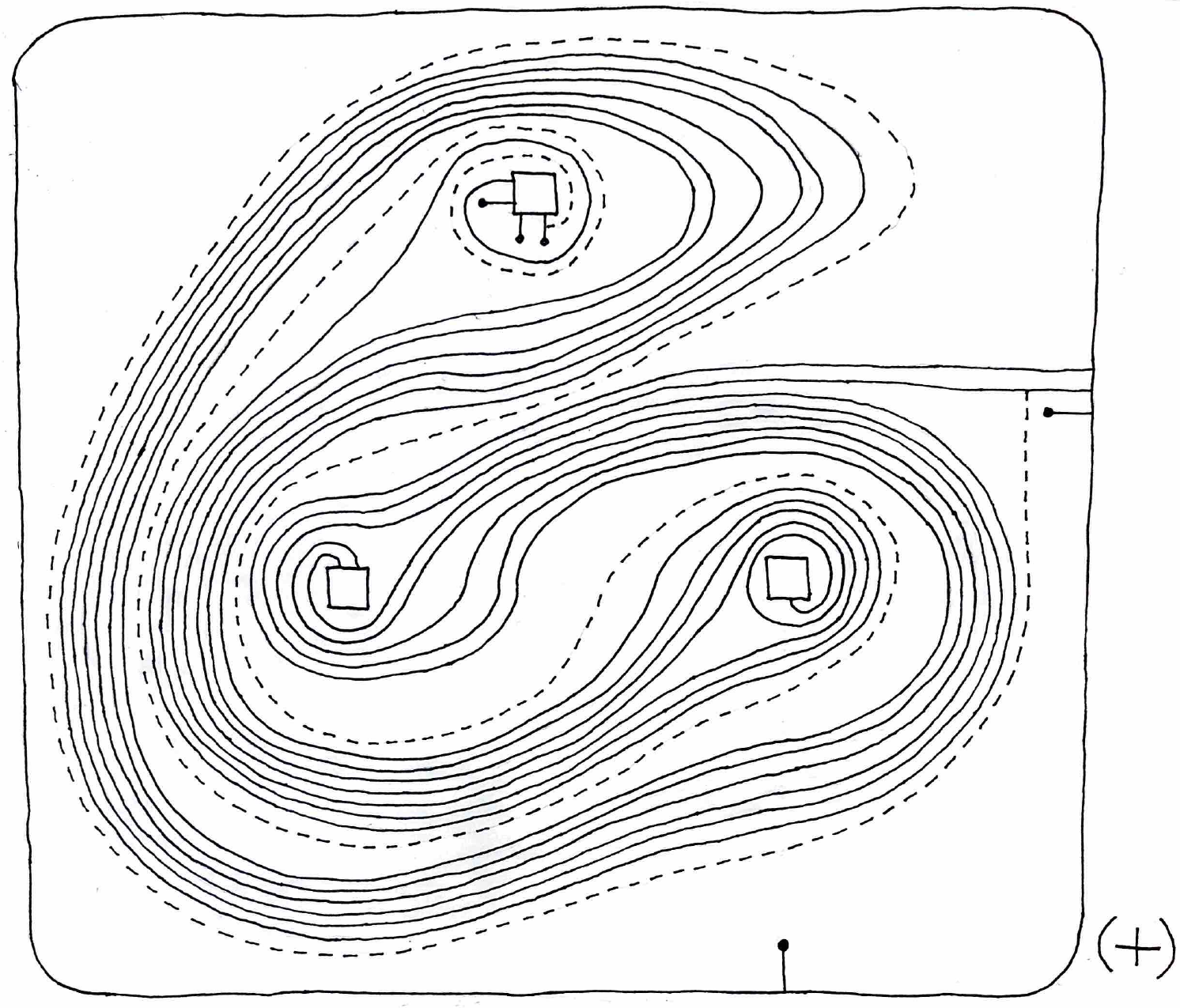} & Shift \(\xi_2\) to have slope \(-10/7\) and positive endpoint \(z_2\), realizing FDTC \(n_1=2\) at \(B_1\). \\
    \includegraphics[scale = 0.3]{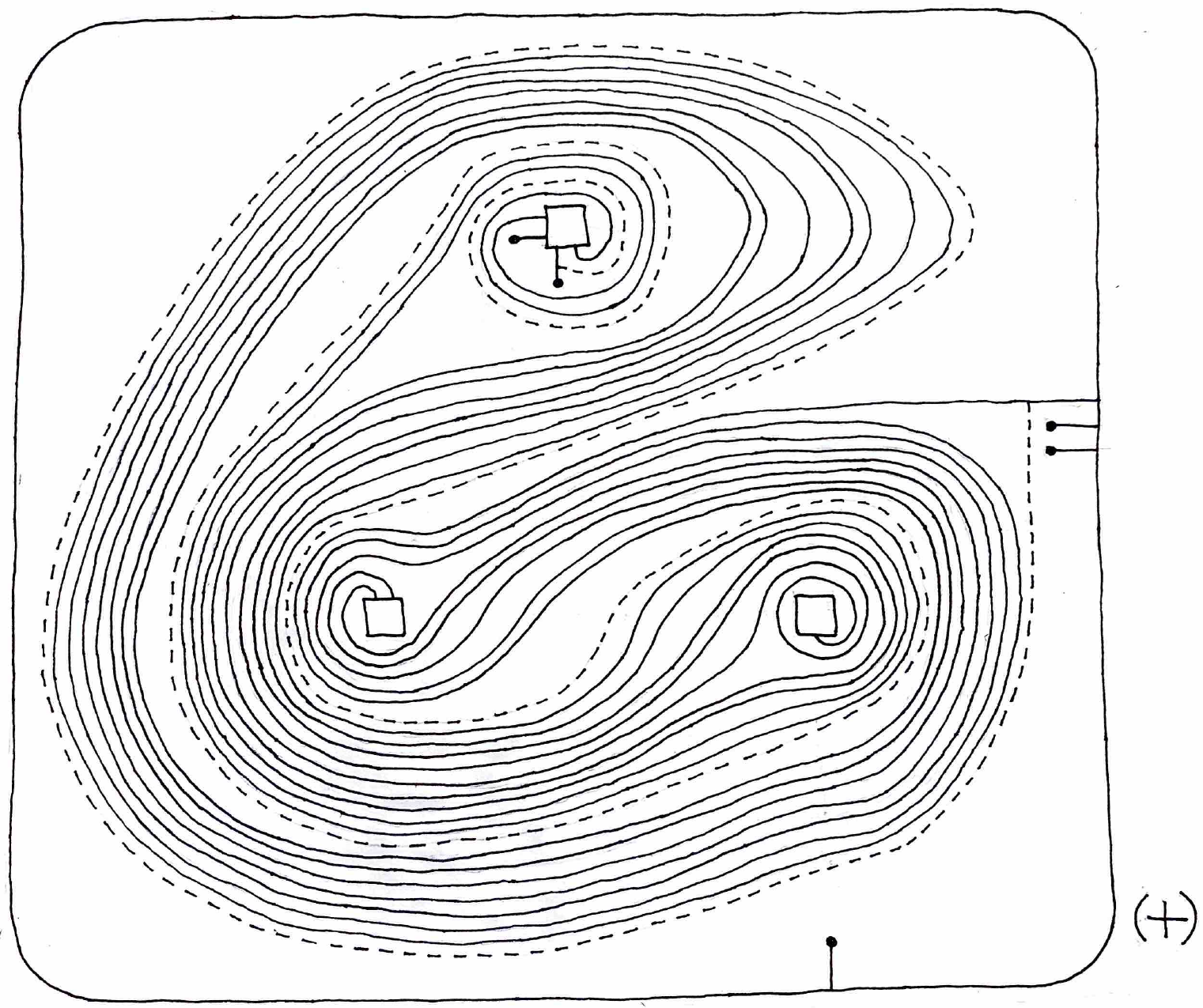} & Shift \(\xi_1\) to have slope \(-10/7\) and positive endpoint \(z_1\), realizing FDTC \(n_1=2\) at \(B_1\). \\
    \includegraphics[scale = 0.3]{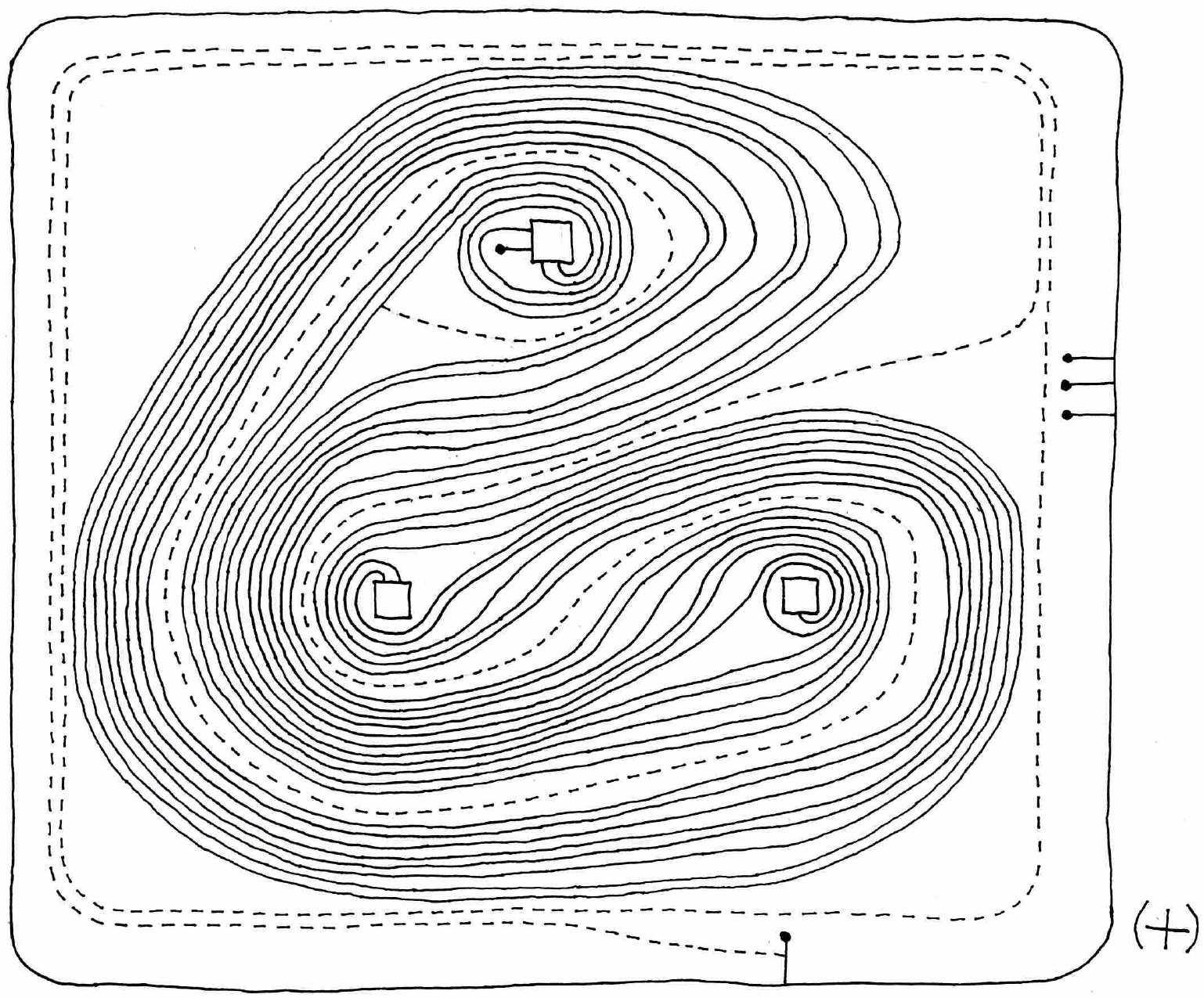} & Shift \(\eta\) to have slope \(-10/7\) and positive endpoint \(w\), realizing FDTC \(n_4=2\) at \(B_4\). \\
    \includegraphics[scale = 0.3]{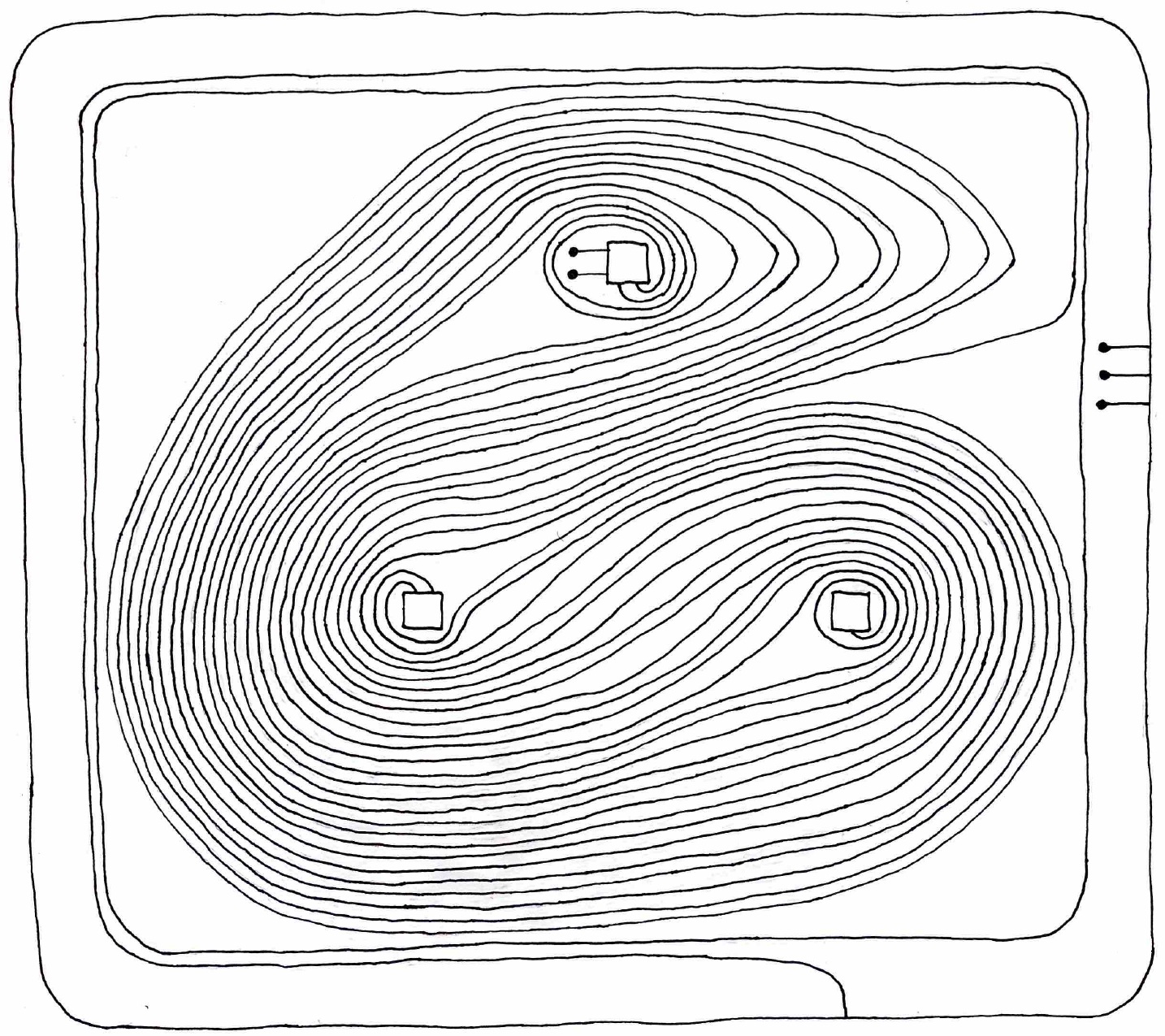} & The page \(\Sigma_{2\pi -\epsilon}\).
    \end{longtblr}
    \captionof{figure}{}
    \label{fig:stage_two}
\end{center}

Note applying the monodromy \(f\) to the last picture in Figure \ref{fig:stage_two}, one obtains Figure \ref{fig:page_zero}. Thus the arcs on \(\Sigma_{2\pi-\varepsilon}\) match up with those on \(\Sigma_0\). By tracing through the above movie, one finds that the arcs  trace out a disk \(D\) whose open book foliation is as in Figure \ref{fig:foliation_two}. 
\begin{figure}[ht]
    \centering
    \includegraphics[scale = 1]{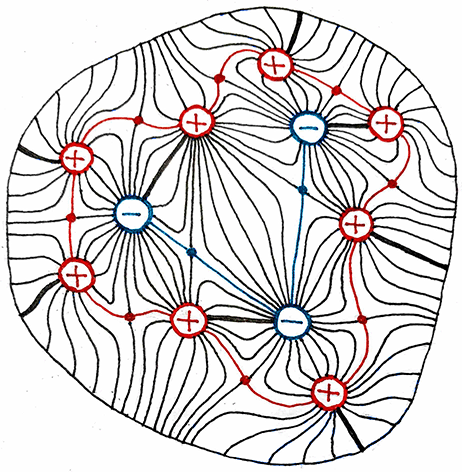}
    \caption{The same conventions as in Figure \ref{fig:foliation_one} are used. The \(G_{--}\) graph is shown in blue, and the \(G_{++}\) graph is shown in red.}
    \label{fig:foliation_two}
\end{figure}

As with characteristic foliations, the open book foliation on \(D\) is oriented via the orientation of \(D\) and the co-orientation of the pages. The convention is so that an oriented tangent vector to a leaf followed by the direction of increasing \(\varphi\) forms an oriented frame for \(D\) at non-singular points. 

One way to confirm that the above movie presentation always yields a disk instead of a higher genus surface is to note the following properties: 
\begin{itemize}
\item The open book foliation on the surface is outward pointing at the boundary.
\item The number of elliptic singularities is one more than the number of hyperbolic singularities.
\end{itemize}
The above two properties ensure the surface has Euler characteristic \(1\), so must be a disk. 

Let us call a hyperbolic singularity \textit{non-cyclic} if it has no separatrices which exit and then return to the same singularity. We recall some definitions from \cite{ItoKawamuro}. The \textit{\(G_{--}\) graph} of \(\mathcal{F}_{\mathrm{ob}}(D)\) consists of the union of unstable separatrices of negative, non-cyclic hyperbolic singularities. This is thought of as a graph whose vertices are negative elliptic singularities and ends of separatrices on \(\partial D\). (The negative hyperbolic singularities are thought of as part of the edges.) The vertices on \(\partial D\) are called \textit{fake vertices}. Analogously the \textit{\(G_{++}\) graph} is the union of stable separatrices of positive, non-cyclic hyperbolic singularities. These graphs are highlighted in blue and red in Figure \ref{fig:foliation_two}.

Finally, a \textit{transverse overtwisted disk} is a disk \(D\) in general position whose open book foliation \(\mathcal{F}_{\mathrm{ob}}(D)\) satisfies the following properties: 
\begin{itemize}
\item \(G_{--}\) is a connected tree with no fake vertices.
\item \(G_{++}\) is homeomorphic to \(S^1\) (and necessarily encircles \(G_{--}\)).
\item \(\mathcal{F}_{\mathrm{ob}}(D)\) contains no closed regular leaves. (This rules out the existence of cyclic hyperbolic singularities.)
\end{itemize}

The above construction always yields a transverse overtwisted disk whose \(G_{--}\) graph is a tree consisting of one central vertex \(y\) connected to \(n_3\) leaves (in the graph theory sense) \(x_1,\ldots,x_{n_3}\). The features of the movie presentation which ensure this are the following:
\begin{itemize}
\item The arc \(\eta\) has exactly one negative singularity with each of the arcs \(\xi_1,\ldots,\xi_{n_3}\). These are all of the singularities formed between properly embedded arcs. 
\item Each hair has exactly one positive singularity with an embedded arc and no other singularities. 
\end{itemize}

In \cite{ItoKawamuro}, it is shown that the existence of a transverse overtwisted disk in an open book implies the supported contact structure is overtwisted. Briefly, the argument is as follows. The fact that \(\mathcal{F}_{\mathrm{ob}}(D)\) is Morse--Smale implies there is a contact structure \(\xi\) compatible with the open book so that the induced characteristic foliation \(\mathcal{F}_{\xi}(D)\) has the following properties:
\begin{itemize}
\item As oriented singular foliations, \(\mathcal{F}_{\xi}(D)\) and \(\mathcal{F}_{\mathrm{ob}}(D)\) are topologically conjugate. 
\item Every hyperbolic singularity of \(\mathcal{F}_{\xi}(D)\) has the same sign as the corresponding singularity in \(\mathcal{F}_{\mathrm{ob}}(D)\). 
\end{itemize}
Applying the Giroux elimination lemma to cancel elliptic/hyperbolic pairs of the same sign then yields an overtwisted disk.
\end{proof}
\begin{remark}
The above construction does not quite fall under the definition of \textit{twist-left-veering arc system} introduced in \cite{twistleft}. In the setup of the above proof, if \(n_3=1\) then the properly embedded arcs \(\Gamma\) on the page \(\Sigma_{2\pi-\varepsilon}\) do indeed form a 2-twist left-veering arc system. If \(n_3\geq 2\), then \(\Gamma\) almost forms an \((n_3+1)\)-twist left-veering arc system, save for the fact that the boundary based region \(R(\Gamma,f(\Gamma))\) may be merely immersed instead of topologically embedded. In addition \(f(\Gamma)\) intersects this region, which is forbidden.

The appearance of \(R(\Gamma,f(\Gamma))\) depends only on \(n_3\) and the slope \(s_1\). It is a \((2n_3+2)\)-gon. If \(n_3=1\) it is always the same rectangle (regardless of \(s_1\)). In Figure \ref{fig:boundary_region} we show the appearance of \(R(\Gamma,f(\Gamma))\) in the case \(n_3 =1\), and also in the case \(n_3=2\), \(s_1 = -4/3\).
\begin{figure}[ht]
    \centering 
    \includegraphics[scale = 0.35]{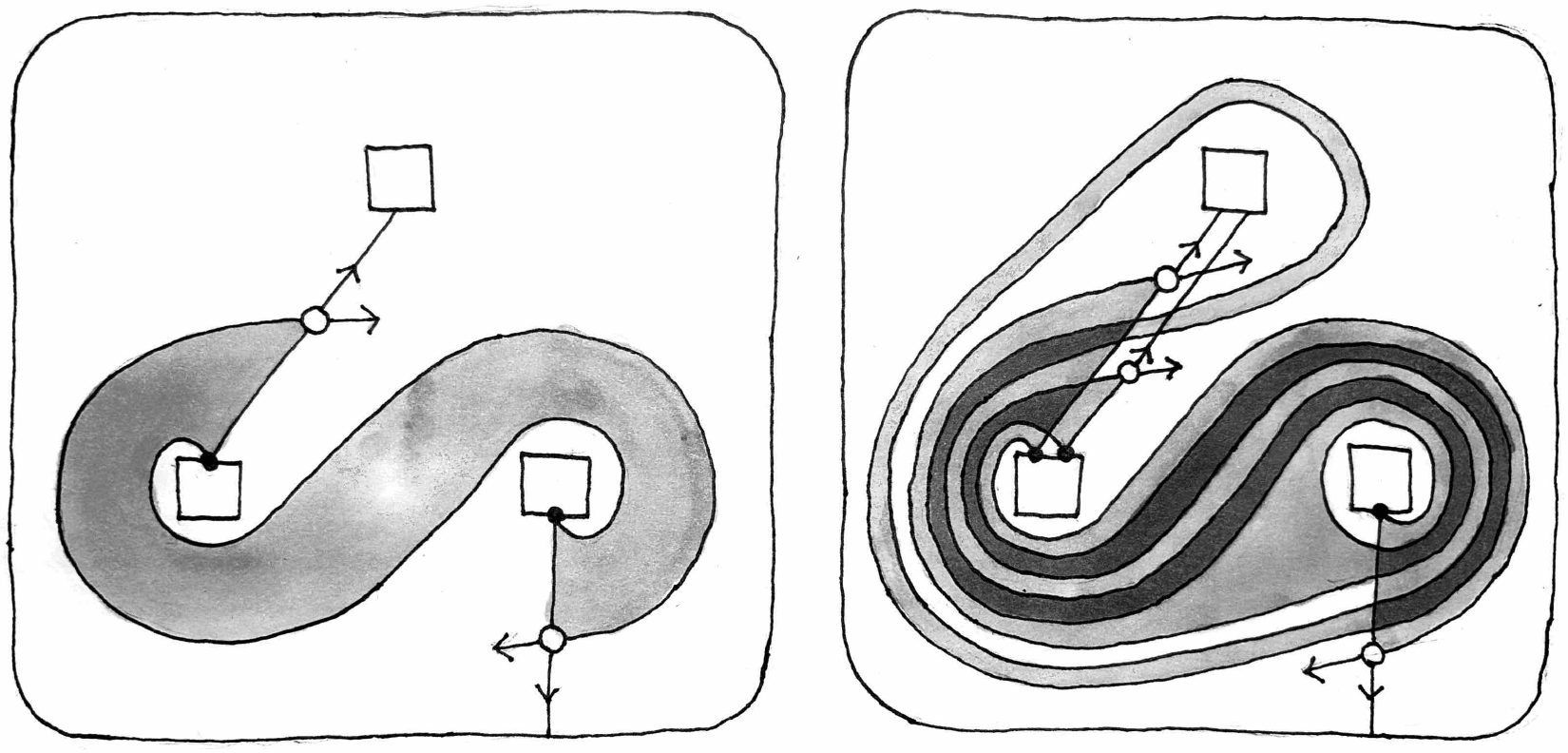}
    \caption{We show the boundary based region \(R(\Gamma,f(\Gamma))\) in two cases. On the left is the case \(n_3 =1\), where \(R(\Gamma,f(\Gamma))\) is an embedded rectangle. On the right is the case \(n_3=2\), \(s_1 = -4/3\), where \(R(\Gamma,f(\Gamma))\) is an immersed hexagon.}
    \label{fig:boundary_region}
\end{figure}
\end{remark}
The following proposition shows the condition that at least one \(r_j\) is not \(-2\) in Theorem \ref{thm:main_one} cannot be removed. 
\begin{proposition}\label{prop:tight_all_twos}
Suppose \(f\in \mathrm{Mod}(\Sigma_{0,4},\partial \Sigma_{0,4})\) satisfies
\[\pi(f) = \begin{bmatrix}p' & q'\\ p & q\end{bmatrix}\]
with \(p,q,p',q'\geq 0\) and \(p>q\). Additionally, suppose the continued fraction expansion of \(-p/q\) is of the form
\[-\frac{p}{q}=[\underbrace{-2,-2,\ldots,-2}_{k+1\text{\normalfont{ times}}}].\]
The number \(k=2\ell\) is even by the parities of \(p\) and \(q\). If the minimum FDTC of \(f\) is \(\geq 1\), then \(f\) factors into positive Dehn twists. 
\end{proposition}
\begin{proof}
In this case, one has 
\[\frac{p}{q}=\frac{k+2}{k+1},\]
and hence 
\[\pi(f) = \begin{bmatrix}k+1+2m(k+2)&k+2m(k+1) \\ k+2 & k+1\end{bmatrix}\qquad\text{for some }m\geq 0.\]
It follows that \(f\) factors as 
\[f=\tau_{a_1}^{n_1}\tau_{a_2}^{n_2}\tau_{a_3}^{n_3}\tau_{a_4}^{n_4}\tau_b^m \tau_d^{\ell} \tau_c^{-1}\qquad\text{for some }n_1,n_2,n_3,n_4,\]
where \(b,c,d\) are the closed curves in Figures \ref{fig:page} and \ref{fig:square_holes}. The monodromy \(\tau_b^m\tau_d^{\ell}\tau_c^{-1}\) has all FDTCs equal to \(0\), since it sends any arc of slope \(0\) to the left and any arc of slope \(\infty\) to the right. The latter statement is true since any arc of slope \(\infty\) is fixed by \(\tau_c^{-1}\), and so such an arc only ``sees'' positive Dehn twists. In Figure \ref{fig:FDTC_verify1}, we verify the statement about an arc of slope \(0\).
\begin{figure}[ht]
    \centering 
    \includegraphics[scale = 0.5]{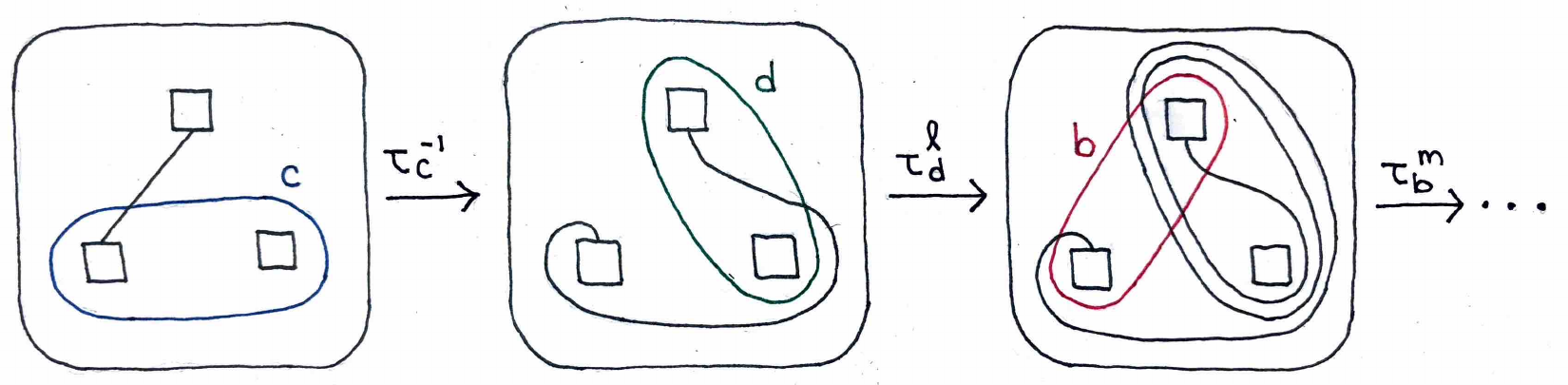}
    \caption{In the rightmost picture, it remains to perform \(m\) positive Dehn twists about the curve \(b\). It is clear that no matter how many positive Dehn twists are performed, the resulting arc will remain to the left of the original arc.}
    \label{fig:FDTC_verify1}
\end{figure}

Accordingly, the FDTCs of \(f\) are given by \(n_i\). Ergo 
\[\min \{n_1,n_2,n_3,n_4\} \geq 1. \]
We now use the lantern relation in \(\mathrm{Mod}(\Sigma_{0,4},\partial \Sigma_{0,4})\), which states
\[\tau_{a_1}\tau_{a_2}\tau_{a_3}\tau_{a_4}=\tau_b\tau_c\tau_d = \tau_c\tau_d \tau_b = \tau_d \tau_b\tau_c.\]
This allows us to factor \(f\) into positive Dehn twists as 
\begin{align*}
f &= \tau_{a_1}^{n_1}\tau_{a_2}^{n_2}\tau_{a_3}^{n_3}\tau_{a_4}^{n_4}\tau_b^m \tau_d^{\ell} \tau_c^{-1}\\
&= \tau_{a_1}^{n_1-1}\tau_{a_2}^{n_2-1}\tau_{a_3}^{n_3-1}\tau_{a_4}^{n_4-1}\tau_b^m \tau_d^k \tau_{a_1}\tau_{a_2}\tau_{a_3}\tau_{a_4}\tau_c^{-1}\\
& = \tau_{a_1}^{n_1-1}\tau_{a_2}^{n_2-1}\tau_{a_3}^{n_3-1}\tau_{a_4}^{n_4-1}\tau_b^m \tau_d^{\ell+1}\tau_b.\qedhere
\end{align*}
\end{proof}

With Theorem \ref{thm:main_one}, we can now prove Theorem \ref{thm:classify_tight_reducible}.
\begin{proof}[Proof of Theorem {\normalfont\ref{thm:classify_tight_reducible}}] \phantomsection\label{proof:classify_tight_reducible}
If \(f\) satisfies (i), then it is tight by Theorem \ref{thm: Ito-Kawamuro}. If \(f\) satisfies (ii), we show in Proposition \ref{prop: Stein fillable} that \(f\) factors into positive Dehn twists. Hence \(f\) is Stein fillable. Conversely, suppose \(f\) does not satisfy (i) or (ii). Then \(f\) is either not right-veering or is isotopic to  
\[\tau_{a_1}^{n_1}\tau_{a_2}^{n_2}\tau_{a_3}^{n_3}\tau_{a_4}^{n_4}\tau_{\gamma}^{n_{\gamma}},\qquad \min \{n_1,n_2,n_3,n_4\} = 1,\, n_{\gamma}\leq -2.\]
In the latter case, we may assume up to conjugacy that \(\gamma =c\). Then these monodromies all satisfy the criterion of Theorem \ref{thm:main_one} with 
\[-\frac{p}{q}= 2n_{\gamma},\]
so are overtwisted
\end{proof}

We now describe the second overtwisted family of monodromies.
\begin{theorem}\label{thm:main_two}
Suppose \(f\in \mathrm{Mod}(\Sigma_{0,4},\partial \Sigma_{0,4})\) satisfies
\[\pi(f) = \begin{bmatrix}p' & q'\\ p & q\end{bmatrix}\]
with \(p,q,p',q'\geq 0\) and \(p>q\). Additionally, suppose \(f\) has minimum FDTC \(1\) and one of the following conditions is satisfied: 
\begin{enumerate}[label={\normalfont(\arabic*)}]
\item The continued fraction expansion of \(-p/q\) is of the form
\[-\frac{p}{q}=[r_0,r_1,-2,r_3,-2,r_5,\ldots,r_{k-2},-2,r_k]\]
where \(r_j<-1\), \(r_1<-3\), and \(r_0\) and \(k\) are odd. {\normalfont(}The case \(k=1\) is allowed.{\normalfont)}
\item The continued fraction expansion of \(-p/q\) is of the form 
\[-\frac{p}{q} = [r_0,-3]\]
where \(r_0 \leq -5\) is odd and a binding component with FDTC \(1\) is connected by an arc of slope \(0\) to a binding component with FDTC \(\leq \frac{1}{2}(\lvert r_0\rvert-3)\). 
\end{enumerate}
Then \(f\) is overtwisted.
\end{theorem}
\begin{remark}\label{re:two_term}By the parity of \(p\) and \(q\), if
\[-\frac{p}{q}=[r_0,r_1],\]
then both \(r_0\) and \(r_1\) are necessarily odd. 
\end{remark}
\begin{proof}[Proof of {\normalfont(1)}]
The proof strategy is the same as that of Theorem \ref{thm:main_one}, that is, we give a movie presentation for a transverse overtwisted disk \(D\). We draw the page in the same way as in Figure \ref{fig:square_holes}. As before, let \(n_i\) denote the FDTC of \(f\) at binding component \(B_i\). This time, we use conjugation by hyperelliptic involutions to ensure \(n_1 =1\). The intersection of \(D\) with \(\Sigma_0\) is shown in Figure \ref{fig:page_zero_again}.
\begin{figure}[ht]
\centering
    \includegraphics[scale = 0.2]{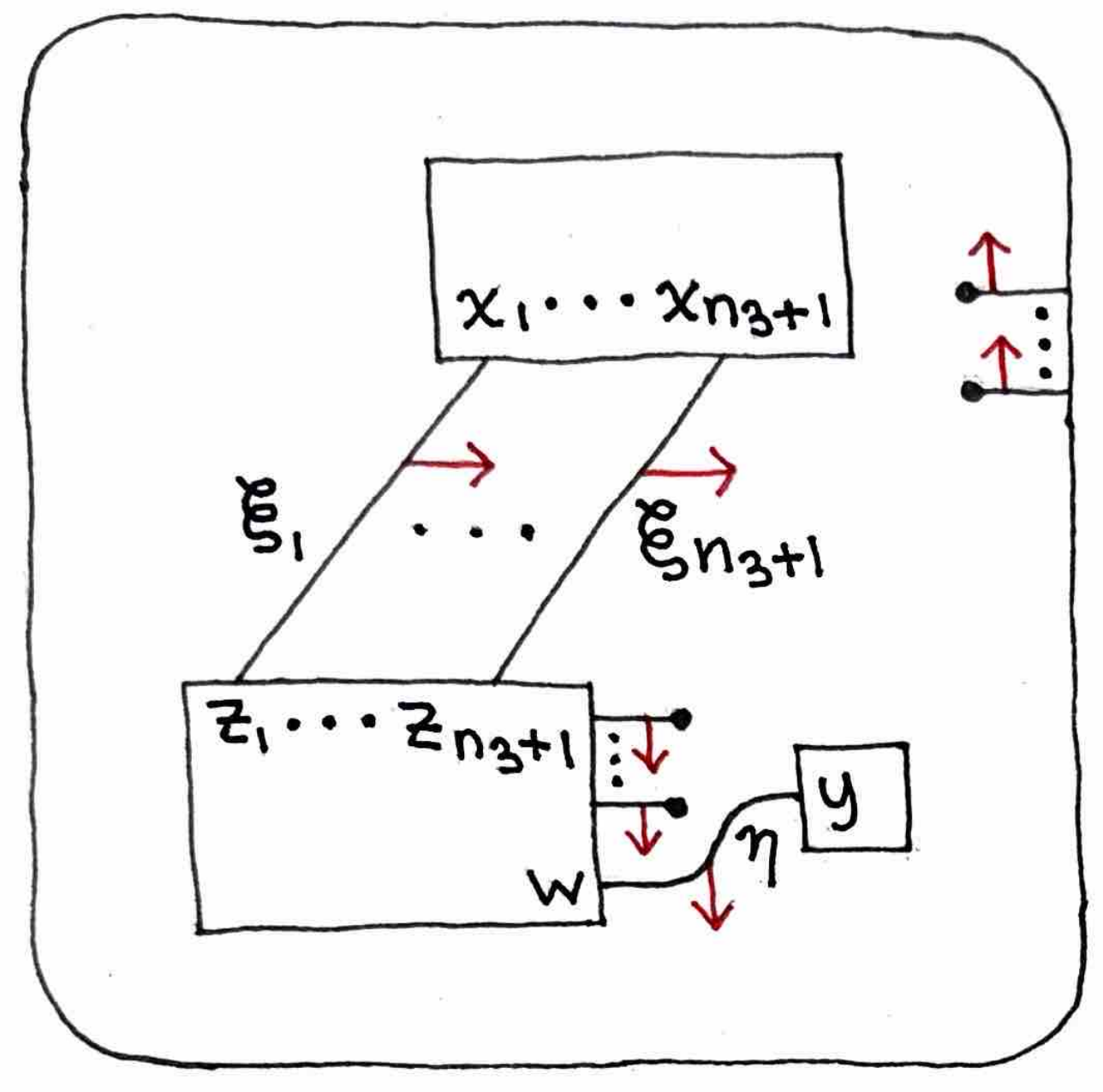}
    \caption{Boundary components have been stretched to fit labels.}
    \label{fig:page_zero_again}
\end{figure}
This intersection consists of: 
\begin{itemize}
\item properly embedded arcs \(\xi_1,\ldots,\xi_{n_3+1}\) from \(B_1\) to \(B_2\) of slope \(0\),
\item a properly embedded arc \(\eta\) from \(B_3\) to \(B_2\) of slope \(\infty\),
\item hairs with one endpoint in the interior of the page and the other endpoint on \(B_2\) or \(B_4\). 
\end{itemize}

As before, the red arrows indicate the (co)orientation of \(D\). Note \(x_1,\ldots,x_{n_3+1},y\) are negative intersections of \(D\) with the binding \(B\), whereas \(z_1,\ldots,z_{n_3+1},w\) have positive sign. (The positions of \(x_j\) and \(z_j\) have changed from the proof of Theorem \ref{thm:main_one}) The endpoint of each hair on a binding component has positive sign. 

Away from singular values of \(\varphi\) there will always be \(n_3+2\) properly embedded arcs on \(\Sigma_{\varphi}\), each with an endpoint at one of the negative elliptic singularities \(x_{n_1},\ldots,x_{n_3+1},y\). Let us refer to these arcs as \(\xi_1,\ldots,\xi_{n_3+1},\eta\). We keep the notation
\[-\frac{p}{q}=s_N<s_{N-1}<\cdots<s_1<s_0 =-1\]
with \(s_i = p_i/q_i\) from Lemma \ref{prop:fractions}. Our current assumptions necessitate \(N\geq \lvert r_0\rvert +2\). We describe the general case in text, and illustrate the case 
\[\pi(f) = \begin{bmatrix}17 & 6\\ 14 & 5\end{bmatrix}\]
with \(n_1=1\) and \(n_2=n_3=2\). (The value of \(n_4\) does not affect the movie presentation.) In this case,
\[-\frac{p}{q}=-\frac{14}{5}=[-3,-5]\]
and
\[s_0 =-1,\qquad s_1 = -2,\qquad s_2 = -\frac{5}{2},\qquad s_3 = -\frac{8}{3},\qquad s_4 = -\frac{11}{4},\qquad s_5 = -\frac{14}{5}.\]\par
Shift the arcs \(\xi_{1},\ldots,\xi_{n_3+1}\) in order using positive singularities with hairs to have the slopes \[s_1=-2,s_3=-4,\ldots, s_{\lvert r_0\rvert-2}=r_0+1.\]
We illustrate this process below for our current case.
\begin{center}
    \captionsetup{type=figure}
    \begin{longtblr}{
    colspec = {X[c,h]X[l]},
    }
    \includegraphics[scale = 0.2]{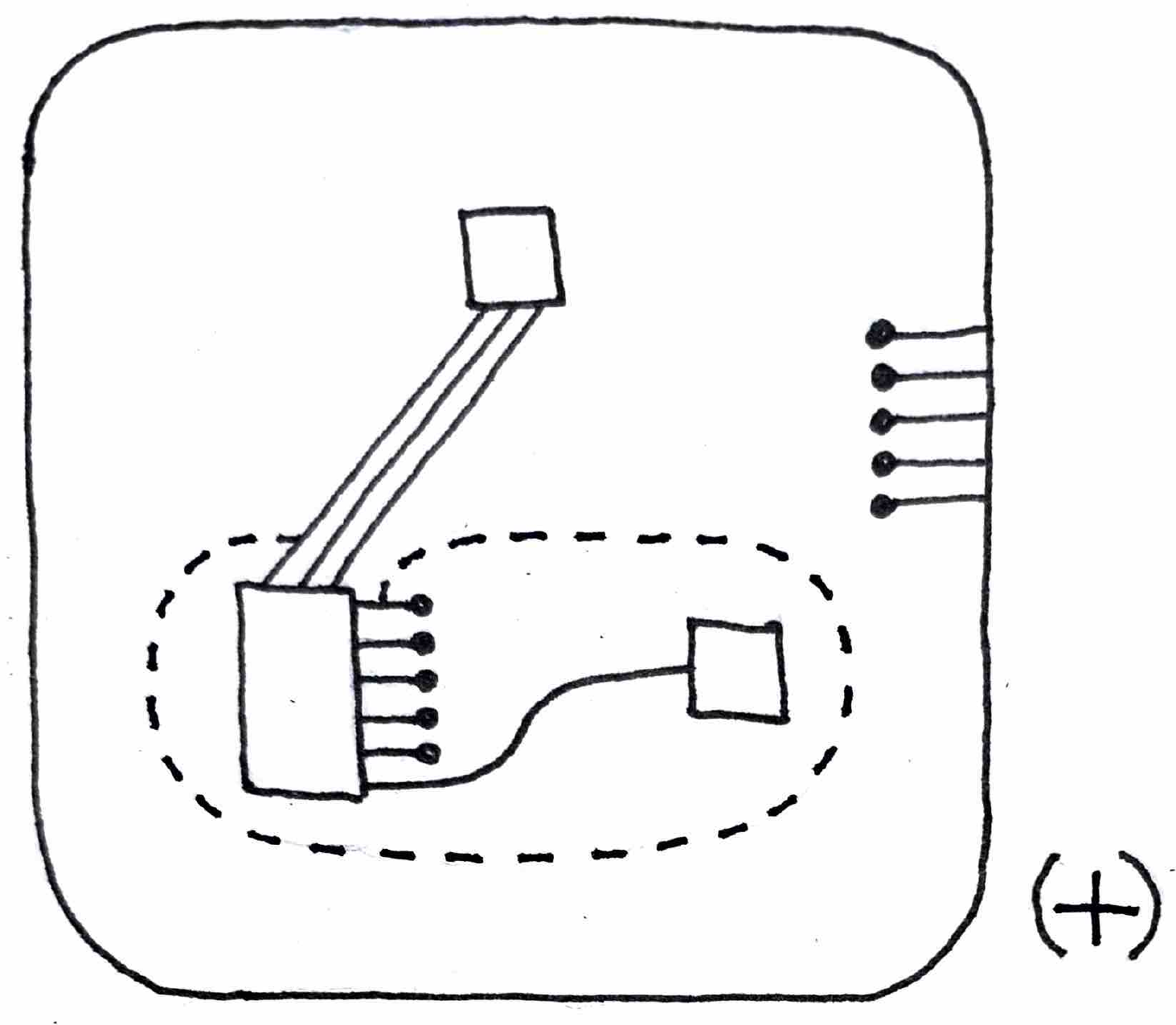} & Shift \(\xi_1\) to have slope \(-2\). \\
    \includegraphics[scale = 0.2]{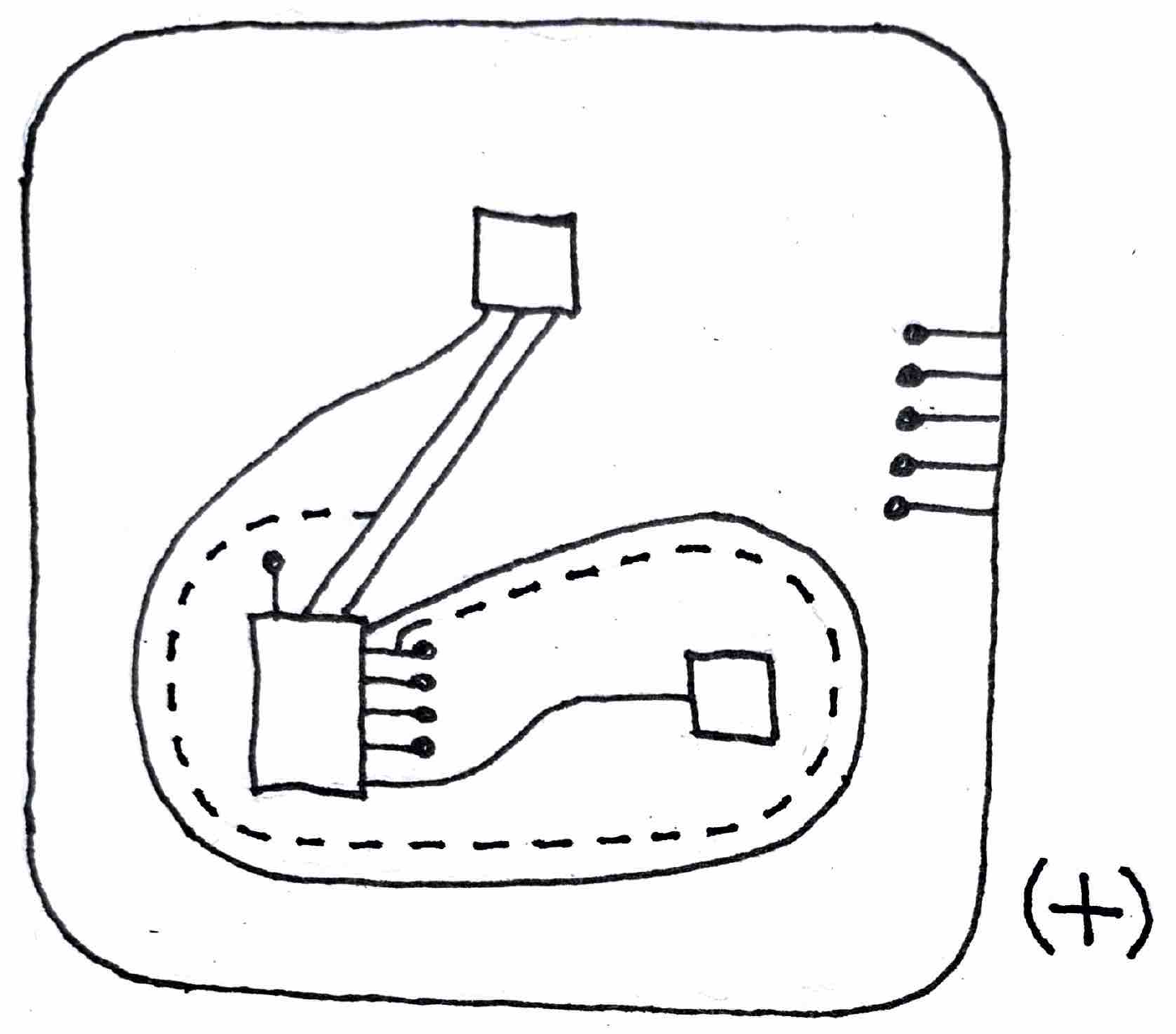} & Shift \(\xi_2\) to have slope \(-2\). \\
    \includegraphics[scale = 0.2]{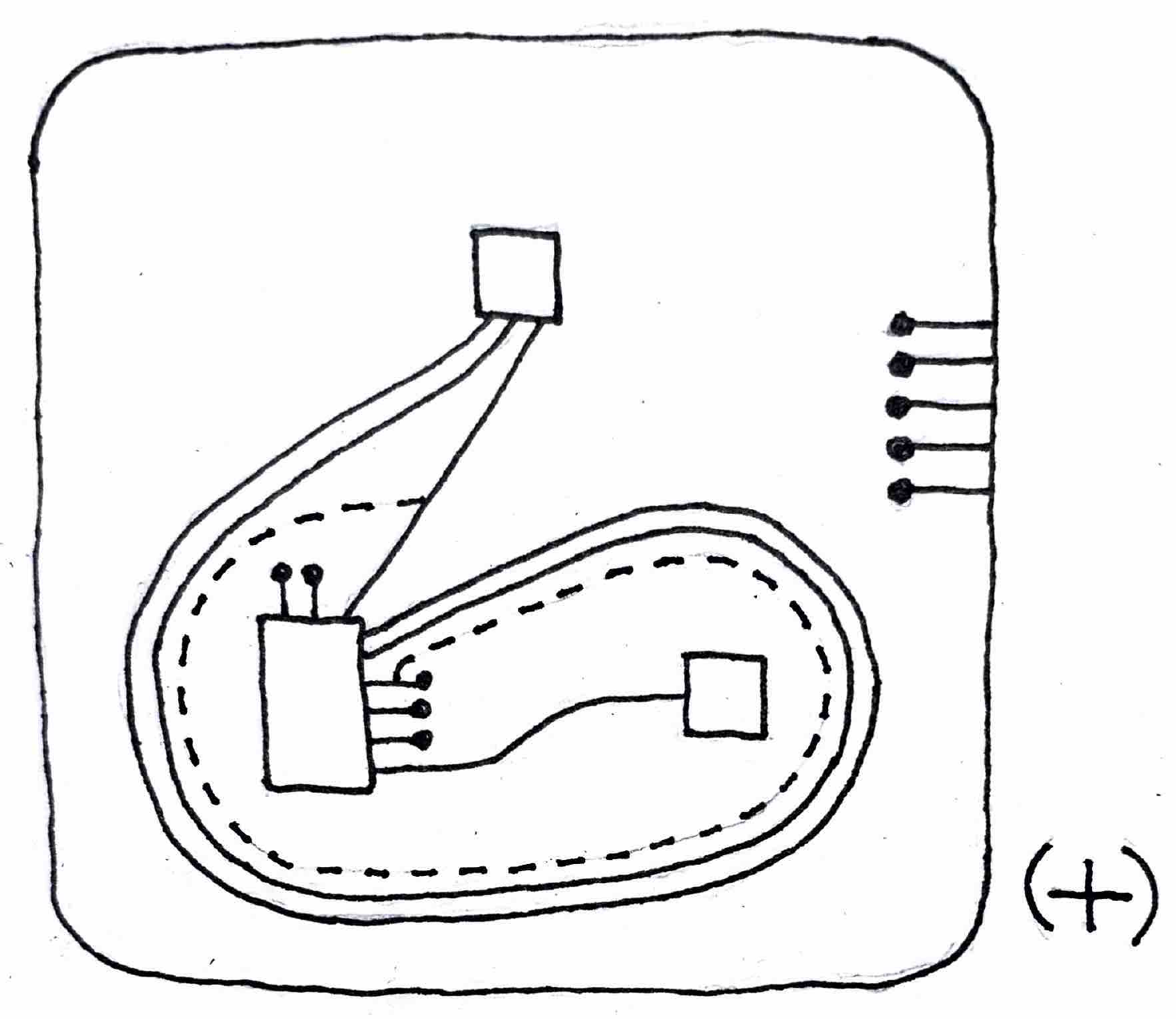} & Shift \(\xi_3\) to have slope \(-2\). 
    \end{longtblr}
    \captionof{figure}{}
    \label{fig:stage2_one}
\end{center}

During the next steps, we form negative singularities between \(\eta\) and \(\xi_{n_3+1},\ldots,\xi_1\), in this order. This is done in such a manner so that after the singularities are formed, 
\begin{itemize}
    \item \(\eta\) has slope \(s_{\lvert r_0\rvert -1}=r_0+\frac{1}{2}\),
    \item \(\xi_1,\ldots,\xi_{n_3}\) have slope \(s_{\lvert r_0\rvert}=r_0 +\frac{1}{3}\), 
    \item \(\xi_{n_3+1}\) has slope \(s_{\lvert r_0\rvert-2}=r_0+1\).
\end{itemize}
Each additional negative singularity formed with \(\eta\) makes \(\eta\) twist to the left around \(B_3\). We show this below.
\begin{center}
    \captionsetup{type=figure}
    \begin{longtblr}{
    colspec = {X[c,h]X[l]},
    }
    \includegraphics[scale = 0.25]{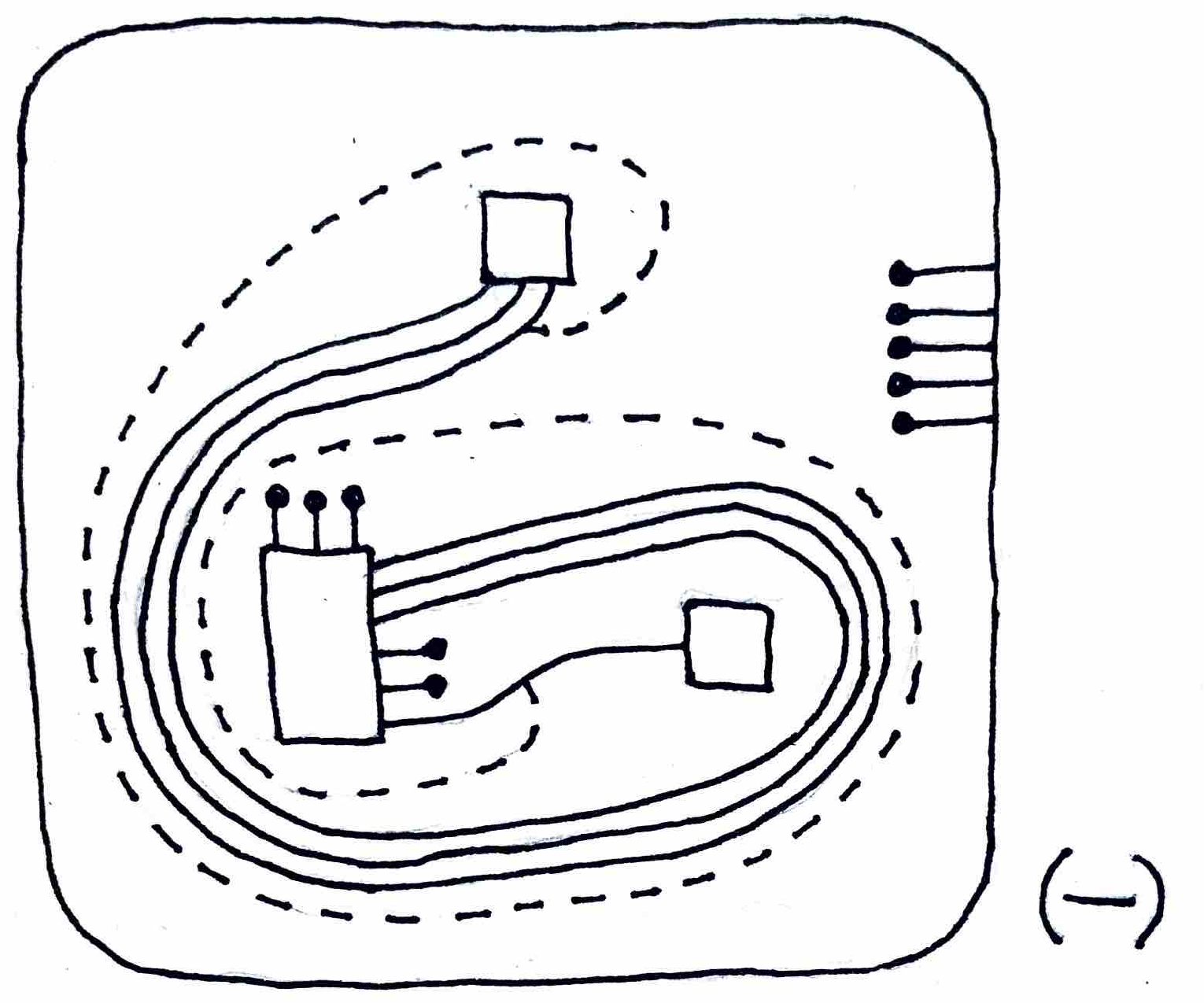} & Form a negative singularity between \(\xi_{3}\) and \(\eta\).\\
    \includegraphics[scale = 0.25]{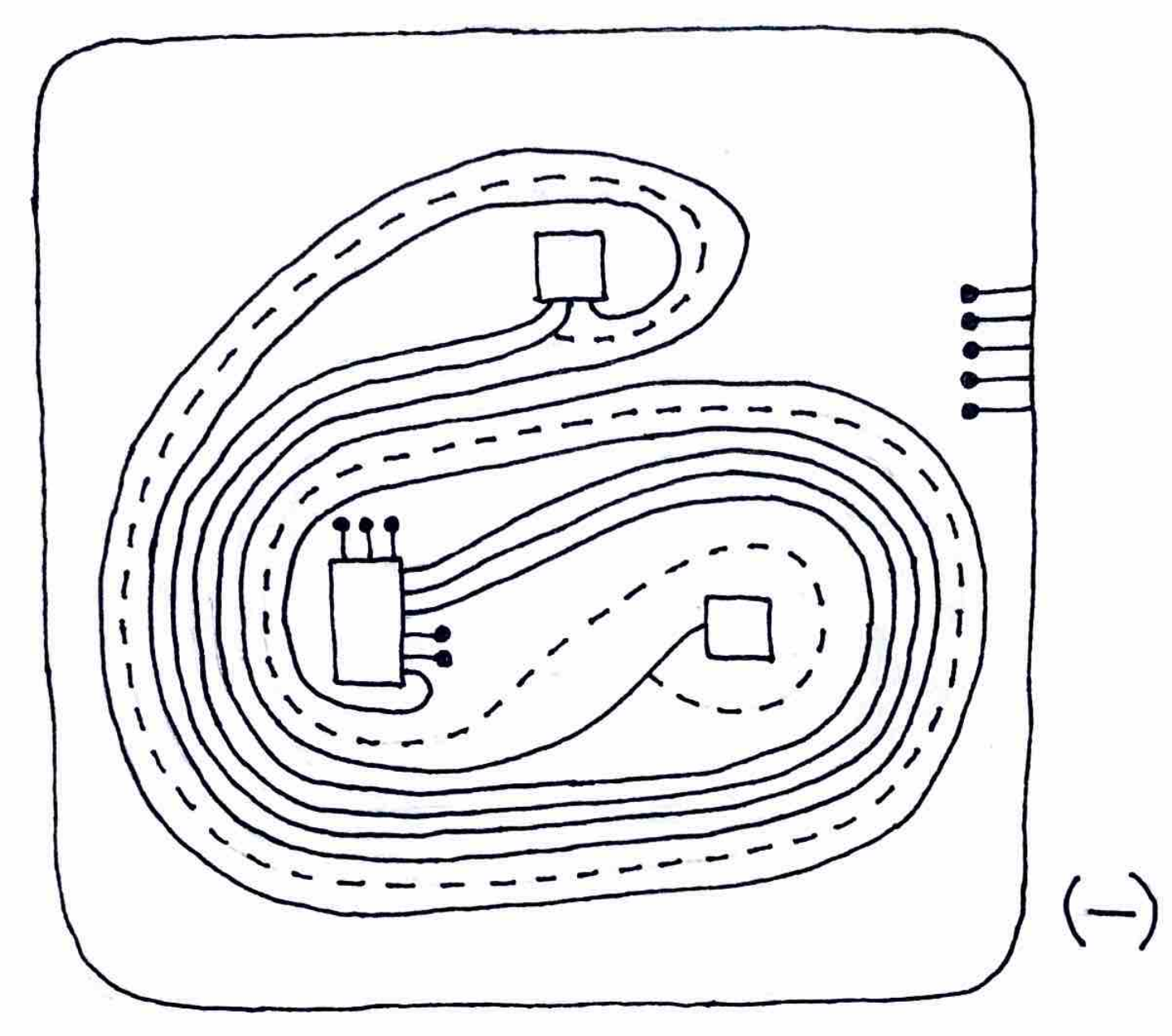} & Form a negative singularity between \(\xi_{2}\) and \(\eta\).\\
    \includegraphics[scale = 0.25]{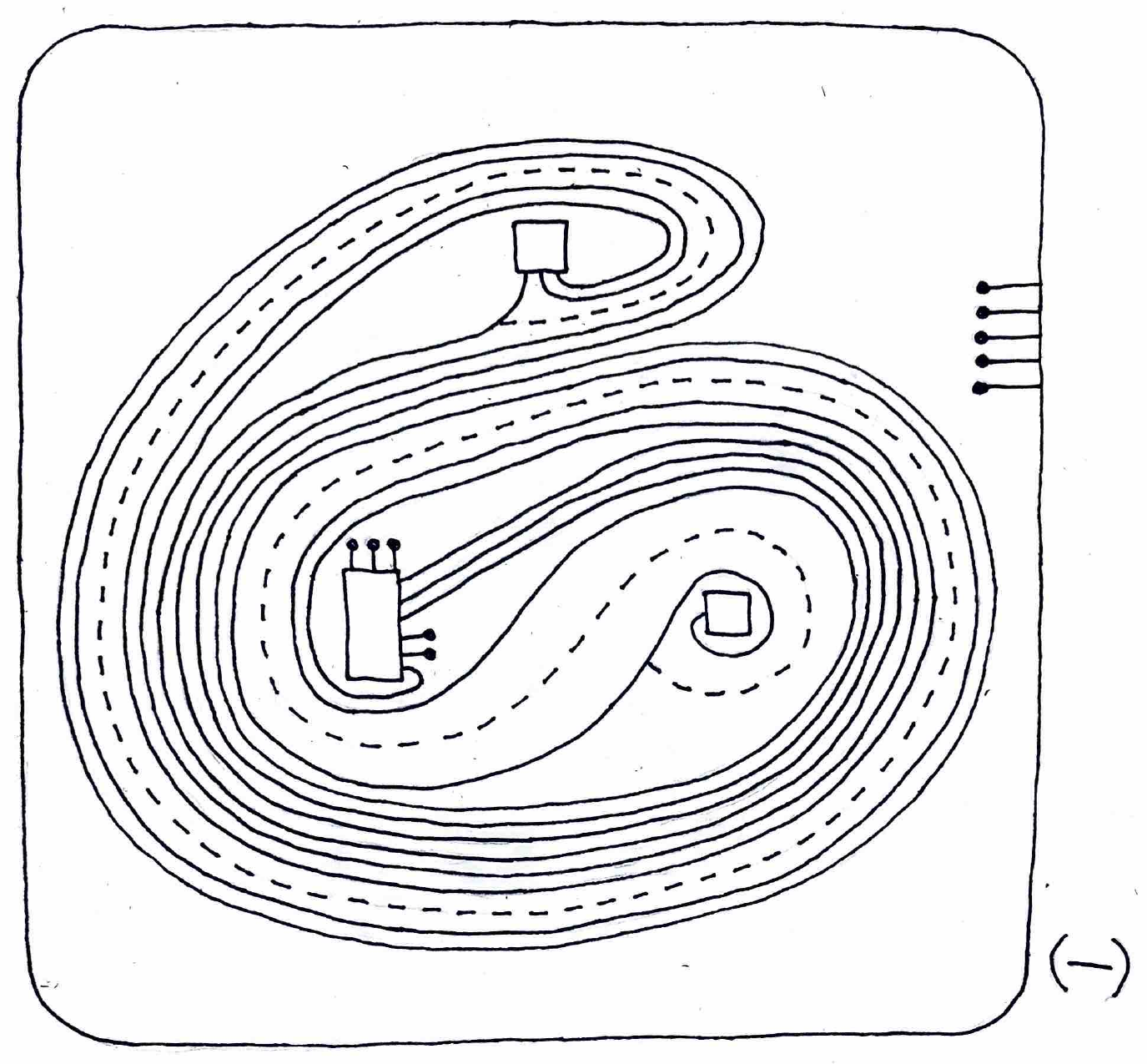} & Form a negative singularity between \(\xi_{1}\) and \(\eta\).
    \end{longtblr}
    \captionof{figure}{}
    \label{fig:stage2_one_cont}
\end{center}

Next, we shift \(\xi_{n_3+1},\eta\) (in this order) to have slope \(s_{\lvert r_0\rvert}\). Subsequently, we repeatedly shift the arcs 
\(\xi_{1},\xi_{2},\ldots, \xi_{n_3+1},\eta\) in order to have the slopes \(s_{\lvert r_0\rvert +1},\ldots,s_N,\). For \(\lvert r_0\rvert-1 \leq i\leq N\), one of \(p_i\) or \(q_i\) is even. Consequently, the positive endpoint of each properly embedded arc is shifted to \(B_2\) or \(B_4\). This prevents the situation where \(\eta\) and some \(\xi_j\) run parallel between \(B_1\) and \(B_3\), block each other from moving. At the end, we shift \(\eta\) to have slope \(-p'/q'\). As in the proof of Theorem \ref{thm:main_one}, during this process we have enough freedom to realize appropriate FDTCs about each of the binding components. 

We show this below for our current case. As a matter of short-hand, we will use a single red arc to denote \(\xi_1,\xi_2,\xi_3\) altogether once they are all parallel. We will allow ourselves to shift this red arc in a single step, with the understanding that this is an abbreviation for first shifting \(\xi_1\), then \(\xi_2\), then \(\xi_3\) as in the first three steps above. Red hairs will also be used to abbreviate three hairs sitting next to each other.
\begin{center}
    \captionsetup{type=figure}
    \begin{longtblr}{
    colspec = {X[c,h]X[l]},
    }
    \includegraphics[scale = 0.3]{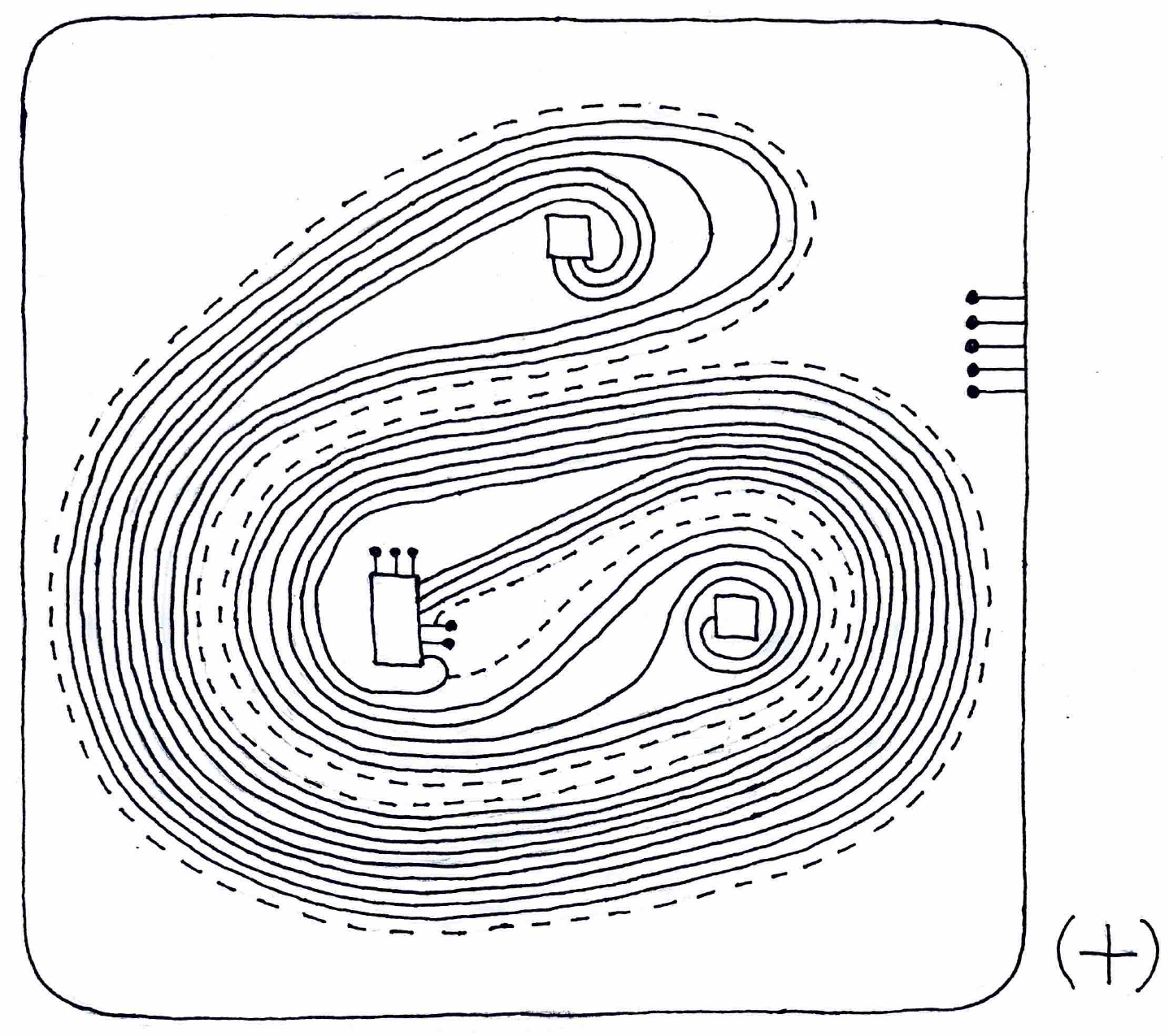} & Shift \(\xi_3\) to have slope \(-8/3\).\\
    \includegraphics[scale = 0.9]{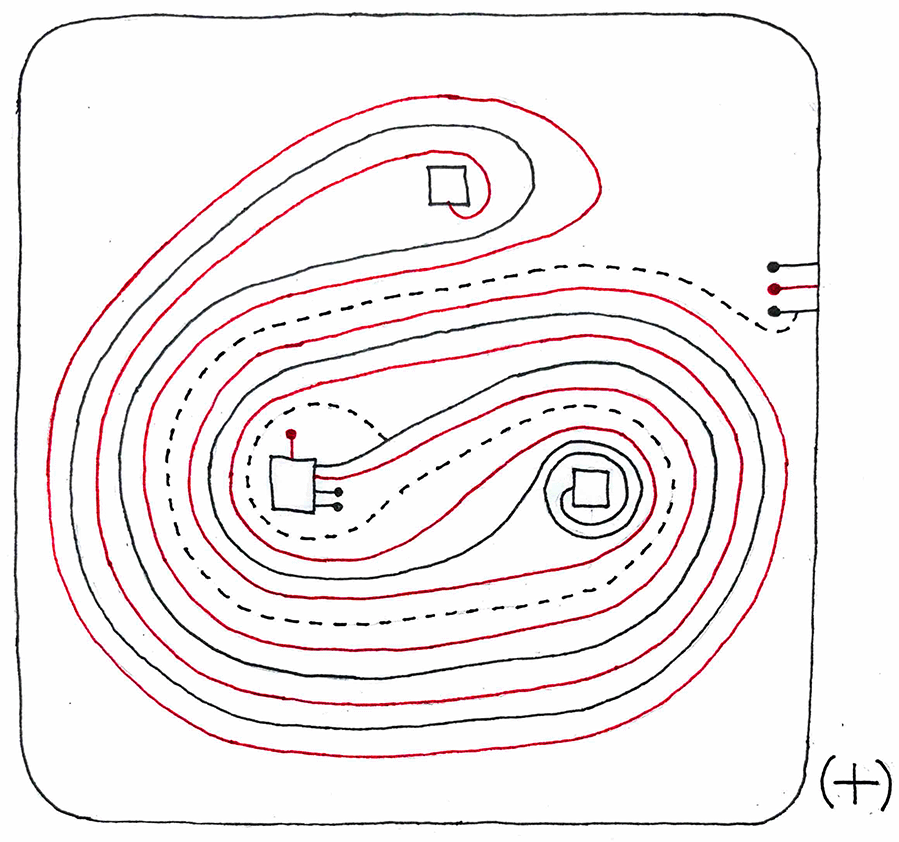} & Shift \(\eta\) to have slope \(-8/3\).\\
    \includegraphics[scale = 0.9]{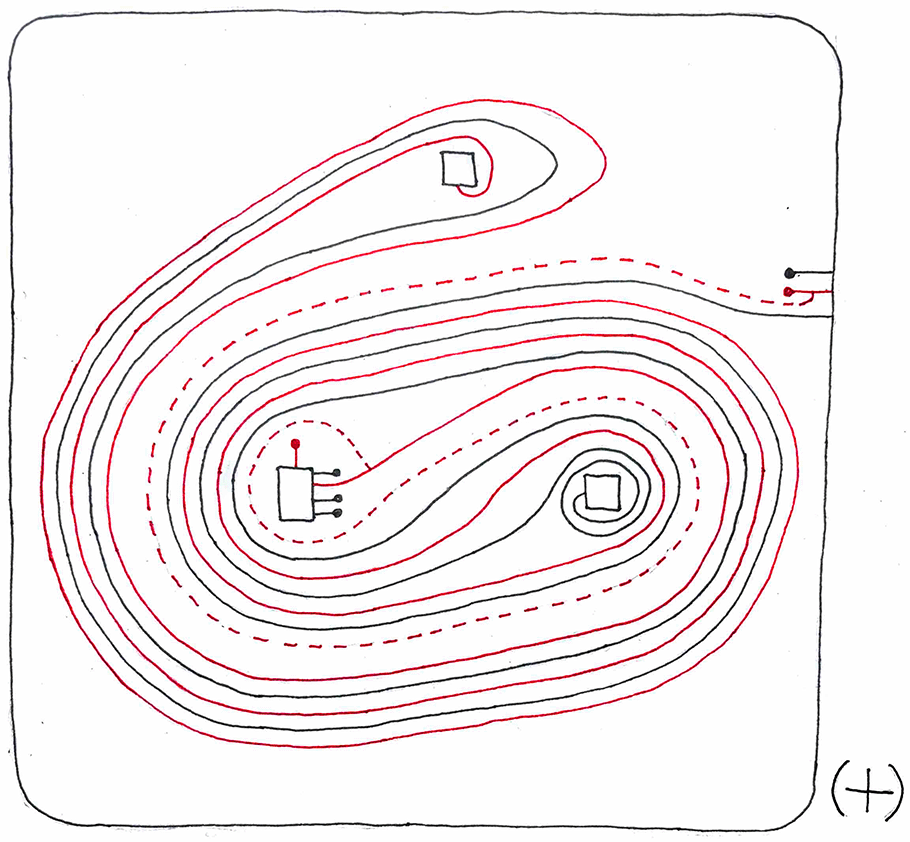} & Shift \(\xi_1,\xi_2,\xi_3\) to have slope \(-11/4\).\\
    \includegraphics[scale = 0.9]{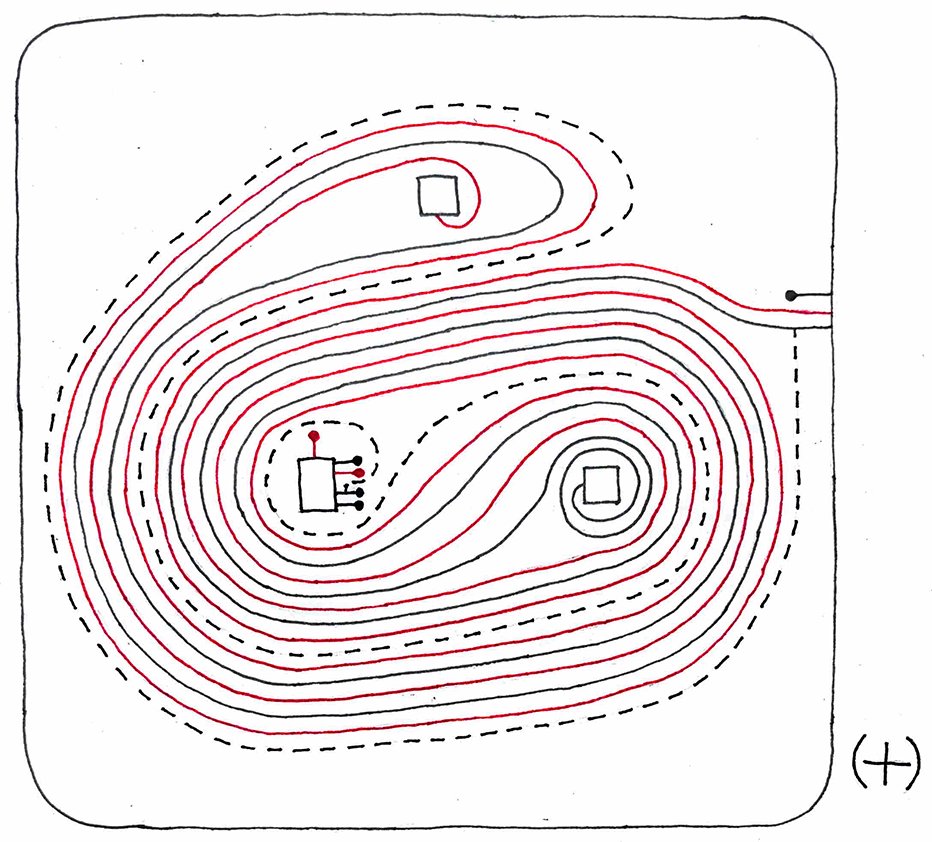} & Shift \(\eta\) to have slope \(-11/4\), twisting left around \(B_2\).\\
    \includegraphics[scale = 0.9]{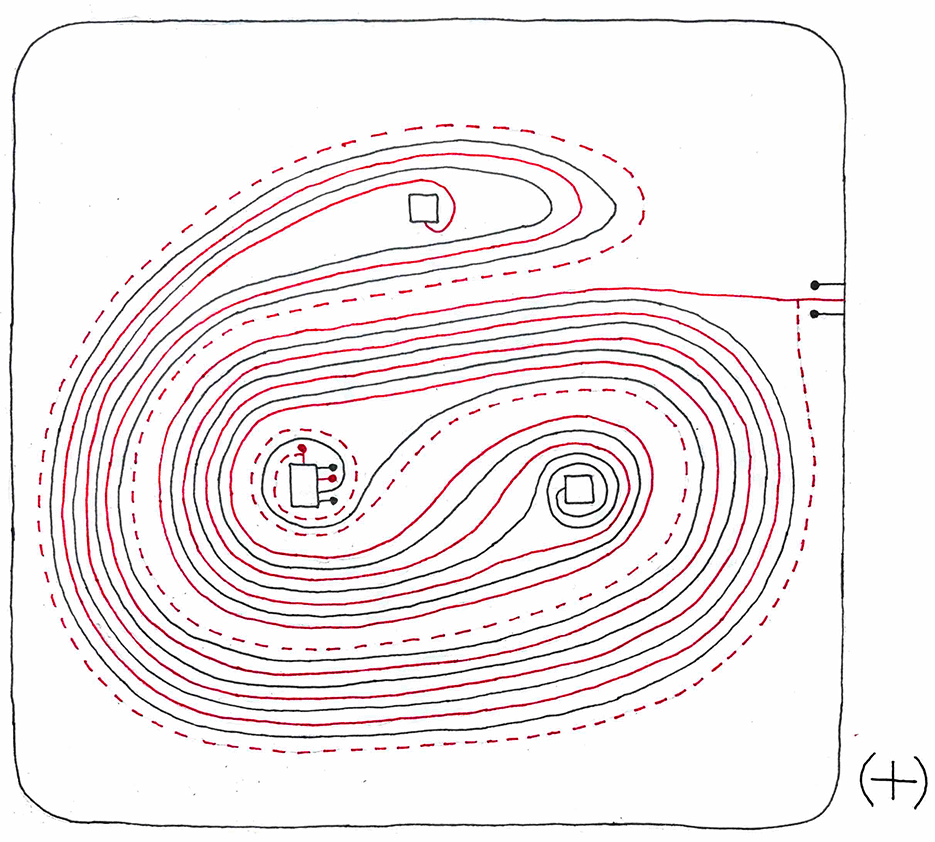} & Shift \(\xi_1,\xi_2,\xi_3\) to have slope \(-14/5\) and positive endpoint \(z_1,z_2,z_3\), realizing FDTC \(n_2=2\) at \(B_2\).\\
    \includegraphics[scale = 0.9]{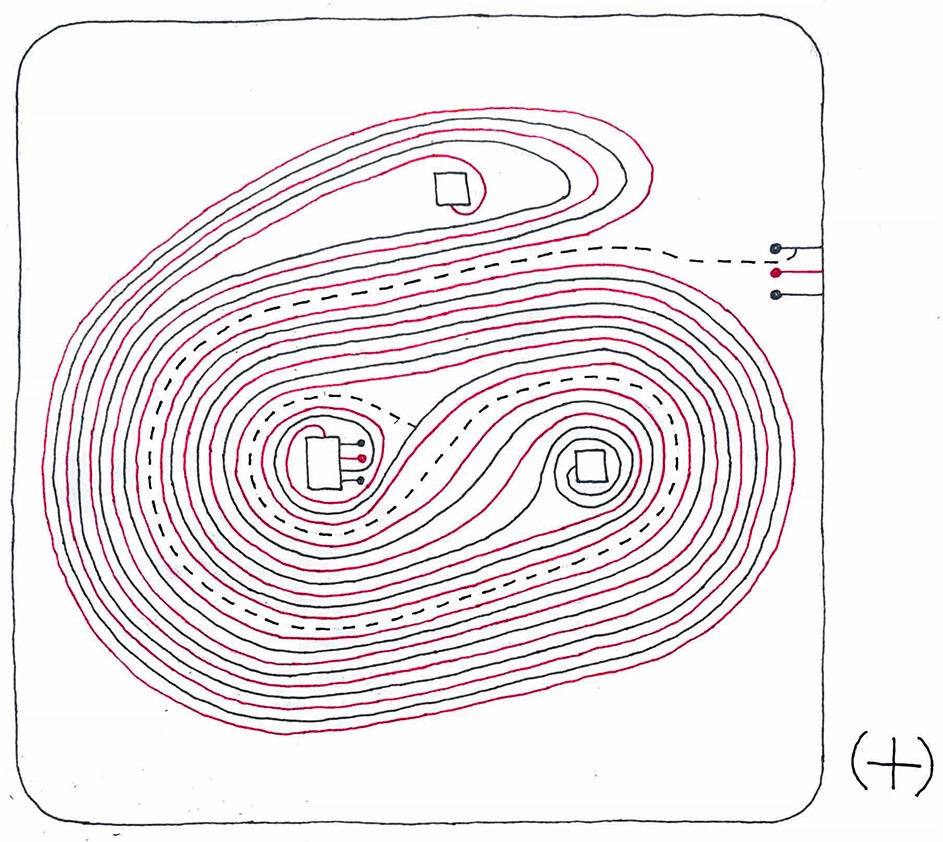} & Shift \(\eta\) to have slope \(-14/5\).\\
    \includegraphics[scale = 0.9]{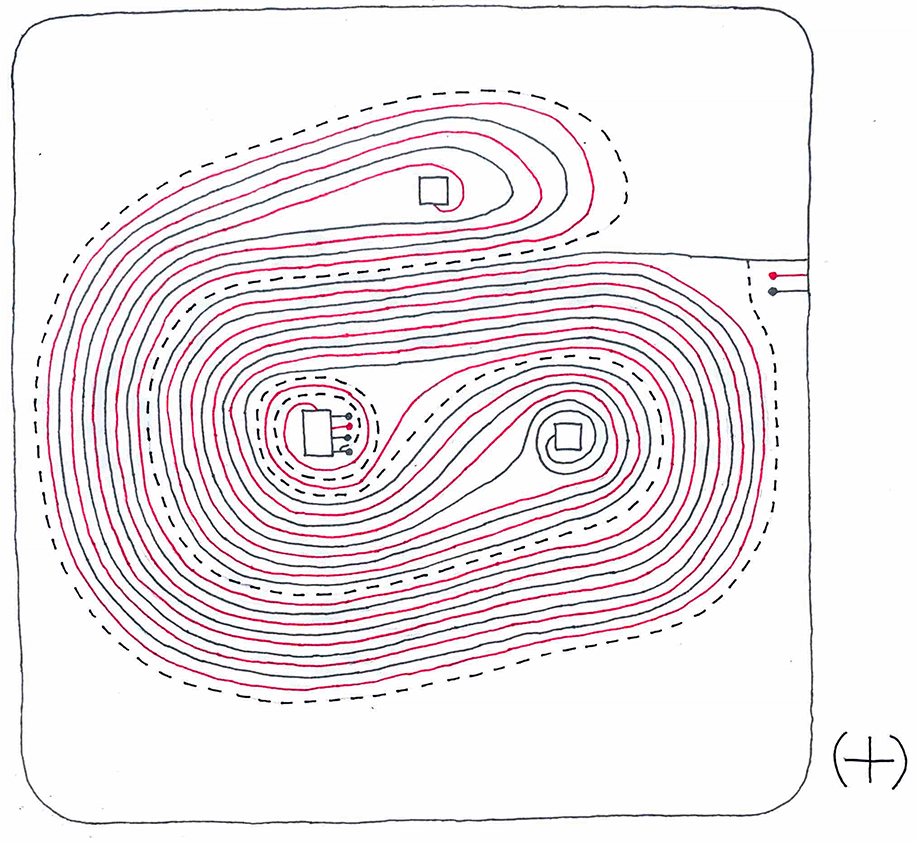} & Shift \(\eta\) to have slope \(-17/6\) and positive endpoint \(w\), realizing FDTC \(n_2=2\) at \(B_2\).\\
    \includegraphics[scale = 0.9]{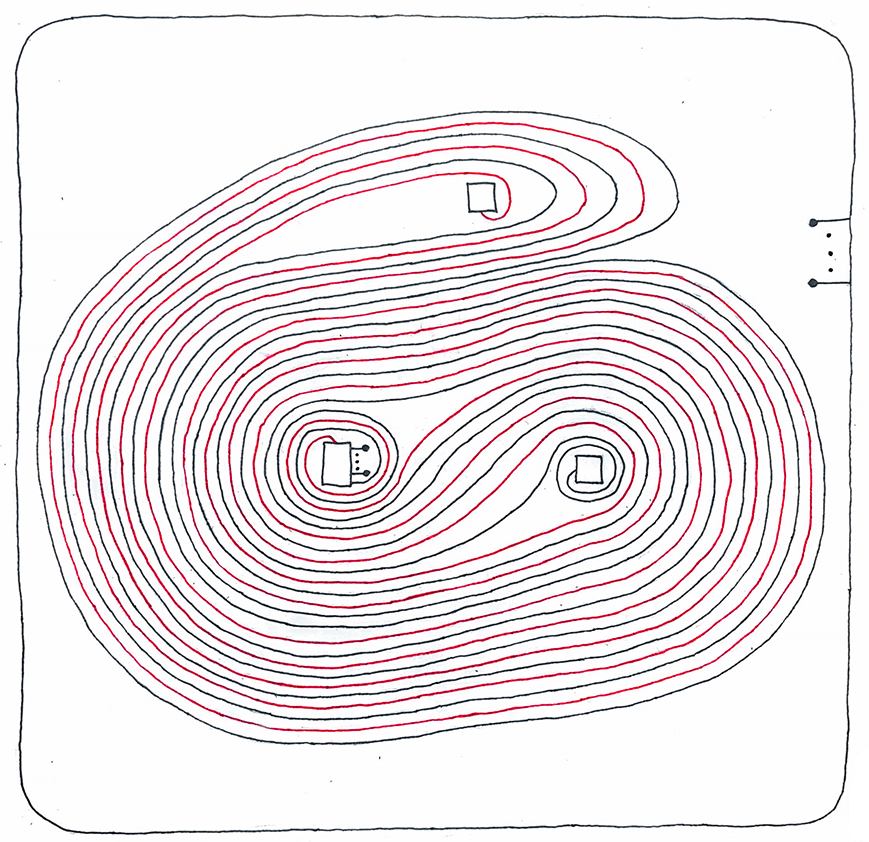} & The page \(\Sigma_{2\pi-\epsilon}\).
    \end{longtblr}
    \captionof{figure}{}
    \label{fig:stage2_two}
\end{center}

The above movie presentation specifies a transverse overtwisted disk \(D\) whose \(G_{--}\) graph consists of a single central vertex \(y\) connected to \(n_3+1\) leaves \(x_1,\ldots,x_{n_3+1}\). 
\end{proof}
\begin{proof}[Proof of {\normalfont(2)}]
We continue to use the same notation as above. We conjugate \(f\) so that \(n_1=1\) and \(n_2\leq \frac{1}{2}(\lvert r_0\rvert -3)\). If we mimic the proof strategy of (1) in this case, we find that the arcs \(\xi_1,\ldots,\xi_{n_3}\) have their final slopes after the negative singularities are formed. However, at this stage they have not realized the appropriate FDTC at \(B_2\). To remedy this, we perform a certain twist move on \(\eta\) which twists \(\eta\) to the right around \(B_3\) and to the left around \(B_2\). This allows us to twist the \(\xi_j\) left around \(B_2\) before the negative singularities are formed.

We have \(\frac{1}{2}(\lvert r_0\rvert -3)\) opportunities to perform this twist move. Each time it is applied, the realizable FDTC on \(B_2\) is increased by 1 (hence the bound on \(n_2\)). However, in exchange \(\eta\) also twists to the right one more time around \(B_3\). We compensate for this by starting with \(n_2 + n_3 + 1\) arcs \(\xi_1,\ldots,\xi_{n_2+n_3+1}\). (We still use the same labeling scheme as in Figure \ref{fig:page_zero_again}.) Other than this, the movie presentation is unchanged from the proof of (1). Below we show the simplest case, when 
\[\pi(f) = \begin{bmatrix}19 & 4 \\ 14 & 3\end{bmatrix}\]
and \(n_1=n_2=n_3=1\). (The value of \(n_4\) does not affect the movie presentation.) In this case
\[-\frac{p}{q}=-\frac{14}{3}=[-5,-3]\]
and 
\[s_0 = -1,\quad s_1=-2,\quad s_2=-3,\quad s_3=-4, \quad s_4 = -\frac{14}{3}.\]
\begin{center}
    \captionsetup{type=figure}
    \begin{longtblr}{
    colspec = {X[c,h]X[l]},
    }
    \includegraphics[scale = 0.2]{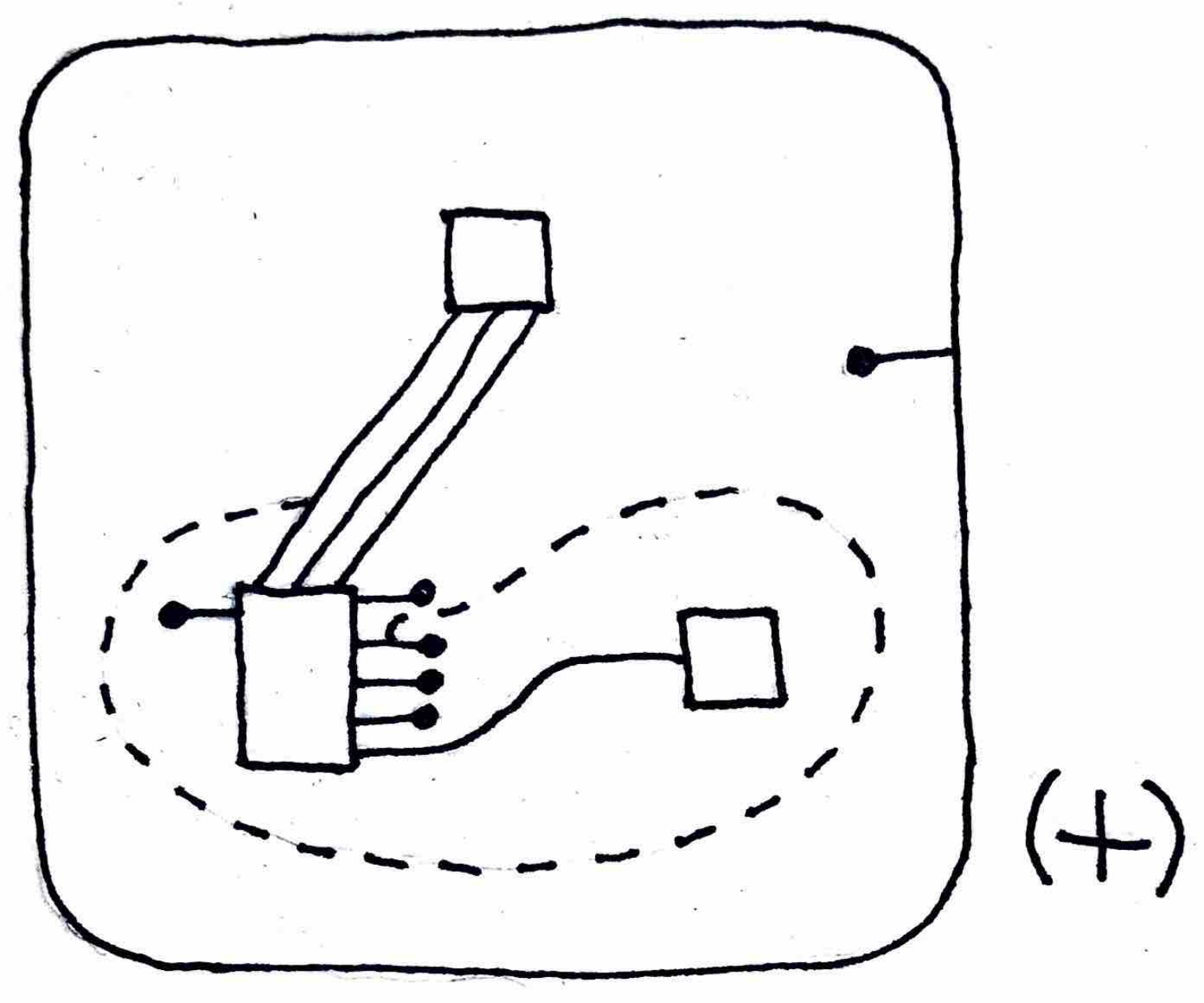} & Shift \(\xi_1\) to have slope \(-2\). \\
    \includegraphics[scale = 0.2]{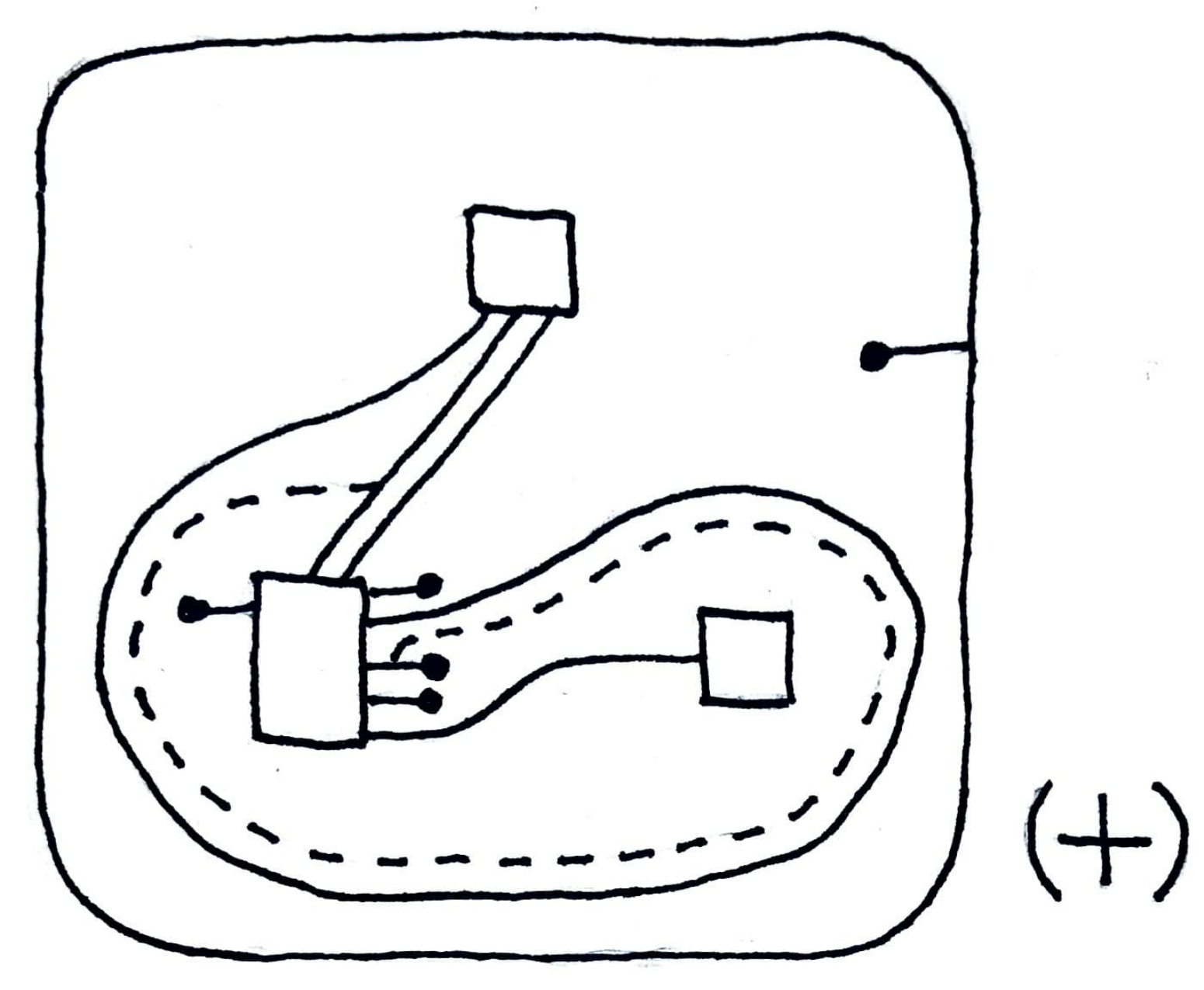} & Shift \(\xi_2\) to have slope \(-2\). \\
    \includegraphics[scale = 0.2]{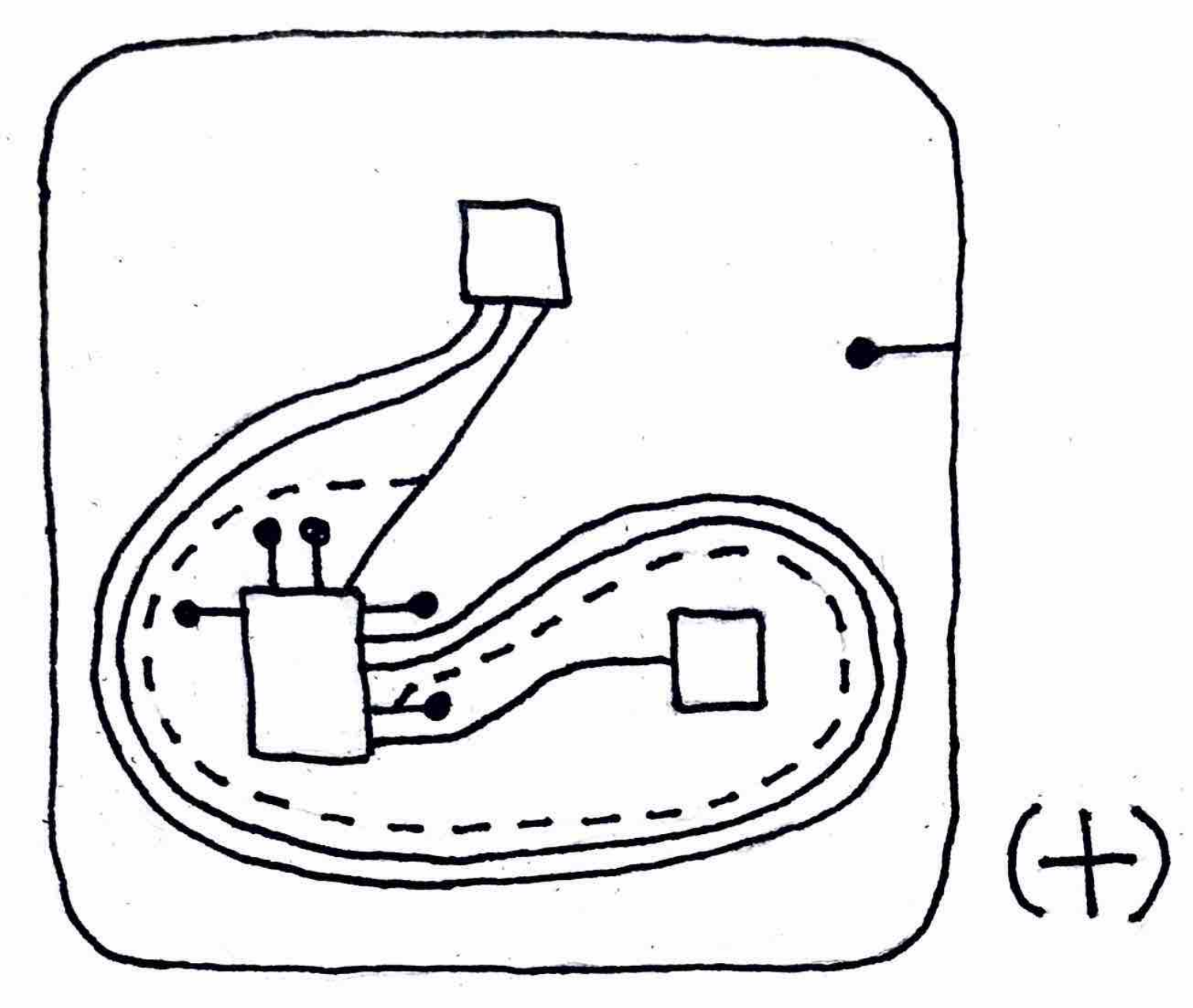} & Shift \(\xi_3\) to have slope \(-2\). \\
    \includegraphics[scale = 0.2]{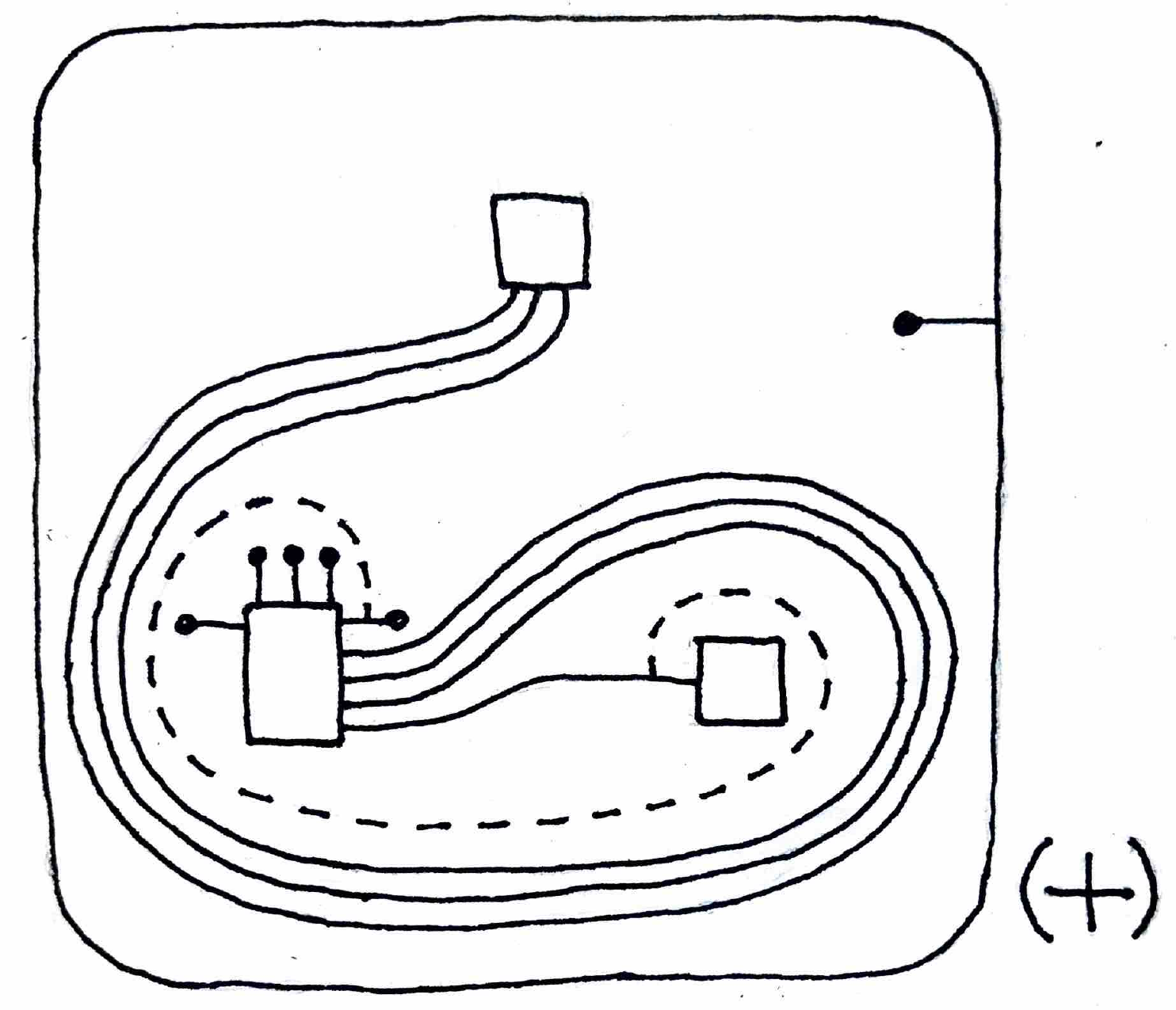} & Perform the twist move on \(\eta\).\\
    \includegraphics[scale = 0.25]{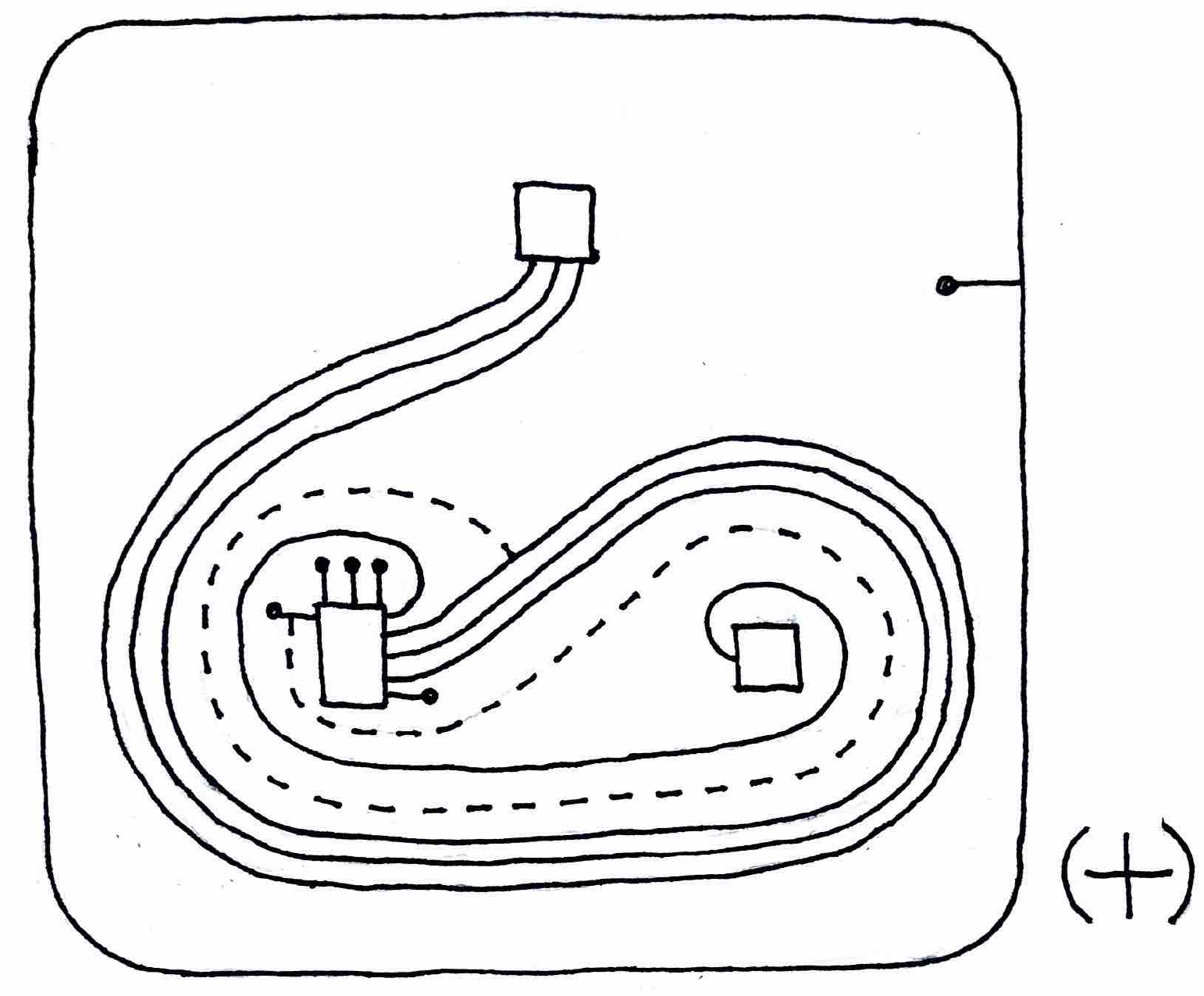} & Shift \(\xi_1\) to have slope \(-4\), twisting left around \(B_2\).\\
    \includegraphics[scale = 0.25]{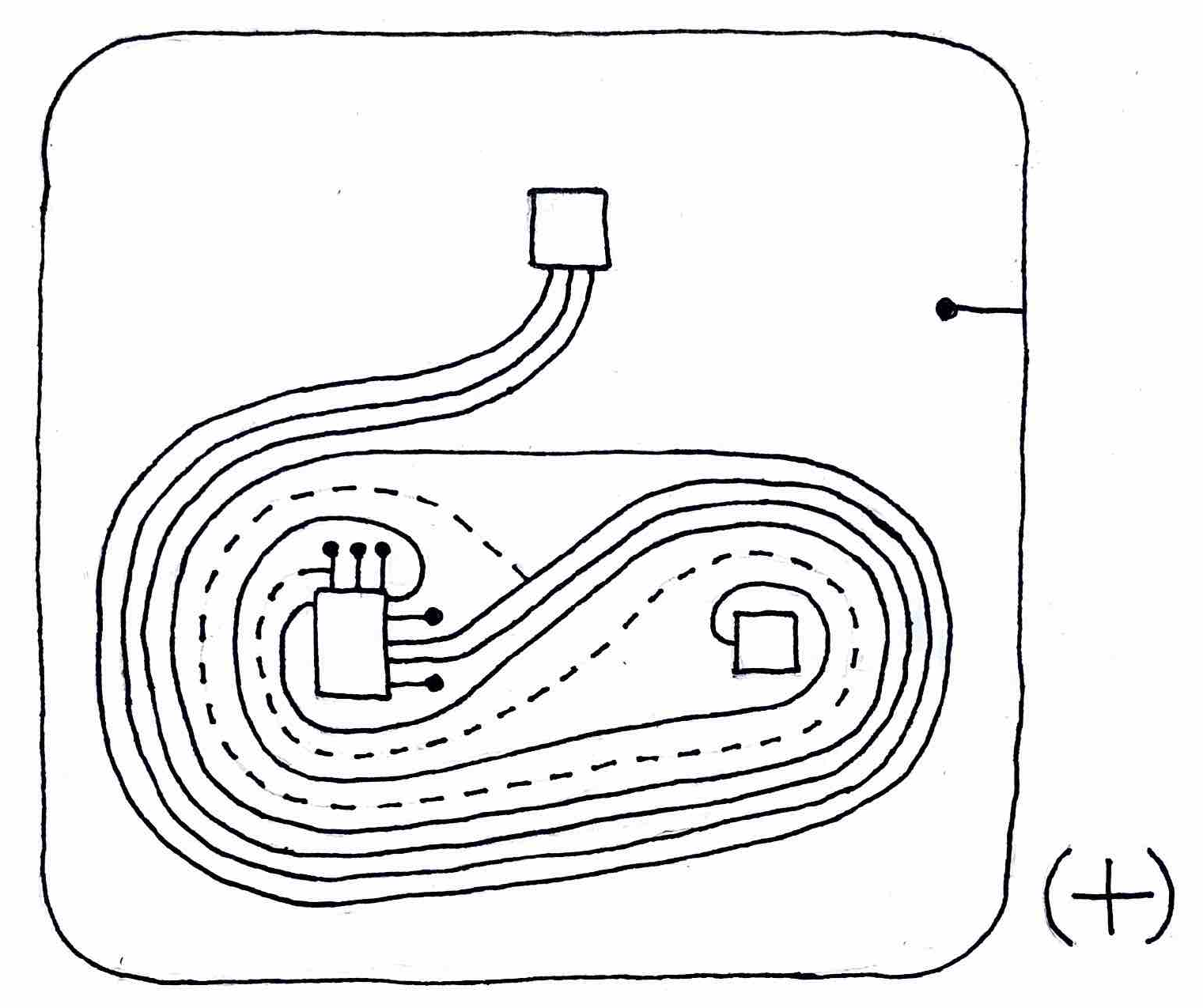} & Shift \(\xi_2\) to have slope \(-4\) and positive endpoint \(z_1\) [sic], twisting left around \(B_2\).\\
    \includegraphics[scale = 0.25]{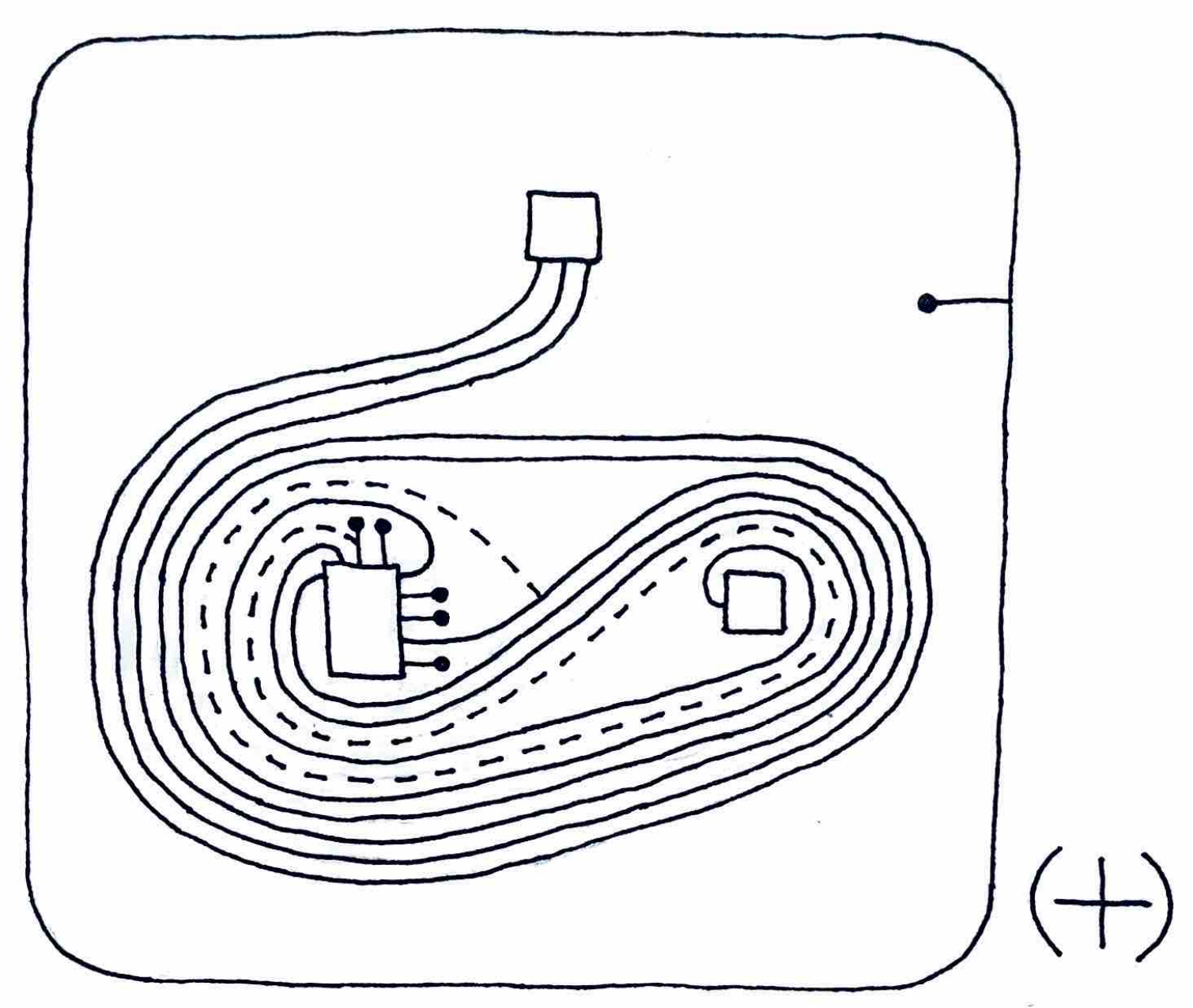} & Shift \(\xi_3\) to have slope \(-4\) and positive endpoint \(z_2\) [sic], twisting left around \(B_2\).\\
    \includegraphics[scale = 0.3]{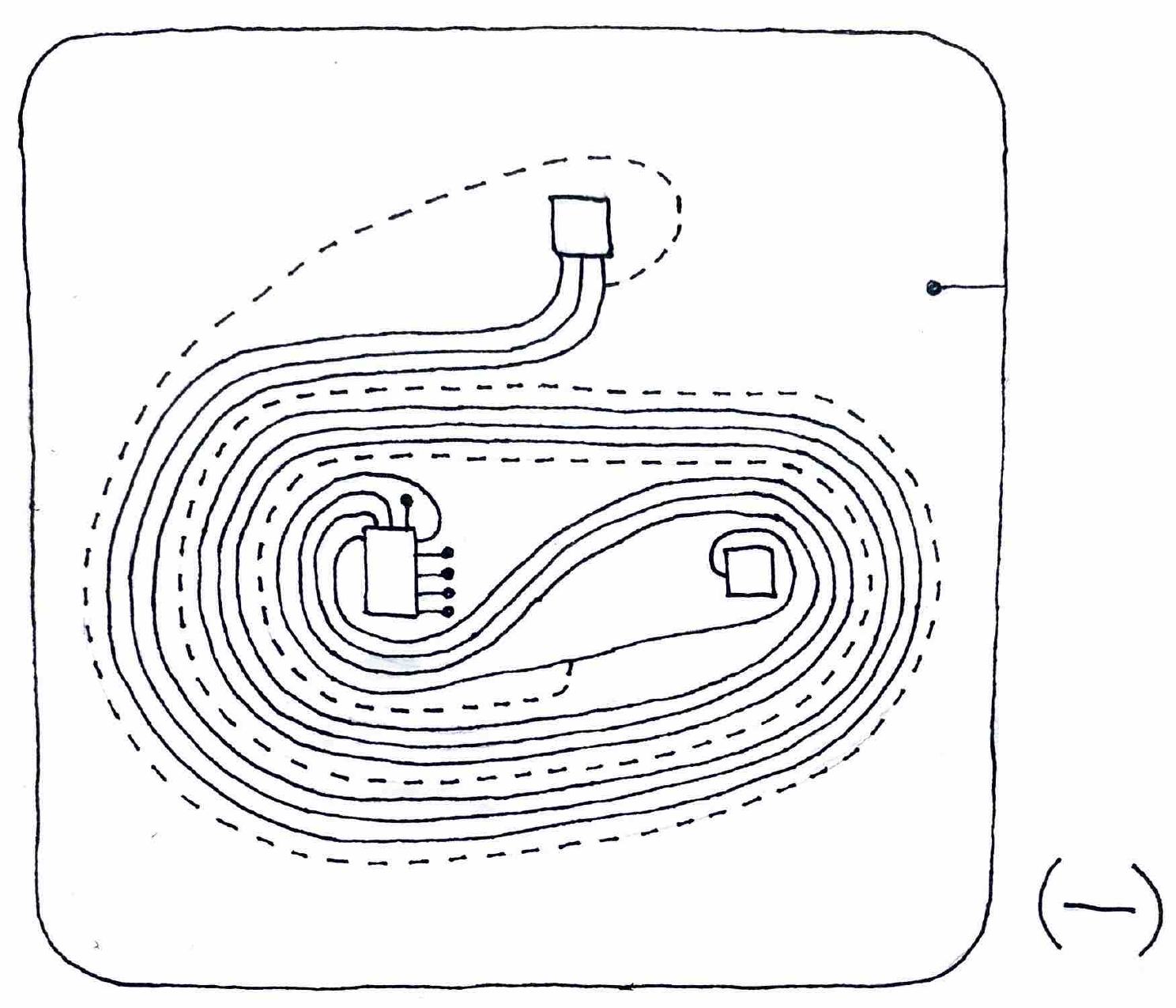} & Form a negative singularity between \(\xi_{3}\) and \(\eta\).\\
    \includegraphics[scale = 0.3]{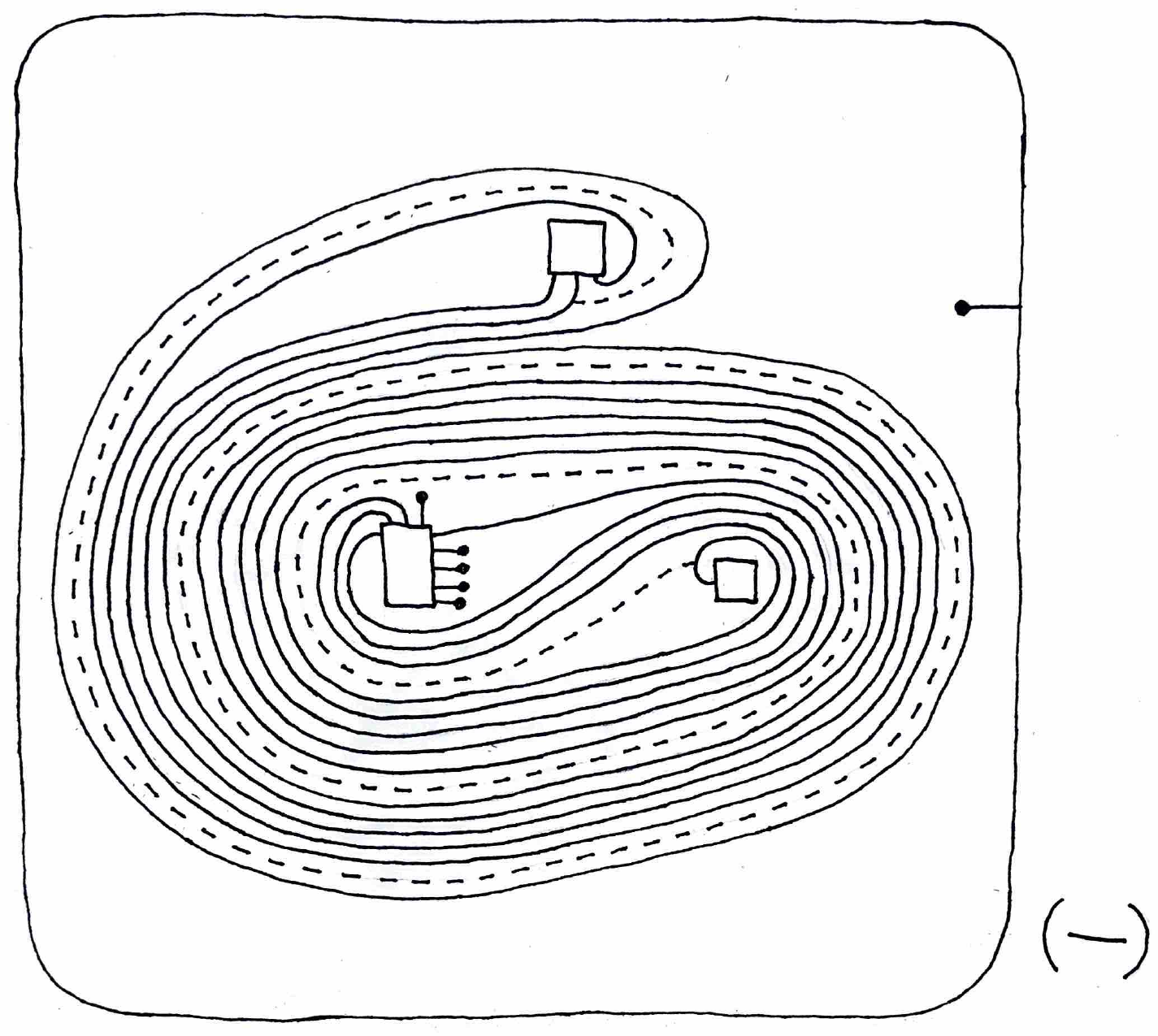} & Form a negative singularity between \(\xi_{2}\) and \(\eta\).\\
    \includegraphics[scale = 0.3]{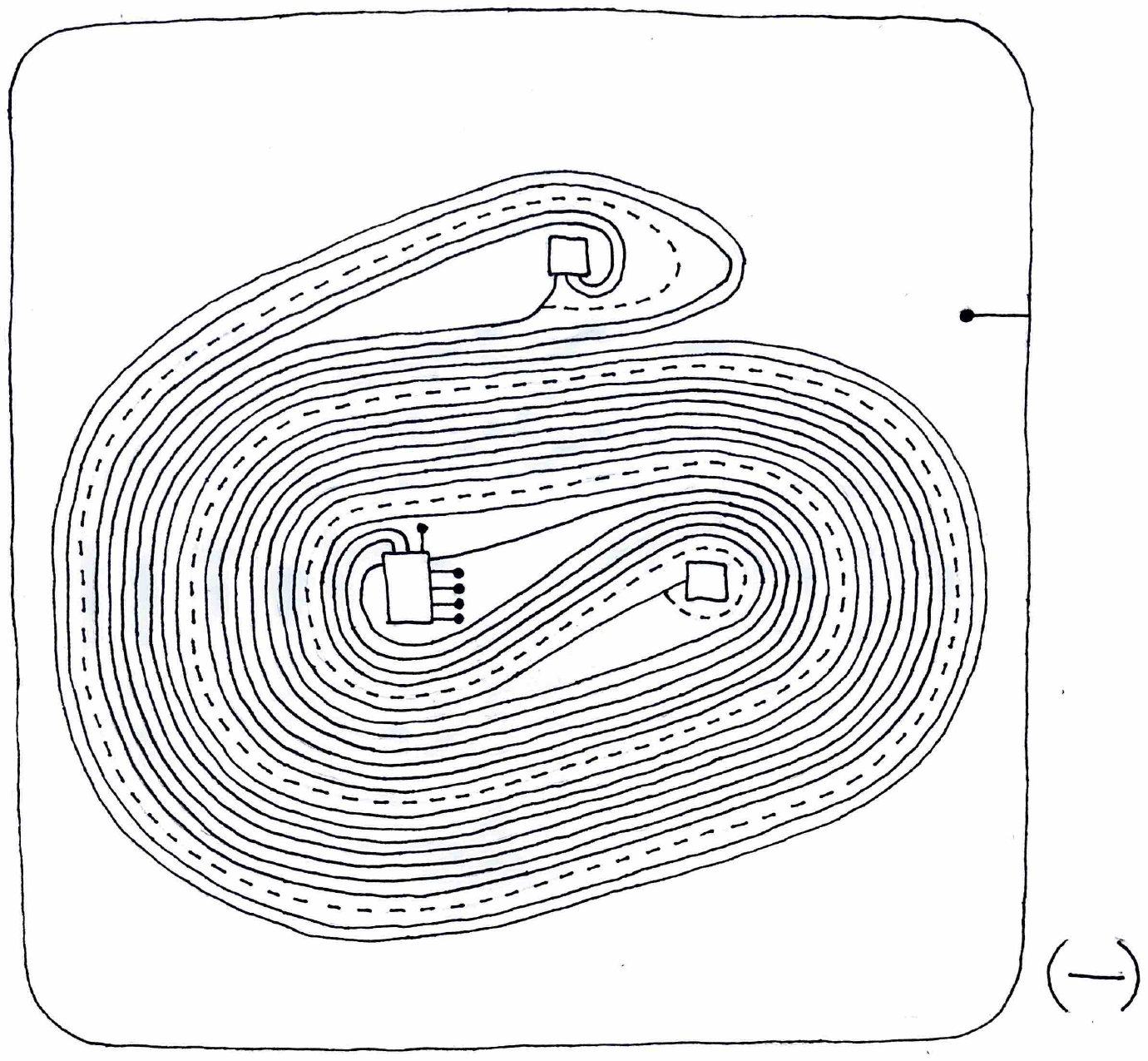} & Form a negative singularity between \(\xi_{1}\) and \(\eta\).\\
    \includegraphics[scale = 0.3]{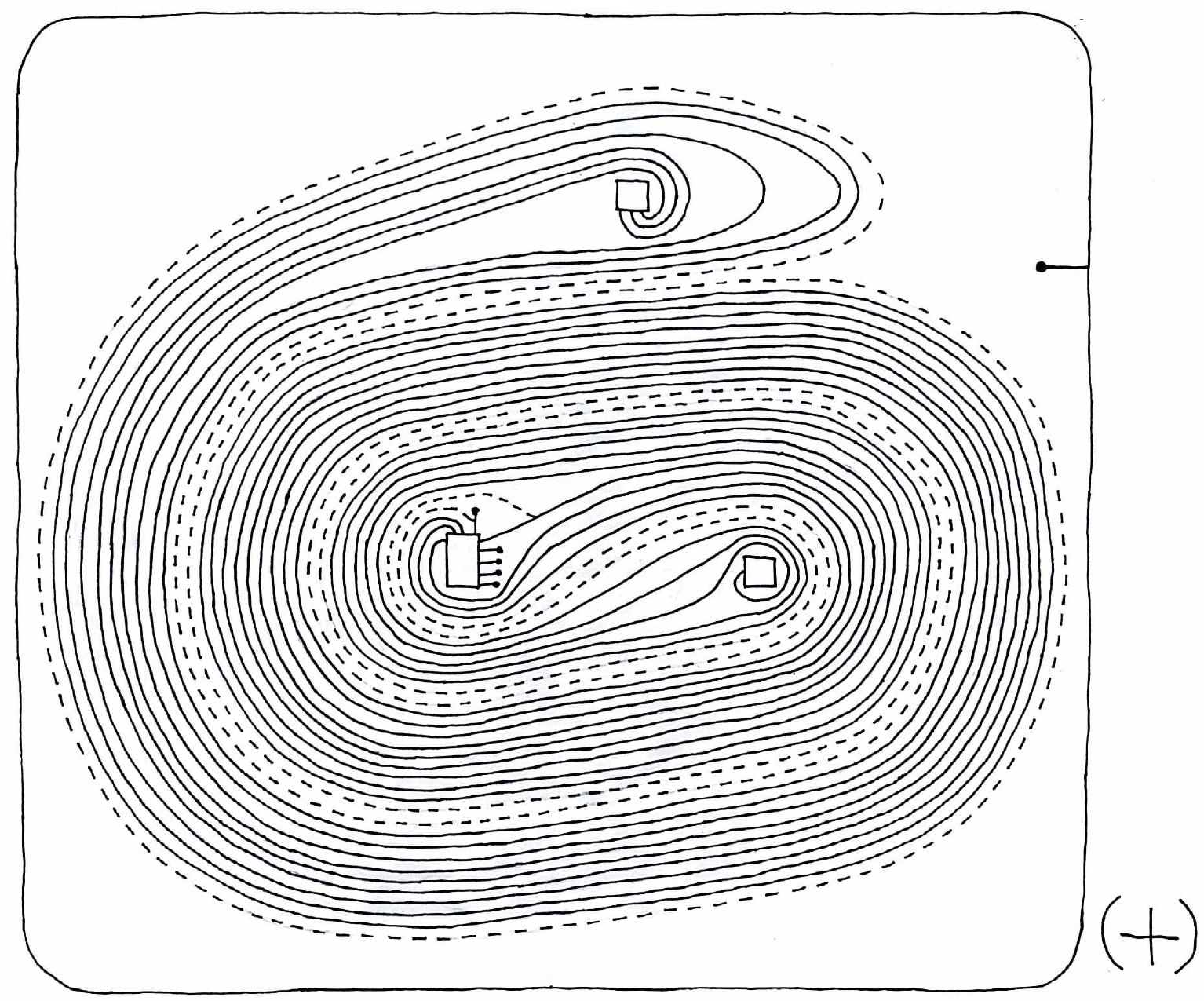} & Shift \(\xi_3\) to have slope \(-14/3\) and positive endpoint \(z_3\), realizing FDTC \(n_2=1\) at \(B_2\).\\
    \includegraphics[scale = 0.9]{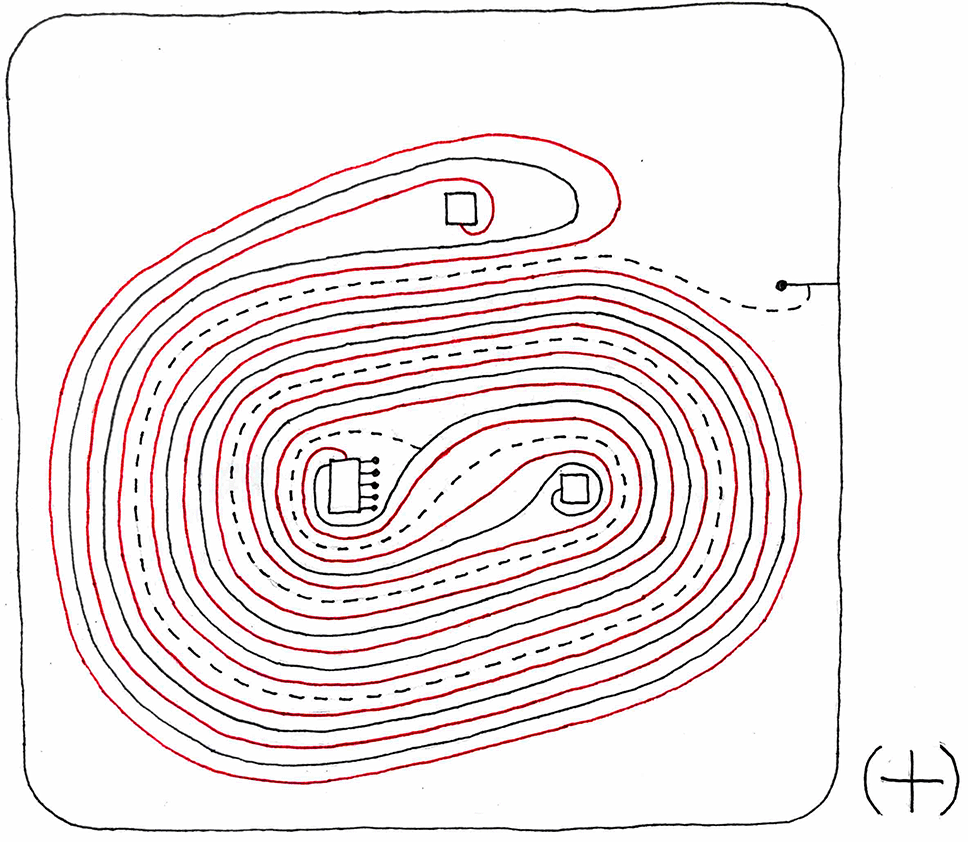} & Shift \(\eta\) to have slope \(-14/3\). (Henceforth we represent \(\xi_1,\xi_2,\xi_3\) with a single red arc.)\\
    \includegraphics[scale = 0.9]{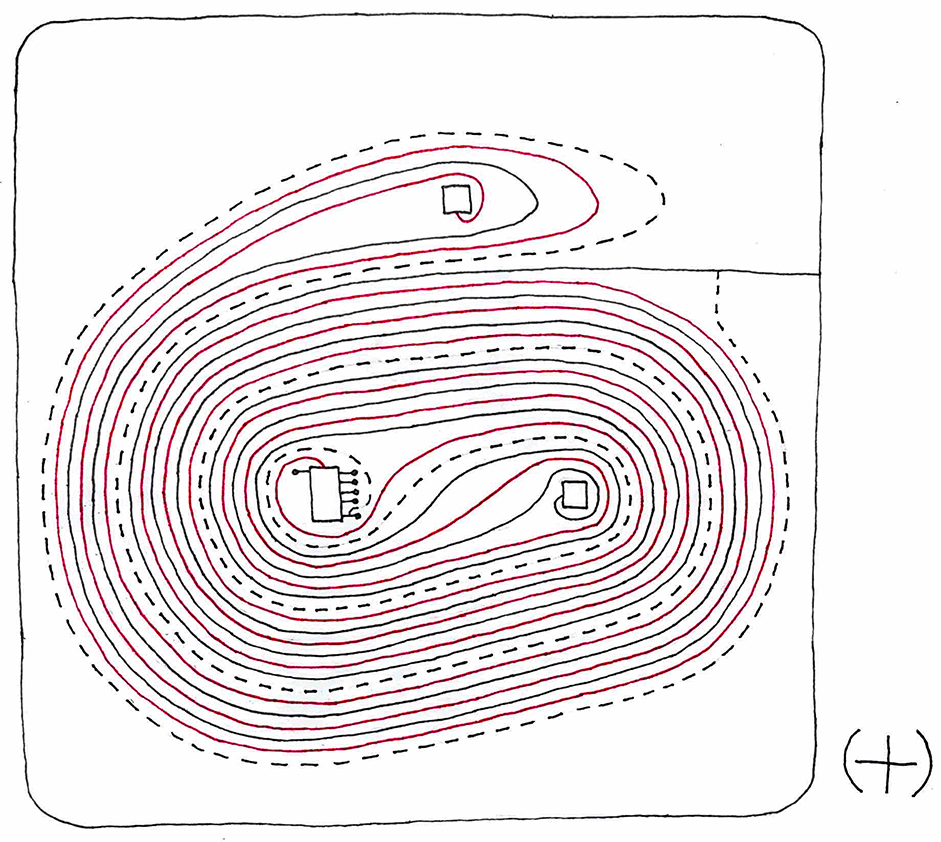} & Shift \(\eta\) to have slope \(-10/4\) and positive endpoint \(w\), realizing FDTC \(n_2=1\) at \(B_2\).\\
    \includegraphics[scale = 0.9]{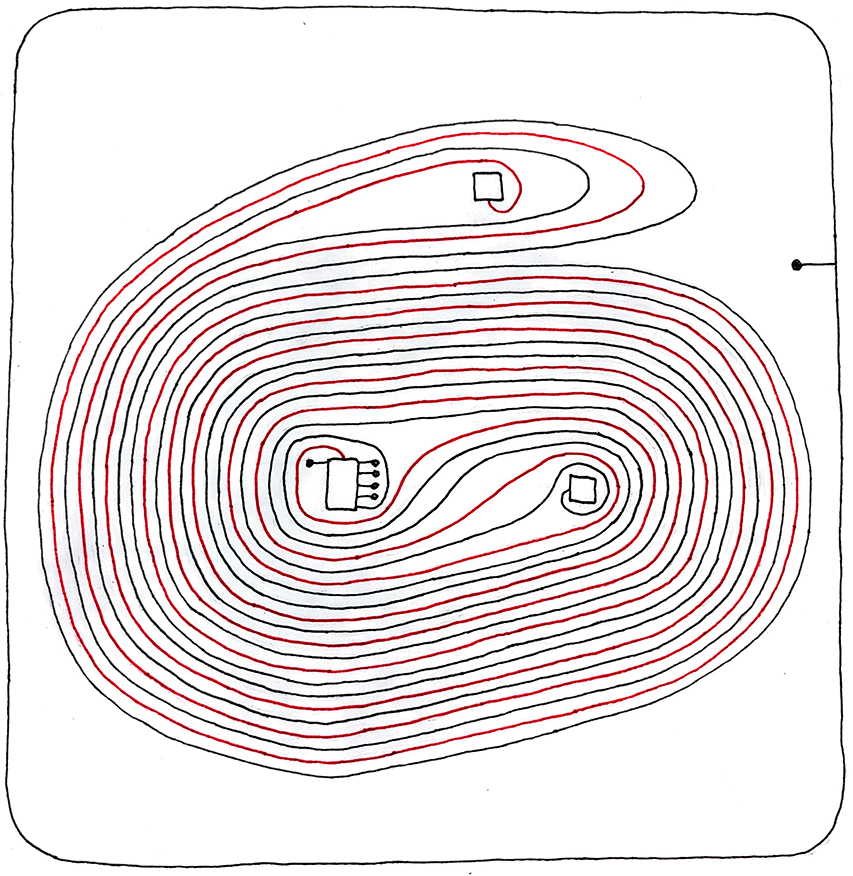} & The page \(\Sigma_{2\pi-\epsilon}\).
    \end{longtblr}
    \captionof{figure}{}
    \label{fig:movie3}
\end{center}

The above movie presentation specifies a transverse overtwisted disk whose \(G_{--}\) graph consists of a single central vertex connected to \(n_2+n_3+1\) leaves. 
\end{proof}
\begin{remark}\label{re:boundary_region_2}
The movie presentation in the proof of (1) is constructed from a boundary-based region depending only on \(n_3\) and \(r_0\). It is an immersed \((2n_3+4)\)-gon. In Figure \ref{fig:boundary_region_2}, we show the boundary-based region in the case \(r_0=3\) and \(n_3=2\).
\begin{figure}[ht]
\centering
    \includegraphics[scale = 0.3]{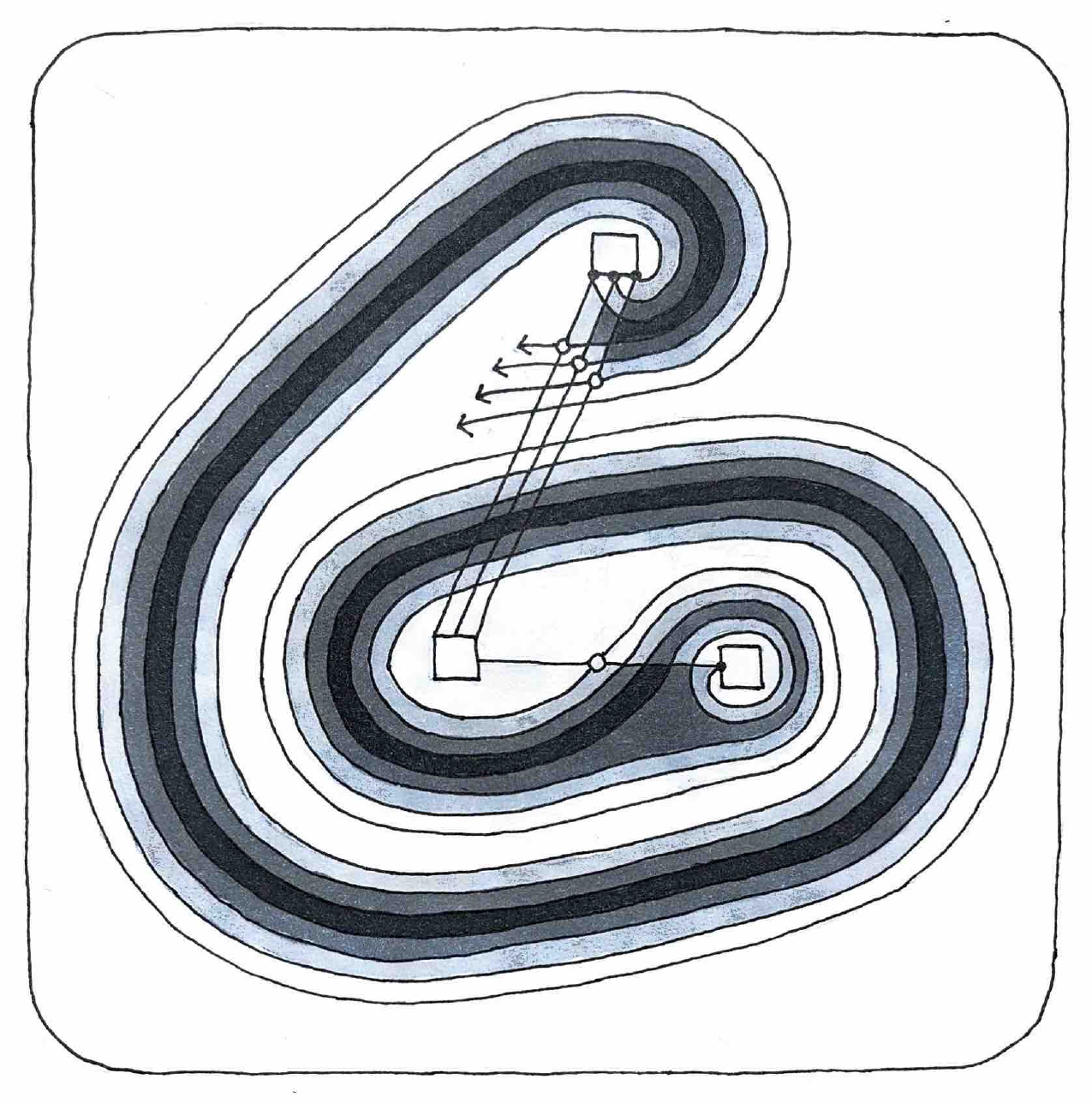}
    \caption{The boundary-based region in the proof of (1) when \(n_3=2\) and \(r_0=3\). It is an immersed octagon. }
    \label{fig:boundary_region_2}
\end{figure}
\end{remark}
Observe that the case \(-p/q=[-3,-3]\) is excluded from Theorem \ref{thm:main_two}. The next proposition shows that in this case, the monodromy is Stein fillable.  
\begin{proposition}\label{prop:tight_two_threes}
Suppose \(f\in \mathrm{Mod}(\Sigma_{0,4},\partial \Sigma_{0,4})\) satisfies
\[\pi(f) = \begin{bmatrix}p' & q'\\ 8 & 3\end{bmatrix}\]
with \(p',q'\geq 0\). If the minimum FDTC of \(f\) is \(\geq 1\), then \(f\) factors into positive Dehn twists. 
\end{proposition}
\begin{proof}
In this case, 
\[\pi(f) = \begin{bmatrix} 16m + 11 & 6m + 4\\ 8 & 3 \end{bmatrix}\qquad\text{for some }m\geq 0.\]
It follows that \(f\) factors as 
\[f = \tau_{a_1}^{n_1}\tau_{a_2}^{n_2}\tau_{a_3}^{n_3}\tau_{a_4}^{n_4}\tau_b^{m+1}\tau_e \qquad\text{for some }n_1,n_2,n_3,n_4,\]
where \(b\) is as in Figure \ref{fig:square_holes} and \(e\) be an embedded circle of slope \(-2\) as shown in Figure \ref{fig:slope_minus_two}.
\begin{figure}[ht]
\centering
    \includegraphics[scale = 0.2]{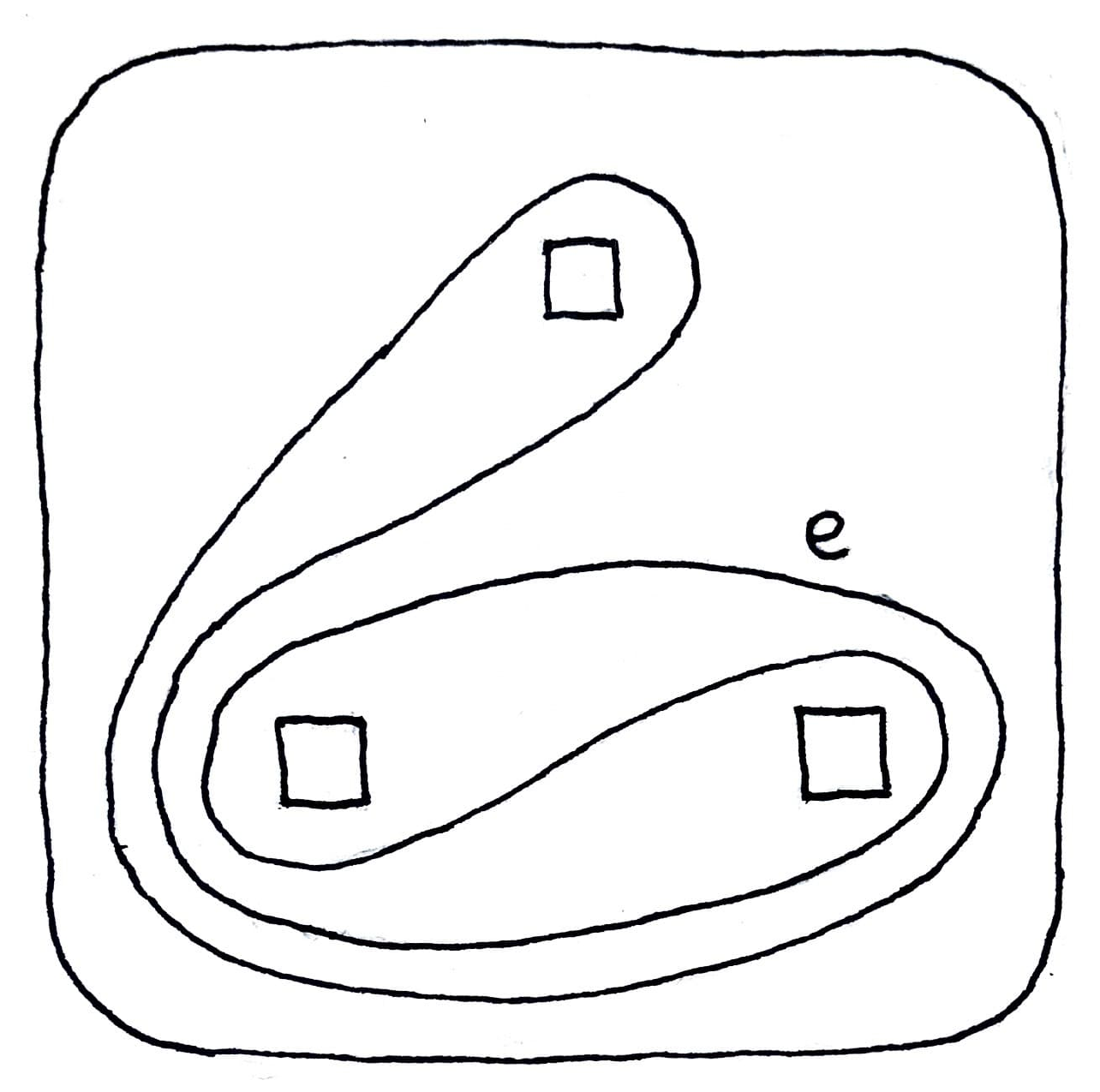}
    \caption{}
    \label{fig:slope_minus_two}
\end{figure}

 One can check that all the FDTCs of \(\tau_b^{m+1}\tau_e\) are equal to \(1\) by probing with arcs. However, this is a bit tedious, so we opt for a different method. Both \(\tau_{b}^{m+1}\) and \(\tau_e\) have all FDTCs equal to \(0\). Using that FDTCs are a quasi-morphism with defect \(1\) \cite[Corollary 4.17]{ItoKawamuroFDTC}, it follows that the FDTCs of \(\tau_b^{m+1}\tau_e\) all lie in \(\{-1,0,1\}\). On the other hand, \(\tau_b^{m+1}\tau_e\) is a right-veering pseudo-Anosov homeomorphism, so must have all FDTCs strictly positive \cite[Proposition 3.1]{HKMRVI}. We conclude  all the FDTCs of \(\tau_b^{m+1}\tau_e\) are \(1\). Accordingly, the FDTCs of \(f\) are \(n_1+1,\ldots,n_4+1\). In particular, \(n_1,\ldots,n_4\geq 0\), so \(f\) factors into positive Dehn twists. 
\end{proof}
\begin{remark}
The author does not know if the condition on FDTCs in (2) of Theorem \ref{thm:main_two} can be removed.
\end{remark}
\begin{remark}
To find the overtwisted monodromies that appear in Theorem \ref{thm:main_two}, the author was helped by the Sage program \texttt{hf-hat-obd} from \cite{computing}. More specifically, the author inputted various monodromies with FDTCs all equal to 1 into \texttt{hf-hat-obd}. When the continued fraction decomposition of \(-p/q\) is as in Remark \ref{re:two_term}, the program \texttt{hf-hat-obd} found an annulus domain which kills the Heegaard Floer contact invariant. (This domain is more-or-less the immersed image of the boundary-based region discussed in Remark \ref{re:boundary_region_2} when \(n_3 =1\).) Using this annulus, the author found the movie presentations used in the proof of Theorem \ref{thm:main_two}. 

The program \texttt{hf-hat-obd-nice} from the same paper indicated that of the remaining right-veering pseudo-Anosov mondoromies on \(\Sigma_{0,4}\), there are both ones with vanishing and non-vanishing Heegaard Floer invariant. Some data is included in Appendix \ref{sec:data}.
\end{remark}
\begin{question}
Which right-veering pseudo-Anosov monodromies on \(\Sigma_{0,4}\) not covered by Theorems \ref{thm:main_one} and \ref{thm:main_two} are overtwisted?
\end{question}
\part{The contact invariant for reducible monodromies}
\section{Proof of Theorem \ref{thm: Lekili}} 
The rest of the paper is dedicated to reproving Theorem \ref{thm: Lekili} using bordered contact invariants. Henceforth we take \(f\in \mathrm{Mod}(\Sigma_{0,4},\partial \Sigma_{0,4})\) to be a reducible monodromy fixing the curve \(b\) in Figure \ref{fig:page}, which we show again in Figure \ref{fig: open book}.
\begin{figure}[ht]
\centering
    \includegraphics[scale = 0.15]{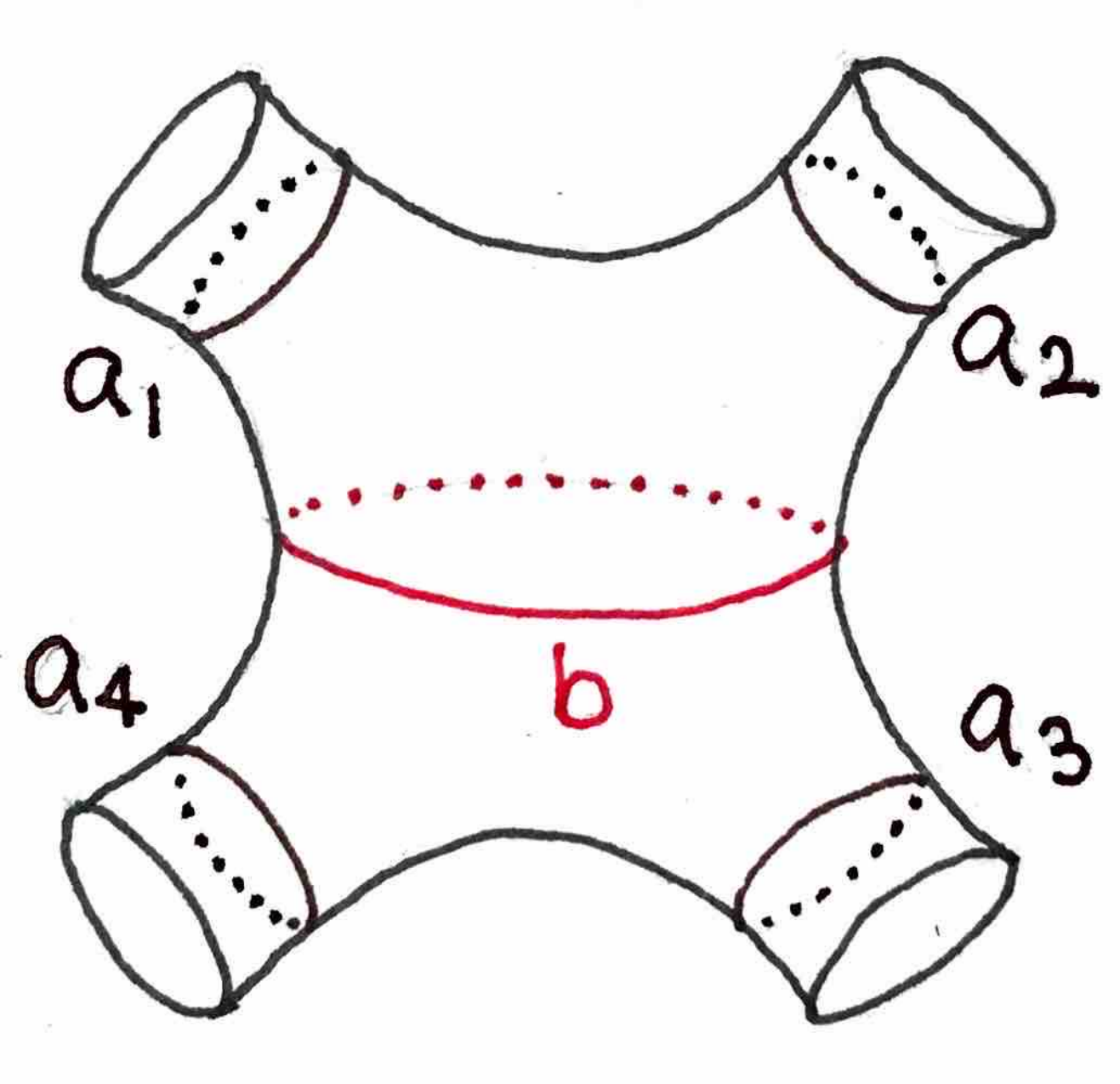}
    \caption{}
    \label{fig: open book}
\end{figure}
Hence \(f\) takes the form shown in (\ref{eq: monodromy}).
\begin{equation}f = \tau_{a_1}^{n_1}\tau_{a_2}^{n_2}\tau_{a_3}^{n_3}\tau_{a_4}^{n_4}\tau_b^{n_b}\qquad\text{for some }n_1,n_2,n_3,n_4,n_b.\label{eq: monodromy}\end{equation}
\begin{proposition}
\label{prop: Stein fillable}
The monodromy {\normalfont (\ref{eq: monodromy})} is Stein fillable if 
\[\min \{n_1,n_2,n_3,n_4\} \geq \max\{-n_b,0\}.\]
\end{proposition}
\begin{proof} If \(n_b\geq 
0\), then the monodromy 
\[f = \tau_{a_1}^{n_1} \tau_{a_2}^{n_2} \tau_{a_3}^{n_3} \tau_{a_4}^{n_4} \tau_b^{n_b}\]
is evidently factored into positive Dehn twists, so the corresponding contact structure is Stein fillable. If \(n_b<0\), using the lantern relation we can write
\begin{align*}f &= \tau_{a_1}^{n_1} \tau_{a_2}^{n_2} \tau_{a_3}^{n_3} \tau_{a_4}^{n_4} =\tau_{a_1}^{n_1+n_b} \tau_{a_2}^{n_2+n_b}\tau_{n_3}^{n_3+n_b}\tau_{a_4}^{n_4 + n_b}(\tau_b^{-1} \tau_{a_1}\tau_{a_2}\tau_{a_3}\tau_{a_4})^{-n_b} 
\\&= \tau_{a_1}^{n_1 +n_b}\tau_{a_2}^{n_2 +n_b}\tau_{a_3}^{n_3 + n_b}\tau_{a_4}^{n_4+n_b}(\tau_c\tau_d)^{-n_b}.
\end{align*}
 We see again that \(f\) factors into positive Dehn twists.
\end{proof}
\begin{remark}\label{rem:Lekili}
In \cite{Lekili}, Lekili uses a more intricate argument with the lantern relation to prove that any monodromy of the form 
\[\tau_{a_1}^{n_1}\tau_{a_2}^{n_2}\tau_{a_3}^{n_3}\tau_{a_4}^{n_4}\tau_c^{C_k}\tau_b^{B_k}\cdot\ldots\cdot\tau_c^{C_1}\tau_b^{B_1}\]
factors into positive Dehn twists if
\[\min\{n_1,n_2,n_3,n_4\}\geq \sum_{i=1}^k\max \{-B_i,-C_i,0\}.\]
Any monodromy in the relative mapping class group \(\mathrm{Mod}(\Sigma,\partial \Sigma)\) can be put uniquely into the above form. Using this fact together with Baldwin's capping off result, Lekili shows that any monodromy of the form
\[\tau_{a_1}^{n_1}\tau_{a_2}^{n_2}\tau_{a_3}\tau_{a_4}\tau_b^{n_b}\tau_c^{n_c}\]
is Stein fillable if and only if it has non-vanishing contact invariant. This occurs precisely when 
\[\min \{n_1,n_2,n_3,n_4\} \geq \min\{0,-n_b,-n_c\}.\]
\end{remark}
 Now, if 
\begin{enumerate}[label = (\roman*)]
\item \(\min \{n_1,n_2,n_3,n_4\}<0\) or 
\item \(n_b<0\) and \(\min \{n_1,n_2,n_3,n_4 \}= 0\),
\end{enumerate}
then \(f\) is not right-veering, hence the corresponding contact structure is overtwisted. To prove Theorem \ref{thm: Lekili} it remains to show the following. 
\begin{theorem}\label{thm: vanishing}
Let \(f\) be the monodromy in {\normalfont(\ref{eq: monodromy})}. Then \(f\) has vanishing Heegaard Floer contact invariant whenever
\[0<\min \{n_1,n_2,n_3,n_4\} < -n_b .\]
\end{theorem}
In fact, by the naturality of the contact invariant under \((-1)\) Legendrian surgery it suffices to prove the contact invariant vanishes whenever 
\[0<\min \{n_1,n_2,n_3,n_4\} =-n_b-1,\]
but the general case will not be much harder. We prove Theorem \ref{thm: vanishing} in the subsequent sections. 

\section{Splitting the open book}\label{sec: splitting}
Given a compact surface \(\Sigma\) with non-empty boundary and a monodromy  \(f\in\mathrm{Mod}(\Sigma,\partial \Sigma)\), let us write \(Y(f)\) for the corresponding open book. We implicitly think of this as a contact manifold equipped with a contact structure supported by the open book. We also write \(c(f)\in \widehat{\mathit{HF}}\big({-Y(f)}\big)\) for the Heegaard Floer invariant for this contact structure. Our main strategy going forward is as follows. We may choose a curve \(\sigma\) isotopic to \(b\) which has an annulus neighborhood fixed pointwise by \(f\), as shown in Figure \ref{fig: cutting}.  
\begin{figure}[ht]
\centering
    \includegraphics[scale = 0.15]{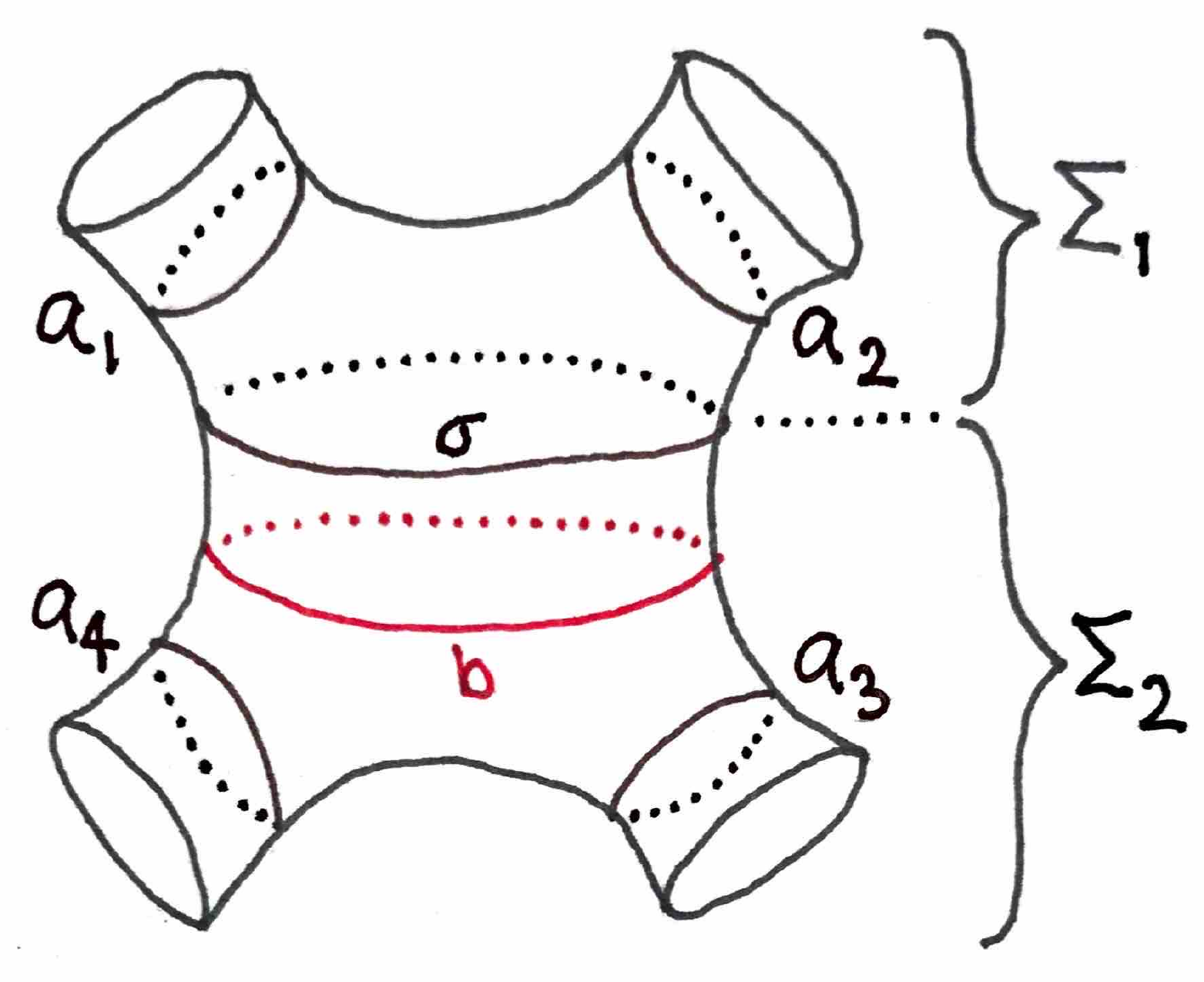}
    \caption{}
    \label{fig: cutting}
\end{figure}

The curve \(\sigma\) traces out a torus in \(Y(f)\). We may simultaneously Legendrian realize \(\sigma\) on each page, so that this torus becomes pre-Lagrangian in the corresponding contact manifold. Although this is a standard procedure, we outline the details below as we will need these details later. 
\begin{proposition}
\label{prop: full}
One can pick a compatible contact structure on \(Y(f)\) for which the curve \(\sigma\) traces out a pre-Lagrangian torus. 
\end{proposition}
\begin{proof}
This is a slight modification of the Thurston--Winkelnkemper construction \cite{ThurstonWinkelnkemper}. Let \(\Sigma\) denote the page. Note \(\sigma\) separates \(\Sigma\) into two compact submanifolds \(\Sigma_1,\Sigma_2\) with \(\Sigma_1\cap \Sigma_2=\sigma\), as shown in Figure \ref{fig: cutting}. We identify collar neighborhoods \[ A\cong [-2,0]_t\times \partial \Sigma \subset \Sigma,\qquad  B\cong [-2,2]_t\times \sigma \subset \Sigma \]
where \(\partial\Sigma\) is identified with \(\{0\}\times \partial \Sigma\), and similarly \(\sigma\) is identified with \(\{0\}\times \sigma\). Additionally, we orient the \([-2,2]_t\) factor of \(B\) so that \(\Sigma_1\cap B= [-2,0]_t\times \sigma\). We choose \(A\) and \(B\) to be small enough so that they are disjoint and \(f\) is the identity on \(A\cup B\). 

Let \(\partial_i\Sigma\) be the component of \(\partial\Sigma\) parallel to \(a_i\). Choose an angle coordinate \(\theta_i\) for \(\partial_i \Sigma\) valued in \(\mathbb{R}/2\pi\mathbb{Z}\cong S^1\) which is consistent with the boundary orientation on \(\partial \Sigma\). Additionally let \(\theta\) be an angle coordinate on \(\sigma\). Orient \(\sigma\) so that the product orientation on \(B\) coincides with the orientation of \(\Sigma\). Let \(S\subset \Omega^1(\Sigma)\) denote the collection of 1-forms \(\beta\) for which
\begin{enumerate}[label=(\roman*)]
\item the 2-form \(d\beta\) is a positive area form on \(\Sigma\),
\item on the collar \([-\frac{3}{2},0]_t\times \partial_i \Sigma\) one has \(\beta = e^t\,d\theta_i\),
\vspace{1pt}
\item on the collar \([-\frac{3}{2},\frac{3}{2}]_t\times \sigma \) one has \(\beta = \frac{1}{10}t\, d\theta\).
\end{enumerate}

It is easy to see that \(S\) is convex. To show \(S\) is non-empty, start out with any 1-form \(\beta_0\in \Omega^1(\Sigma)\) which satisfies 
\begin{enumerate}[label=(\roman*)]
\item \(\beta_0 = e^t d\theta_i\) on \([-2,0]_t\times \partial_i\Sigma\), 
\item  \(\beta_0 = \frac{1}{10}t\,d\theta\) on \([-2,2]_t\times \sigma\). 
\end{enumerate}
Let \(\omega\in \Omega^2(\Sigma)\) be a positive area form which agrees with \(d\beta_0\) on \(A\cup B\), and which also satisfies 
\[\int_{\Sigma_1}\omega = \int_{\Sigma_2}\omega = 4\pi.\]
Note such a 2-form exists, since \(d\beta_0\) is a positive area form on \(A\cup B\) with
\[\int_{(A\,\cup\, B)\,\cap\, \Sigma_i}d\beta_0 <4\pi \qquad\text{for}\qquad i=1,2.\]
Now \(\omega - d\beta_0\) has compact support in the open surface 
\[\Sigma' = \Sigma \setminus \Big(\big([-{\textstyle\frac{3}{2}},0]_t\times \partial \Sigma \big)\cup \big([-{\textstyle\frac{3}{2}},{\textstyle\frac{3}{2}}]_t\times \sigma \big)\Big).\]
Moreover, by Stoke's theorem, for \(i=1,2\)
\[\int_{\Sigma_i} (\omega - d\beta_0) = \int_{\Sigma_i}\omega - \int_{\partial \Sigma_i}\beta_0 =0.\]
Therefore the class \([\omega -d\beta_0]\) is trivial in \(H_c^2(\Sigma';\mathbb{R})\cong \mathbb{R}^2\). Hence there is a compactly supported 1-form \(\beta_1 \in \Omega^1_{c}(\Sigma')\) so that \(\omega - d\beta_0 = d\beta_1\). Then 
\[\beta:=\beta_0+\beta_1\in S.\]\par
As \(f\) fixes \(A\cup B\) pointwise, we see \(f^*\beta\) also lies in \(S\). Let \(\mu \colon [0,2\pi]\to [0,1]\) be a smooth function which is identically \(0\) near \(0\), and identically \(1\) near \(2\pi\). The form
\[ \mu(\varphi)\beta  + \big(1-\mu(\varphi)\big )f^*\beta\qquad\text{on}\qquad  \Sigma\times [0,2\pi]_\varphi  \]
descends to a 1-form on the mapping torus \[\Sigma(f):= \Sigma \times [0,2\pi]_{\varphi}/(x,2\pi)\sim (f(x),0).\] By abuse of notation, we also call this 1-form \(\beta\). Note the restriction of \(\beta\) to each fiber of \(\Sigma(f)\) lies in \(S\). In particular \(d\varphi\wedge d\beta\) is a positive volume form on \(\Sigma(f)\), so for \(C>0\) sufficiently large the 1-form \(\alpha = \beta + Cd\varphi\) will be a contact form on \(\Sigma(f)\). Note \(d\alpha\) restricts to a positive area form on each fiber of \(\Sigma(f)\). Moreover, with respect to this contact form, the torus traced out by \(\sigma\) is pre-Lagrangian. 

It remains to extend \(\alpha\) from \(\Sigma(f)\) to \(Y(f)\). We think of \(Y(f)\) as a quotient of 
\[ \Sigma(f) \sqcup \coprod_{i=1}^4 (\partial_i \Sigma \times D_2^2).\]
Above, \(D_2^2\) denotes the closed \(2\)-disk of radius \(2\). If \((r,\varphi)\) are polar coordinates on \(D_2^2\), then we identify \[\partial_i \Sigma \times (D_2^2 - \Int D^2)\qquad\text{and}\qquad [-1,0]_t\times \partial_i \Sigma \times S^1_\varphi \subset \Sigma(f)\] 
by setting the point \((\theta_i,r,\varphi)\) in the LHS equal to the point \((1-r,\theta_i,\varphi)\) in the RHS. To extend the contact form \(\alpha\) over the solid tori, we make the ansatz 
\[\alpha = h_1(r)\,d\theta_i + h_2(r)\,d\varphi\qquad\text{on}\qquad  \partial_i\Sigma \times D_2^2.\]
In order for \(\alpha\) to satisfy the properties we would like, we insist the following:
\begin{enumerate}[label = (\roman*)]
\item \(h_1(r) = e^{1-r}\) and \(h_2(r) = C\) for \(r\in [1,2]\),
\item \(h_1(r)=2-r^2\) and \(h_2(r) =r^2\) near \(r=0\) *so that \(\alpha\) is non-singular on \(\partial_i\Sigma \times \{0\}\subset \partial_i\Sigma\times D_2^2\) and that \(\partial_i\Sigma\times \{0\}\) is a positively transverse knot),
\item the determinant
\[\det \begin{bmatrix}h_1(r) &  h_1'(r)\\ h_2(r) & h_2'(r)\end{bmatrix}\]
never vanishes for \(r>0\) (so that \(\alpha\) is a contact form),
\item \(h_1'(r)<0\) for \(r>0\) (so that \(d\alpha\) restricts to a positive area form on the interior of each page).
\end{enumerate}
\begin{figure}[ht]
\centering
    \includegraphics[scale =0.2]{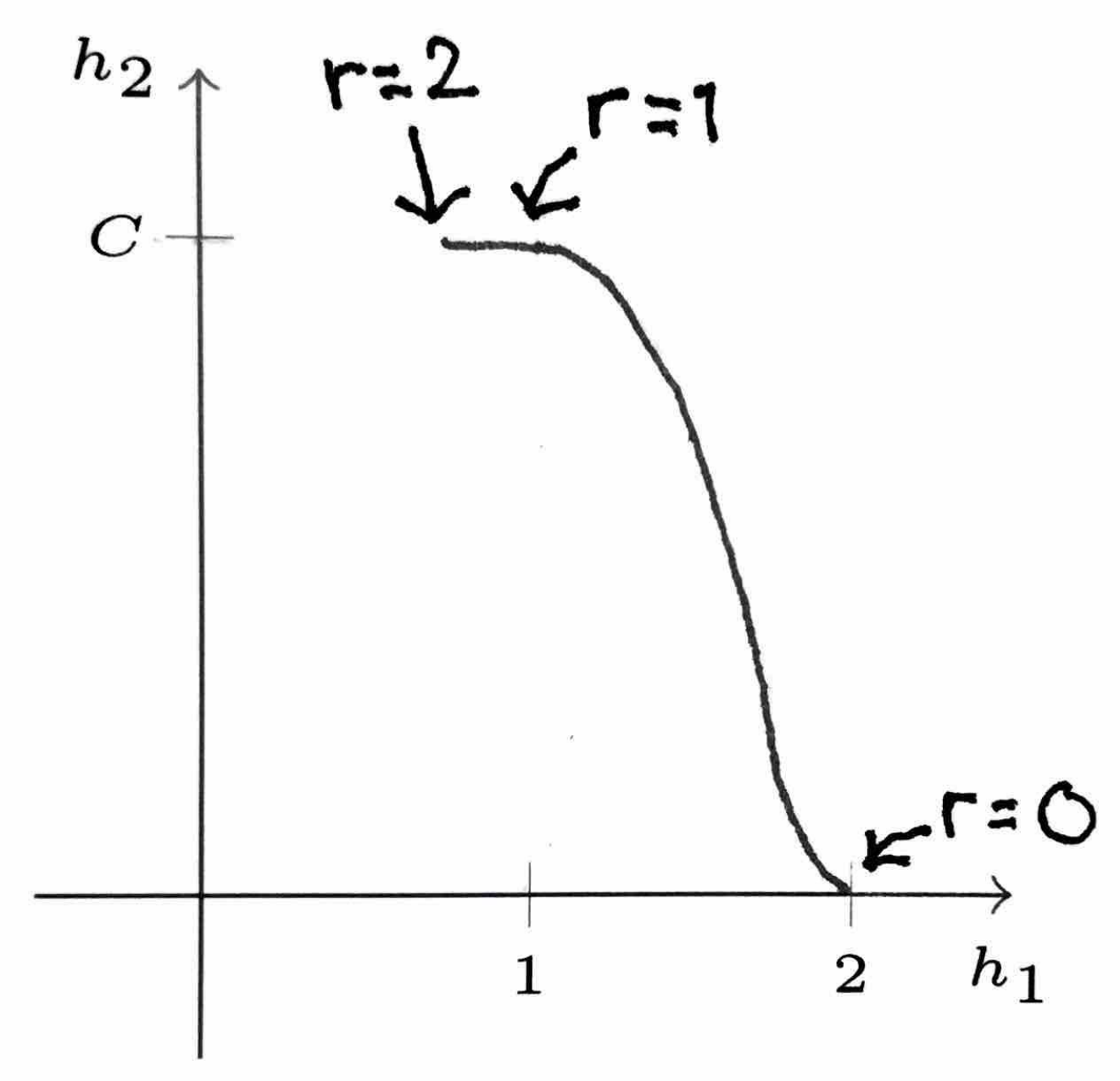}
    \caption{}
    \label{fig: binding extension}
\end{figure}
These conditions ensure \(\alpha\) descends to a smooth contact form on \(Y(f)\) which is supported by its open book structure. We draw the parametric plot \(r\mapsto \big(h_1(r),h_2(r)\big)\) in Figure \ref{fig: binding extension}
\end{proof}

Perturb the pre-Lagrangian torus above to be convex with two dividing curves and then cut to produce contact manifolds with boundary, \(Y_1,Y_2\), so that \(Y_1\) contains the binding components parallel to \(a_1\) and \(a_2\). We now glue a solid torus on \(Y_i\) in such a way that the resulting closed manifold has a natural open book decomposition. For each \(n,m\in \mathbb{Z}\), consider the monodromy
\[g_{n,m}=\tau_\gamma^n \tau_\delta^m\]
on a pair-of-pants, where \(\gamma\) and \(\delta\) are the curves shown in Figure \ref{fig: pair-of-pants}. 
\begin{figure}[ht]
\centering
    \includegraphics[scale =0.12]{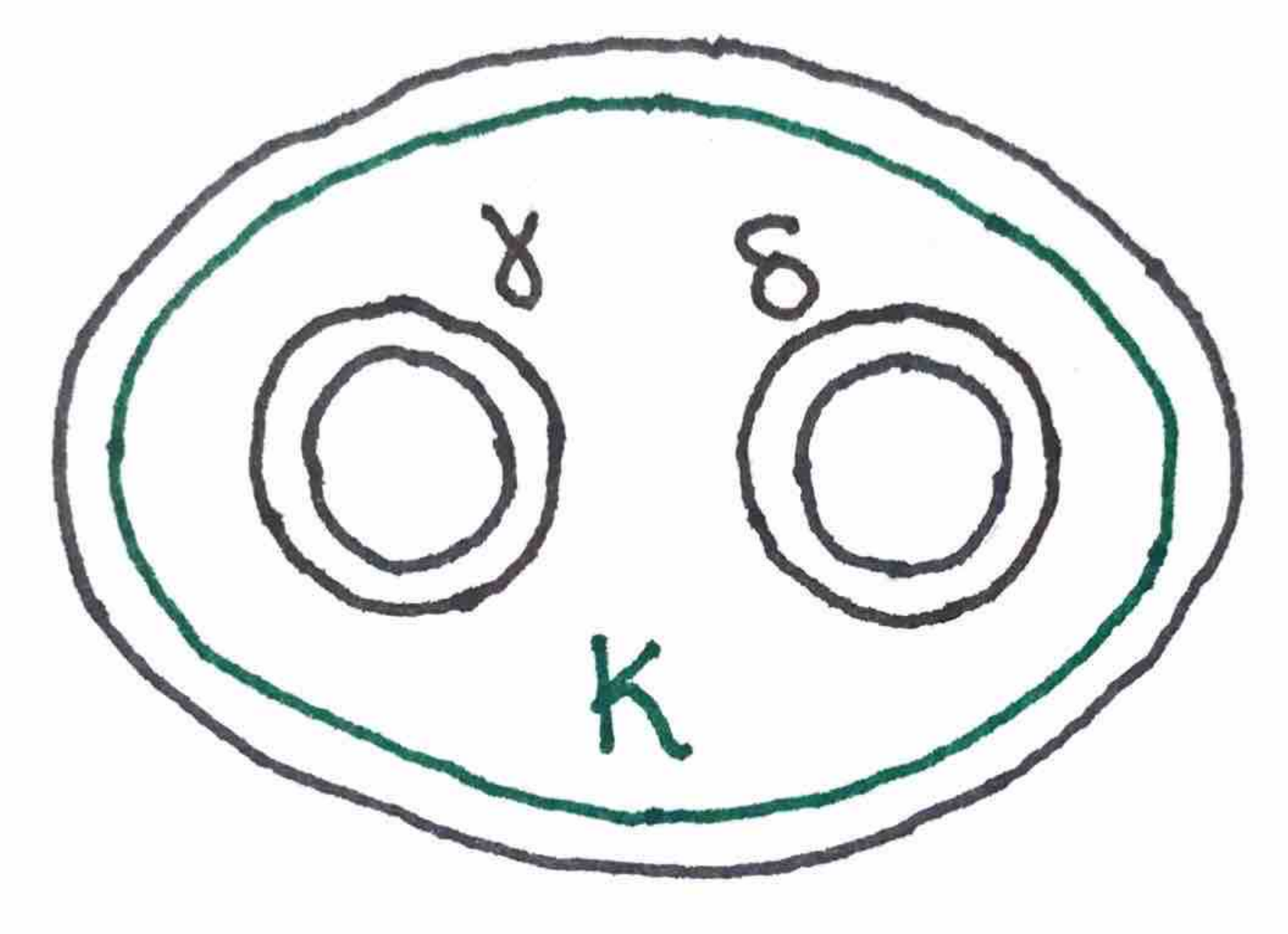}
    \caption{}
    \label{fig: pair-of-pants}
\end{figure}
Topologically, 
\[Y(g_{n,m})\cong L(n,-1)\# L(m,-1).\]\par
In \(Y(g_{n,m})\), Legendrian realize the curve \(K\) shown above on a page. (This is possible as the page has binding components other than the one parallel to \(K\).) Take the complement of a standard tubular neighborhood of \(K\) in \(Y(g_{n,m})\), and call the resulting contact manifold-with-boundary \(Y(n,m)\). 
\begin{proposition}
\label{prop: half}
As contact manifolds with convex boundary, \(Y_1\cong Y(n_1,n_2)\) and \(Y_2\cong Y(n_3,n_4)\). 
\end{proposition}
\begin{proof} We give \(Y(f)\) the same contact from \(\alpha\) constructed in Proposition \ref{prop: full}. Cut \(Y(f)\) along the pre-Lagrangian torus traced out by \(\sigma\) without perturbing it to be convex. Call the resulting halves \(Y_1',Y_2'\). 
Let \((r,\varphi)\) denote polar coordinates on the disk \(D_2^2\). Consider the quotient of \(Y_1'\sqcup (\sigma\times D_2^2),\) where we identify 
\[\sigma\times (D_2^2- \Int D^2)\qquad \text{and}\qquad [-1,0]_t\times \sigma \times S^1_{\varphi}\subset Y_1'\]
by setting the point \((\theta,r,\varphi)\) in the LHS equal to the point \((1-r,\theta,\varphi)\) in the RHS. 
The resulting manifold has an open book decomposition making it diffeomorphic to \(Y(n_1,n_2)\). We wish to extend the contact form on \(Y_1'\) to \(Y(g_{n_1,n_2})\) in a way which is compatible with the open book structure. \par 
We again make the ansatz 
\[\alpha = h_1(r)\,d\theta + h_2(r)\,d\varphi \qquad \text{on}\qquad \sigma \times D_2^2\]
This time, we require 
\begin{enumerate}[label=(\roman*)]
    \item \(h_1(r) = \frac{1}{10}(1-r)\) and \(h_2(r)=C\) for \(r\in [1,2]\),
    \item \(h_1(r) = 1-r^2\) and \(h_2(r) = r^2\) near \(r=0\),
    \item the determinant 
    \[\det \begin{bmatrix}h_1(r) &  h_1'(r)\\ h_2(r) & h_2'(r)\end{bmatrix}\]
    never vanishes for \(r>0\),
    \item \(h_1'(r)<0\) for \(r>0\). 
\end{enumerate}
See Figure \ref{fig: binding extension one}
\begin{figure}[ht]
\centering
    \includegraphics[scale =0.17]{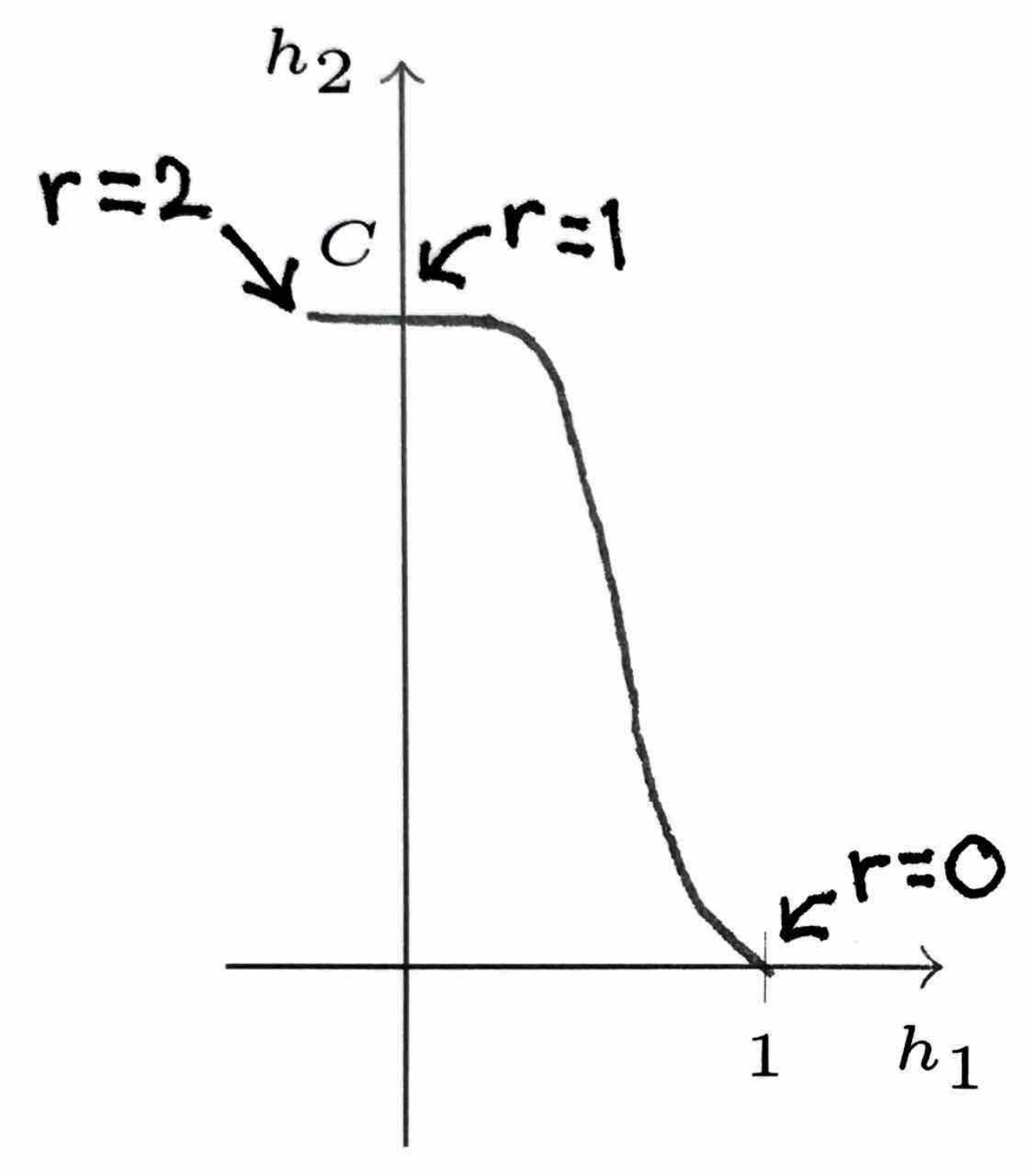}
    \caption{}
    \label{fig: binding extension one}
\end{figure}

In this way, we think of \(Y_1'\) as a contact submanifold of \(Y(g_{n_1,n_2})\). Within the contact manifold \(Y(g_{n_1,n_2})\) perturb \(\partial Y_1'\) to make it convex with two dividing curves. This convex torus bounds \(Y_1\) on one side and a solid torus on the other. Moreover, the dividing curves on the torus are parallel to the boundary of a page. Hence the solid torus is contact isotopic to a standard neighborhood of the Legendrian \(K\) above. This proves the first part of the proposition.

Let \(-\sigma\) denote \(\sigma\) with its reverse orientation. We give \(\sigma\) the angle coordinate \(\overline{\theta}=-\theta\). Consider the quotient of \(Y_2'\sqcup \big((-\sigma)\times D_2^2\big)\) where we identify 
\[(-\sigma)\times (D_2^2 -\Int D^2)\qquad\text{and}\qquad [0,1]_t\times \sigma\times S^1_{\varphi}\subset Y_2'\]
by setting the point \((\overline{\theta},r,\varphi)\) in the LHS equal to the point \((r-1, -n_b\varphi-\overline{\theta},\varphi)\) in the RHS. (We flip the orientation of \(\sigma\) as it is oriented oppositely from \(\partial Y_2'\).) The resulting manifold has an open book decomposition making it diffeomorphic to \(Y(n_3,n_4)\). We wish to extend the contact form on \(Y_2'\)  to \(Y(n_3,n_4)\) in a way which is compatible with the open book structure. \par We make the ansatz 
\[\alpha = h_1(r)\,d\overline{\theta} + h_2(r)\,d\varphi \qquad \text{on}\qquad (-\sigma) \times D_2^2\]
and require (ii) through (iv) as above, but change (i) to 
\begin{enumerate}
\item[(i')] \(h_1(r) = \frac{1}{10}(1-r)\) and \(h_2(r) = C+\frac{1}{10}n_b(1-r)\) for \(r\in [1,2]\).
\end{enumerate}
See Figure \ref{fig: binding extension two}.
\begin{figure}[ht]
\centering
    \includegraphics[scale =0.17]{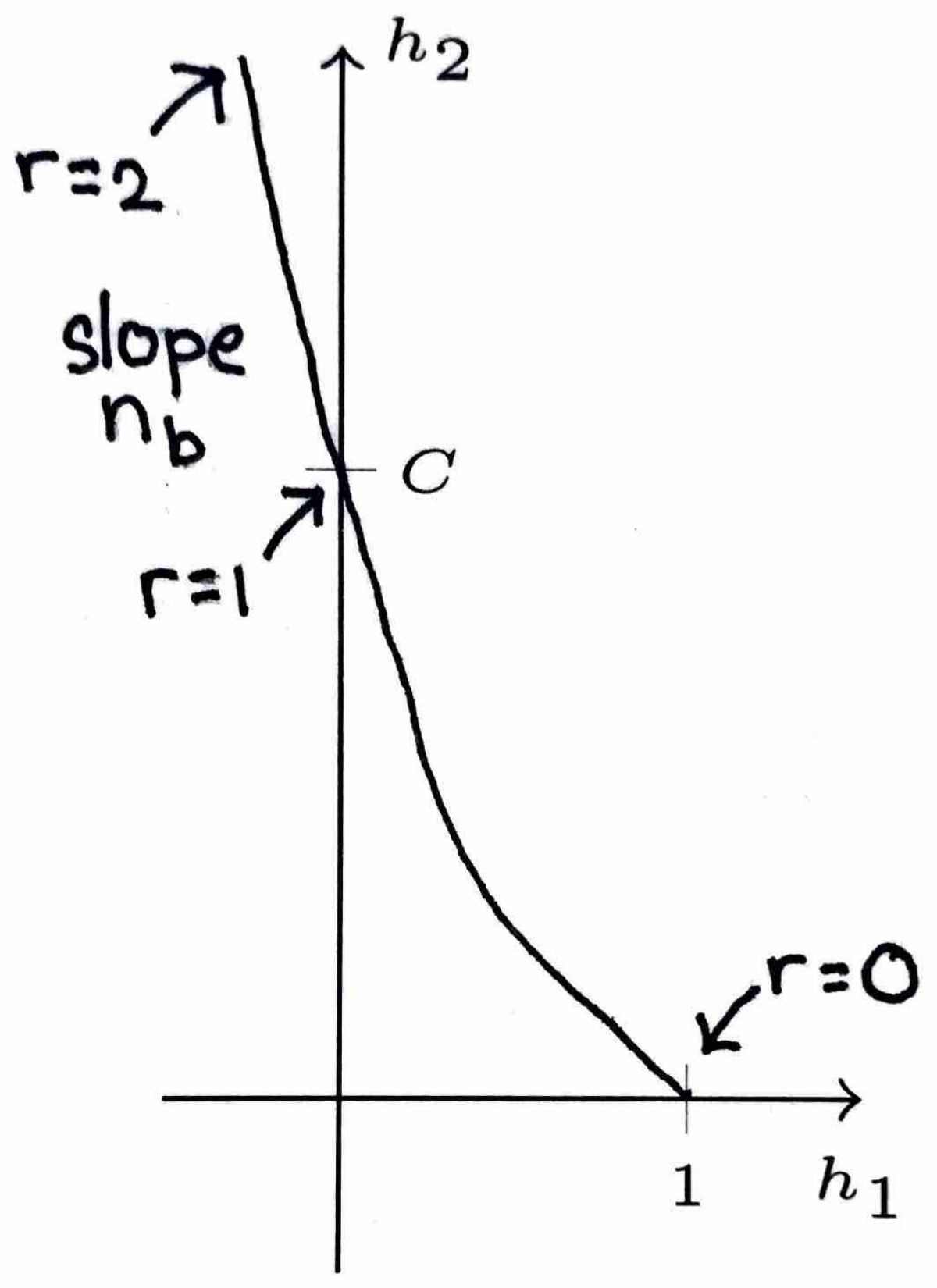}
    \caption{The case when \(n_b<0\).}
    \label{fig: binding extension two}
\end{figure}
We now conclude in the same way as above.
\end{proof}

\begin{remark}
In the case \(n_3,n_4\geq 0\) and \(n_b>0\), the above construction together with the fact that the contact structure on \(Y(n_3,n_4)\) is tight implies that in the proof of Theorem \ref{prop: full} one must take \(C>\frac{1}{10}n_b\) in order for \(\alpha\) to be a contact form. 
\end{remark}

In Section \ref{sec: identifying} we identify the bordered contact invariant for \(Y(n,m)\). In Section \ref{sec: gluing} we then calculate the contact invariant of \(Y(f)\) by pairing the bordered contact invariants.

\section{Identifying the bordered contact invariant}\label{sec: identifying}
In the previous section, we have been vague as to how we orient the page. A pair-of-pants has an orientation-reversing diffeomorphism which restricts to a self-diffeomorphism on each boundary component. We can use such a diffeomorphism to construct an \textit{orientation-preserving} diffeomorphism of \(Y(g_{n,m})\) which maps each page to a page and flips the direction of the Reeb vector field. The induced effect on the contact structure (up to isotopy) is to flip its co-orientation.

We would like to similarly flip the co-orientation of the contact structure on \(Y(n,m)\). To do this in a way so that the two resulting contact structures induce the same (oriented) dividing set on \(\partial Y(n,m)\), we think of \(Y(n,m)\) as being obtained by perturbing the boundary of a contact manifold \(Y_{n,m}'\) with pre-Lagrangian boundary. (This is possible by the proof of Proposition \ref{prop: half}.)

Temporarily, let \(\xi_1\) denote the contact structure on \(Y_{n,m}'\), and let \(\xi_2\) denote the same contact structure with the opposite co-orientation. We now perturb \(Y_{n,m}'\) to have convex boundary with two dividing curves to obtain \(Y(n,m)\). However, we do this in complementary ways for \(\xi_1\) and \(\xi_2\), i.e., where we push the boundary ``outward'' for \(\xi_1\), we push ``inward'' for \(\xi_2\), and vice-versa. The resulting manifolds with boundary can be identified in a natural way. After this identification, \(\xi_1\) and \(\xi_2\) induce the same dividing set on \(\partial Y(n,m)\).

As with \(Y(g_{n,m})\), we can define an orientation-preserving diffeomorphism
\[\Phi\colon Y(n,m)\to Y(n,m)\]
which preserves the sutured structure on \(\partial Y(n,m)\) and so \(\Phi_*(\xi_1)\) is isotopic to \(\xi_2\) through contact structures inducing the same dividing set on \(\partial Y(n,m)\). When restricted to the torus boundary, \(\Phi\) is isotopic to a hyperelliptic involution. Let \(\mathcal{F}\) be a pointed matched circle on \(\partial Y(n,m)\) with one parameterizing arc which agrees with the sutures induced by \(\xi_1,\xi_2\), and let \(\mathcal{F}'\) be the same pointed matched circle
but with the basepoint moved to the diametrically opposite position. The diffeomorphism \(\Phi\) induces an equivalence 
\[\Phi_*\colon {\widehat{\mathit{CFA}}}\big({-Y(n,m),\mathcal{F}}\big)\to {\widehat{\mathit{CFA}}}\big({-Y(n,m),\mathcal{F}'}\big)\]
sending the Type A contact invariant \(c_A(\xi_1, \mathcal{F})\) to \(c_A(\xi_2,\mathcal{F}')\), and \(c_A(\xi_2, \mathcal{F})\) to \(c_A(\xi_1,\mathcal{F}')\). It will to be beneficial to keep track of both \(\xi_1\) and \(\xi_2\). From here on out we will be slightly sloppy; when we refer to \(Y(n,m)\) as a contact manifold, we implicitly mean with respect to one of the contact structures \(\xi_1\) or \(\xi_2\). 

In \(Y(n,m)\), orient the Lagrangian \(K\) in the same way as the binding component parallel to it. (The binding is a positively transverse knot, so the orientation of \(K\) depends on the co-orientation of the contact structure.) We now attach a basic slice to the boundary of \(Y(n,m)\) so that the resulting dividing set consists of meridians of \(K\). If one chooses the sign of the bypass attachment correctly, then by work of Stipsicz--V\'{e}rtesi \cite{StipsiczVertesi} the induced map on sutured Floer homology 
\begin{equation} \mathit{SFH}\big({-Y(n,m)}, -\Gamma_\lambda\big)\to \mathit{SFH}\big({-Y(n,m),-\Gamma_\mu}\big)\cong \widehat{\mathit{HFK}}\big({-Y(g_{n,m}),K}\big)\label{eq: SV}\end{equation}
sends the sutured contact invariant \(\mathit{EH}\big(Y(n,m)\big)\) \cite{HKMsuture} to the LOSS invariant \({\widehat{\mathcal{L}}(K)}\) \cite{LOSS}. Above \(\Gamma_\lambda\) and \(\Gamma_\mu\) are the longitudinal and meridional sutures on \(\partial Y(n,m)\) determined by the Legendrian \(K\). (Note \(K\) is rationally null-homologous, so the LOSS invariant of \(K\) is well-defined.)
\begin{lemma} 
\label{lemma: LOSS}
Depending on the co-orientation for the contact structure on \(Y(g_{n,m})\), the LOSS invariant \(\mathcal{L}(K)\in \mathit{HFK}^-\big({-Y(g_{n,m}),K}\big)\) is represented by the pictured generator in one of the two doubly-pointed Heegaard diagrams in Figure {\normalfont\ref{fig: LOSS}}.
\begin{figure}[ht]
\centering
    \includegraphics[scale =0.2]{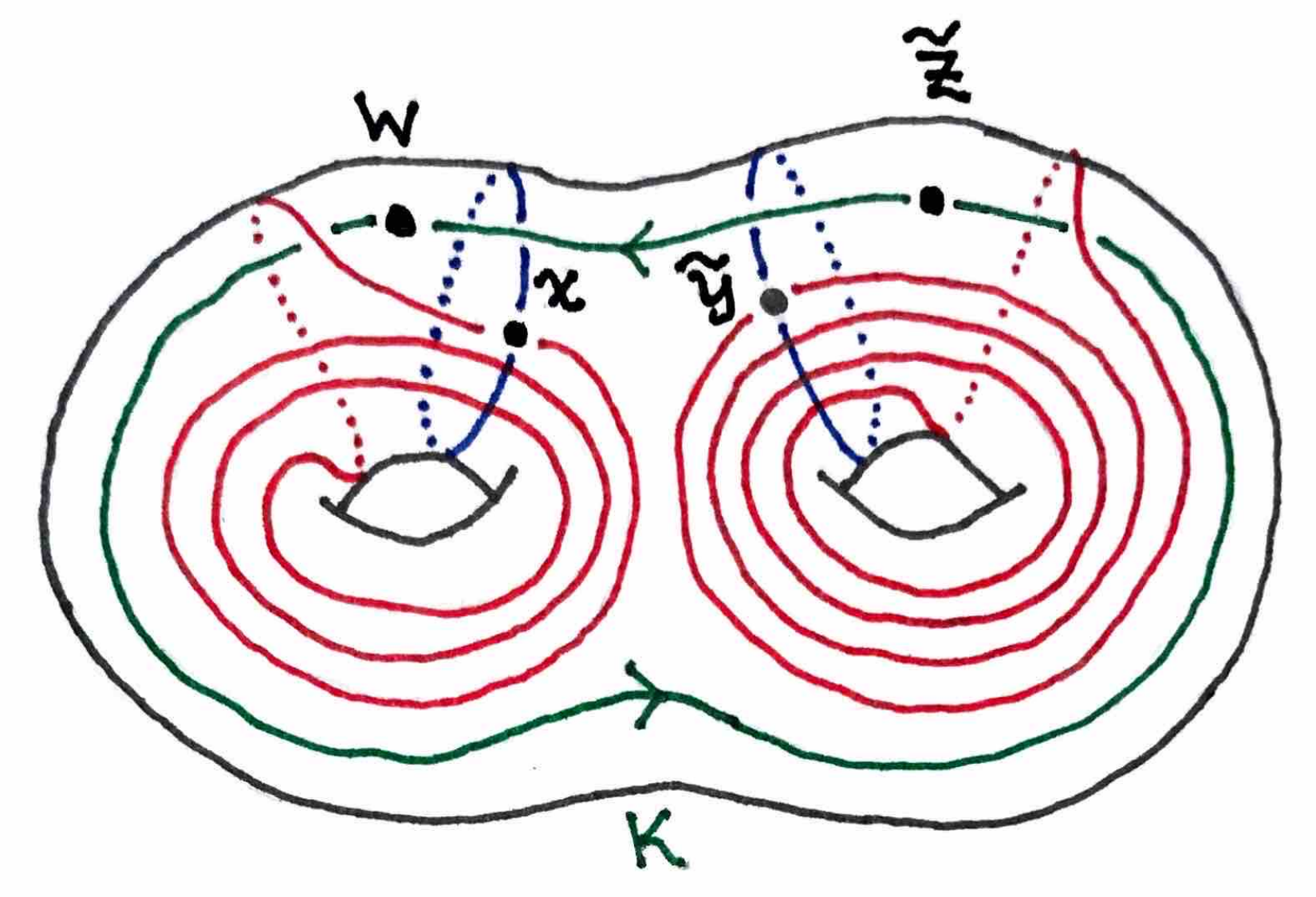}\qquad\qquad\includegraphics[scale =0.18]{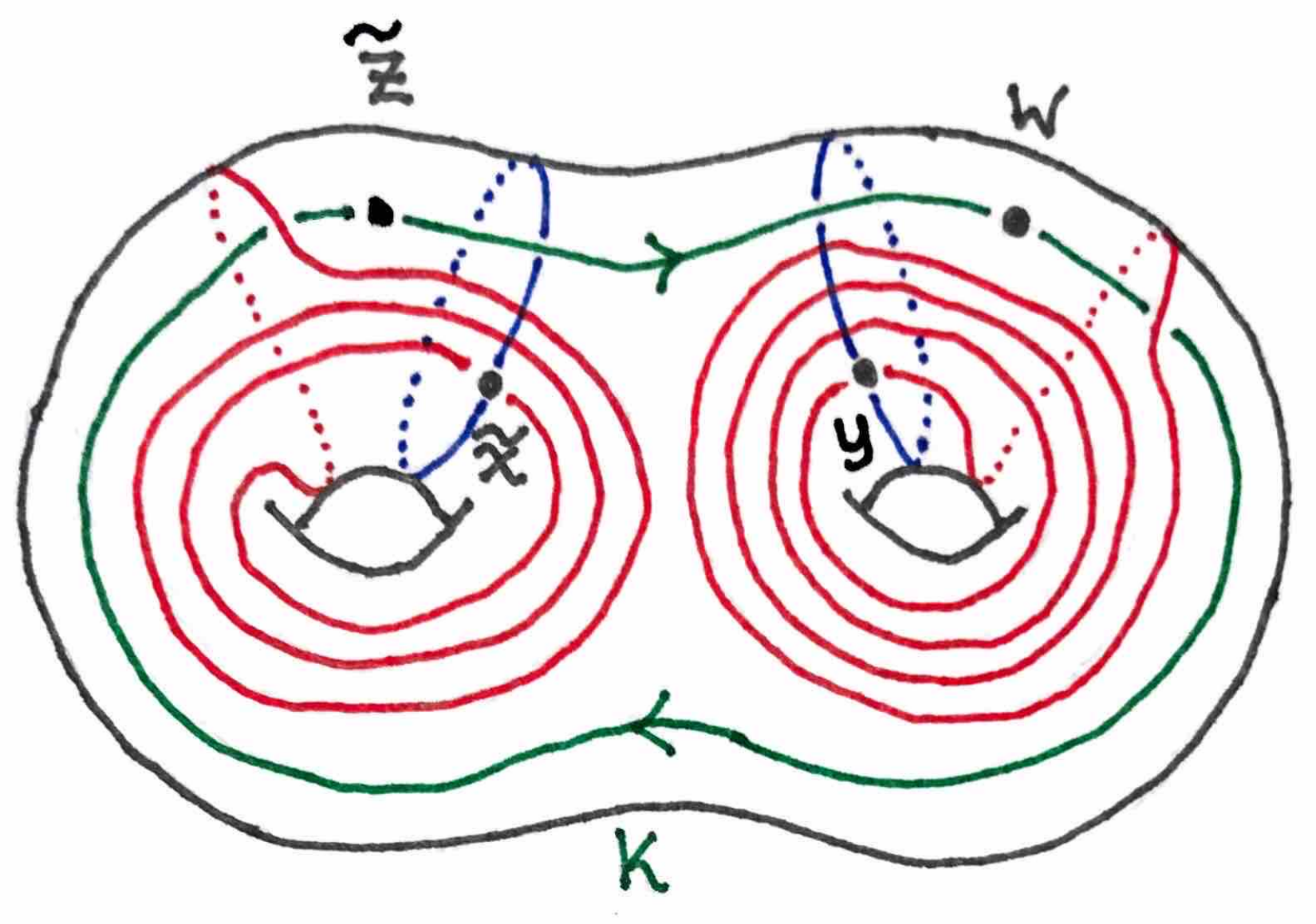}
    \caption{The case \(n=3\), \(m=4\)}
    \label{fig: LOSS}
\end{figure}
We color the curves in these Heegaard diagrams according to the usual convention for \(-Y(g_{n,m})\), instead of \(Y(g_{n,m})\). 
\end{lemma}
\begin{proof}
Instead of directly identifying \(\mathcal{L}(K)\) via the procedure in \cite{LOSS}, we use the fact that the natural map
\begin{equation}\mathit{HFK}^{-}\big({-Y(g_{n,m})},K\big) \to \widehat{\mathit{HF}}\big({-Y(g_{n,m})}\big) \label{eq: forget}\end{equation}
takes the class \(\mathcal{L}(K)\) to \(c(g_{n,m})\). To find \(c(g_{n,m})\), we follow the procedure in \cite{HKM}. This is illustrated for each choice of co-orientation of the contact structure in Figure \ref{fig: singly pointed}. 
\begin{figure}[ht]
\centering
    \includegraphics[scale=1]{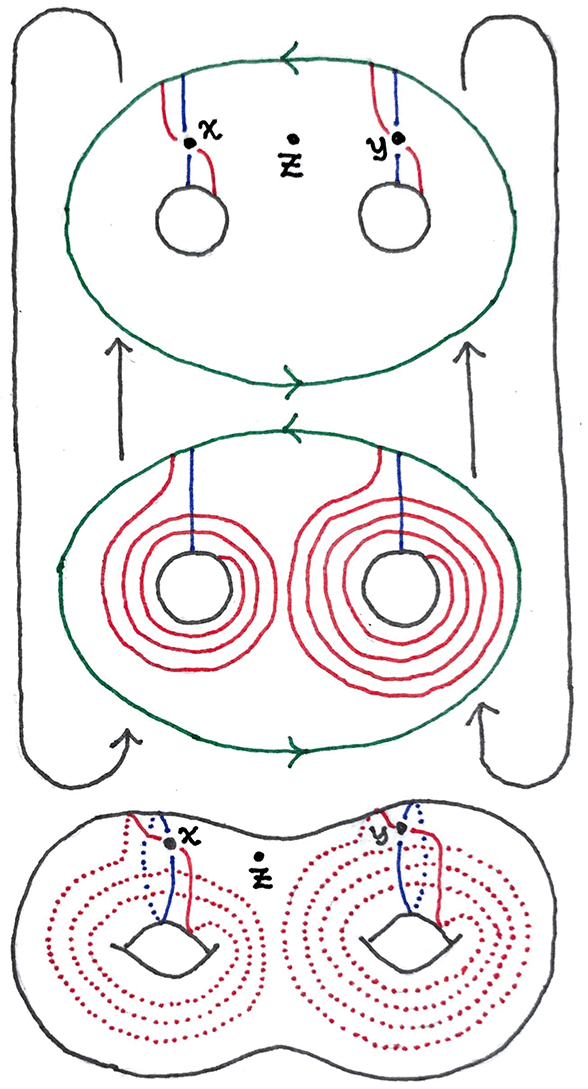}\qquad\qquad\includegraphics[scale=1]{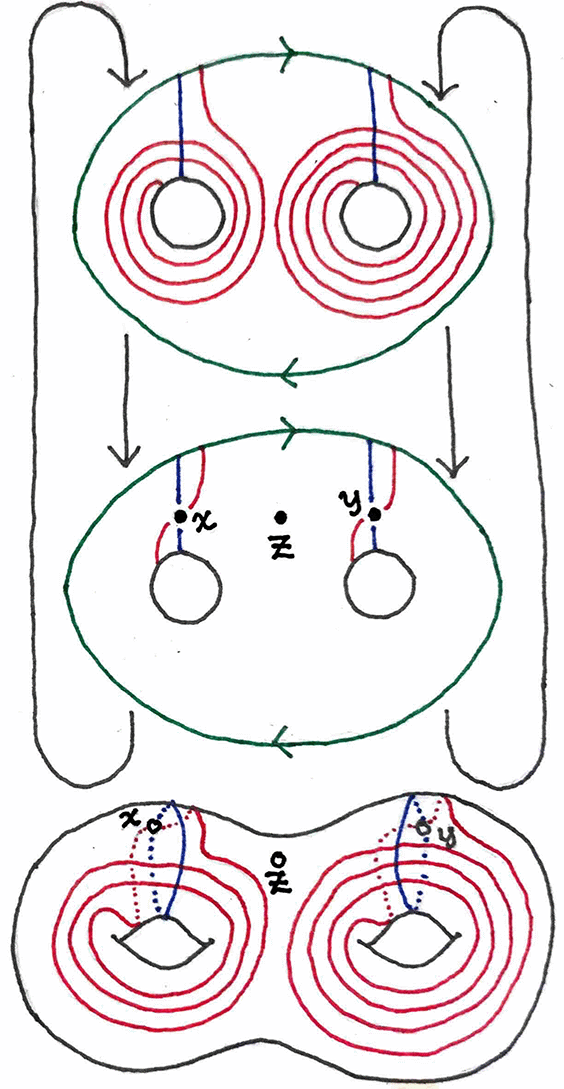}
    \caption{}
    \label{fig: singly pointed}
\end{figure}

In both cases, we can isotope the red \(\alpha\) curves to obtain Figure \ref{fig: isotopy}. To get from the left-hand side of Figure \ref{fig: singly pointed} to \ref{fig: isotopy} we do not ``flip the Heegaard surface upside down''. 
\begin{figure}[ht]
\centering
    \includegraphics[scale=0.2]{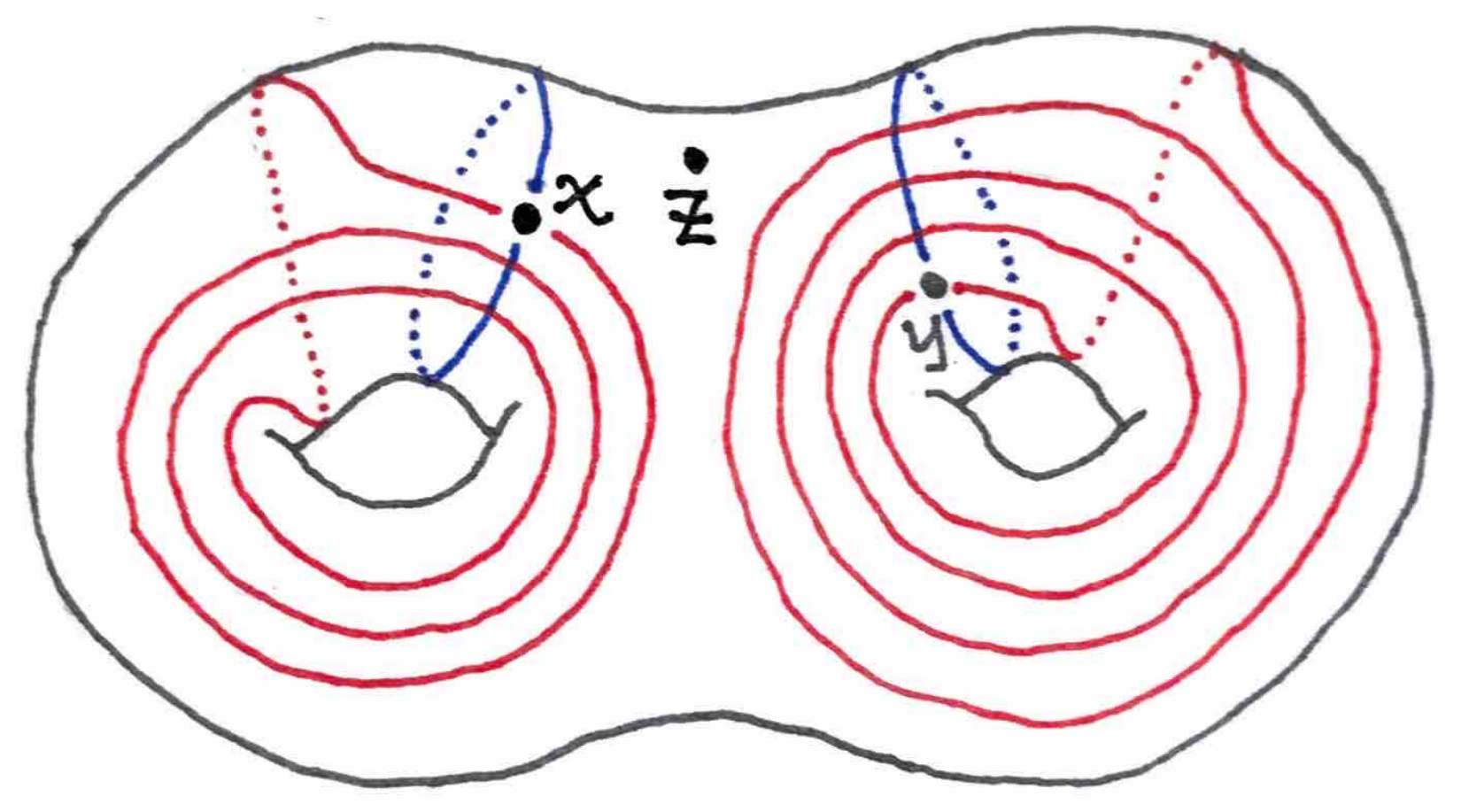}
    \caption{}
    \label{fig: isotopy}
\end{figure}

Each generator in Figure \ref{fig: isotopy} belongs to a different \(\mathrm{spin}^c\) structure. Indeed, \({\mathrm{Spin}^c\big(Y(g_{n,m})\big)}\) is a torsor over the abelian group 
\[H_1\big({Y(g_{n,m})}\big)\cong \langle \mu_1,\mu_2\,|\,n\mu_1 = 0,m\mu_2 = 0\rangle\cong \mathbb{Z}/n\oplus \mathbb{Z}/m.\]
In Figure \ref{fig: isotopy}, \(\mu_1\) and \(\mu_2\) are represented by dual curves to the blue \(\beta\) curves. Using this, one can see each of the \(nm\) generators lives in a distinct \(\mathrm{spin}^c\) class.

The map (\ref{eq: forget}) forgets the basepoint \(w\). Hence in Figure \ref{fig: isotopy} we should move the basepoint \(z\) to \({\widetilde{z}}\) as in Figure \ref{fig: LOSS}. Note this shifts the assignment of generators to \(\mathrm{spin}^c\) structures by \(\mathrm{PD}([\gamma])\), where \(\gamma\) is the curve in Figure \ref{fig: shift}. To compensate for this, we either shift the point \(x\) to \({\widetilde{x}}\) or the point \(y\) to \({\widetilde{y}}\). 
\begin{figure}[ht]
\centering
    \includegraphics[scale=0.18]{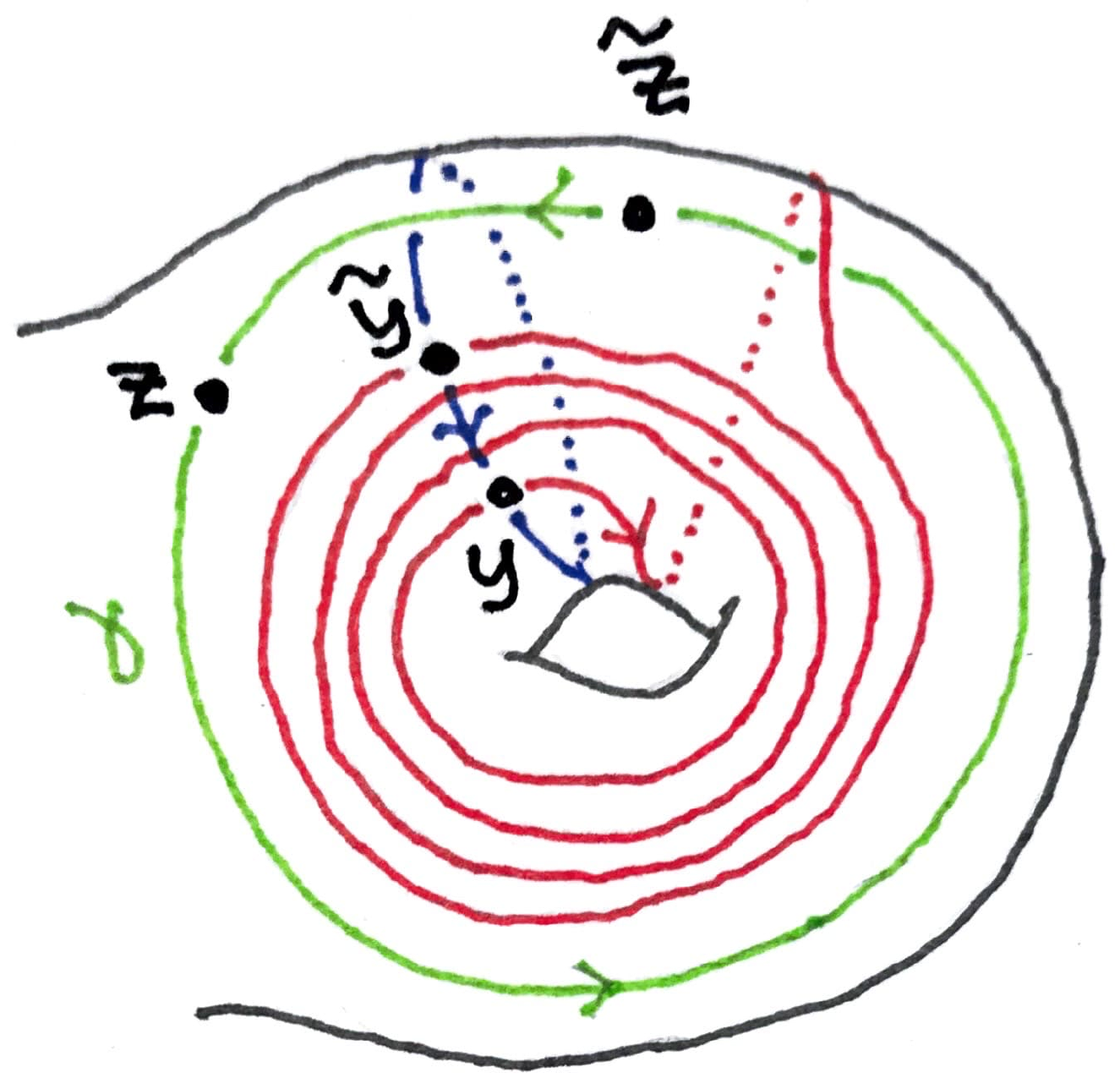}\qquad\qquad\includegraphics[scale=0.18]{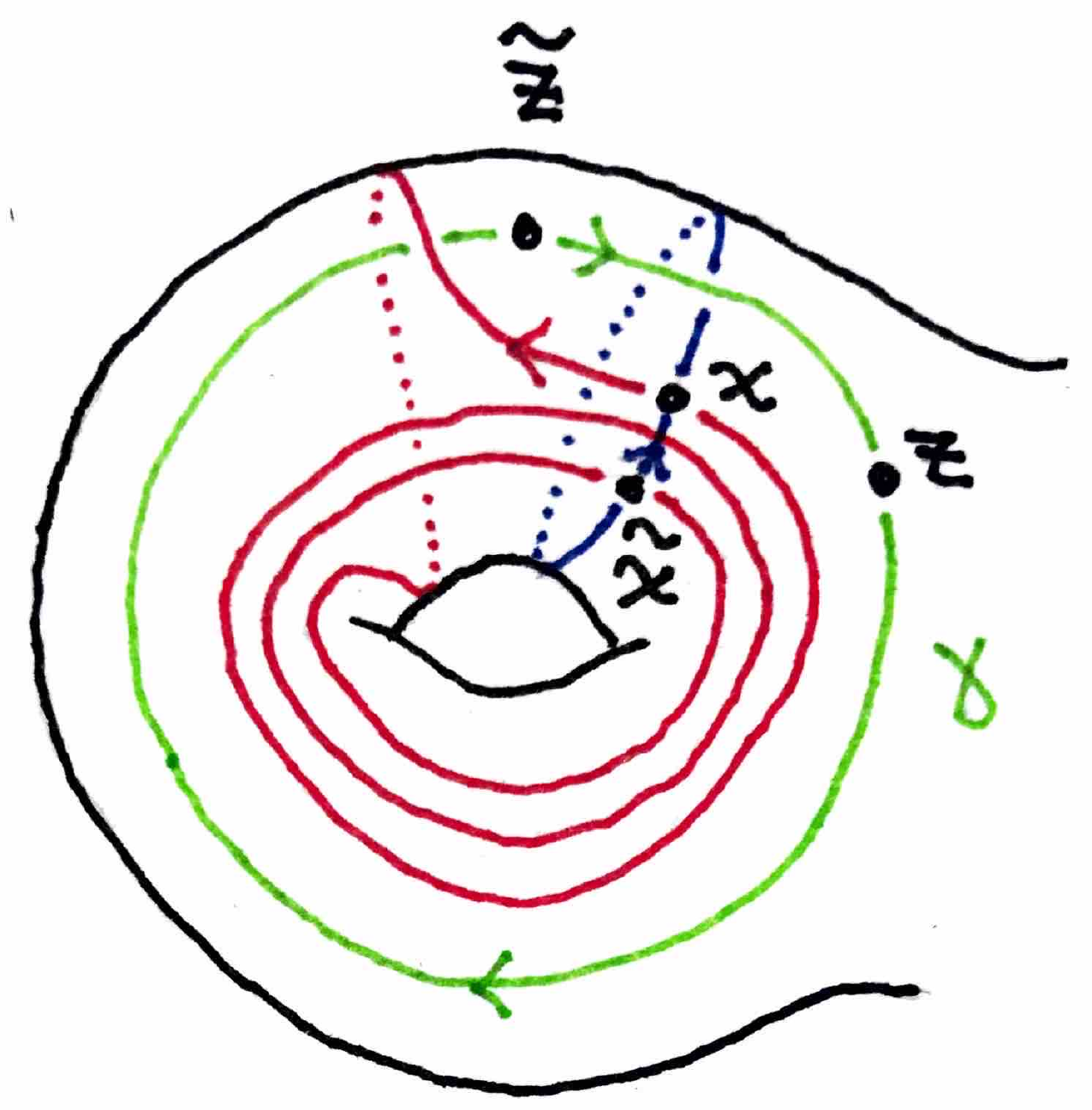}
    \caption{}
    \label{fig: shift}
\end{figure}
\end{proof}
Since the map 
\[\mathit{HFK}^-\big({-Y(g_{n,m}),K}\big)\to \widehat{\mathit{HFK}}\big({-Y(g_{n,m}),K}\big),\]
obtained by setting \(U=0\), sends \(\mathcal{L}(K)\) to \({\widehat{L}(K)}\), we can think of Figure \ref{fig: LOSS} as specifying \[{\widehat{L}(K)}\in \widehat{\mathit{HFK}}\big({-Y(g_{n,m}),K}\big)\cong \mathit{SFH}\big({-Y(n,m),-\Gamma_{\mu}}\big).\]

Before moving to the bordered setting, we refer the reader to Appendix \ref{sec: algebra} for a review of Type A and D modules over the torus algebra (in part to set up notation). In this paper, all Type A modules will be right modules. Except when tensoring a Type A module with a Type DD bimodule, all Type D modules will be left modules.  

\begin{lemma} 
Let \(\mathcal{F}\) be a framing of \(\partial\big({-Y(n,m)}\big)\) with one arc parallel to the boundary of a page, and the other arc parallel to a meridian of \(K\). There are 8 such framings depending on the choice of basepoint and choice of \(\alpha\) or \(\beta\) type. Consider the isomorphism 
\[\widehat{\mathit{HFA}}\big({-Y(n,m),\mathcal{F}}\big)\cdot \iota \to \mathit{SFH}\big({-Y(n,m),-\Gamma_{\mu}}\big)\]
obtained by attaching the appropriate sutured cap {\normalfont\cite{Zarev}}, where \(\iota\) is the corresponding idempotent in the torus algebra {\normalfont(}depending on the basepoint{\normalfont)}. Under this isomorphism, \({\widehat{\mathcal{L}}(K)}\) is represented by either \(\{p,x_1,y_1\}\) or \(\{p,x_n,y_m\}\) in one of the bordered diagrams in Figure {\normalfont\ref{fig: bordered}}.
\begin{figure}
\centering
    \includegraphics[scale=0.2]{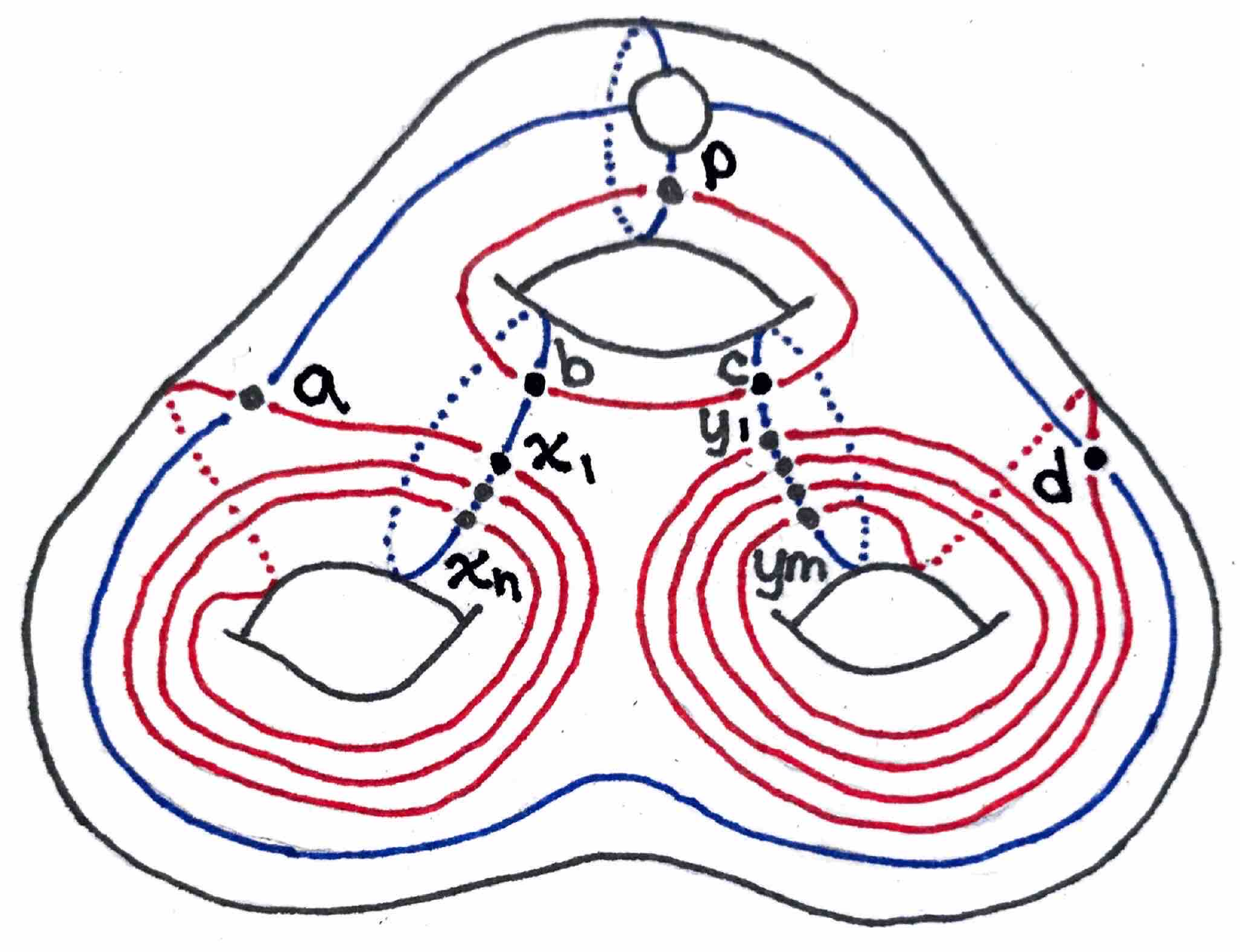}\qquad \includegraphics[scale=0.2]{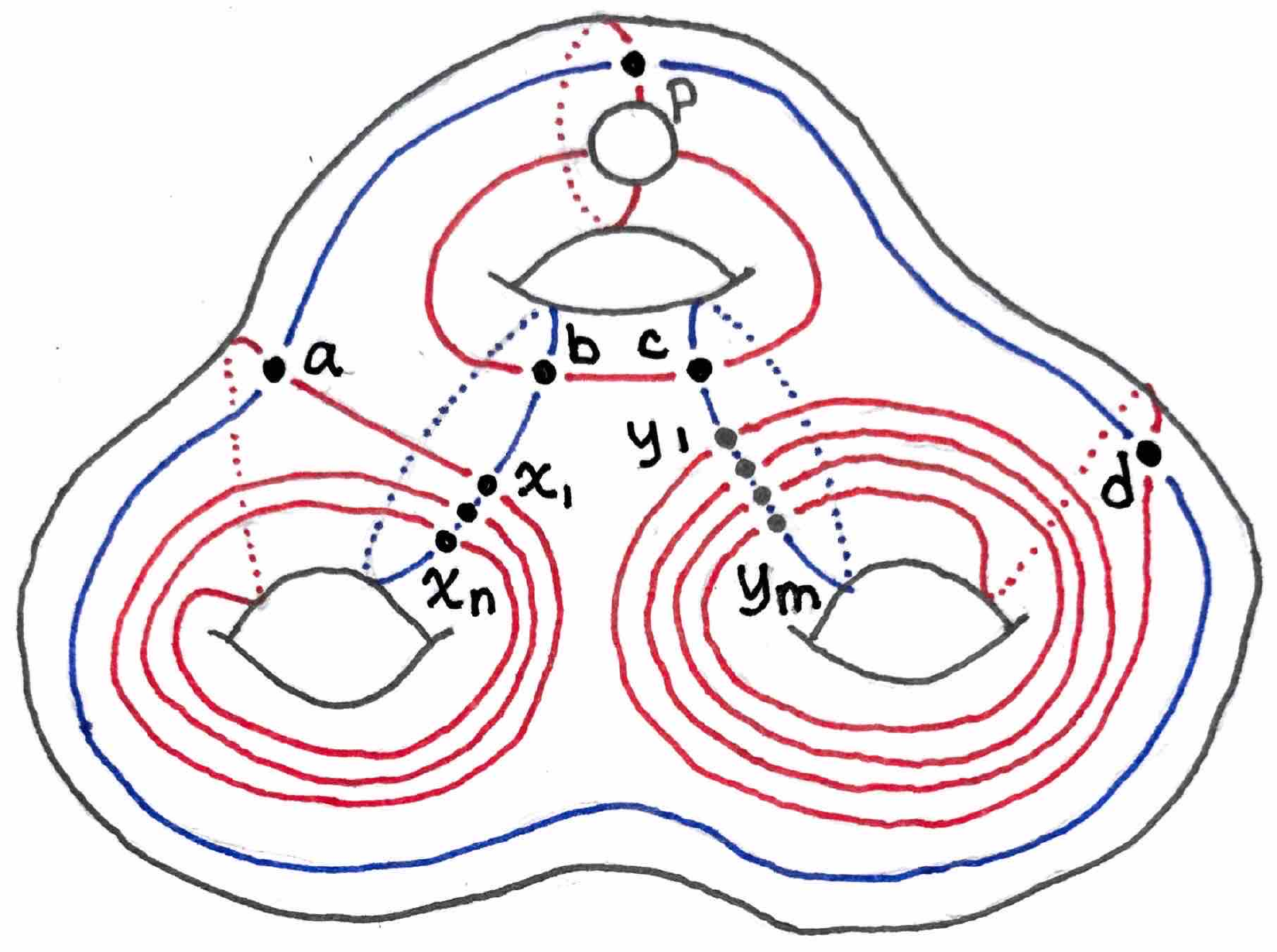}
    \caption{}
    \label{fig: bordered}
\end{figure}
\end{lemma}
For more on the distinction between \(\alpha\)-type and \(\beta\)-type diagrams, see Remark \ref{re:alpha_beta}. 
\begin{proof}
We focus on \(\beta\)-type framing. (The argument for \(\alpha\)-type framing is similar.) To attached a sutured cap to a bordered diagram, one first converts the bordered diagram to a bordered-sutured diagram by replacing the basepoint with a small green arc, as show in Figure \ref{fig: bordered sutured}. (We represent sutures in green.)
\begin{figure}[ht]
\centering
    \includegraphics[scale=0.6]{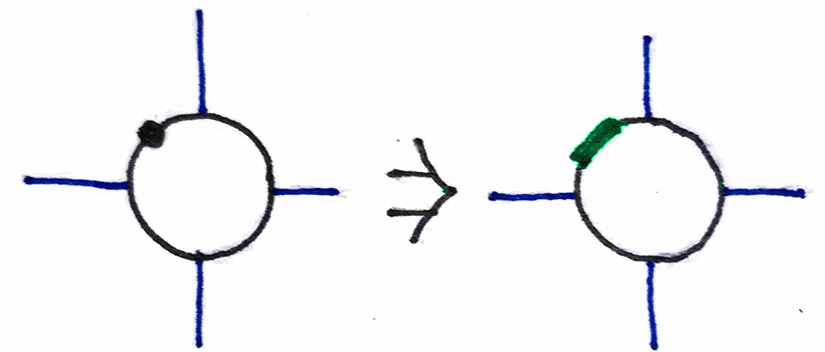}
    \caption{}
    \label{fig: bordered sutured}
\end{figure}

To realize the suture \(\Gamma_{\lambda}\) we now attach one of the two  sutured caps shown in Figure \ref{fig: caps}, matching the black parts of the boundaries. The cap we choose depends on the choice of basepoint. The result in all cases is the left-hand side of Figure \ref{fig: attached cap}.
\begin{figure}[ht]
\centering
    \includegraphics[scale=0.6]{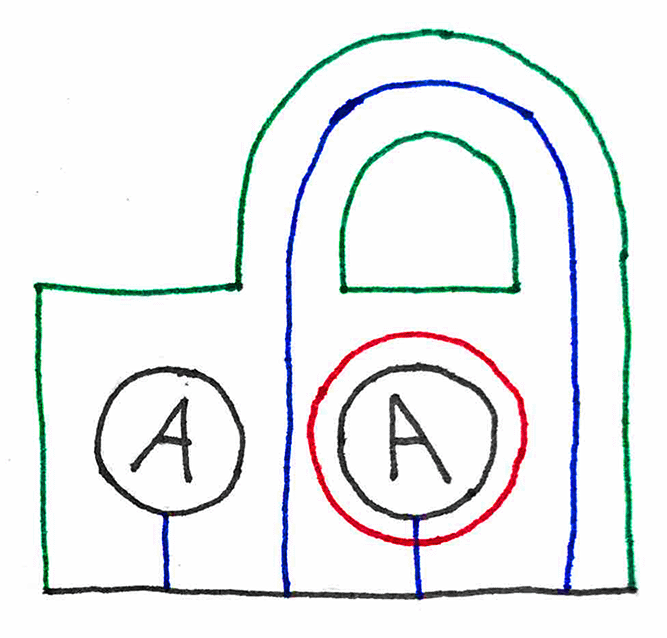}\qquad \includegraphics[scale=0.6]{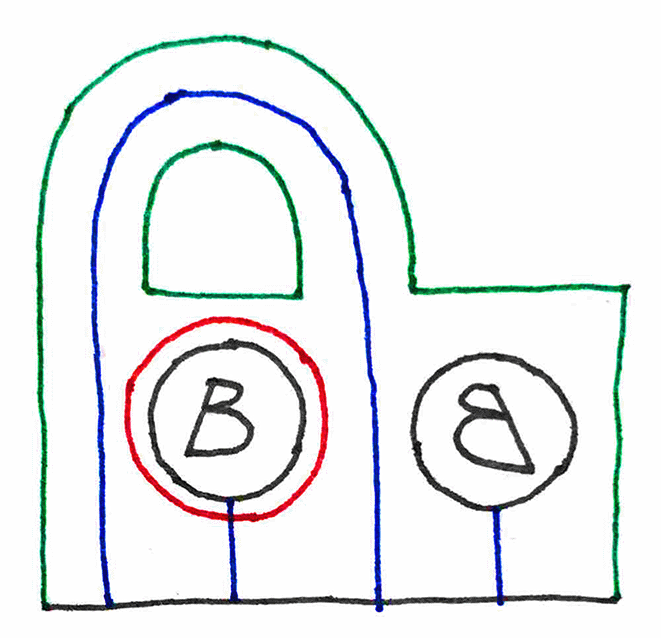}
    \caption{}
    \label{fig: caps}
\end{figure}
\begin{figure}[ht]
\centering
    \includegraphics[scale=0.8]{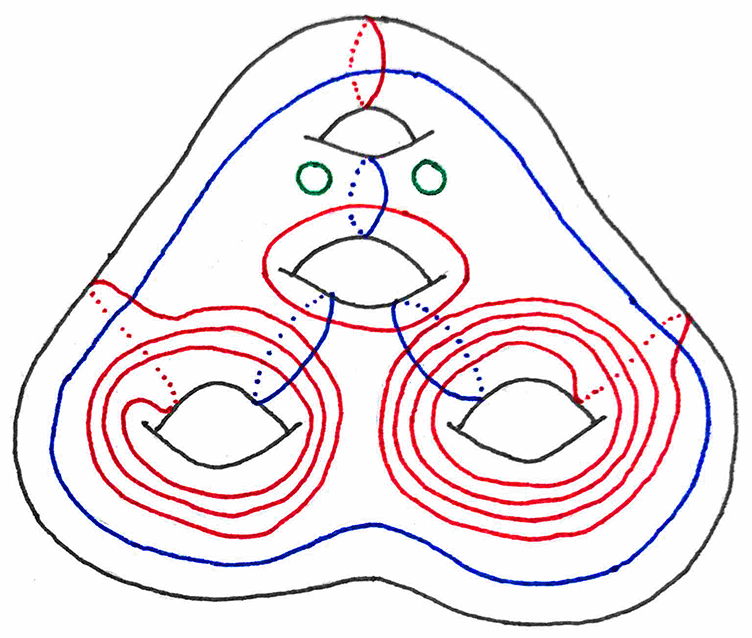}\qquad \qquad \includegraphics[scale=0.8]{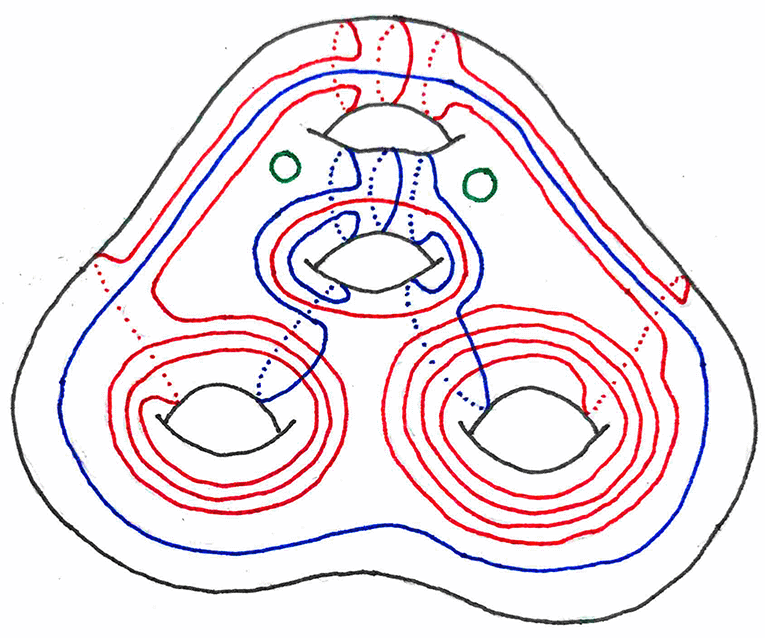}
    \caption{}
    \label{fig: attached cap}
\end{figure}

We would like to de-stabilize this diagram twice to get back to Figure \ref{fig: LOSS}. To do this, we first perform two \(\alpha\)-handleslides and two \(\beta\)-handleslides to obtain the right-hand side of Figure \ref{fig: attached cap}. We can then de-stabilize twice to obtain Figure \ref{fig: LOSS}. (Recall that a doubly-pointed Heegaard diagram is made into a sutured diagram by puncturing the basepoints \(z\) and \(w\).)

The changes in the pseudo-gradient vector field on \(Y(n,m)\) resulting from these handleslides and de-stabilizations occur within two 3-balls away from \(\partial Y(n,m)\). Hence the natural bijection between generators in the left-hand side of Figure \ref{fig: attached cap} and generators in Figure \ref{fig: LOSS}  preserves \(\mathrm{spin}^c\) classes. Moreover, the generators in both of these diagrams live in distinct \(\mathrm{spin}^c\) classes. Indeed, \({\mathrm{Spin}^c\big({-Y(n,m),-\Gamma_\mu}\big)}\) is a torsor over the abelian group 
\[H_1\big({Y(n,m)}\big)\cong \langle \mu_1,\mu_2\,|\, n\mu_1 = m\mu_2\rangle\cong \mathbb{Z}/(n,m)\oplus \mathbb{Z}.\]
In Figure \ref{fig: LOSS}, \(\mu_1\) and \(\mu_2\) are represented by dual curves to the blue \(\beta\) curves. Using this, one can see each of the \(nm\) generators lives in a distinct \(\mathrm{spin}^c\) class. From this, the corollary follows. 
\end{proof}
\begin{remark} 
It is more natural to think of the \(\alpha\)-type diagram in Figure \ref{fig: bordered} as shown in Figure \ref{fig: moustache}.
\begin{figure}[ht]
\centering
    \includegraphics[scale=0.15]{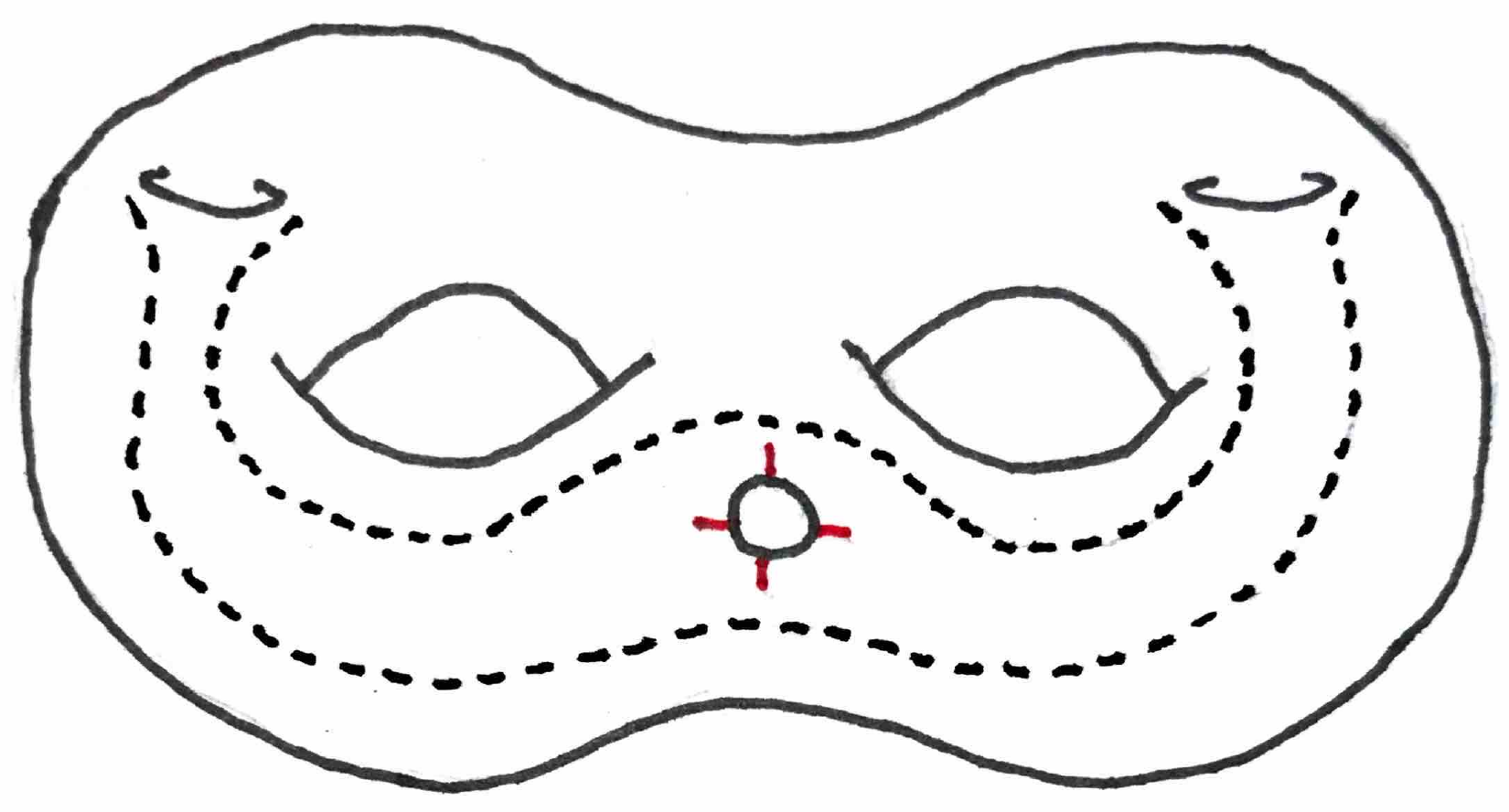}
    \caption{}
    \label{fig: moustache}
\end{figure}
As with the \(\beta\)-type diagram in Figure \ref{fig: bordered}, in this diagram the \(\beta\) handlebody is ``visible'' in \(\mathbb{R}^3\). Both diagrams arise from the link surgery diagram in \(S^3\) for \(\big(Y(g_{n,m}),K\big)\) shown in Figure \ref{fig: surgery one}.
\begin{figure}[ht]
\centering
    \includegraphics[scale=0.14]{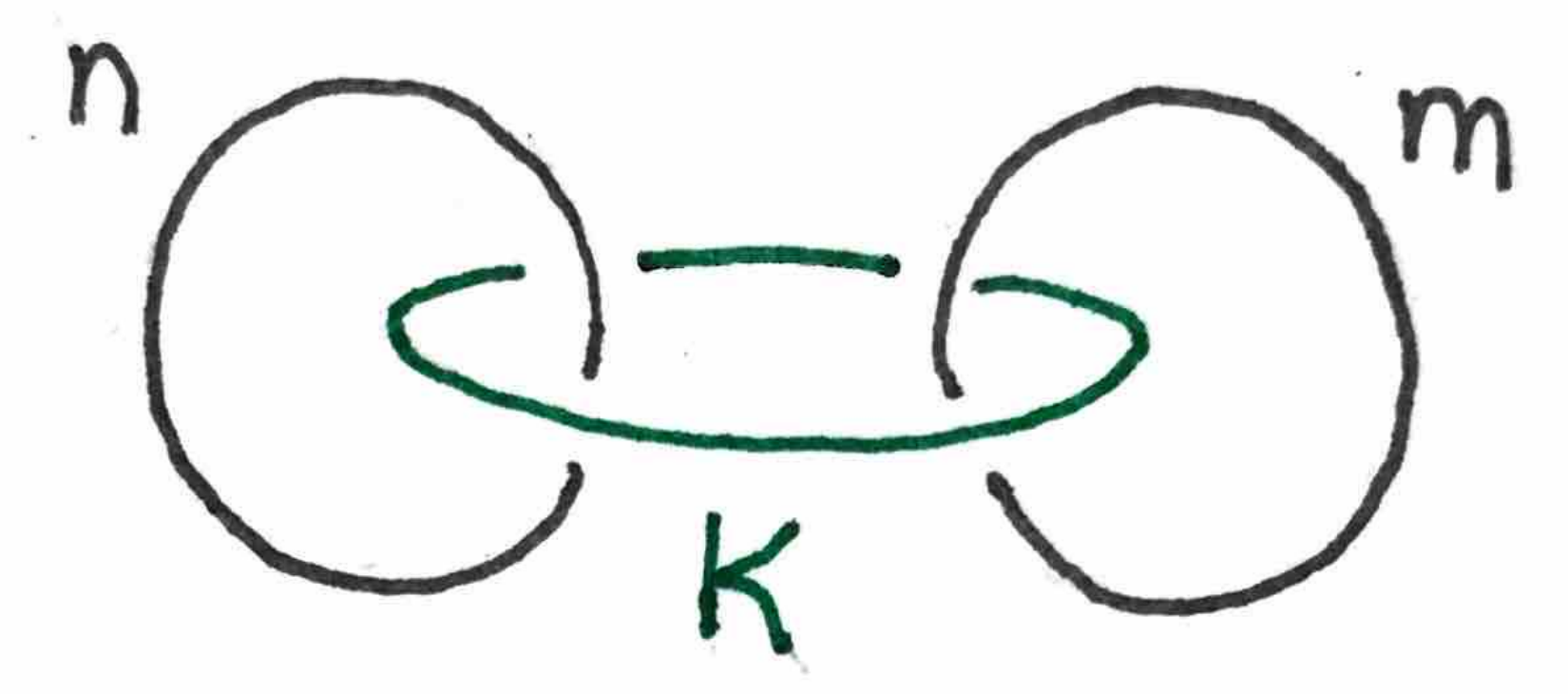}
    \caption{}
    \label{fig: surgery one}
\end{figure}
\end{remark}
Label the four possible basepoint choices for both the \(\beta\)-type and \(\alpha\)-type diagrams in Figure \ref{fig: bordered} as shown in Figure \ref{fig: beta basepoint}. We will label the resulting framings \(\mathcal{F}_{\mathrm{I},\beta}\), etc. The labels are chosen so that e.g., \(\mathcal{F}_{\mathrm{I},\beta}\) and \(\mathcal{F}_{\mathrm{I},\alpha}\) induce the same pointed matched circle on \(\partial Y(n,m)\).
\begin{figure}[ht]
    \centering
    \includegraphics[scale=0.08]{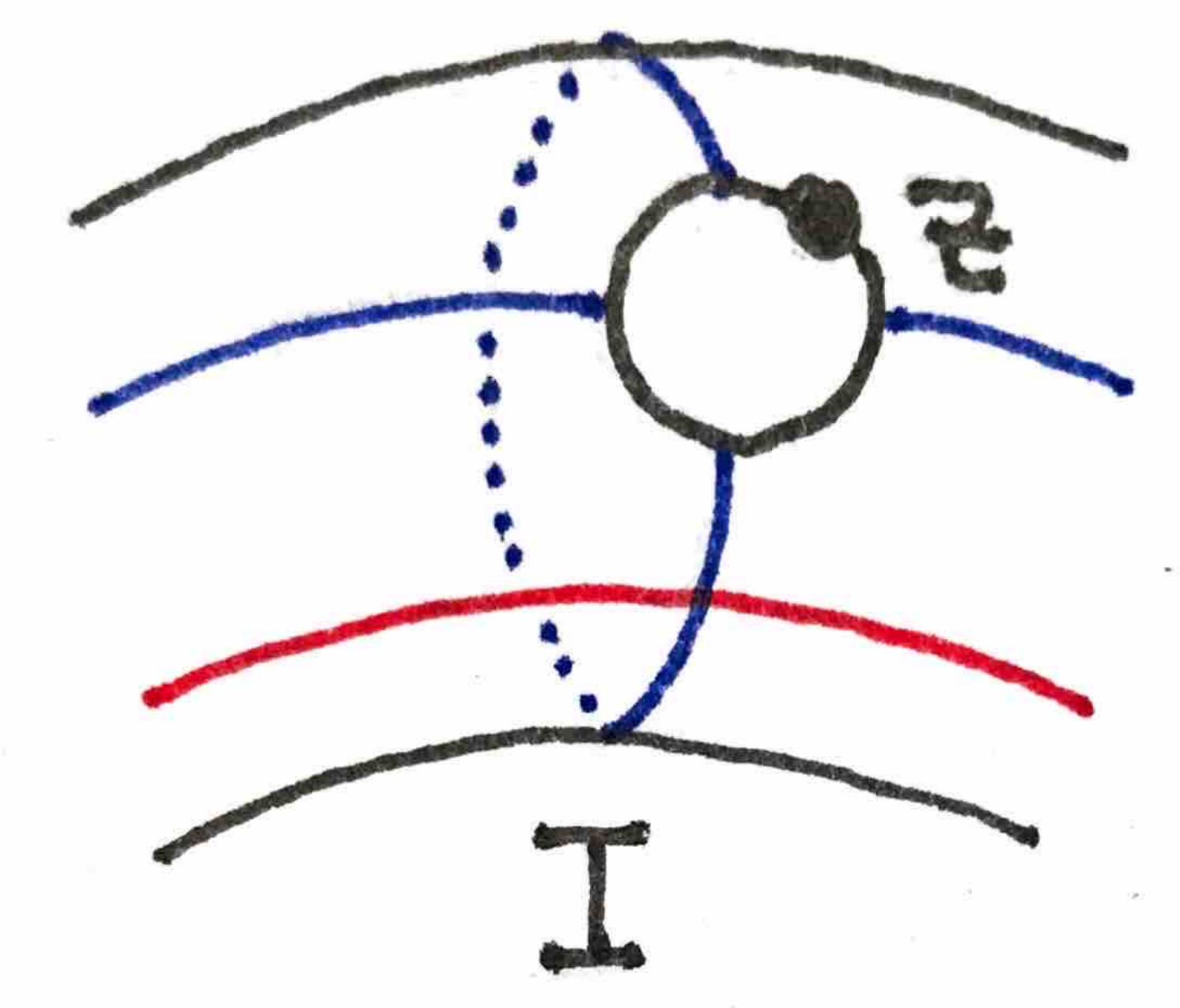}\quad\includegraphics[scale=0.08]{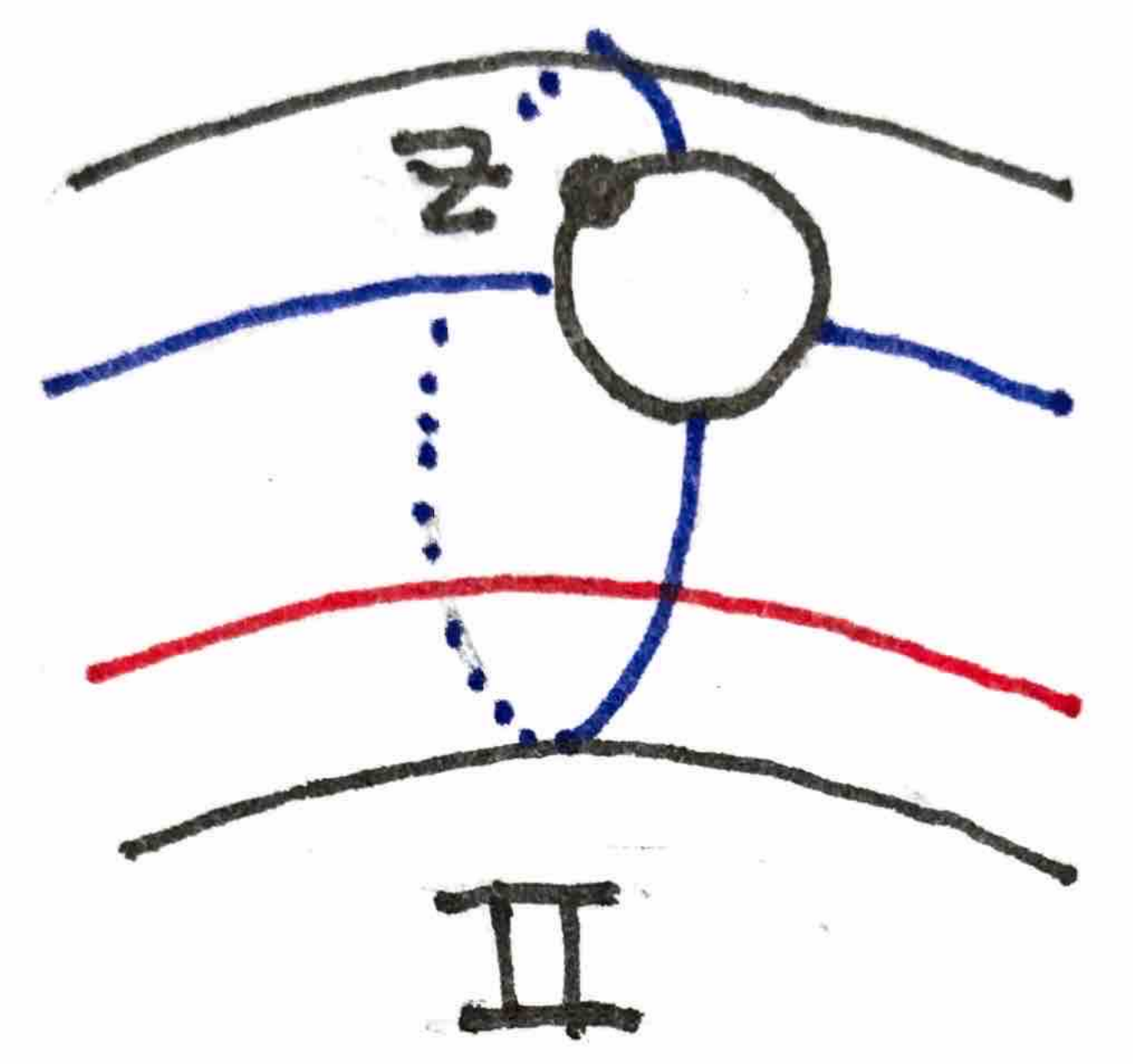}\quad\includegraphics[scale=0.08]{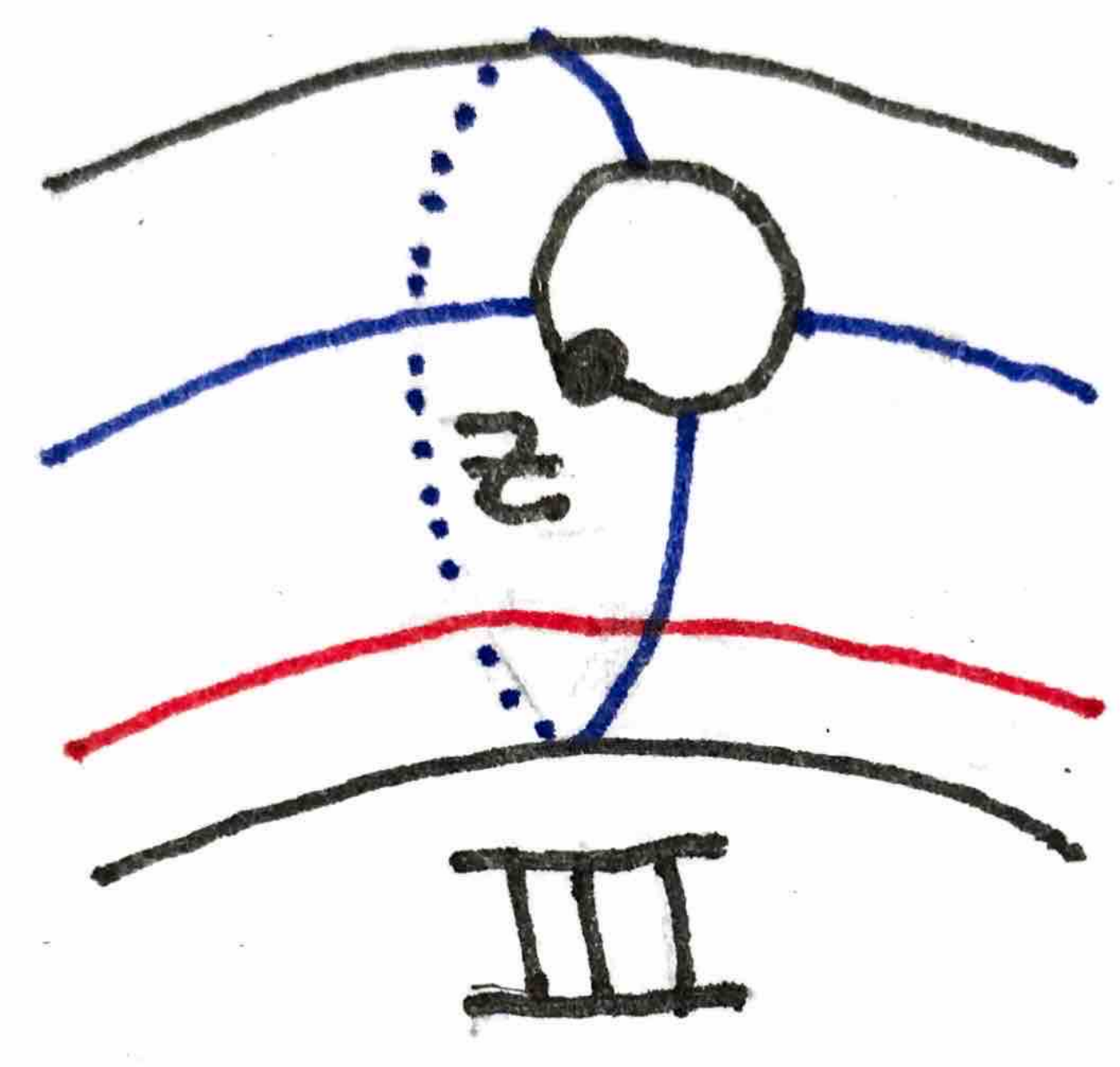}\quad\includegraphics[scale=0.08]{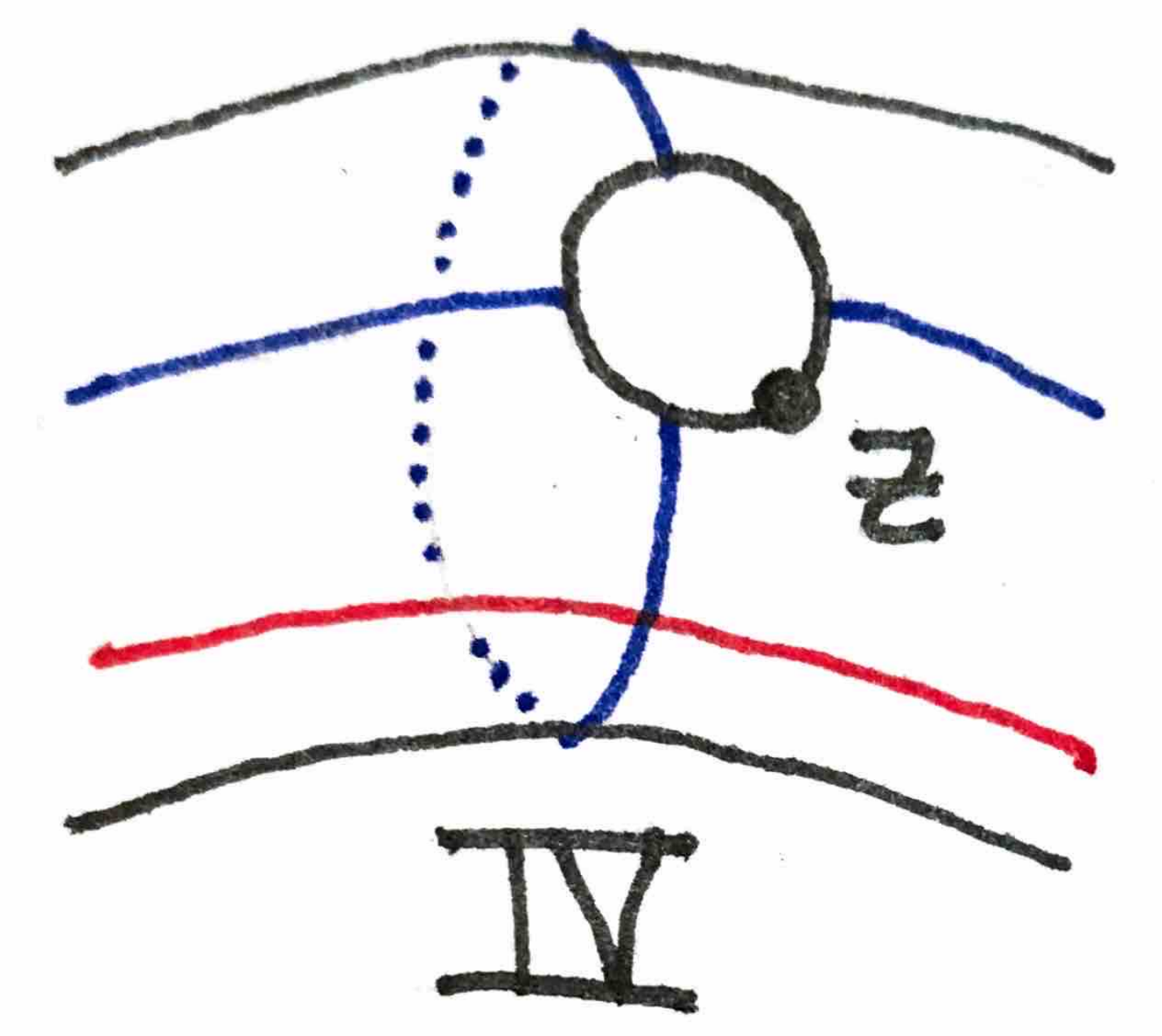}\\
    \includegraphics[scale=0.08]{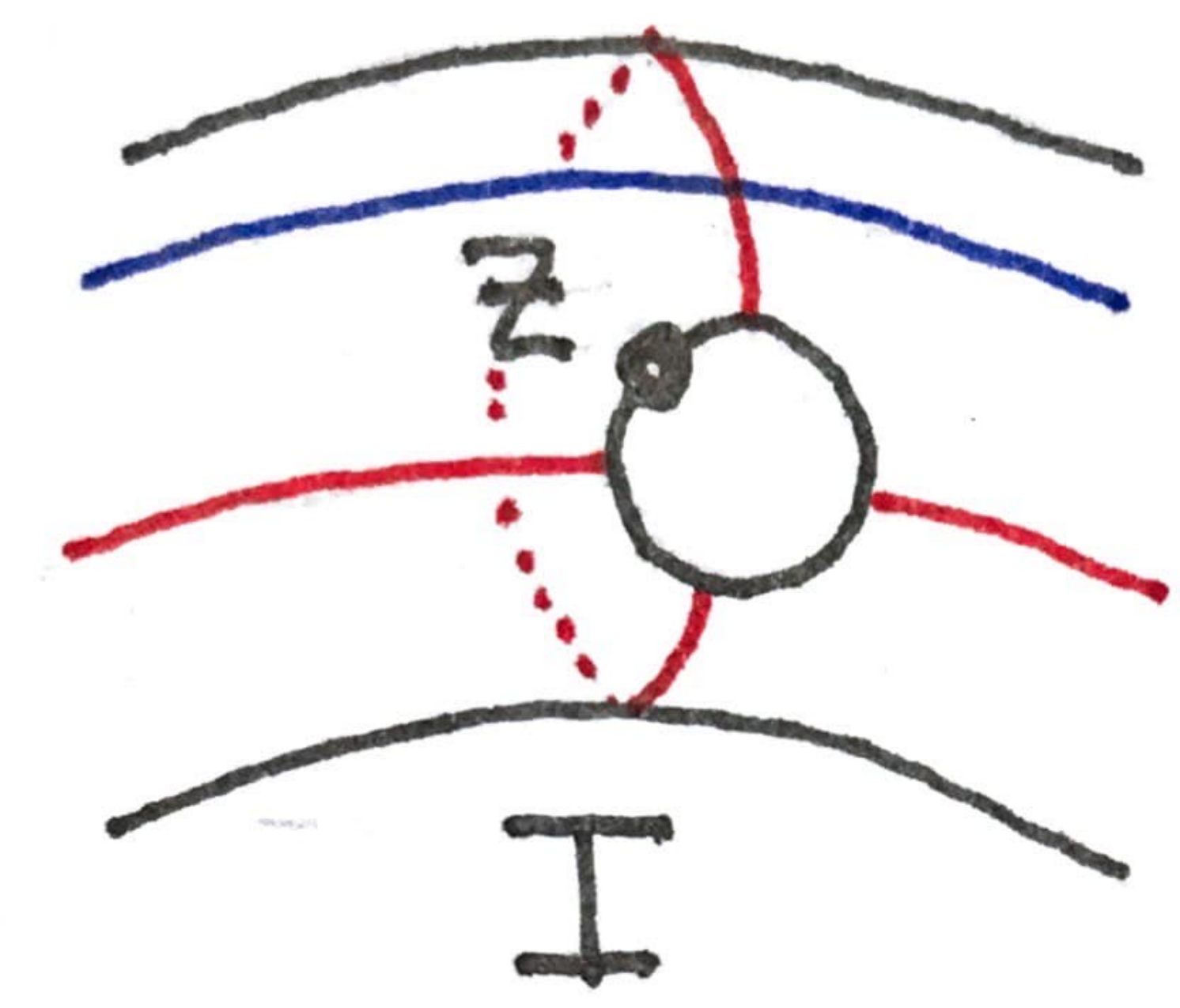}\quad\includegraphics[scale=0.08]{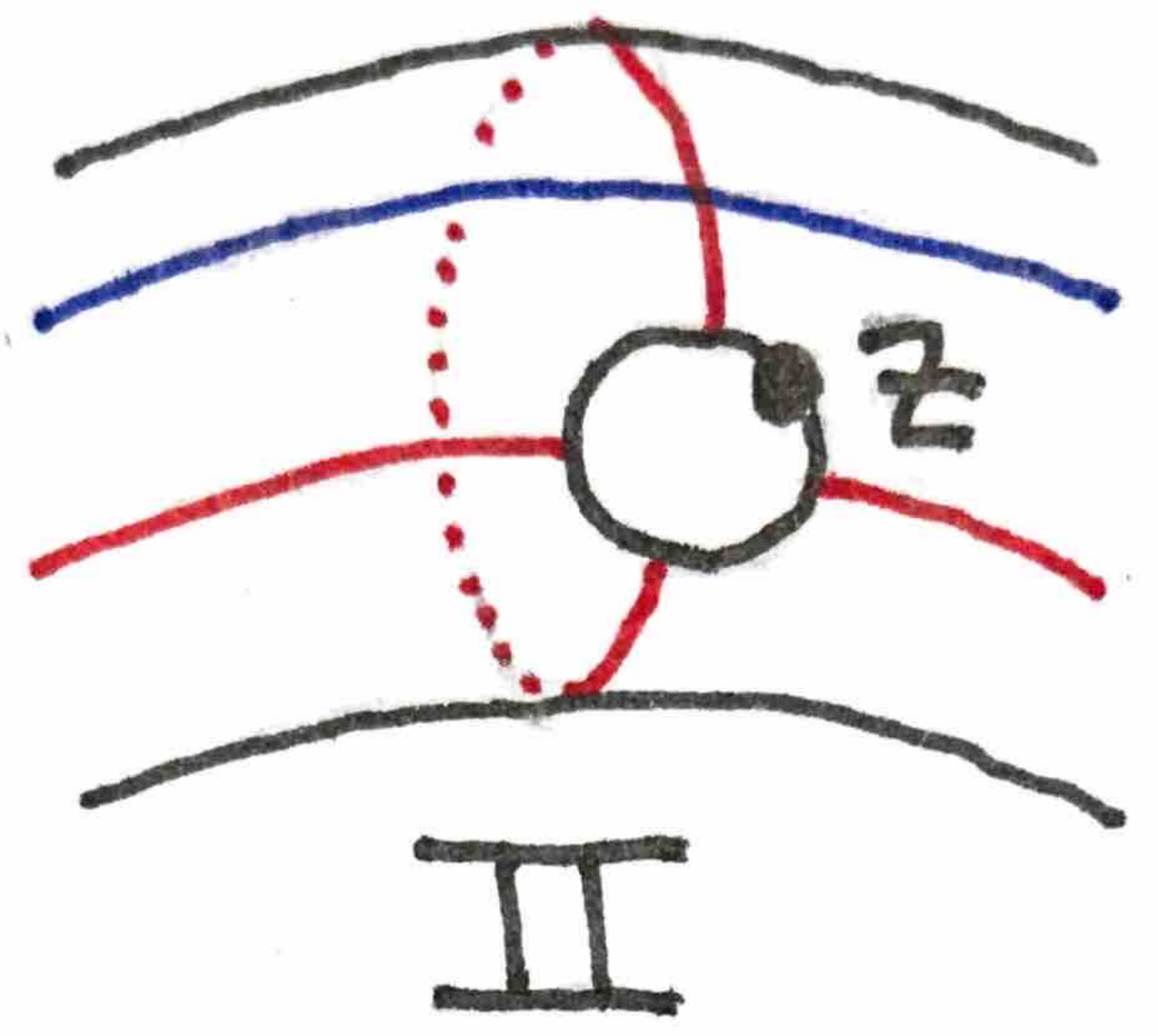}\quad\includegraphics[scale=0.08]{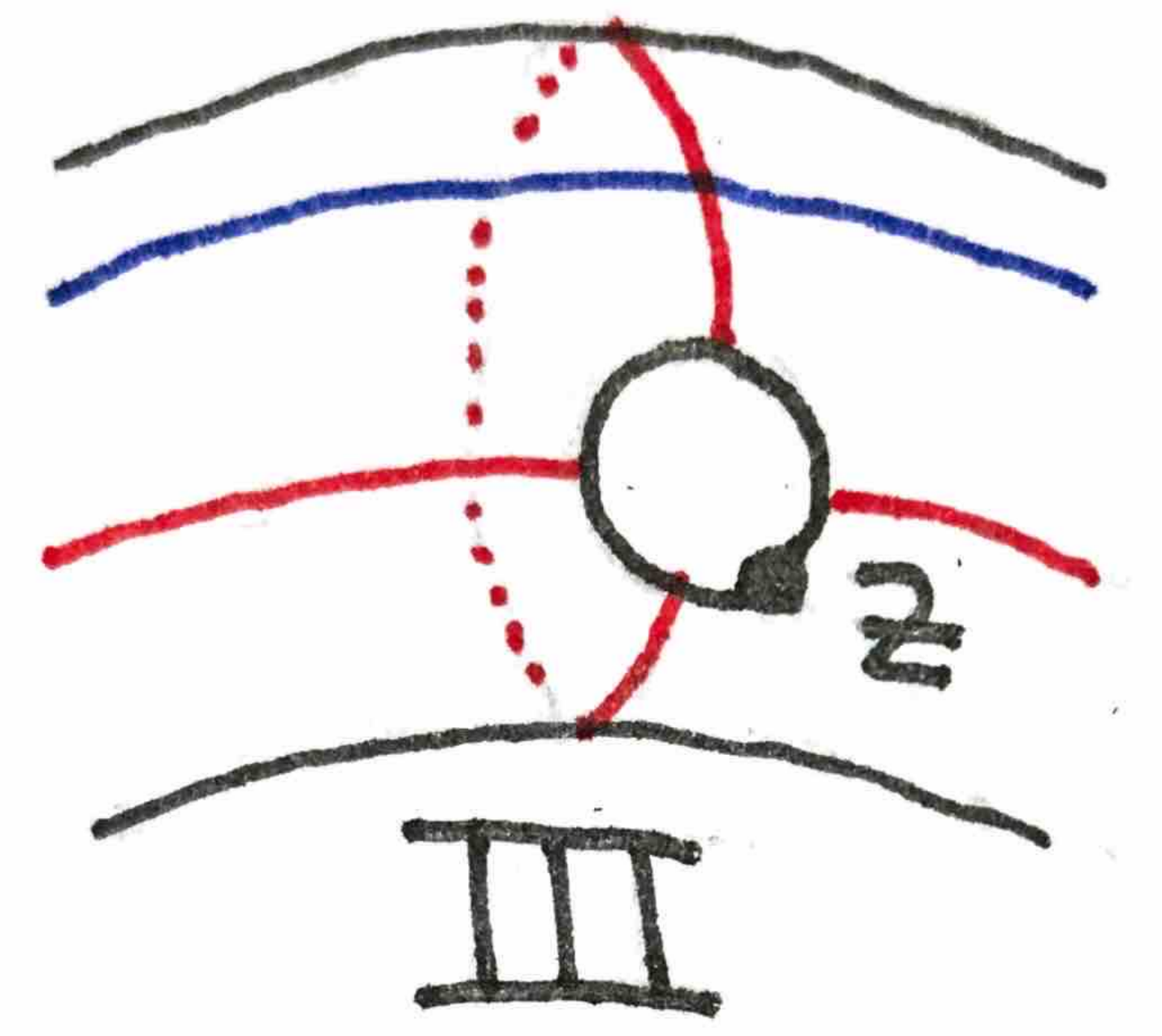}\quad\includegraphics[scale=0.08]{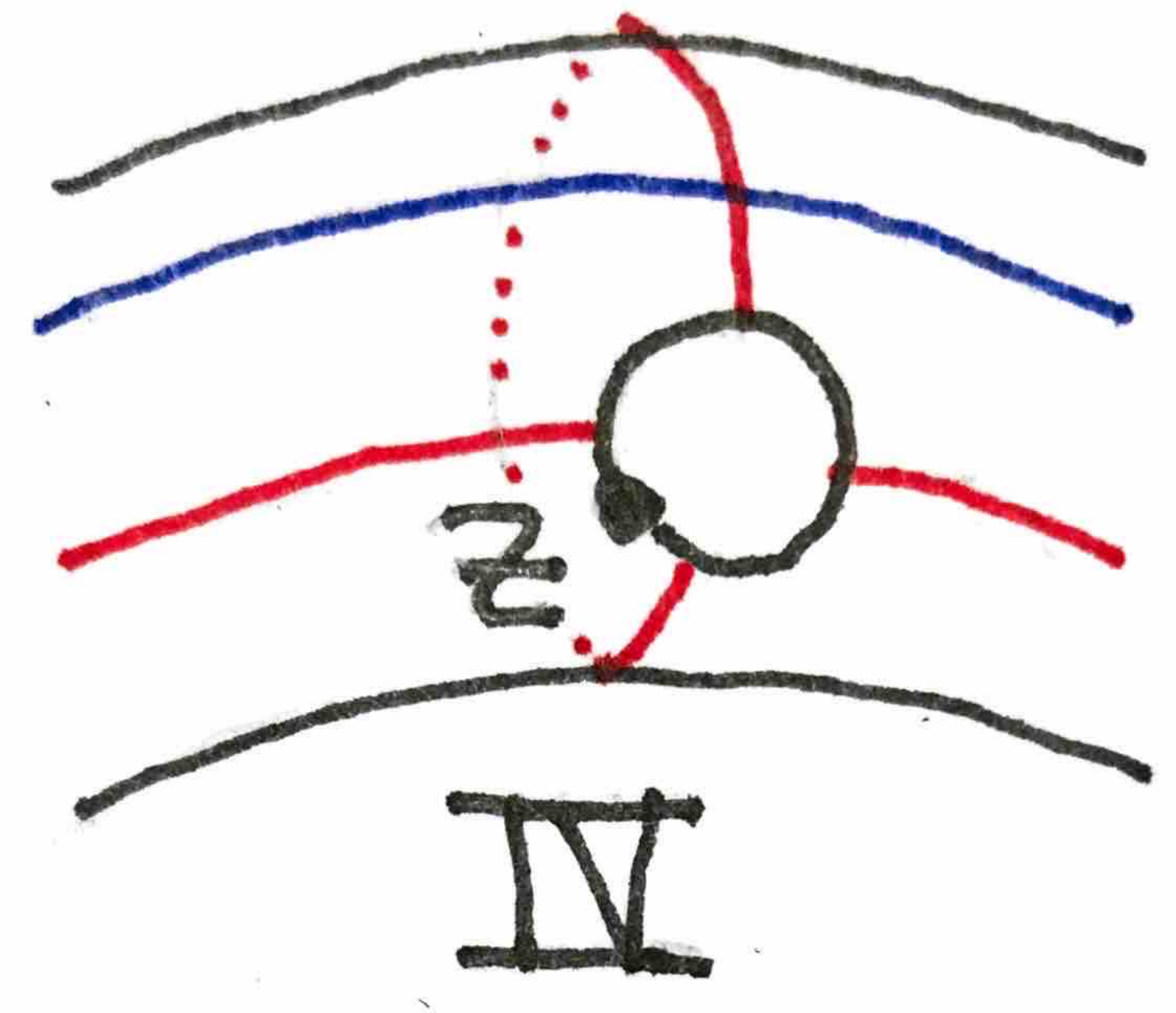}
    \caption{}
    \label{fig: beta basepoint}
\end{figure}
In Appendix \ref{sec: calculating}, we calculate that if \(\mathcal{F}\) is any one of these framings, then the Type A module \({{\widehat{\mathit{CFA}}\big({-Y(n,m),\mathcal{F}}\big)}}\) has a model given by one of the decorated graphs in Figure \ref{fig: initial type A}.
\begin{figure}[ht]
\centering
\begin{tikzpicture}
\def\n{3}
\def\m{4}
\def\gap{1}
\def\biggap{2}
\def\border{0.35}
\def\epsilon{0.15}
\draw[thick] (-\border,\border) rectangle ({(\n+2)*\gap}, {-(\m+2)*\gap});
\node at ({(\n+2)*\gap-1.3},0) {Type A, I/III};
\draw[thick] ({\biggap+(\n+2)*\gap-\border},\border) rectangle ({\biggap+(2*\n+4)*\gap}, {-(\m+2)*\gap});
\node at ({\biggap+(2*\n+4)*\gap-1.3},0) {Type A, II/IV};
\draw[thick, ->] ({-\border},{\border}) -- ({-\epsilon/sqrt(2)},{\epsilon/sqrt(2)});
\draw[thick] ({\biggap+(\n+2)*\gap-\border},{\border}) -- ({\biggap+(\n+2)*\gap-\epsilon/sqrt(2)},{\epsilon/sqrt(2)});
\filldraw[thick, fill=white] (0,0) circle (2pt);
\filldraw[black] ({\biggap+(\n+2)*\gap},0) circle (2pt);
\draw[thick, ->] ({\epsilon/sqrt(2)},{-\epsilon/sqrt(2)}) -- ({\gap-\epsilon/sqrt(2)}, {\epsilon/sqrt(2)-\gap})node[fill=white,inner sep=2pt,midway] {\scriptsize \(2\)};
\draw[thick, ->] ({\biggap+(\n+2)*\gap+\epsilon/sqrt(2)},{-\epsilon/sqrt(2)}) -- ({\biggap+(\n+3)*\gap-\epsilon/sqrt(2)}, {\epsilon/sqrt(2)-\gap})node[fill=white,inner sep=2pt,midway] {\scriptsize \(1\)};
\foreach\x in {2,...,\n}
    {
    \filldraw[thick, fill=white] ({(\x-1)*\gap},0) circle (2pt);
    \filldraw[black] ({\biggap+(\n+\x+1)*\gap},0) circle (2pt);
    \draw[thick, ->] ({(\x-1)*\gap +\epsilon/sqrt(2)},{-\epsilon/sqrt(2)}) -- ({\x*\gap-\epsilon/sqrt(2)}, {\epsilon/sqrt(2)-\gap})node[fill=white,inner sep=2pt,midway] {\scriptsize \(2\)};
    \draw[thick, ->] ({\biggap+(\n+\x+1)*\gap +\epsilon/sqrt(2)},{-\epsilon/sqrt(2)}) -- ({\biggap+(\n+2+\x)*\gap-\epsilon/sqrt(2)}, {\epsilon/sqrt(2)-\gap})node[fill=white,inner sep=2pt,midway] {\scriptsize \(1\)};
    \draw[thick] ({(\x-1)*\gap}, {-\m*\gap-\epsilon}) -- ({(\x-1)*\gap},{-(\m+2)*\gap})node[fill=white,inner sep=2pt,midway] {\scriptsize \(321\)};
    \draw[thick, ->] ({\biggap+(\n+\x+1)*\gap},{-(\m+2)*\gap}) -- ({\biggap +(\n+\x+1)*\gap}, {-\m*\gap-\epsilon}) node[fill=white,inner sep=2pt,midway] {\scriptsize \(3\)};
    \draw[thick, ->] ({(\x-1)*\gap}, {\border})--({(\x-1)*\gap}, {\epsilon});
    \draw[thick, -] ({\biggap+(\n+\x+1)*\gap}, {\border})--({\biggap+(\n+\x+1)*\gap}, {\epsilon});
    }
\foreach\y in {2,...,\m} 
    {
    \filldraw[thick, fill=white] (0,{(1-\y)*\gap}) circle (2pt);
    \filldraw[black] ({\biggap+(\n+2)*\gap},{(1-\y)*\gap}) circle (2pt);
    \draw[thick, ->] ({\epsilon/sqrt(2)},{(1-\y)*\gap -\epsilon/sqrt(2)}) -- ({\gap-\epsilon/sqrt(2)}, {\epsilon/sqrt(2)-\y*\gap})node[fill=white,inner sep=2pt,midway] {\scriptsize \(2\)};
    \draw[thick, ->] ({\biggap + (\n+2)*\gap + \epsilon/sqrt(2)},{(1-\y)*\gap -\epsilon/sqrt(2)}) -- ({\biggap + (\n+3)*\gap-\epsilon/sqrt(2)}, {\epsilon/sqrt(2)-\y*\gap})node[fill=white,inner sep=2pt,midway] {\scriptsize \(1\)};
    \draw[thick] ({\n*\gap+\epsilon},{(1-\y)*\gap})--({(\n+2)*\gap},{(1-\y)*\gap}) node[fill=white,inner sep=2pt,midway] {\scriptsize \(321\)};
    \draw[thick, ->] ({\biggap+(2*\n+4)*\gap},{(1-\y)*\gap})--({\biggap+(2*\n+2)*\gap+\epsilon},{(1-\y)*\gap}) node[fill=white,inner sep=2pt,midway] {\scriptsize \(3\)};
    \draw[thick, ->]({-\border},{(1-\y)*\gap})--({-\epsilon},{(1-\y)*\gap});
    \draw[thick]({\biggap+(\n+2)*\gap-\border},{(1-\y)*\gap})--({\biggap+(\n+2)*\gap-\epsilon},{(1-\y)*\gap});
    }
\foreach \x in {1,...,\n}
    \foreach \y in {1,...,\m}{
        \filldraw[black] ({\x*\gap},{-\y*\gap}) circle (2pt);
        \filldraw[thick, fill = white] ({\biggap+(\n+2+\x)*\gap},{-\y*\gap}) circle (2pt);
        }
\foreach \x in {2,...,\n}
    \foreach \y in {2,...,\m}{
        \draw [thick, ->] ({(\x-1)*\gap+\epsilon/sqrt(2)},{(1-\y)*\gap-\epsilon/sqrt(2)}) --({\x*\gap-\epsilon/sqrt(2)},{-\y*\gap+\epsilon/sqrt(2)}) node[fill=white,inner sep=2pt,midway] {\scriptsize \(32\)};
        \draw [thick, ->] ({\biggap+(\n+\x+1)*\gap+\epsilon/sqrt(2)},{(1-\y)*\gap-\epsilon/sqrt(2)}) --({\biggap+(\n+2+\x)*\gap-\epsilon/sqrt(2)},{-\y*\gap+\epsilon/sqrt(2)}) node[fill=white,inner sep=2pt,midway] {\scriptsize \(21\)};
        }
\filldraw[thick, fill=white] ({(\n+1)*\gap},{-(\m+1)*\gap}) circle (2pt);
\filldraw[black] ({\biggap+(2*\n+3)*\gap},{-(\m+1)*\gap}) circle (2pt);
\draw[thick, ->] ({\n*\gap+\epsilon/sqrt(2)},{-\m*\gap-\epsilon/sqrt(2)}) -- ({(\n+1)*\gap-\epsilon/sqrt(2)},{\epsilon/sqrt(2)-(\m+1)*\gap}) node[fill=white,inner sep=2pt,midway] {\scriptsize \(321\)};
\draw[thick, ->] ({\biggap+(2*\n+3)*\gap-\epsilon/sqrt(2)},{\epsilon/sqrt(2)-(\m+1)*\gap}) -- ({\biggap + (2*\n+2)*\gap+\epsilon/sqrt(2)},{-\m*\gap-\epsilon/sqrt(2)}) node[fill=white,inner sep=2pt,midway] {\scriptsize \(3\)};
\draw[thick] ({(\n+1)*\gap+\epsilon/sqrt(2)},{-(\m+1)*\gap-\epsilon/sqrt(2)}) -- ({(\n+2)*\gap},{-(\m+2)*\gap}) node[fill=white,inner sep=2pt,midway] {\scriptsize \(21\)};
\draw[thick, ->] ({\biggap+(2*\n+4)*\gap},{-(\m+2)*\gap}) -- ({\biggap+(2*\n+3)*\gap+\epsilon/sqrt(2)},{-(\m+1)*\gap-\epsilon/sqrt(2)})node[fill=white,inner sep=2pt,midway] {\scriptsize \(32\)};
\end{tikzpicture}
\caption{The case \(n=3\), \(m=4\). Both decorated graphs should be interpreted as lying on a torus. (We only draw the graphs on a torus to avoid self-intersections.)}
\label{fig: initial type A}
\end{figure}
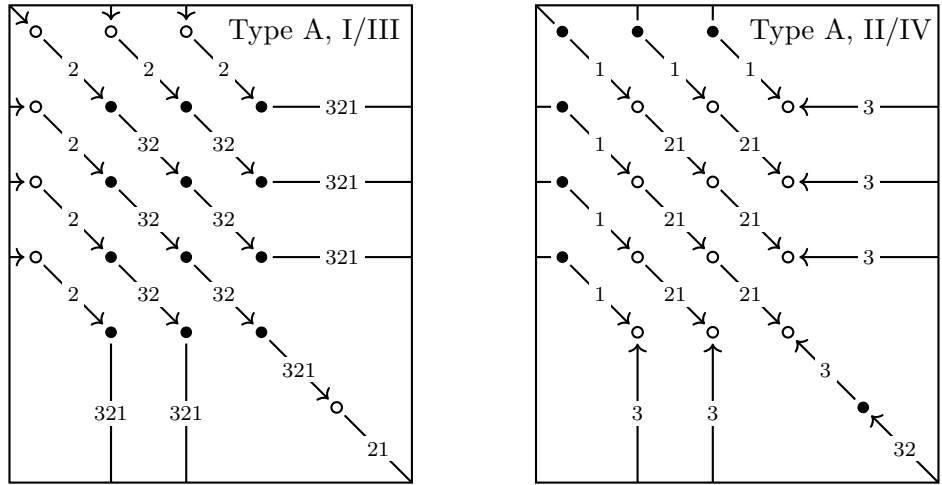

Here we are representing \({\widehat{\mathit{CFA}}}\) by a decorated graph using the conventions from \cite{calculus}. Namely, vertices in the graph correspond to generators. A vertex labeled \(\bullet\) admits a non-trivial action of the idempotent \(\iota_0\), and a vertex labeled \(\circ\) is admits on non-trivial action of \(\iota_1\). There is no \(m_1\) action. To read off \(m_k\) for \(k\geq 2\), one looks at directed paths. Suppose there is a directed path from a generator \(\mathbf{x}\) to a generator \(\mathbf{y}\). Starting at \(\mathbf{x}\), concatenate all the labels in the path into a string \(I\). Next, separate \(I\) into substrings \(I=I_1\cdots I_k\) so that each \(I_1,\ldots,I_k\) lies in \(\{1,2,3,12,23,123\}\). Do this in such a way to minimize \(k\). From this data, one reads off 
\[m_{k+1}(\mathbf{x},\rho_{I_1},\ldots,\rho_{I_k}) = \mathbf{y}.\]
For instance, a path \({\bullet \xrightarrow{\,32\,}\bullet \xrightarrow{\,321\,}\bullet}\) corresponds the \(m_5\) action \[m_5(\mathbf{x},\rho_3,\rho_{23},\rho_2,\rho_1)=\mathbf{y}.\]
We remark that since the decorated graph on the right has no directed cycles, the corresponding model for \({\widehat{\mathit{CFA}}}\big({-Y(n,m),\mathcal{F}_{\mathrm{II/IV}}}\big)\) is bounded as defined in Appendix \ref{sec: algebra}.

To identify the LOSS invariant \({\widehat{\mathcal{L}}(K)}\) in the above models, we will use a grading argument. We make a few preliminary observations. First, the group of periodic domains for a bordered Heegaard diagram representing \(-Y(n,m)\) is isomorphic to 
\[H_2\big({Y(n,m), \partial Y(n,m)}\big)\cong H^1\big(Y(n,m)\big) \cong \mathbb{Z}.\]
We identify a generating periodic domain for Figure \ref{fig: bordered} with the values \(n=3\), \(m=4\) and framing \(\mathcal{F}_{\mathrm{I},\beta}\) in Figure \ref{fig: periodic}.

\begin{figure}[ht]
\centering
    \includegraphics[scale=0.75]{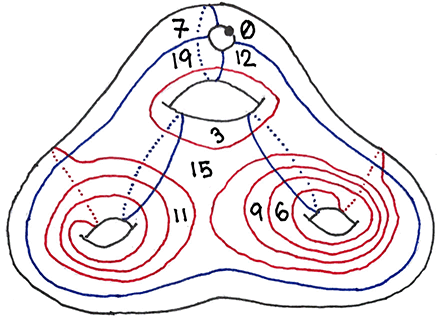}
    \caption{Despite appearances, all regions above have been labeled!}
    \label{fig: periodic}
\end{figure}

From here on out, in order to draw domains simultaneously for both the \(\alpha\) and \(\beta\)-type diagrams in Figure \ref{fig: bordered} we draw domains in the ``nodal diagram'' shown in Figure \ref{fig: nodal}. 
\begin{figure}[ht]
\centering
    \includegraphics[scale=0.13]{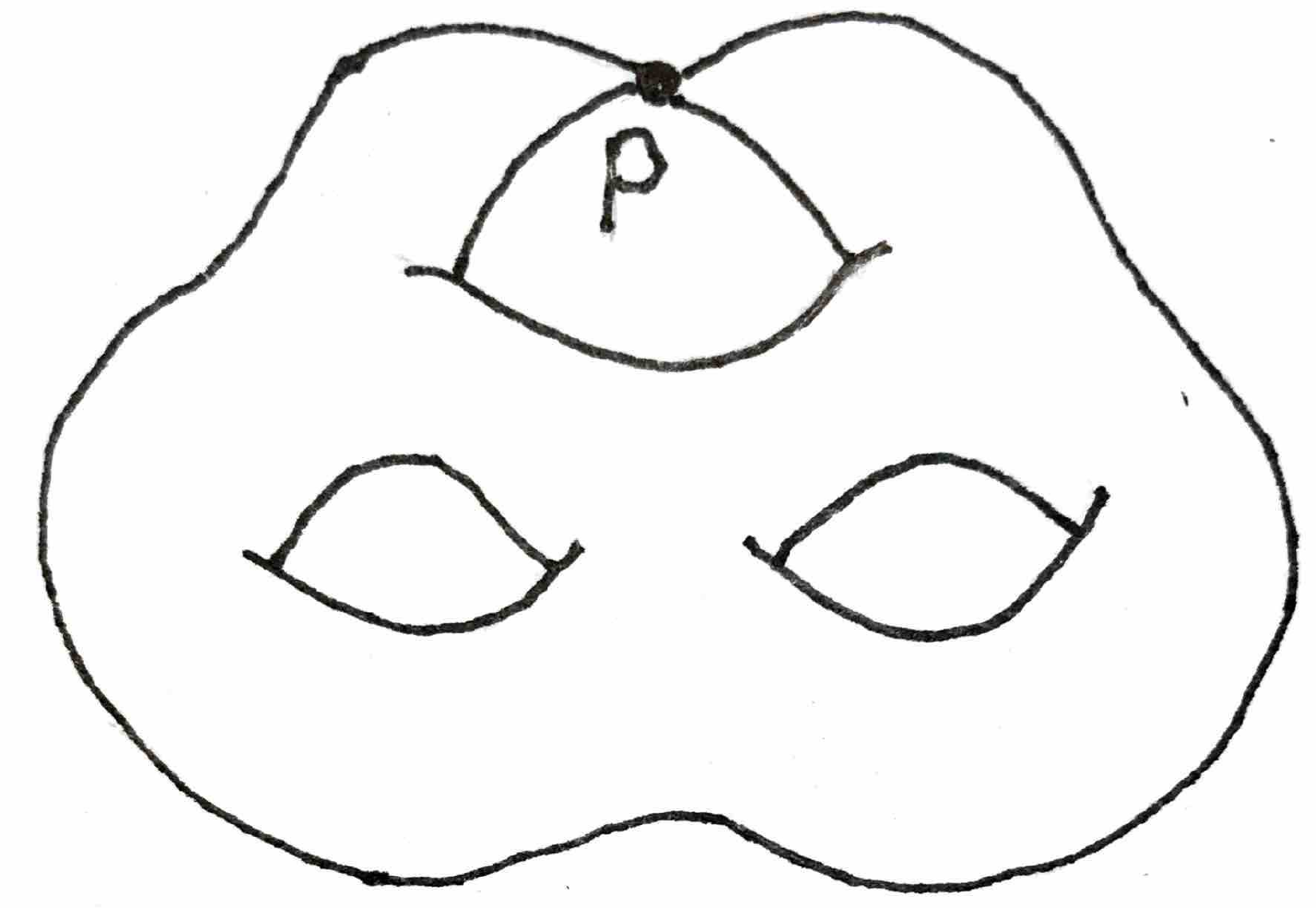}
    \caption{}
    \label{fig: nodal}
\end{figure}

In general, a generating periodic domain for the diagram in Figure \ref{fig: bordered} has multiplicites in regions adjacent to the pointed matched circle as shown in Figure \ref{fig: weights}.
\begin{figure}[ht]
\centering
    \includegraphics[scale=0.15]{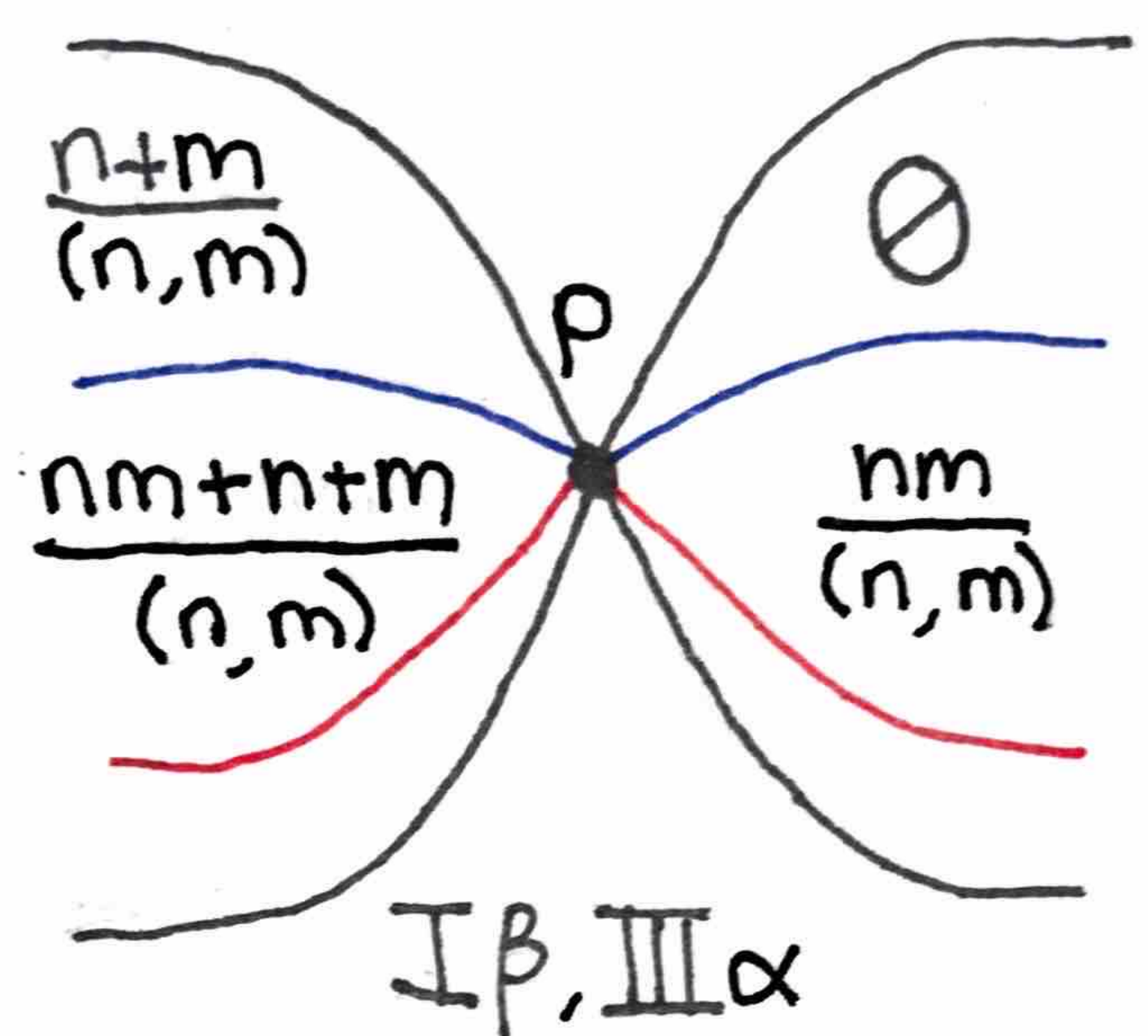}\quad\includegraphics[scale=0.15]{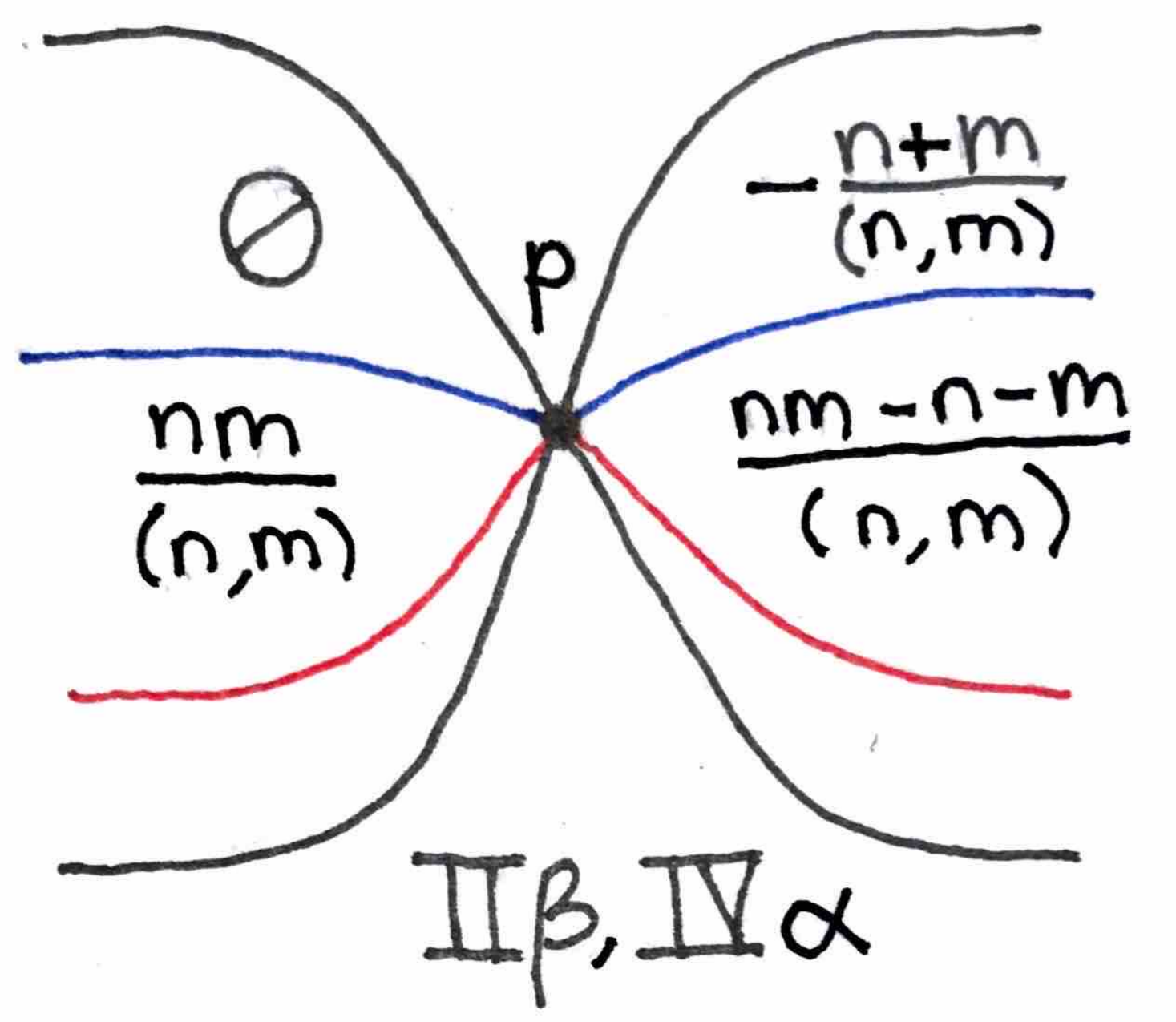}\quad\includegraphics[scale=0.15]{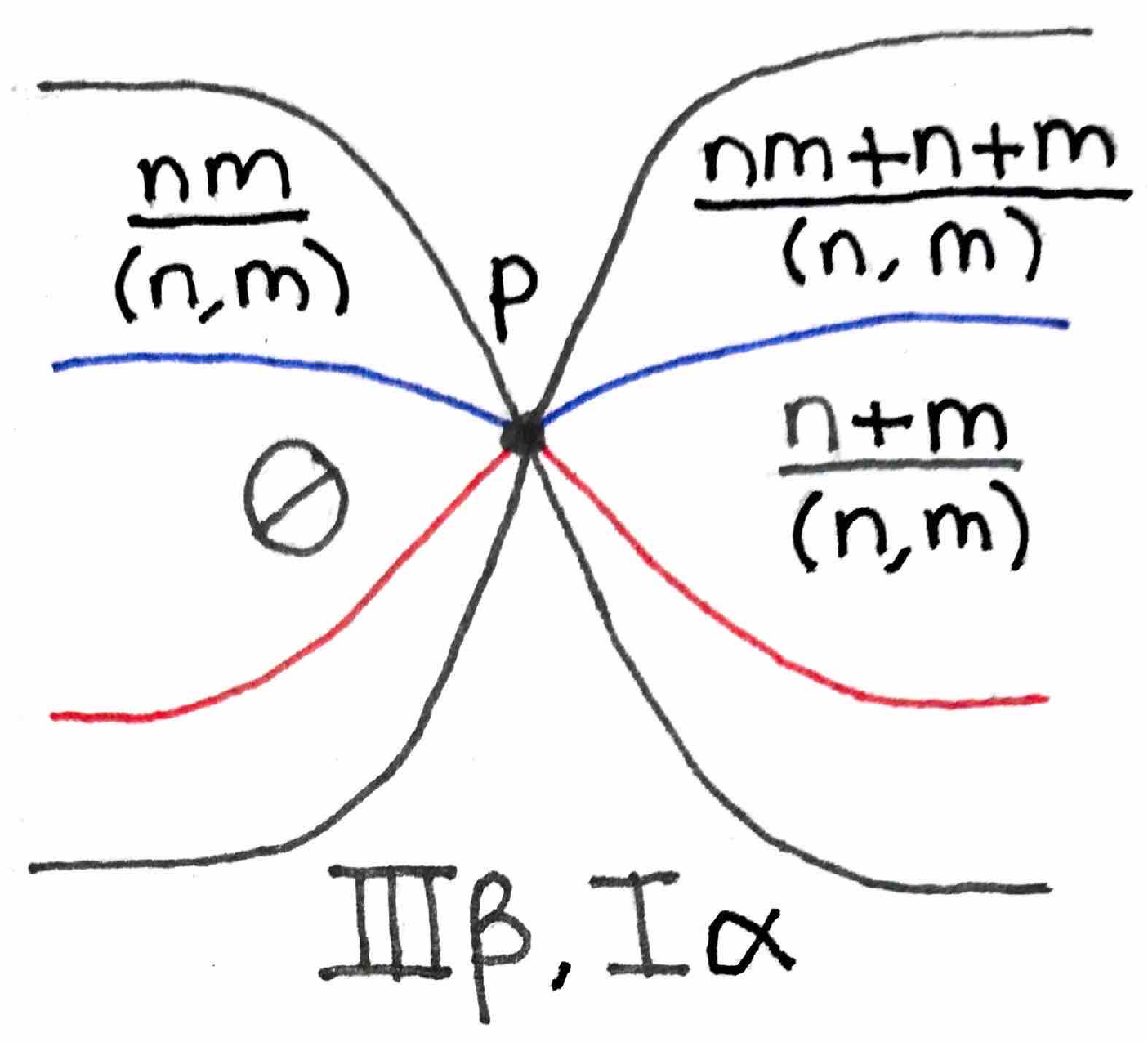}\quad\includegraphics[scale=0.15]{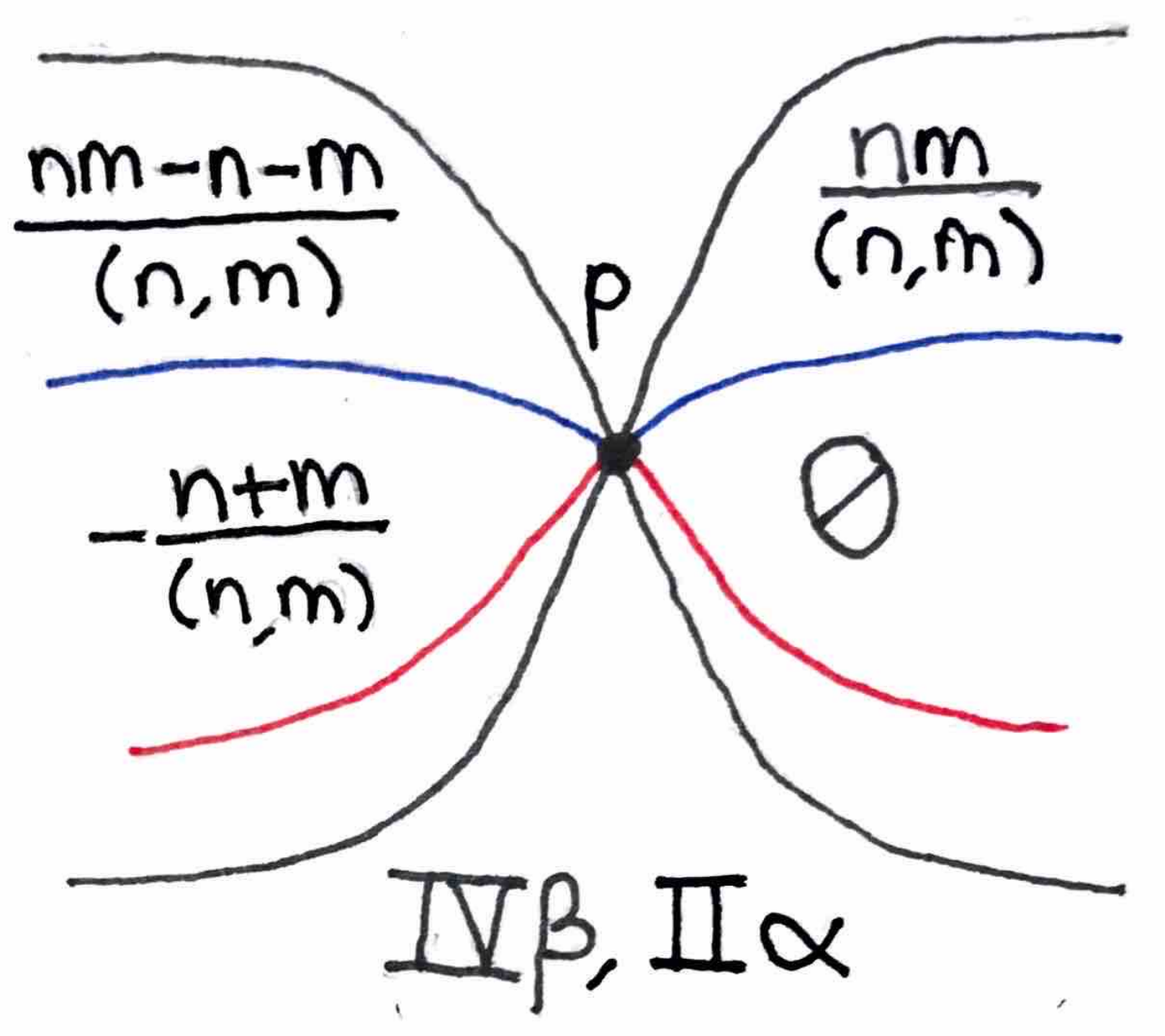}
    \caption{}
    \label{fig: weights}
\end{figure}

Recall a periodic domain is called \textit{provincial} if its multiplicities at regions adjacent to the pointed matched circle are \(0\). From Figure \ref{fig: weights} we see there are no provincial periodic domains, hence both diagrams in \ref{fig: bordered} are \textit{provincially admissible.} (This also follows from the fact that \(H_2\big(Y(n,m)\big) =0\).) 

Next, both diagrams in Figure \ref{fig: bordered} contain no \(m_1\) action. To see this, one can either argue using \(\mathrm{spin}^c\) structures or observe that the number of generators in Figure \ref{fig: bordered} is the same as in Figure \ref{fig: initial type A}. We now identify some index 1 simple domains which possibly contribute to higher differentials. In Figure \ref{fig: domains one} we show two such domains. 
\begin{figure}[ht]
\centering
    \includegraphics[scale=0.17]{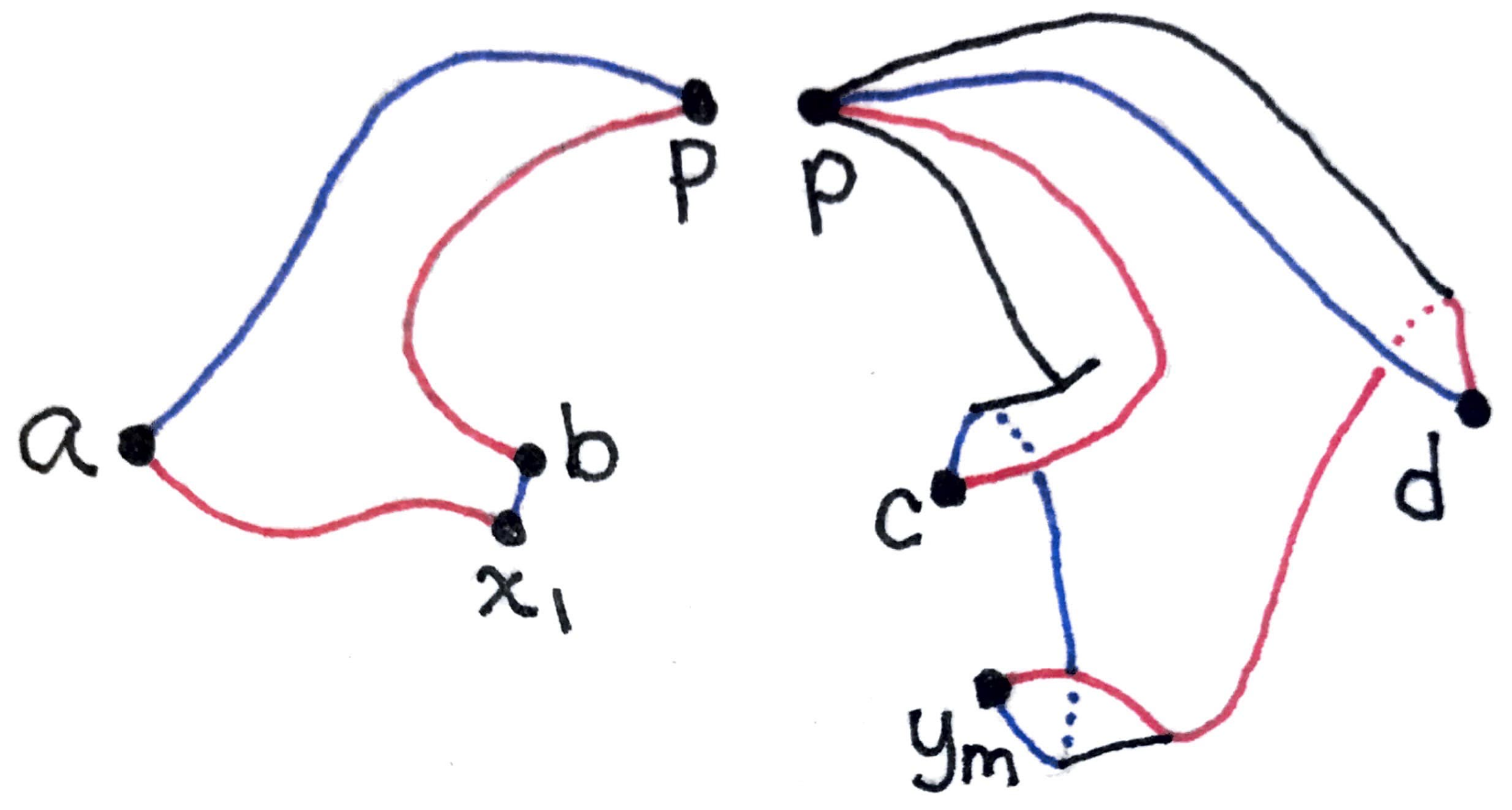}
    \caption{}
    \label{fig: domains one}
\end{figure}

The domain on the left of Figure \ref{fig: domains one} contributes to the following counts: 
\begin{alignat*}{2}
&\text{For }\mathcal{F}_{\mathrm{II},\beta}\text{ and }\mathcal{F}_{\mathrm{IV},\alpha},&&\quad \big[m_2(\{a,b,y_j\}, \rho_1); \{p,x_1,y_j\}\big],\\
&\text{For }\mathcal{F}_{\mathrm{I},\beta}\text{ and }\mathcal{F}_{\mathrm{III},\alpha},&&\quad \big[m_2(\{a,b,y_j\}, \rho_2); \{p,x_1,y_j\}\big],\qquad\qquad  1\leq j\leq m.\\
&\text{For }\mathcal{F}_{\mathrm{IV},\beta}\text{ and }\mathcal{F}_{\mathrm{II},\alpha},&&\quad \big[m_2(\{a,b,y_j\}, \rho_3); \{p,x_1,y_j\}\big],
\end{alignat*}
On the other hand, for \(\mathcal{F}_{\mathrm{III},\beta}\) and \(\mathcal{F}_{\mathrm{I},\alpha}\) the \textit{complement} of the same domain has index \(1\) if considered as contributing (potentially zero) to the count 
\[\big[ m_4(\{p,x_1,y_j\}, \rho_3,\rho_2,\rho_1\big);\{a,b,y_j\}\big] ,\qquad 1\leq j\leq m.\]\par
Similarly, the domain on the right of Figure \ref{fig: domains one} contributes to the following counts: 
\begin{alignat*}{2}
&\text{For }\mathcal{F}_{\mathrm{IV},\beta}\text{ and }\mathcal{F}_{\mathrm{II},\alpha},&&\quad \big[m_2(\{c,d,x_i\}, \rho_1); \{p,x_i,y_m\}\big],\\
&\text{For }\mathcal{F}_{\mathrm{III},\beta}\text{ and }\mathcal{F}_{\mathrm{I},\alpha},&&\quad \big[m_2(\{c,d,x_i\}, \rho_2); \{p,x_i,y_m\}\big],\qquad\qquad  1\leq i\leq n.\\
&\text{For }\mathcal{F}_{\mathrm{II},\beta}\text{ and }\mathcal{F}_{\mathrm{IV},\alpha},&&\quad \big[m_2(\{c,d,x_i\}, \rho_3); \{p,x_i,y_m\}\big],
\end{alignat*}
For \(\mathcal{F}_{\mathrm{I},\beta}\) and \(\mathcal{F}_{\mathrm{III},\alpha}\) the complement of the same domain has index \(1\) if considered as contributing to the count 
\[\big[m_4(\{p,x_i,y_m\}, \rho_3,\rho_2,\rho_1); \{c,d,x_i\}\big],\qquad 1\leq i\leq n.\]\par
Next, consider the complementary simple domains in Figure \ref{fig: domains two}
\begin{figure}[ht]
\centering
\includegraphics[scale=0.17]{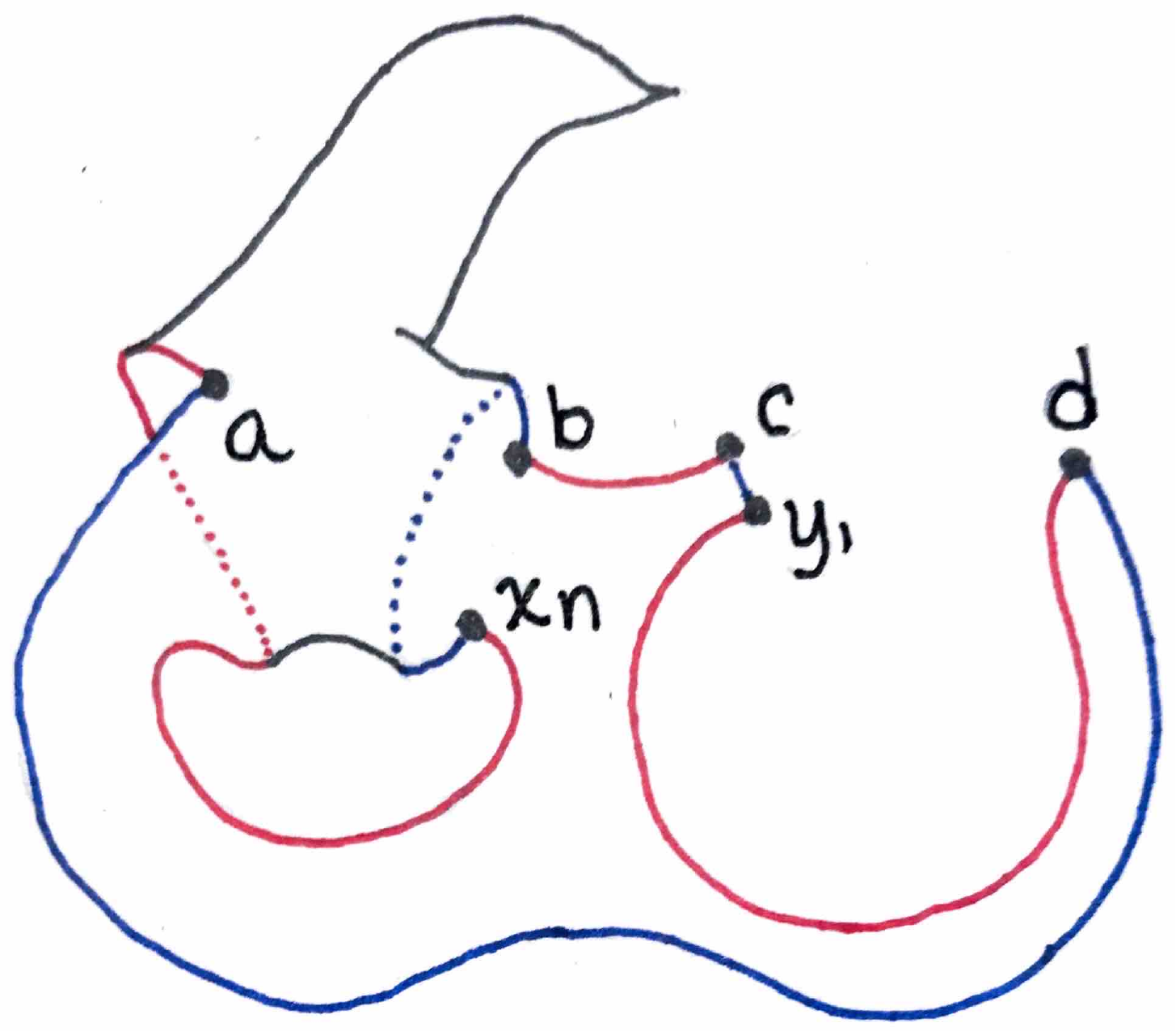}\qquad\includegraphics[scale=0.17]{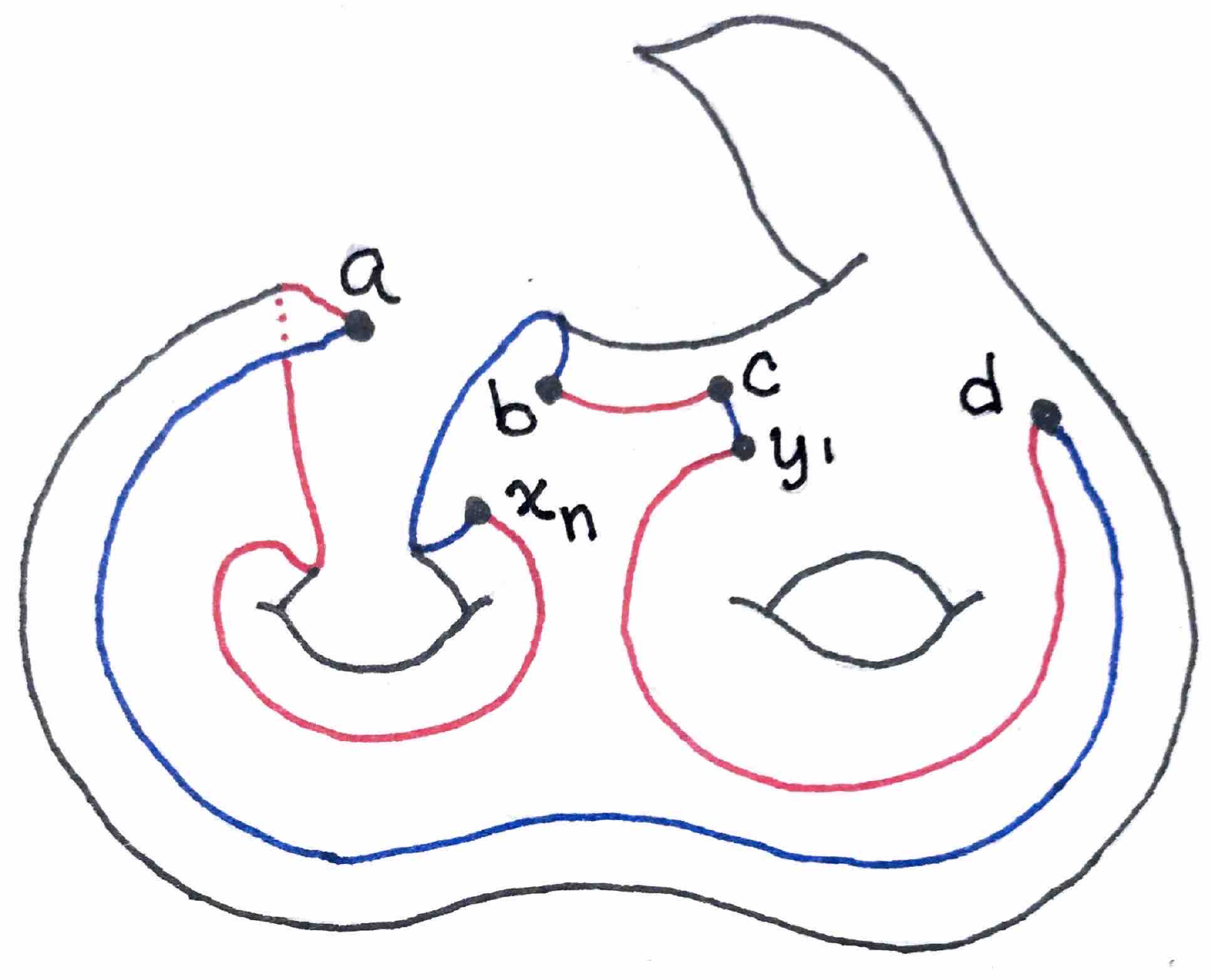}
\caption{}
\label{fig: domains two}
\end{figure}

The domain on the left of Figure \ref{fig: domains two} has index \(1\) if considered as contributing to the following counts:
\begin{alignat*}{2}
&\text{For }\mathcal{F}_{\mathrm{I},\beta}\text{ and }\mathcal{F}_{\mathrm{III},\alpha},&&\quad \big[m_3(\{c,d,x_n\},\rho_2,\rho_1); \{a,b,y_1\}\big],\\
&\text{For }\mathcal{F}_{\mathrm{IV},\beta}\text{ and }\mathcal{F}_{\mathrm{II},\alpha},&&\quad \big[m_3(\{c,d,x_n\},\rho_3,\rho_2); \{a,b,y_1\}\big].
\end{alignat*}
The complement of the same domain, i.e., the domain on the right, has index \(1\) if considered as contributing to the following counts:
\begin{alignat*}{2}
&\text{For }\mathcal{F}_{\mathrm{III},\beta}\text{ and }\mathcal{F}_{\mathrm{I},\alpha},&&\quad \big[m_3(\{a,b,y_1\},\rho_2,\rho_1); \{c,d,x_n\}\big],\\
&\text{For }\mathcal{F}_{\mathrm{II},\beta}\text{ and }\mathcal{F}_{\mathrm{IV},\alpha},&&\quad \big[m_3(\{a,b,y_1\},\rho_3,\rho_2); \{c,d,x_n\}\big].
\end{alignat*}\par
Before identifying the LOSS invariant, we introduce a quick definition.
\begin{definition} Let \(\mathcal{F}\) be any framing of \(\partial \big({-Y(n,m)}\big)\) with one of the arcs forming the boundary of a page. Call the framing \(\mathcal{F}\) \textbf{consistent} with the co-orientation of the contact structure if the basepoint lies on the side of the page determined by moving slightly in the positive Reeb direction. 
\end{definition}
\begin{proposition}
\label{prop: LOSS in graph}
The LOSS invariant \({\widehat{\mathcal{L}}(K)}\) is given by one of the cyan or pink generators in Figure {\normalfont\ref{fig: LOSS in graph}}, depending on the choice of co-orientation for the contact structure on \(Y(n,m)\) and the choice of framing. 
\begin{figure}[ht]
\centering
\begin{tikzpicture}
\def\n{3}
\def\m{4}
\def\gap{1}
\def\biggap{2}
\def\border{0.35}
\def\epsilon{0.15}
\draw[thick] (-\border,\border) rectangle ({(\n+2)*\gap}, {-(\m+2)*\gap});
\node at ({(\n+2)*\gap-1.3},0) {\normalfont Type A, I/III};
\draw[thick] ({\biggap+(\n+2)*\gap-\border},\border) rectangle ({\biggap+(2*\n+4)*\gap}, {-(\m+2)*\gap});
\node at ({\biggap+(2*\n+4)*\gap-1.3},0) {\normalfont Type A, II/IV};
\draw[thick, ->] ({-\border},{\border}) -- ({-\epsilon/sqrt(2)},{\epsilon/sqrt(2)});
\draw[thick] ({\biggap+(\n+2)*\gap-\border},{\border}) -- ({\biggap+(\n+2)*\gap-\epsilon/sqrt(2)},{\epsilon/sqrt(2)});
\filldraw[thick, fill=white] (0,0) circle (2pt);
\filldraw[black] ({\biggap+(\n+2)*\gap},0) circle (2pt);
\draw[thick, ->] ({\epsilon/sqrt(2)},{-\epsilon/sqrt(2)}) -- ({\gap-\epsilon/sqrt(2)}, {\epsilon/sqrt(2)-\gap})node[fill=white,inner sep=2pt,midway] {\scriptsize \(2\)};
\draw[thick, ->] ({\biggap+(\n+2)*\gap+\epsilon/sqrt(2)},{-\epsilon/sqrt(2)}) -- ({\biggap+(\n+3)*\gap-\epsilon/sqrt(2)}, {\epsilon/sqrt(2)-\gap})node[fill=white,inner sep=2pt,midway] {\scriptsize \(1\)};
\foreach\x in {2,...,\n}
    {
    \filldraw[thick, fill=white] ({(\x-1)*\gap},0) circle (2pt);
    \filldraw[black] ({\biggap+(\n+\x+1)*\gap},0) circle (2pt);
    \draw[thick, ->] ({(\x-1)*\gap +\epsilon/sqrt(2)},{-\epsilon/sqrt(2)}) -- ({\x*\gap-\epsilon/sqrt(2)}, {\epsilon/sqrt(2)-\gap})node[fill=white,inner sep=2pt,midway] {\scriptsize \(2\)};
    \draw[thick, ->] ({\biggap+(\n+\x+1)*\gap +\epsilon/sqrt(2)},{-\epsilon/sqrt(2)}) -- ({\biggap+(\n+2+\x)*\gap-\epsilon/sqrt(2)}, {\epsilon/sqrt(2)-\gap})node[fill=white,inner sep=2pt,midway] {\scriptsize \(1\)};
    \draw[thick] ({(\x-1)*\gap}, {-\m*\gap-\epsilon}) -- ({(\x-1)*\gap},{-(\m+2)*\gap})node[fill=white,inner sep=2pt,midway] {\scriptsize \(321\)};
    \draw[thick, ->] ({\biggap+(\n+\x+1)*\gap},{-(\m+2)*\gap}) -- ({\biggap +(\n+\x+1)*\gap}, {-\m*\gap-\epsilon}) node[fill=white,inner sep=2pt,midway] {\scriptsize \(3\)};
    \draw[thick, ->] ({(\x-1)*\gap}, {\border})--({(\x-1)*\gap}, {\epsilon});
    \draw[thick, -] ({\biggap+(\n+\x+1)*\gap}, {\border})--({\biggap+(\n+\x+1)*\gap}, {\epsilon});
    }
\foreach\y in {2,...,\m} 
    {
    \filldraw[thick, fill=white] (0,{(1-\y)*\gap}) circle (2pt);
    \filldraw[black] ({\biggap+(\n+2)*\gap},{(1-\y)*\gap}) circle (2pt);
    \draw[thick, ->] ({\epsilon/sqrt(2)},{(1-\y)*\gap -\epsilon/sqrt(2)}) -- ({\gap-\epsilon/sqrt(2)}, {\epsilon/sqrt(2)-\y*\gap})node[fill=white,inner sep=2pt,midway] {\scriptsize \(2\)};
    \draw[thick, ->] ({\biggap + (\n+2)*\gap + \epsilon/sqrt(2)},{(1-\y)*\gap -\epsilon/sqrt(2)}) -- ({\biggap + (\n+3)*\gap-\epsilon/sqrt(2)}, {\epsilon/sqrt(2)-\y*\gap})node[fill=white,inner sep=2pt,midway] {\scriptsize \(1\)};
    \draw[thick] ({\n*\gap+\epsilon},{(1-\y)*\gap})--({(\n+2)*\gap},{(1-\y)*\gap}) node[fill=white,inner sep=2pt,midway] {\scriptsize \(321\)};
    \draw[thick, ->] ({\biggap+(2*\n+4)*\gap},{(1-\y)*\gap})--({\biggap+(2*\n+2)*\gap+\epsilon},{(1-\y)*\gap}) node[fill=white,inner sep=2pt,midway] {\scriptsize \(3\)};
    \draw[thick, ->]({-\border},{(1-\y)*\gap})--({-\epsilon},{(1-\y)*\gap});
    \draw[thick]({\biggap+(\n+2)*\gap-\border},{(1-\y)*\gap})--({\biggap+(\n+2)*\gap-\epsilon},{(1-\y)*\gap});
    }
\foreach \x in {1,...,\n}
    \foreach \y in {1,...,\m}{
        \filldraw[black] ({\x*\gap},{-\y*\gap}) circle (2pt);
        \filldraw[thick, fill = white] ({\biggap+(\n+2+\x)*\gap},{-\y*\gap}) circle (2pt);
        }
\foreach \x in {2,...,\n}
    \foreach \y in {2,...,\m}{
        \draw [thick, ->] ({(\x-1)*\gap+\epsilon/sqrt(2)},{(1-\y)*\gap-\epsilon/sqrt(2)}) --({\x*\gap-\epsilon/sqrt(2)},{-\y*\gap+\epsilon/sqrt(2)}) node[fill=white,inner sep=2pt,midway] {\scriptsize \(32\)};
        \draw [thick, ->] ({\biggap+(\n+\x+1)*\gap+\epsilon/sqrt(2)},{(1-\y)*\gap-\epsilon/sqrt(2)}) --({\biggap+(\n+2+\x)*\gap-\epsilon/sqrt(2)},{-\y*\gap+\epsilon/sqrt(2)}) node[fill=white,inner sep=2pt,midway] {\scriptsize \(21\)};
        }
\filldraw[thick, fill=white] ({(\n+1)*\gap},{-(\m+1)*\gap}) circle (2pt);
\filldraw[black] ({\biggap+(2*\n+3)*\gap},{-(\m+1)*\gap}) circle (2pt);
\draw[thick, ->] ({\n*\gap+\epsilon/sqrt(2)},{-\m*\gap-\epsilon/sqrt(2)}) -- ({(\n+1)*\gap-\epsilon/sqrt(2)},{\epsilon/sqrt(2)-(\m+1)*\gap}) node[fill=white,inner sep=2pt,midway] {\scriptsize \(321\)};
\draw[thick, ->] ({\biggap+(2*\n+3)*\gap-\epsilon/sqrt(2)},{\epsilon/sqrt(2)-(\m+1)*\gap}) -- ({\biggap + (2*\n+2)*\gap+\epsilon/sqrt(2)},{-\m*\gap-\epsilon/sqrt(2)}) node[fill=white,inner sep=2pt,midway] {\scriptsize \(3\)};
\draw[thick] ({(\n+1)*\gap+\epsilon/sqrt(2)},{-(\m+1)*\gap-\epsilon/sqrt(2)}) -- ({(\n+2)*\gap},{-(\m+2)*\gap}) node[fill=white,inner sep=2pt,midway] {\scriptsize \(21\)};
\draw[thick, ->] ({\biggap+(2*\n+4)*\gap},{-(\m+2)*\gap}) -- ({\biggap+(2*\n+3)*\gap+\epsilon/sqrt(2)},{-(\m+1)*\gap-\epsilon/sqrt(2)})node[fill=white,inner sep=2pt,midway] {\scriptsize \(32\)};
\filldraw[cyan] ({\gap},{-\gap}) circle (2pt);
\filldraw[magenta] ({\n*\gap},{-\m*\gap}) circle (2pt);
\filldraw[ultra thick, color = cyan, fill=white] ({\biggap+(\n+3)*\gap},{-\gap}) circle (2pt);
\filldraw[ultra thick, color = magenta, fill=white] ({\biggap+(2*\n+2)*\gap},{-\m*\gap}) circle (2pt);
\end{tikzpicture}
\caption{}
\label{fig: LOSS in graph}
\end{figure}

Specifically, the cyan generators correspond to consistent framing choices for \(\beta\)-type diagrams and inconsistent choices for \(\alpha\)-type diagrams {\normalfont(}and vice-versa for the pink generators{\normalfont)}.
\begin{proof}
We argue using the refined grading on \({\widehat{\mathit{CFA}}}\). We recall the definition here, but for a more detailed discussion see \cite[Chapter 10, 11]{bordered} and \cite{properties}. The refined grading is defined only up to an overall shift on each \(\mathrm{spin}^c\) component. While the full refined grading takes values in a non-abelian group, for our purposes it will suffice to consider the ``\(\mathrm{spin}^c\) component'' of the refined grading. From here on out, we simply refer to this component as ``the grading''. 

In the case of torus boundary, the grading takes values in a quotient of the group \(\frac{1}{2}\mathbb{Z}^2\). The quotient is determined by the group of periodic domains in a diagram. For I/III basepoints, the observations above imply this quotient is 
\[\frac{1}{2}\mathbb{Z}^2\Big /\Big\langle\Big(\frac{n+m}{(n,m)},\frac{nm}{(n,m)}\Big)\Big\rangle.\]
For II/IV basepoints, the quotient is instead
\[\frac{1}{2}\mathbb{Z}^2\Big /\Big\langle\Big(\frac{nm}{(n,m)},-\frac{n+m}{(n,m)}\Big)\Big\rangle.\]
The Reeb elements of the torus algebra are graded by \(\frac{1}{2}\mathbb{Z}^2\) by setting
\[\mathrm{gr}(\rho_1) = ({\textstyle\frac{1}{2}},-{\textstyle\frac{1}{2}}),\qquad \mathrm{gr}(\rho_2) = ({\textstyle\frac{1}{2}},{\textstyle\frac{1}{2}}),\qquad \mathrm{gr}(\rho_3) = (-{\textstyle\frac{1}{2}},{\textstyle\frac{1}{2}})\]
and imposing \(\mathrm{gr}(\rho_I\rho_J)=\mathrm{gr}(\rho_I)\mathrm{gr}(\rho_J)\). The grading on \({\widehat{\mathit{CFA}}}\) then satisfies
\[\mathrm{gr}\big(m_{k+1}(\mathbf{x},\rho_{I_1},\ldots,\rho_{I_k})\big)=\mathrm{gr}(\mathbf{x}) + \mathrm{gr}(\rho_{I_1})+\cdots + \mathrm{gr}(\rho_{I_k}).\]\par
In Table \ref{fig: table}, we show how to calculate grading differences from a Type A decorated graph.
\begin{table}[ht]
\centering
\begin{tblr}{c|l||c|l}
Labeled edge & Grading change & Labeled edge & Grading change \\ 
\hline 
\(\mathbf{x}\xrightarrow{\,1\,}\mathbf{y}\) & \(\mathrm{gr}(\mathbf{y})=\mathrm{gr}(\mathbf{x}) + (\frac{1}{2},-\frac{1}{2})\) & \(\mathbf{x}\xrightarrow{\,21\,}\mathbf{y}\) &  \(\mathrm{gr}(\mathbf{y})=\mathrm{gr}(\mathbf{x}) + (1,0)\)\\ 
\(\mathbf{x}\xrightarrow{\,2\,}\mathbf{y} \) & \(\mathrm{gr}(\mathbf{y})=\mathrm{gr}(\mathbf{x}) + (\frac{1}{2},\frac{1}{2})\) & \(\mathbf{x}\xrightarrow{\,32\,}\mathbf{y} \) & \(\mathrm{gr}(\mathbf{y})=\mathrm{gr}(\mathbf{x}) + (0,1)\)\\
\(\mathbf{x}\xrightarrow{\,3\,}\mathbf{y} \) & \(\mathrm{gr}(\mathbf{y})=\mathrm{gr}(\mathbf{x}) + (-\frac{1}{2},\frac{1}{2})\) & \(\mathbf{x}\xrightarrow{\,321\,}\mathbf{y} \) & \(\mathrm{gr}(\mathbf{y})=\mathrm{gr}(\mathbf{x}) + (\frac{1}{2},\frac{1}{2})\)
\end{tblr}
\caption{Grading differences in a Type A decorated graph}
\label{fig: table}
\end{table}\par
For any choice of framing \(\mathcal{F}\), there are \((n,m)\) many \(\mathrm{spin}^c\) summands in \({{\widehat{\mathit{CFA}}\big({-Y(n,m),\mathcal{F}}\big)}}\). To see this, one can count the number of connected components in Figure \ref{fig: LOSS in graph} or observe \[H_1\big(Y(n,m),\partial Y(n,m)\big)\cong \mathbb{Z}/(n,m).\]\par 
Let us first focus on the I/III case. We start with the \(\mathrm{spin}^c\) summand which contains the top-leftmost generator in the I/III graph of Figure \ref{fig: LOSS in graph}. As we go around the directed cycle starting from this generator, we pass by
\[\frac{n+m}{(n,m)}-1,\frac{n+m}{(n,m)}-1,\frac{nm-n-m}{(n,m)}+1, 1\]
arrow(s) labeled \(2\), \(321\), \(32\), \(21\), respectively. Traversing this cycle once, we ``path lift'' the gradings to \(\frac{1}{2}\mathbb{Z}^2\) starting from \((0,0)\). Both components of the grading are non-decreasing as the grading changes from \((0,0)\) to 
\[\Big(\frac{n+m}{(n,m)},\frac{nm}{(n,m)}\Big). \label{eq: fraction} \]
The second component is only stationary as we pass the unique arrow labeled \(21\). \par 
Suppose we play the same game starting from a generator in any other \(\mathrm{spin}^c\) class. We pass by 
\[\frac{n+m}{(n,m)},\frac{n+m}{(n,m)},\frac{nm-n-m}{(n,m)}\]
arrows labeled \(2\), \(321\), \(32\) respectively. When we path lift, both components are non-decreasing as the grading changes from \((0,0)\) to the same value as before. This time, the second coordinate strictly increases. It follows that the top left and bottom right generators in the I/III graph are distinguished as the unique pair belonging to the same \(\mathrm{spin}^c\) class whose gradings differ by 
\[ (1,0) +\mathbb{Z}\Big(\frac{n+m}{(n,m)},\frac{nm}{(n,m)}\Big).\]
A similar argument shows that the top left and bottom right generators in the II/IV graph are distinguished as the unique pair belonging to the same \(\mathrm{spin}^c\) class whose gradings differ by 
\[ (0,1) +\mathbb{Z}\Big(\frac{nm}{(n,m)},-\frac{n+m}{(n,m)}\Big).\]\par
Now consider the bordered diagram in Figure \ref{fig: bordered} with framing \(\mathcal{F}_{\mathrm{I},\beta}\). (The case of other framings is similar). The existence of the index 1 domain on the left in Figure \ref{fig: domains two} (considered as contributing to an \(m_3(-,\rho_2,\rho_1)\) action) implies 
\[\mathrm{gr}\{a,b,y_1\} = \mathrm{gr}\{c,d,x_n\}+(1,0).\]
We conclude any graded equivalence of Type A modules must send \(\{a,b,y_1\}\) and \(\{c,d,x_n\}\) to the top left and bottom right generators in Figure \ref{fig: LOSS in graph}, respectively. Using the domain on the left in Figure \ref{fig: domains one} and the complement of the domain on the right in the same figure, we can conclude again by gradings that \(\{p,x_1,y_1\}\) and \(\{p,x_n,y_n\}\) must be sent to the cyan and pink generators in Figure \ref{fig: LOSS in graph}, respectively.
\end{proof}
\end{proposition}
\begin{corollary}
The Type A contact invariant for \(Y(n,m)\) is given by one of the generators shown in \ref{fig: type A invariant}, depending on the choice of co-orientation for the contact structure on \(Y(n,m)\) and the choice of framing. 
\begin{figure}[ht]
\centering
\begin{tikzpicture}
\def\n{3}
\def\m{4}
\def\gap{1}
\def\biggap{2}
\def\border{0.35}
\def\epsilon{0.15}
\draw[thick] (-\border,\border) rectangle ({(\n+2)*\gap}, {-(\m+2)*\gap});
\node at ({(\n+2)*\gap-1.3},0) {\normalfont Type A, I/III};
\draw[thick] ({\biggap+(\n+2)*\gap-\border},\border) rectangle ({\biggap+(2*\n+4)*\gap}, {-(\m+2)*\gap});
\node at ({\biggap+(2*\n+4)*\gap-1.3},0) {\normalfont Type A, II/IV};
\draw[thick, ->] ({-\border},{\border}) -- ({-\epsilon/sqrt(2)},{\epsilon/sqrt(2)});
\draw[thick] ({\biggap+(\n+2)*\gap-\border},{\border}) -- ({\biggap+(\n+2)*\gap-\epsilon/sqrt(2)},{\epsilon/sqrt(2)});
\filldraw[thick, fill=white] (0,0) circle (2pt);
\filldraw[black] ({\biggap+(\n+2)*\gap},0) circle (2pt);
\draw[thick, ->] ({\epsilon/sqrt(2)},{-\epsilon/sqrt(2)}) -- ({\gap-\epsilon/sqrt(2)}, {\epsilon/sqrt(2)-\gap})node[fill=white,inner sep=2pt,midway] {\scriptsize \(2\)};
\draw[thick, ->] ({\biggap+(\n+2)*\gap+\epsilon/sqrt(2)},{-\epsilon/sqrt(2)}) -- ({\biggap+(\n+3)*\gap-\epsilon/sqrt(2)}, {\epsilon/sqrt(2)-\gap})node[fill=white,inner sep=2pt,midway] {\scriptsize \(1\)};
\foreach\x in {2,...,\n}
    {
    \filldraw[thick, fill=white] ({(\x-1)*\gap},0) circle (2pt);
    \filldraw[black] ({\biggap+(\n+\x+1)*\gap},0) circle (2pt);
    \draw[thick, ->] ({(\x-1)*\gap +\epsilon/sqrt(2)},{-\epsilon/sqrt(2)}) -- ({\x*\gap-\epsilon/sqrt(2)}, {\epsilon/sqrt(2)-\gap})node[fill=white,inner sep=2pt,midway] {\scriptsize \(2\)};
    \draw[thick, ->] ({\biggap+(\n+\x+1)*\gap +\epsilon/sqrt(2)},{-\epsilon/sqrt(2)}) -- ({\biggap+(\n+2+\x)*\gap-\epsilon/sqrt(2)}, {\epsilon/sqrt(2)-\gap})node[fill=white,inner sep=2pt,midway] {\scriptsize \(1\)};
    \draw[thick] ({(\x-1)*\gap}, {-\m*\gap-\epsilon}) -- ({(\x-1)*\gap},{-(\m+2)*\gap})node[fill=white,inner sep=2pt,midway] {\scriptsize \(321\)};
    \draw[thick, ->] ({\biggap+(\n+\x+1)*\gap},{-(\m+2)*\gap}) -- ({\biggap +(\n+\x+1)*\gap}, {-\m*\gap-\epsilon}) node[fill=white,inner sep=2pt,midway] {\scriptsize \(3\)};
    \draw[thick, ->] ({(\x-1)*\gap}, {\border})--({(\x-1)*\gap}, {\epsilon});
    \draw[thick, -] ({\biggap+(\n+\x+1)*\gap}, {\border})--({\biggap+(\n+\x+1)*\gap}, {\epsilon});
    }
\foreach\y in {2,...,\m} 
    {
    \filldraw[thick, fill=white] (0,{(1-\y)*\gap}) circle (2pt);
    \filldraw[black] ({\biggap+(\n+2)*\gap},{(1-\y)*\gap}) circle (2pt);
    \draw[thick, ->] ({\epsilon/sqrt(2)},{(1-\y)*\gap -\epsilon/sqrt(2)}) -- ({\gap-\epsilon/sqrt(2)}, {\epsilon/sqrt(2)-\y*\gap})node[fill=white,inner sep=2pt,midway] {\scriptsize \(2\)};
    \draw[thick, ->] ({\biggap + (\n+2)*\gap + \epsilon/sqrt(2)},{(1-\y)*\gap -\epsilon/sqrt(2)}) -- ({\biggap + (\n+3)*\gap-\epsilon/sqrt(2)}, {\epsilon/sqrt(2)-\y*\gap})node[fill=white,inner sep=2pt,midway] {\scriptsize \(1\)};
    \draw[thick] ({\n*\gap+\epsilon},{(1-\y)*\gap})--({(\n+2)*\gap},{(1-\y)*\gap}) node[fill=white,inner sep=2pt,midway] {\scriptsize \(321\)};
    \draw[thick, ->] ({\biggap+(2*\n+4)*\gap},{(1-\y)*\gap})--({\biggap+(2*\n+2)*\gap+\epsilon},{(1-\y)*\gap}) node[fill=white,inner sep=2pt,midway] {\scriptsize \(3\)};
    \draw[thick, ->]({-\border},{(1-\y)*\gap})--({-\epsilon},{(1-\y)*\gap});
    \draw[thick]({\biggap+(\n+2)*\gap-\border},{(1-\y)*\gap})--({\biggap+(\n+2)*\gap-\epsilon},{(1-\y)*\gap});
    }
\foreach \x in {1,...,\n}
    \foreach \y in {1,...,\m}{
        \filldraw[black] ({\x*\gap},{-\y*\gap}) circle (2pt);
        \filldraw[thick, fill = white] ({\biggap+(\n+2+\x)*\gap},{-\y*\gap}) circle (2pt);
        }
\foreach \x in {2,...,\n}
    \foreach \y in {2,...,\m}{
        \draw [thick, ->] ({(\x-1)*\gap+\epsilon/sqrt(2)},{(1-\y)*\gap-\epsilon/sqrt(2)}) --({\x*\gap-\epsilon/sqrt(2)},{-\y*\gap+\epsilon/sqrt(2)}) node[fill=white,inner sep=2pt,midway] {\scriptsize \(32\)};
        \draw [thick, ->] ({\biggap+(\n+\x+1)*\gap+\epsilon/sqrt(2)},{(1-\y)*\gap-\epsilon/sqrt(2)}) --({\biggap+(\n+2+\x)*\gap-\epsilon/sqrt(2)},{-\y*\gap+\epsilon/sqrt(2)}) node[fill=white,inner sep=2pt,midway] {\scriptsize \(21\)};
        }
\filldraw[thick, fill=white] ({(\n+1)*\gap},{-(\m+1)*\gap}) circle (2pt);
\filldraw[black] ({\biggap+(2*\n+3)*\gap},{-(\m+1)*\gap}) circle (2pt);
\draw[thick, ->] ({\n*\gap+\epsilon/sqrt(2)},{-\m*\gap-\epsilon/sqrt(2)}) -- ({(\n+1)*\gap-\epsilon/sqrt(2)},{\epsilon/sqrt(2)-(\m+1)*\gap}) node[fill=white,inner sep=2pt,midway] {\scriptsize \(321\)};
\draw[thick, ->] ({\biggap+(2*\n+3)*\gap-\epsilon/sqrt(2)},{\epsilon/sqrt(2)-(\m+1)*\gap}) -- ({\biggap + (2*\n+2)*\gap+\epsilon/sqrt(2)},{-\m*\gap-\epsilon/sqrt(2)}) node[fill=white,inner sep=2pt,midway] {\scriptsize \(3\)};
\draw[thick] ({(\n+1)*\gap+\epsilon/sqrt(2)},{-(\m+1)*\gap-\epsilon/sqrt(2)}) -- ({(\n+2)*\gap},{-(\m+2)*\gap}) node[fill=white,inner sep=2pt,midway] {\scriptsize \(21\)};
\draw[thick, ->] ({\biggap+(2*\n+4)*\gap},{-(\m+2)*\gap}) -- ({\biggap+(2*\n+3)*\gap+\epsilon/sqrt(2)},{-(\m+1)*\gap-\epsilon/sqrt(2)})node[fill=white,inner sep=2pt,midway] {\scriptsize \(32\)};
\filldraw[ultra thick, color = blue, fill=white] (0,0) circle (2pt);
\filldraw[ultra thick, color = red, fill=white] ({(\n+1)*\gap},{-(\m+1)*\gap}) circle (2pt);
\filldraw[blue] ({\biggap+(\n+2)*\gap},{0}) circle (2pt);
\filldraw[red] ({\biggap+(2*\n+3)*\gap},{-(\m+1)*\gap}) circle (2pt);
\end{tikzpicture}
\caption{}
\label{fig: type A invariant}
\end{figure}

The blue generators correspond to consistent framing choices for \(\beta\)-type diagrams and inconsistent choices for \(\alpha\)-type diagrams {\normalfont(}and vice-versa for the red generators{\normalfont)}.
\end{corollary}
\begin{proof}
There are two basic slices we can attach to the boundary of \(Y(n,m)\) to obtain sutures which are meridians of \(K\). One of these basic slices will be the correct sign for the Stipsicz--V\'{e}rtesi map (\ref{eq: SV}) for one choice of co-orientation of contact structure on \(Y(n,m)\), and the incorrect sign for the other (and vice-versa for the other basic slice). According to \cite{MinVarvarezos},
\begin{itemize}
    \item For a I/III basepoint, one of these basic slices corresponds to an \(m_2(-,\rho_2)\) action, while the other has no representative in the torus algebra. (This has to do with the fact that elements of the torus algebra correspond to contact structures on a thickened \textit{punctured} torus.)
    \item For a II/IV basepoint, one of these basic slices corresponds to an \(m_2(-,\rho_1)\) action, while the other corresponds to an \(m_2(-,\rho_3)\) action. 
\end{itemize}
Let us focus on the case of a I/III basepoint. (The case for a II/IV basepoint is similar and easier.) Since the cyan LOSS invariants is hit by a single \(\rho_2\) action and the pink LOSS invariant is not hit by any \(\rho_2\) action, this allows us to identify the blue generator in the I/III graph of Figure \ref{fig: type A invariant} as a contact invariant. (This sort of trickery would not be necessary if the author had more carefully read about signs of basic slices.) 

To identify the other contact invariant, recall in the bordered diagram of Figure \ref{fig: bordered} with \(\mathcal{F}_{\mathrm{I},\beta}\) framing the blue contact invariant corresponds to \(\{a,b,y_1\}\). The grading argument in Proposition \ref{prop: LOSS in graph} implies that in the same bordered Heegaard diagram with \(\mathcal{F}_{\mathrm{III},\beta}\) framing, \(\{a,b,y_1\}\) corresponds to the red generator in the I/III graph of Figure \ref{fig: type A invariant}. Note this is the contact invariant for the same contact structure with \(\mathcal{F}_{\mathrm{III},\beta}\) framing, since attaching a sutured cap to a bordered Heegaard diagram forgets any information about the basepoint. 
\end{proof}
\begin{remark}\label{re:alpha_beta}
It is tempting to assume in that the Type A contact invariant does not change when switching from an \(\alpha\)-type to a \(\beta\)-type framing, since the corresponding \(\mathit{\widehat{CFA} }\) modules do not change. However, switching from an \(\alpha\)-type bordered diagram to a \(\beta\)-type bordered diagram exchanges the roles of \(R_+\) and \(R_{-}\) when a sutured cap is attached. In the case of a higher genus boundary, this makes it clear that when switching from an \(\alpha\)-type to \(\beta\)-type diagram, the idempotent in the surface algebra corresponding to the Type A contact invariant cannot remain the same.
\end{remark}
\section{Reparameterizing and gluing} \label{sec: gluing}
In this section, we finish the reproof of Theorem \ref{thm: Lekili} by pairing bordered contact invariants. We refer the reader to Appendix \ref{sec: algebra} for a review of Type DD and DA bimodules, as well as various types of box tensor products. We make the precursory comment that all bimodules in this section are left and right bounded in the sense of \cite{bimodules}, as they are computed from left and right provincially admissible Heegaard diagrams. Indeed, these diagrams have no left or right provincial periodic domains.  

In order to carry out the pairing we first identify the Type D invariant for \(Y(n,m)\) with the 8 framings from the previous section. If \(\mathcal{F}\) is such a framing, then we calculate in Appendix \ref{sec: calculating} that the module \(\widehat{\mathit{CFD}}\big({-Y(n,m),\mathcal{F}}\big)\) has a model given by one of the decorated graphs in Figure \ref{fig: initial type D}.
\begin{figure}[ht]
\centering
\begin{tikzpicture}
\def\n{3}
\def\m{4}
\def\gap{1}
\def\biggap{2}
\def\border{0.35}
\def\epsilon{0.15}
\draw[thick] (-\border,\border) rectangle ({(\n+2)*\gap}, {-(\m+2)*\gap});
\node at ({(\n+2)*\gap-1.3},0) {Type D, I/III};
\draw[thick] ({\biggap+(\n+2)*\gap-\border},\border) rectangle ({\biggap+(2*\n+4)*\gap}, {-(\m+2)*\gap});
\node at ({\biggap+(2*\n+4)*\gap-1.3},0) {Type D, II/IV};
\draw[thick, ->] ({-\border},{\border}) -- ({-\epsilon/sqrt(2)},{\epsilon/sqrt(2)});
\draw[thick] ({\biggap+(\n+2)*\gap-\border},{\border}) -- ({\biggap+(\n+2)*\gap-\epsilon/sqrt(2)},{\epsilon/sqrt(2)});
\filldraw[thick, fill=white] (0,0) circle (2pt);
\filldraw[black] ({\biggap+(\n+2)*\gap},0) circle (2pt);
\draw[thick, ->] ({\epsilon/sqrt(2)},{-\epsilon/sqrt(2)}) -- ({\gap-\epsilon/sqrt(2)}, {\epsilon/sqrt(2)-\gap})node[fill=white,inner sep=2pt,midway] {\scriptsize \(2\)};
\draw[thick, ->] ({\biggap+(\n+2)*\gap+\epsilon/sqrt(2)},{-\epsilon/sqrt(2)}) -- ({\biggap+(\n+3)*\gap-\epsilon/sqrt(2)}, {\epsilon/sqrt(2)-\gap})node[fill=white,inner sep=2pt,midway] {\scriptsize \(3\)};
\foreach\x in {2,...,\n}
    {
    \filldraw[thick, fill=white] ({(\x-1)*\gap},0) circle (2pt);
    \filldraw[black] ({\biggap+(\n+\x+1)*\gap},0) circle (2pt);
    \draw[thick, ->] ({(\x-1)*\gap +\epsilon/sqrt(2)},{-\epsilon/sqrt(2)}) -- ({\x*\gap-\epsilon/sqrt(2)}, {\epsilon/sqrt(2)-\gap})node[fill=white,inner sep=2pt,midway] {\scriptsize \(2\)};
    \draw[thick, ->] ({\biggap+(\n+\x+1)*\gap +\epsilon/sqrt(2)},{-\epsilon/sqrt(2)}) -- ({\biggap+(\n+2+\x)*\gap-\epsilon/sqrt(2)}, {\epsilon/sqrt(2)-\gap})node[fill=white,inner sep=2pt,midway] {\scriptsize \(3\)};
    \draw[thick] ({(\x-1)*\gap}, {-\m*\gap-\epsilon}) -- ({(\x-1)*\gap},{-(\m+2)*\gap})node[fill=white,inner sep=2pt,midway] {\scriptsize \(123\)};
    \draw[thick, ->] ({\biggap+(\n+\x+1)*\gap},{-(\m+2)*\gap}) -- ({\biggap +(\n+\x+1)*\gap}, {-\m*\gap-\epsilon}) node[fill=white,inner sep=2pt,midway] {\scriptsize \(1\)};
    \draw[thick, ->] ({(\x-1)*\gap}, {\border})--({(\x-1)*\gap}, {\epsilon});
    \draw[thick, -] ({\biggap+(\n+\x+1)*\gap}, {\border})--({\biggap+(\n+\x+1)*\gap}, {\epsilon});
    }
\foreach\y in {2,...,\m} 
    {
    \filldraw[thick, fill=white] (0,{(1-\y)*\gap}) circle (2pt);
    \filldraw[black] ({\biggap+(\n+2)*\gap},{(1-\y)*\gap}) circle (2pt);
    \draw[thick, ->] ({\epsilon/sqrt(2)},{(1-\y)*\gap -\epsilon/sqrt(2)}) -- ({\gap-\epsilon/sqrt(2)}, {\epsilon/sqrt(2)-\y*\gap})node[fill=white,inner sep=2pt,midway] {\scriptsize \(2\)};
    \draw[thick, ->] ({\biggap + (\n+2)*\gap + \epsilon/sqrt(2)},{(1-\y)*\gap -\epsilon/sqrt(2)}) -- ({\biggap + (\n+3)*\gap-\epsilon/sqrt(2)}, {\epsilon/sqrt(2)-\y*\gap})node[fill=white,inner sep=2pt,midway] {\scriptsize \(3\)};
    \draw[thick] ({\n*\gap+\epsilon},{(1-\y)*\gap})--({(\n+2)*\gap},{(1-\y)*\gap}) node[fill=white,inner sep=2pt,midway] {\scriptsize \(123\)};
    \draw[thick, ->] ({\biggap+(2*\n+4)*\gap},{(1-\y)*\gap})--({\biggap+(2*\n+2)*\gap+\epsilon},{(1-\y)*\gap}) node[fill=white,inner sep=2pt,midway] {\scriptsize \(1\)};
    \draw[thick, ->]({-\border},{(1-\y)*\gap})--({-\epsilon},{(1-\y)*\gap});
    \draw[thick]({\biggap+(\n+2)*\gap-\border},{(1-\y)*\gap})--({\biggap+(\n+2)*\gap-\epsilon},{(1-\y)*\gap});
    }
\foreach \x in {1,...,\n}
    \foreach \y in {1,...,\m}{
        \filldraw[black] ({\x*\gap},{-\y*\gap}) circle (2pt);
        \filldraw[thick, fill = white] ({\biggap+(\n+2+\x)*\gap},{-\y*\gap}) circle (2pt);
        }
\foreach \x in {2,...,\n}
    \foreach \y in {2,...,\m}{
        \draw [thick, ->] ({(\x-1)*\gap+\epsilon/sqrt(2)},{(1-\y)*\gap-\epsilon/sqrt(2)}) --({\x*\gap-\epsilon/sqrt(2)},{-\y*\gap+\epsilon/sqrt(2)}) node[fill=white,inner sep=2pt,midway] {\scriptsize \(12\)};
        \draw [thick, ->] ({\biggap+(\n+\x+1)*\gap+\epsilon/sqrt(2)},{(1-\y)*\gap-\epsilon/sqrt(2)}) --({\biggap+(\n+2+\x)*\gap-\epsilon/sqrt(2)},{-\y*\gap+\epsilon/sqrt(2)}) node[fill=white,inner sep=2pt,midway] {\scriptsize \(23\)};
        }
\filldraw[thick, fill=white] ({(\n+1)*\gap},{-(\m+1)*\gap}) circle (2pt);
\filldraw[black] ({\biggap+(2*\n+3)*\gap},{-(\m+1)*\gap}) circle (2pt);
\draw[thick, ->] ({\n*\gap+\epsilon/sqrt(2)},{-\m*\gap-\epsilon/sqrt(2)}) -- ({(\n+1)*\gap-\epsilon/sqrt(2)},{\epsilon/sqrt(2)-(\m+1)*\gap}) node[fill=white,inner sep=2pt,midway] {\scriptsize \(123\)};
\draw[thick, ->] ({\biggap+(2*\n+3)*\gap-\epsilon/sqrt(2)},{\epsilon/sqrt(2)-(\m+1)*\gap}) -- ({\biggap + (2*\n+2)*\gap+\epsilon/sqrt(2)},{-\m*\gap-\epsilon/sqrt(2)}) node[fill=white,inner sep=2pt,midway] {\scriptsize \(1\)};
\draw[thick] ({(\n+1)*\gap+\epsilon/sqrt(2)},{-(\m+1)*\gap-\epsilon/sqrt(2)}) -- ({(\n+2)*\gap},{-(\m+2)*\gap}) node[fill=white,inner sep=2pt,midway] {\scriptsize \(23\)};
\draw[thick, ->] ({\biggap+(2*\n+4)*\gap},{-(\m+2)*\gap}) -- ({\biggap+(2*\n+3)*\gap+\epsilon/sqrt(2)},{-(\m+1)*\gap-\epsilon/sqrt(2)})node[fill=white,inner sep=2pt,midway] {\scriptsize \(12\)};
\end{tikzpicture}
\caption{The case \(n=3\), \(m=4\). Both decorated graphs should be interpreted as lying on a torus.}
\label{fig: initial type D}
\end{figure}

The interpretation of a decorated graph is different for a left Type D module. Vertices have the same interpretation as before. (Recall these are labeled according to non-trivial idempotent actions, not which arcs the generators lie on in a diagram). A directed path of length \(k\) from a generator \(\mathbf{x}\) to a generator \(\mathbf{y}\) with arrows labeled \(I_1,\ldots,I_k\) corresponds to a differential 
\[\delta^k (\mathbf{x}) = \rho_{I_k}\otimes \cdots\otimes \rho_{I_k}\otimes \mathbf{y}.\]
An unlabeled arrow corresponds to a pure differential, i.e., \(\delta^1(x) = \mathbf{1}\otimes \mathbf{y}\). For example, a path \({\bullet \xrightarrow{\,12\,}\bullet \xrightarrow{\,123\,}\bullet}\) corresponds to the \(\delta^2\) differential
\[\delta^2(\mathbf{x})=\rho_{12}\otimes \rho_{123}\otimes \mathbf{y}.\]\par
To identify the Type D invariant, we follow the procedure in \cite{MinVarvarezos} and pair with \({\widehat{\mathit{BSDD}}(\mathcal{TW}_{\mathcal{F}}^+)}\), where \(\mathcal{TW}^+_{\mathcal{F}}\) is  the \textit{twisting slice}. In the case of torus boundary, \(\mathcal{TW}^+_{\mathcal{F}}\) has a bordered-sutured Heegaard diagram shown in Figure \ref{fig: twisting slice}.
\begin{figure}[ht]
\centering
\includegraphics[scale=0.8]{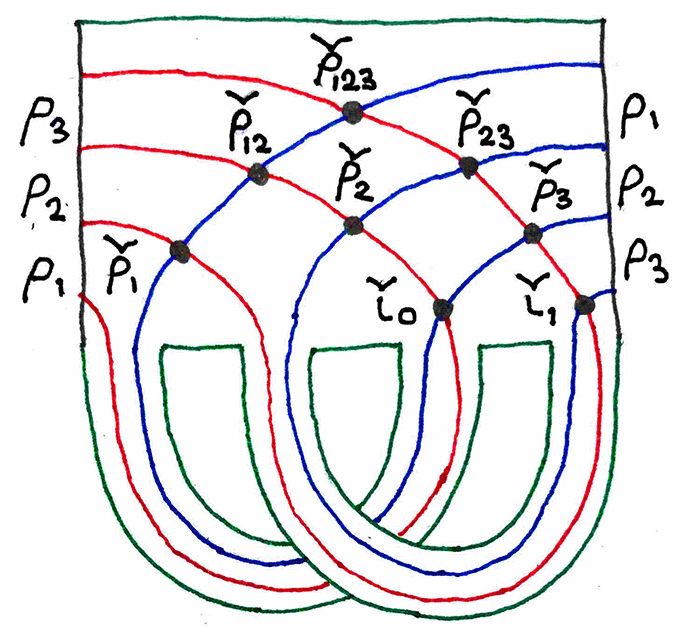}
\caption{}
\label{fig: twisting slice}
\end{figure}

From the above diagram one can calculate the following model for \(\widehat{\mathit{BSDD}}(\mathcal{TW}_{\mathcal{F}}^+)\). The non-trivial idempotent actions are given by
\begin{alignat*}{4}
\iota_0\cdot \widecheck{\rho}_{123}\cdot \iota_1 &= \widecheck{\rho}_{123},&\qquad \iota_0\cdot \widecheck{\rho}_{23}\cdot \iota_0 &= \widecheck{\rho}_{23},&\qquad  \iota_1\cdot \widecheck{\rho}_{12}\cdot \iota_1 &= \widecheck{\rho}_{12},&\qquad \iota_0 \cdot \widecheck{\rho}_3\cdot \iota_1 &= \widecheck{\rho}_3,\\ 
\iota_0\cdot \widecheck{\rho}_1\cdot \iota_1 &=\widecheck{\rho}_1,&\quad \iota_1\cdot \widecheck{\rho}_2\cdot \iota_0 &=\widecheck{\rho}_2,&\quad \iota_0\cdot \widecheck{\iota}_1\cdot \iota_0&= \widecheck{\iota}_1,&\quad \iota_1\cdot \widecheck{\iota}_0\cdot \iota_1&=\widecheck{\iota}_0.
\end{alignat*}
Next, the non-zero differentials are as follows.
\begin{alignat*}{3}\delta^1 (\widecheck{\rho}_{123})& = \rho_3\otimes \widecheck{\rho}_{12} + \widecheck{\rho}_{23}\otimes \rho_1,&\qquad  \delta^1(\widecheck{\rho}_{23})&=\rho_3\otimes \widecheck{\rho}_2+\widecheck{\rho}_3\otimes \rho_2,&\qquad
\delta^1(\widecheck{\rho}_{12})&=\rho_2\otimes \widecheck{\rho}_1+\widecheck{\rho}_2\otimes \rho_1,\\ \delta^1(\widecheck{\rho}_3)&=\rho_3\otimes \widecheck{\iota}_0+\widecheck{\iota}_1\otimes \rho_3,&\qquad \delta^1(\widecheck{
\rho}_1) &= \rho_1\otimes \widecheck{\iota}_0+\widecheck{\iota}_1\otimes \rho_1,&\qquad \delta^1(\widecheck{\rho}_2)&=\rho_2\otimes \widecheck{\iota}_1+\widecheck{\iota}_0\otimes \rho_2.\end{alignat*}
In general, one has
\[\widehat{\mathit{CFD}}(-Y,\mathcal{F})\cong \widehat{\mathit{CFA}}(-Y,\mathcal{F})\boxtimes \widehat{\mathit{BSDD}}(\mathcal{TW}_{\mathcal{F}}^+).\]\par
To determine the Type D contact invariant, one tensors the Type A invariant \(c_A(\xi,\mathcal{F})\) with \(\widecheck{\iota}_1\) or \(\widecheck{\iota}_0\), depending on whether \(c_A(\xi,\mathcal{F})\) is acted on by the idempotent \(\iota_0\) or \(\iota_1\), respectively [sic]. Note attaching a twisting slice changes an \(\alpha\)-type diagram to \(\beta\)-type diagram, and vice-versa. Performing this construction in our case and simplifying the resulting Type D modules by canceling pure differentials gives the following proposition.

\begin{proposition}
The Type D contact invariant for \(Y(n,m)\) is given by one of the generators highlighted in Figure {\normalfont\ref{fig: type D invariant}}, depending on the choice of co-orientation for the contact structure on \(Y(n,m)\) and the choice of framing. 
\begin{figure}[ht]
\centering
\begin{tikzpicture}
\def\n{3}
\def\m{4}
\def\gap{0.8}
\def\biggap{1}
\def\border{0.35}
\def\epsilon{0.15}
\draw[thick] (-\border,\border) rectangle ({(\n+6)*\gap }, {-(\m+6)*\gap});
\node at ({(\n+6)*\gap-1.3},0) {\normalfont{Type D, I/III}};
\draw[thick] ({\biggap+(\n+6)*\gap-\border},\border) rectangle ({\biggap+(2*\n+10)*\gap }, { -(\m+4)*\gap});
\node at ({\biggap+(2*\n+10)*\gap-1.3},0) {\normalfont{Type D, II/IV}};
\draw[thick, ->] ({-\border},{\border}) -- ({-\epsilon/sqrt(2)},{\epsilon/sqrt(2)});
\draw[thick] ({\biggap+(\n+6)*\gap-\border},{\border}) -- ({\biggap+(\n+6)*\gap-\epsilon/sqrt(2)},{\epsilon/sqrt(2)});
\filldraw[thick, fill = white] (0,0) circle (2pt);
\filldraw[black] ({\biggap+(\n+6)*\gap},0) circle (2pt);
\draw[thick, ->] ({\epsilon/sqrt(2)},{-\epsilon/sqrt(2)}) -- ({\gap-\epsilon/sqrt(2)}, {\epsilon/sqrt(2)-\gap})node[fill=white,inner sep=2pt,midway] {\scriptsize \(2\)};
\draw[thick, ->] ({\biggap+(\n+6)*\gap+\epsilon/sqrt(2)},{-\epsilon/sqrt(2)}) -- ({\biggap+(\n+7)*\gap-\epsilon/sqrt(2)}, {\epsilon/sqrt(2)-\gap})node[fill=white,inner sep=2pt,midway] {\scriptsize \(3\)};
\foreach\x in {2,...,\n}
    {
    \filldraw[thick, fill=white] ({(\x-1)*\gap},0) circle (2pt);
    \filldraw[black] ({\biggap+(\n+\x+5)*\gap},0) circle (2pt);
    \draw[thick, ->] ({(\x-1)*\gap +\epsilon/sqrt(2)},{-\epsilon/sqrt(2)}) -- ({\x*\gap-\epsilon/sqrt(2)}, {\epsilon/sqrt(2)-\gap})node[fill=white,inner sep=2pt,midway] {\scriptsize \(2\)};
    \draw[thick, ->] ({\biggap+(\n+\x+5)*\gap +\epsilon/sqrt(2)},{-\epsilon/sqrt(2)}) -- ({\biggap+(\n+6+\x)*\gap-\epsilon/sqrt(2)}, {\epsilon/sqrt(2)-\gap})node[fill=white,inner sep=2pt,midway] {\scriptsize \(3\)};
    \draw[thick] ({(\x-1)*\gap}, {-\m*\gap-\epsilon}) -- ({(\x-1)*\gap},{-(\m+6)*\gap})node[fill=white,inner sep=2pt,midway] {\scriptsize \(123\)};
    \draw[thick, ->] ({\biggap+(\n+\x+5)*\gap},{-(\m+4)*\gap}) -- ({\biggap +(\n+\x+5)*\gap}, {-\m*\gap-\epsilon}) node[fill=white,inner sep=2pt,midway] {\scriptsize \(1\)};
    \draw[thick, ->] ({(\x-1)*\gap}, {\border})--({(\x-1)*\gap}, {\epsilon});
    \draw[thick, -] ({\biggap+(\n+\x+5)*\gap}, {\border})--({\biggap+(\n+\x+5)*\gap}, {\epsilon});
    }
\foreach\y in {2,...,\m} 
    {
    \filldraw[thick, fill=white] (0,{(1-\y)*\gap}) circle (2pt);
    \filldraw[black] ({\biggap+(\n+6)*\gap},{(1-\y)*\gap}) circle (2pt);
    \draw[thick, ->] ({\epsilon/sqrt(2)},{(1-\y)*\gap -\epsilon/sqrt(2)}) -- ({\gap-\epsilon/sqrt(2)}, {\epsilon/sqrt(2)-\y*\gap})node[fill=white,inner sep=2pt,midway] {\scriptsize \(2\)};
    \draw[thick, ->] ({\biggap + (\n+6)*\gap + \epsilon/sqrt(2)},{(1-\y)*\gap -\epsilon/sqrt(2)}) -- ({\biggap + (\n+7)*\gap-\epsilon/sqrt(2)}, {\epsilon/sqrt(2)-\y*\gap})node[fill=white,inner sep=2pt,midway] {\scriptsize \(3\)};
    \draw[thick] ({\n*\gap+\epsilon},{(1-\y)*\gap})--({(\n+6)*\gap},{(1-\y)*\gap}) node[fill=white,inner sep=2pt,midway] {\scriptsize \(123\)};
    \draw[thick, ->] ({\biggap+(2*\n+10)*\gap },{(1-\y)*\gap})--({\biggap+(2*\n+6)*\gap+\epsilon},{(1-\y)*\gap}) node[fill=white,inner sep=2pt,midway] {\scriptsize \(1\)};
    \draw[thick, ->]({-\border},{(1-\y)*\gap})--({-\epsilon},{(1-\y)*\gap});
    \draw[thick]({\biggap+(\n+6)*\gap-\border},{(1-\y)*\gap})--({\biggap+(\n+6)*\gap-\epsilon},{(1-\y)*\gap});
    }
\filldraw[thick, fill = white] (0, {(\n-\m)*\gap}) circle (2pt);
\filldraw[black] ({\biggap+(\n+6)*\gap}, {(\n-\m)*\gap}) circle (2pt);
\foreach \x in {1,...,\n}
    \foreach \y in {1,...,\m}{
        \filldraw[black] ({\x*\gap},{-\y*\gap}) circle (2pt);
        \filldraw[thick, fill = white] ({\biggap+(\n+6+\x)*\gap},{-\y*\gap}) circle (2pt);
        }
\foreach \x in {2,...,\n}
    \foreach \y in {2,...,\m}{
        \draw [thick, ->] ({(\x-1)*\gap+\epsilon/sqrt(2)},{(1-\y)*\gap-\epsilon/sqrt(2)}) --({\x*\gap-\epsilon/sqrt(2)},{-\y*\gap+\epsilon/sqrt(2)}) node[fill=white,inner sep=2pt,midway] {\scriptsize \(12\)};
        \draw [thick, ->] ({\biggap+(\n+\x+5)*\gap+\epsilon/sqrt(2)},{(1-\y)*\gap-\epsilon/sqrt(2)}) --({\biggap+(\n+6+\x)*\gap-\epsilon/sqrt(2)},{-\y*\gap+\epsilon/sqrt(2)}) node[fill=white,inner sep=2pt,midway] {\scriptsize \(23\)};
        }
\filldraw[ultra thick, color = red, fill=white] ({\biggap+(2*\n+6)*\gap},{-\m*\gap}) circle (2pt);
\filldraw[color = red] ({(\n+1)*\gap},{-(\m+1)*\gap}) circle (2pt);
\filldraw[black] ({\biggap+(2*\n+7)*\gap},{-(\m+1)*\gap}) circle (2pt);
\draw[thick, ->] ({\n*\gap+\epsilon/sqrt(2)},{-\m*\gap-\epsilon/sqrt(2)}) -- ({(\n+1)*\gap-\epsilon/sqrt(2)},{\epsilon/sqrt(2)-(\m+1)*\gap}) node[fill=white,inner sep=2pt,midway] {\scriptsize \(12\)};
\draw[thick, ->] ({\biggap+(2*\n+7)*\gap-\epsilon/sqrt(2)},{\epsilon/sqrt(2)-(\m+1)*\gap}) -- ({\biggap + (2*\n+6)*\gap+\epsilon/sqrt(2)},{-\m*\gap-\epsilon/sqrt(2)}) node[fill=white,inner sep=2pt,midway] {\scriptsize \(1\)};
\filldraw[black] ({(\n+2)*\gap},{-(\m+2)*\gap}) circle (2pt);
\filldraw[thick, fill=white] ({\biggap+(2*\n+8)*\gap},{-(\m+2)*\gap}) circle (2pt);
\draw[thick, ->] ({(\n+2)*\gap-\epsilon/sqrt(2)},{\epsilon/sqrt(2)-(\m+2)*\gap})--({(\n+1)*\gap+\epsilon/sqrt(2)},{-(\m+1)*\gap-\epsilon/sqrt(2)});
\draw[thick, ->] ({\biggap+(2*\n+8)*\gap-\epsilon/sqrt(2)},{\epsilon/sqrt(2)-(\m+2)*\gap}) -- ({\biggap + (2*\n+7)*\gap+\epsilon/sqrt(2)},{-(\m+1)*\gap-\epsilon/sqrt(2)}) node[fill=white,inner sep=2pt,midway] {\scriptsize \(2\)};
\filldraw[thick, fill = white] ({(\n+3)*\gap},{-(\m+3)*\gap}) circle (2pt);
\filldraw[ultra thick, color = blue, fill=white] ({\biggap+(2*\n+9)*\gap},{-(\m+3)*\gap}) circle (2pt);
\draw[thick, ->] ({(\n+2)*\gap+\epsilon/sqrt(2)},{-(\m+2)*\gap-\epsilon/sqrt(2)})--({(\n+3)*\gap-\epsilon/sqrt(2)},{\epsilon/sqrt(2)-(\m+3)*\gap}) node[fill=white,inner sep=2pt,midway] {\scriptsize \(3\)};
\draw[thick, ->] ({\biggap + (2*\n+8)*\gap+\epsilon/sqrt(2)},{-(\m+2)*\gap-\epsilon/sqrt(2)}) -- ({\biggap+(2*\n+9)*\gap-\epsilon/sqrt(2)},{\epsilon/sqrt(2)-(\m+3)*\gap});
\filldraw[color = blue] ({(\n+4)*\gap},{-(\m+4)*\gap}) circle (2pt);
\draw[thick, ->] ({(\n+3)*\gap+\epsilon/sqrt(2)},{-(\m+3)*\gap-\epsilon/sqrt(2)})--({(\n+4)*\gap-\epsilon/sqrt(2)},{\epsilon/sqrt(2)-(\m+4)*\gap}) node[fill=white,inner sep=2pt,midway] {\scriptsize \(2\)};
\draw[thick, ->] ({\biggap+(2*\n+10)*\gap},{-(\m+4)*\gap})--({\biggap + (2*\n+9)*\gap+\epsilon/sqrt(2)},{-(\m+3)*\gap-\epsilon/sqrt(2)}) node[fill=white,inner sep=2pt,midway] {\scriptsize \(1\)};
\filldraw[black] ({(\n+5)*\gap},{-(\m+5)*\gap}) circle (2pt);
\draw[thick, ->] ({(\n+5)*\gap-\epsilon/sqrt(2)},{\epsilon/sqrt(2)-(\m+5)*\gap})--({(\n+4)*\gap+\epsilon/sqrt(2)},{-(\m+4)*\gap-\epsilon/sqrt(2)});
\draw[thick] ({(\n+5)*\gap+\epsilon/sqrt(2)},{-(\m+5)*\gap-\epsilon/sqrt(2)})--({(\n+6)*\gap},{-(\m+6)*\gap}) node[fill=white,inner sep=2pt,midway] {\scriptsize \(3\)};
\end{tikzpicture}
\caption{}
\label{fig: type D invariant}
\end{figure}
Blue generators correspond to consistent framing choices for Type D diagrams which are \(\beta\)-framed. {\normalfont(}By ``Type D diagram'' we mean the diagram after attaching the basic slice.{\normalfont)}
\end{proposition}

The Type D modules in Figure \ref{fig: type D invariant} look slightly different than in Figure \ref{fig: initial type D}. This is because some pure differentials have been left unreduced, so as to preserve the Type D contact invariants. 

Before gluing, we must reparameterize the boundary of one of the bordered 3-manifolds to realize negative Dehn twists about the curve \(b\) in Figure \(\ref{fig: open book}\). Henceforth, we make the arbitrary choice to assign \(-Y(n_1,n_2)\) a Type A diagram with II/IV boundary. This forces us to assign \(-Y(n_3,n_4)\) a Type D diagram with I/III boundary, of the same \(\beta\) or \(\alpha\) type as the diagram for \(-Y(n_1,n_2)\). To realize a single negative Dehn twist about \(b\), we pair with one of the bimodule diagrams in Figure \ref{fig: dehn twists} (depending on whether we choose \(\beta\)-type or \(\alpha\)-Type diagrams).

\begin{figure}[ht]
\centering
\includegraphics[scale = 0.7]{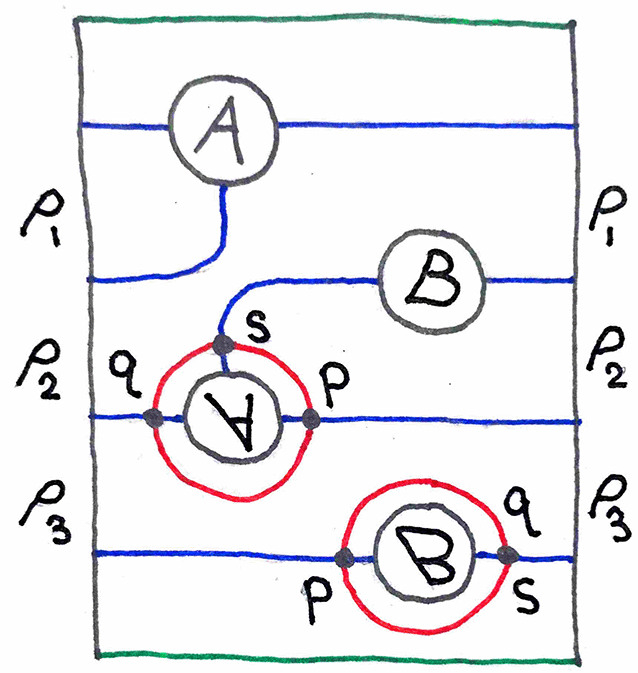}\qquad\qquad\includegraphics[scale = 0.7]{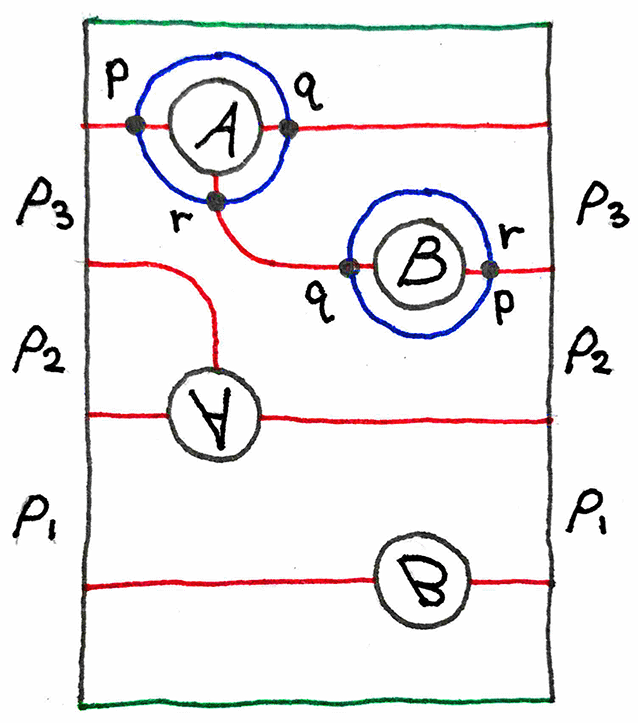}
\caption{}
\label{fig: dehn twists}
\end{figure}

(Unfortunately, it is a headache here to keep track of orientations. Recall that all our Heegaard diagrams involve an orientation-reversal. It is maybe prudent to revert back to the original orientations before figuring out which bimodule diagram to use. This is done by recoloring \(\alpha\) arcs to \(\beta\) arcs, and vice-versa.  It is also important to keep in mind that for \(\alpha\)-type diagrams, the boundary of the 3-manifold is being viewed from the inside. For \(\beta\)-type diagrams, the boundary is viewed from the outside.)

In both diagrams we interpret the left-hand boundary as acting in a Type D fashion, and the right-hand boundary as acting in a Type A fashion. In this way, the above diagrams are assigned isomorphic modules, which we will call \(\widehat{\mathit{CFDA}}(\tau)\). We have 
\[\widehat{\mathit {CF}}\big({-Y(f)}\big)\cong \widehat{\mathit{CFA}}\big({-Y(n_1,n_2),\mathcal{F}_{\mathrm{II/IV}}}\big)\boxtimes {\underbrace{\widehat{\mathit{CFDA}}(\tau)\boxtimes \cdots \boxtimes \widehat{\mathit{CFDA}}(\tau)}_{\lvert n_b\rvert \text{ times}}}\boxtimes{ \widehat{\mathit{CFD}}\big({-Y(n_3,n_4),\mathcal{F}_{\mathrm{I/III}}}\big)}. \]
The bimodule \(\widehat{\mathit{CFDA}}(\tau)\) has three generators, \(\mathbf{p},\mathbf{q},\mathbf{s}\), with idempotent actions
\[\iota_0 \cdot \mathbf{p}\cdot \iota_0=\mathbf{p},\qquad \iota_1\cdot \mathbf{q}\cdot \iota_1 = \mathbf{q},\qquad \iota_0\cdot \mathbf{s}\cdot \iota_1 = \mathbf{s}. \]
In \cite[Section 10.2]{bimodules}, the non-zero differentials are computed to be 
\begin{alignat*}{2}
m_{0,1,2}(\mathbf{q},\rho_2,\rho_1)&=\rho_2\otimes \mathbf{s},&\qquad m_{0,1,2}(\mathbf{q},\rho_2,\rho_{12})&=\rho_{2}\otimes\mathbf{p},\\  m_{0,1,2}(\mathbf{q},\rho_2,\rho_{123})&=\rho_{23}\otimes \mathbf{q},&\qquad
m_{0,1,1}(\mathbf{p},\rho_1) &= \rho_{12}\otimes \mathbf{s},\\ 
m_{0,1,1}(\mathbf{p},\rho_{12})&=\rho_{12}\otimes \mathbf{p},&\qquad m_{0,1,1}(\mathbf{p},\rho_{123})&=\rho_{123}\otimes \mathbf{p},\\ 
m_{0,1,1}(\mathbf{p},\rho_{3})&=\rho_{3}\otimes\mathbf{q},&\qquad m_{0,1,0}(\mathbf{s}) &= \rho_1\otimes \mathbf{q},
\\ m_{0,1,1}(\mathbf{s},\rho_2) &= \mathbf{p},&\qquad m_{0,1,1}(\mathbf{s},\rho_{23})&=\rho_3\otimes \mathbf{q}. 
\end{alignat*}

\begin{proposition}\label{prop:reparameterize}
Let \(Y\) be a compact 3-manifold with torus boundary and framing \(\mathcal{F}\). Let \(\Gamma\) be the suture on \(\partial Y\) corresponding to the idempotent \(\iota_0\). Consider the isomorphisms 
\[\widehat{\mathit{HFA}}(Y,\mathcal{F})\cdot \iota_0 \cong H_*\big(\widehat{\mathit{CFA}}(Y,\mathcal{F})\boxtimes \widehat{\mathit{CFDA}}(\tau)\big)\cdot \iota_0\cong \mathit{SFH}(Y,\Gamma)\]
induced by attaching the sutured cap corresponding to \(\iota_0\). Define the isomorphism 
\[\Phi\colon \mathit{SFH}(Y,\Gamma)\to \mathit{SFH}(Y,\Gamma)\]
by 
\[\Phi([\mathbf{x}]\cdot \iota_0) = [\mathbf{x}\otimes \mathbf{p}]\cdot \iota_0,\qquad \mathbf{x}\in \widehat{\mathit{CFA}}(Y,\mathcal{F}),\, m_1(\mathbf{x}) =0.\]
Then \(\Phi\) preserves \(\mathrm{spin}^c\) structures. 
\end{proposition}
\begin{proof}
We show this by manipulating Heegaard diagrams. We depict the case when the diagrams for \((Y,\mathcal{F})\) and \(\tau\) are of \(\beta\)-type. When attaching the sutured cap to a diagram of \((Y,\mathcal{F})\), the resulting diagram can be destabilized once after possibly handlesliding as shown in Figure \ref{fig:cap_simplify}. If we attach both \(\tau\) and the sutured cap, then the resulting diagram can be destabilized three times after appropriate handleslides. This is shown in Figure \ref{fig:twist_cap}.
\begin{figure}[ht]
\centering
\includegraphics[scale = 1]{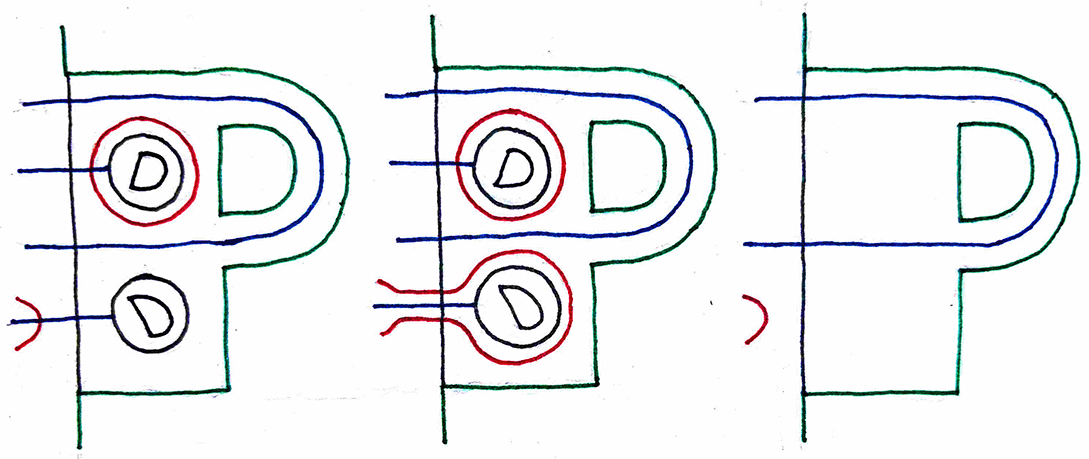}
\caption{}
\label{fig:cap_simplify}
\end{figure}

\begin{figure}[ht]
\centering
\includegraphics[scale = 0.8]{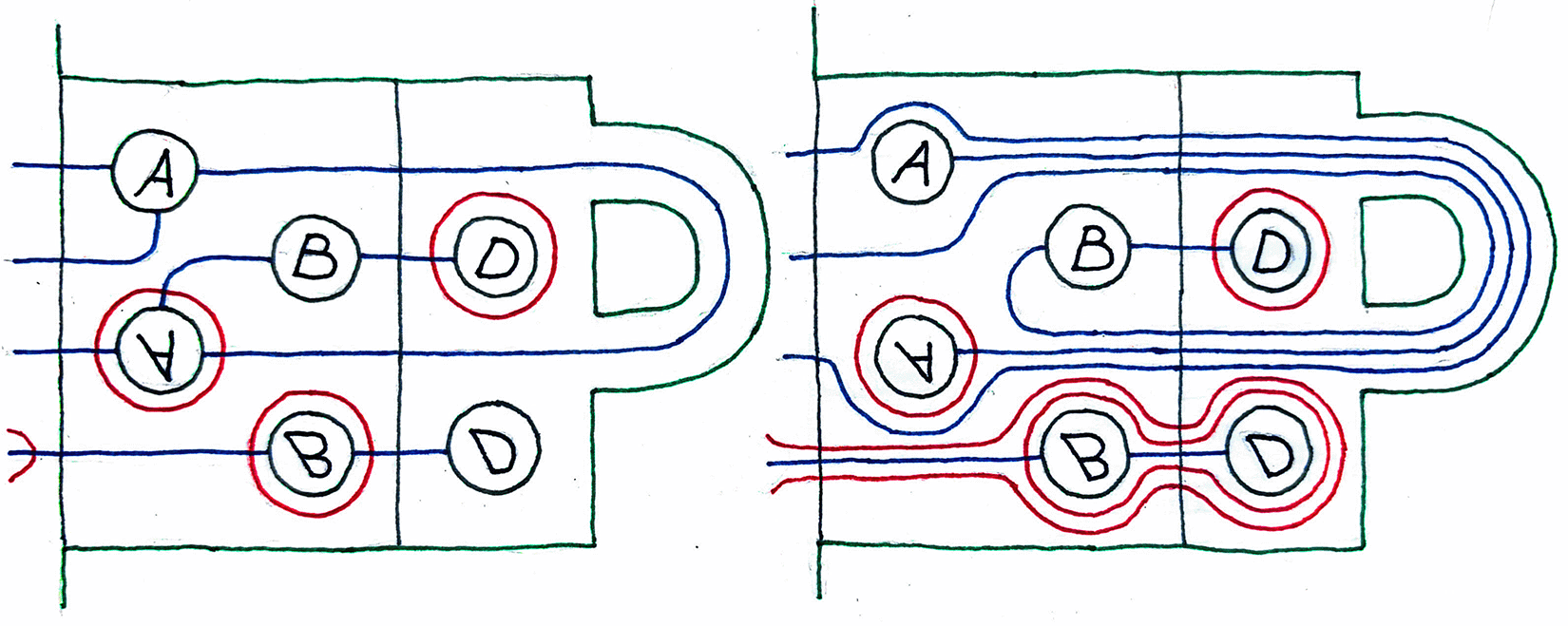}
\caption{The diagram on the right can be destabilized three times to get to the right-most diagram in Figure \ref{fig:cap_simplify}}.
\label{fig:twist_cap}
\end{figure}

In both cases, one obtains the same Heegaard diagram for \((Y,\Gamma)\). Each of the handleslides and de-stabilizations changes the pseudo-gradient vector field on \(Y\) only in 3-balls away from \(\partial Y\). It follows that the isomorphism \(\Phi\) defined above preserves \(\mathrm{spin}^c\) structures.
\end{proof}

Recall \(\Gamma_{\lambda}\) denotes a pair of sutures on \(\partial Y(n_1,n_2)\) which are parallel to the page. The \(n_1+n_2\) generators of
\(\mathit{SFH}\big({-Y(n_1,n_2)},-\Gamma_{\lambda}\big)\) all have distinct \(\mathrm{spin}^c\) structures. This can be seen from Figure \ref{fig: bordered}, noting that \(\mathrm{Spin}^c\big({-Y(n_1,n_2),-\Gamma_{\lambda}}\big)\) is a torsor for the group
\[H_1\big(Y(n_1,n_2)\big)\cong \langle \mu_{1},\mu_2\,|\, n_1\mu_1=n_2\mu_2\rangle\cong \mathbb{Z}/(n_1,n_2)\oplus \mathbb{Z}.\]
In the \(\beta\)-type diagram of Figure \ref{fig: bordered}, \(\mu_1,\mu_2\) are represented by curves intersecting one of the \(\beta\)-circles in once and not intersecting the other \(\beta\)-circle or \(\beta\)-arcs.  For \(i\geq 0\), let \(\mathcal{F}_{1,i}\) denote the framing on \(\partial \big({-Y(n_1,n_2)}\big)\) obtained by attaching \(\tau\) to \(\big({-Y(n_1,n_2),\mathcal{F}_{\mathrm{II/IV}}}\big)\) a total of \(i\) times. By repeatedly applying Proposition \ref{prop:reparameterize}, we have
\[c_A(\xi,\mathcal{F}_{1,i})=c_A(\xi,\mathcal{F}_{\mathrm{II/IV}})\otimes \underbrace{\mathbf{p}\otimes \cdots \otimes \mathbf{p}}_{i\text{ times}} .\]\par
 We are now finally ready to prove the desired vanishing result. 
\begin{proof}[Proof of Theorem \ref{thm: vanishing}]
Let \(\mathbf{x}\in\widehat{\mathit{CFA}}\big({-Y(n_1,n_2}),\mathcal{F}_{ \mathrm{II/IV}}\big)\) and \(\mathbf{y}\in \widehat{\mathit{CFD}}({-Y(n_1,n_2}),\mathcal{F}_{ \mathrm{I/III}}\big)\) denote the blue contact invariants in Figures \ref{fig: type A invariant} and \ref{fig: type D invariant}. From the preceding discussion, we have 
\[ c(f) = \mathbf{x}\otimes {\underbrace{\mathbf{p}\otimes \cdots \otimes \mathbf{p}}_{\lvert n_b\rvert\text{ times}}} \otimes \mathbf{y}.\]
For \(i\geq 0\) let \(\mathcal{F}_{2,i}\) be the framing on \(\partial \big({-Y(n_3,n_4)}\big)\) obtained by attaching \(\tau\) to \(\big({-Y(n_1,n_2)},\mathcal{F}_{\mathrm{I/III}}\big)\) a total of \(i\) times, and set 
\[\mathbf{y}_i ={\underbrace{\mathbf{p}\otimes \cdots \otimes \mathbf{p}}_{\lvert i\rvert\text{ times}} }\otimes \mathbf{y}\in \widehat{\mathit{CFD}}\big({-Y(n_3,n_4),\mathcal{F}_{2,i}}\big). \]
In Figure \ref{fig: changes}, we depict \(\mathbf{y}_i\) in blue for increasing values of \(i\). We only show a portion of the type \(D\) diagram for \(\widehat{\mathit{CFD}}\big({-Y(n_3,n_4),\mathcal{F}_{2,i}}\big)\). Except for the generators at the ends, we depict all incoming and outgoing arrows. 
\begin{figure}[ht]
\centering
\begin{tikzpicture}
\def\gap{1.2}
\def\vertgap{0.75}
\def\epsilon{0.15}
\def\n{3}
\filldraw[black] (0,0) circle (2pt);
\draw[thick, ->] ({\epsilon},0) -- ({\gap-\epsilon},0) node[fill=white,inner sep=2pt,midway] {\scriptsize \(3\)};
\filldraw[thick, fill = white] (\gap,0) circle (2pt);
\draw[thick, ->] ({\gap+\epsilon},0) -- ({2*\gap-\epsilon},0) node[fill=white,inner sep=2pt,midway] {\scriptsize \(2\)};
\filldraw[blue] ({2*\gap},0) circle (2pt) node[anchor=south, black]{\(\mathbf{y}_0\)};
\draw[thick, <-] ({2*\gap+\epsilon},0) -- ({3*\gap-\epsilon},0);
\filldraw[Purple] ({3*\gap},0) circle (2pt)node[anchor=south, black]{\(\mathbf{w}_0\)};
\filldraw[black] (0,{-\vertgap}) circle (2pt);
\draw[thick, ->] ({\epsilon},{-\vertgap}) -- ({\gap-\epsilon},-\vertgap) node[fill=white,inner sep=2pt,midway] {\scriptsize \(3\)};
\filldraw[thick, fill = white] (\gap,{-\vertgap}) circle (2pt);
\draw[thick, <-] ({\gap+\epsilon},{-\vertgap}) -- ({2*\gap-\epsilon},-\vertgap) node[fill=white,inner sep=2pt,midway] {\scriptsize \(1\)};
\filldraw[Green] ({2*\gap},{-\vertgap}) circle (2pt) node[anchor=south, black]{\(\mathbf{z}_1\)};
\draw[thick, ->] ({2*\gap+\epsilon},{-\vertgap}) -- ({3*\gap-\epsilon},-\vertgap);
\filldraw[blue] ({3*\gap},{-\vertgap}) circle (2pt) node[anchor=south, black]{\(\mathbf{y}_1\)};
\draw[thick, <-] ({3*\gap+\epsilon},{-\vertgap}) -- ({4*\gap-\epsilon},-\vertgap);
\filldraw[Purple] ({4*\gap},{-\vertgap}) circle (2pt) node[anchor=south, black]{\(\mathbf{w}_1\)};
\foreach \x in {1,...,\n}
{
\filldraw[black] (0,{-(\x+1)*\vertgap}) circle (2pt);
\draw[thick, ->] ({\epsilon},{-(\x+1)*\vertgap}) -- ({\gap-\epsilon},{-(\x+1)*\vertgap}) node[fill=white,inner sep=2pt,midway] {\scriptsize \(3\)};
\filldraw[thick, fill = white] (\gap,{-(\x+1)*\vertgap}) circle (2pt);
\draw[thick, <-] ({\gap+\epsilon},{-(\x+1)*\vertgap}) -- ({2*\gap-\epsilon},{-(\x+1)*\vertgap}) node[fill=white,inner sep=2pt,midway] {\scriptsize \(1\)};
\foreach \y in {1,...,\x}{
\filldraw[black] ({(\y+1)*\gap},{-(\x+1)*\vertgap}) circle (2pt);
\draw[thick, <-] ({(\y+1)*\gap+\epsilon},{-(\x+1)*\vertgap}) -- ({(\y+2)*\gap-\epsilon},{-(\x+1)*\vertgap}) node[fill=white,inner sep=2pt,midway] {\scriptsize \(12\)};
}
\pgfmathsetmacro\z{\x+1}
\filldraw[Green] ({(\x+2)*\gap},{-(\x+1)*\vertgap}) circle (2pt) node[anchor=south, black]{\(\mathbf{z}_{\pgfmathprintnumber{\z}}\)};
\draw[thick, ->] ({(\x+2)*\gap+\epsilon},{-(\x+1)*\vertgap}) -- ({(\x+3)*\gap-\epsilon},{-(\x+1)*\vertgap});
\filldraw[blue] ({(\x+3)*\gap},{-(\x+1)*\vertgap}) circle (2pt) node[anchor=south, black]{\(\mathbf{y}_{\pgfmathprintnumber{\z}}\)};
\draw[thick, <-] ({(\x+3)*\gap+\epsilon},{-(\x+1)*\vertgap}) -- ({(\x+4)*\gap-\epsilon},{-(\x+1)*\vertgap});
\filldraw[Purple] ({(\x+4)*\gap},{-(\x+1)*\vertgap}) circle (2pt) node[anchor=south, black]{\(\mathbf{w}_{\pgfmathprintnumber{\z}}\)};
}
\end{tikzpicture}
\caption{}
\label{fig: changes}
\end{figure}

For \(i\geq 1\) let \(\mathbf{z}_i\in \widehat{\mathit{CFD}}\big({-Y(n_3,n_4}),\mathcal{F}_{2,i}\big)\) be the green generator which is pictured in Figure \ref{fig: changes}. We claim \(\mathbf{x}\otimes \mathbf{y}_{i}\) is the boundary of \(\mathbf{x}\otimes \mathbf{z}_{i}\) whenever \(i>n_1\), which gives the desired vanishing result. For each \(i\geq 1\), the boundary of \(\mathbf{x}\otimes \mathbf{z}_{i}\) contains \(\mathbf{x}\otimes \mathbf{y}_{i}\) as a term. However if \(i\leq n_1\), then
\[m_{i+1}(\mathbf{x},\underbrace{\rho_{12},\ldots\rho_{12}}_{i-1\text{ times}},\rho_1)=\mathbf{x}_i'\]
for some generator \(\mathbf{x}_i'\), and 
\[\delta^i(\mathbf{z}_i)= {\underbrace{\rho_{21}\otimes \cdots \otimes \rho_{21}}_{i-1\text{ times}}} \otimes \rho_1\otimes  \mathbf{z}_i'\]
for some generator \(\mathbf{z}_i'\). In these cases, it is not too hard to see that
\[\partial^{\boxtimes} (\mathbf{x}\otimes \mathbf{z}_i) = \mathbf{x}\otimes \mathbf{y}_i + \mathbf{x}_i'\otimes \mathbf{z}_i',\qquad 0\leq i \leq n_1.\]
When \(i>n_1\), this additional term in \(\partial^\boxtimes (\mathbf{x}\otimes \mathbf{z}_i)\) vanishes.
\end{proof}
\begin{remark}
At this point, the reader may question why we did not make use of the purple generator in Figure \ref{fig: changes} to kill the contact invariant. Let \(\mathbf{w}_i\in \widehat{\mathit{CFD}}\big({-Y(n_3,n_4}),\mathcal{F}_{2,i}\big)\) denote this generator, and let us assume \(n_3\leq n_4\). It turns out for \(0\leq i<n_3\), one has 
\[\delta^2(\mathbf{w}_i) = \rho_3 \otimes\rho_2\otimes  \mathbf{w}_i'  \]
for some generator \(\mathbf{w}_i'\). In addition
\[m_3(\mathbf{x},\rho_3,\rho_2) = \mathbf{x}''\]
for some generator \(\mathbf{x}''\). When \(i=n_3\), one instead has 
\[\delta^2(\mathbf{w}_{n_3}) = \rho_3\otimes \rho_{23}\otimes \mathbf{w}_{n_3}'\]
for some generator \(\mathbf{w}_{n_3}'\). On the other hand
\[m_3(\mathbf{x},\rho_3,\rho_{23}) = \mathbf{x}'''\]
for some generator \(\mathbf{x}'''\). One can indeed show 
\[\partial^\boxtimes(\mathbf{x}\otimes \mathbf{w}_i)=\mathbf{x}\otimes \mathbf{y}_i+\mathbf{x}''\otimes \mathbf{w}_i'\quad \text{for }i<n_3,\qquad \partial(\mathbf{x}\otimes \mathbf{w}_{n_3}) =\mathbf{x}\otimes \mathbf{y}_i+\mathbf{x}'''\otimes \mathbf{w}_{n_3}'.\]
However, when \(i>n_3\), this additional term in \(\partial^\boxtimes (\mathbf{x}\otimes \mathbf{w}_i)\) vanishes. Thus, this gives another proof of the above proposition by assuming \(n_3 = \min\{n_1,n_2,n_3,n_4\}\). One can make several arguments along these lines. 
\end{remark}
\appendix 
\section{Modules over the torus algebra}\label{sec: algebra}
In this appendix we give a brief review of various modules over the torus algebra. For a more detailed account, see \cite[Chapter 2]{bordered}, \cite{bimodules}, and \cite[Appendix B]{Zarev}. We first define the torus algebra \(\mathcal{A}\). This is an 8-dimensional (unital, associative, non-commutative) algebra over the field of two elements \(\mathbb{F}_2\), spanned as a vector space by the generators
\[\iota_0,\iota_1,\rho_1,\rho_2,\rho_3,\rho_{12},\rho_{23},\rho_{123}.\]\par
The generators \(\iota_0,\iota_1\) are orthogonal, minimal idempotents. The action of the idempotents on the other generators is as follows:
\begin{alignat*}{3}
\rho_1 &= \iota_0\rho_1\iota_1,&\qquad \rho_2 &= \iota_1\rho_2\iota_0,&\qquad \rho_3&=\iota_0\rho_3\iota_1,\\ 
\rho_{12}&= \iota_0\rho_{12}\iota_0, &\qquad \rho_{23}&= \iota_1\rho_{23}\iota_1,&\qquad \rho_{123}&=\iota_0\rho_{123} \iota_1.
\end{alignat*}
Note \(\mathbf{1}=\iota_0+\iota_1\) is an identity. The non-zero products between non-idempotent generators are 
\[\rho_1\rho_2=\rho_{12},\qquad \rho_2\rho_3=\rho_{23},\qquad \rho_1\rho_{23}=\rho_{12}\rho_3 = \rho_{123}.\]\par
Let \(\mathcal{I}\leq \mathcal{A}\) denote the subalgebra spanned by \(\iota_0\) and \(\iota_1\). We note \(\mathcal{I}\) is a commutative, unital algebra; all \(\mathcal{I}\)-modules will be unital and we do not distinguish between left and right \(\mathcal{I}\)-modules. 
\subsection*{Type A modules}
A right Type A module over \(\mathcal{A}\) is an \(\mathcal{I}\)-module \(M\) equipped with ``multiplications''
\[m_{k}\colon M\otimes_\mathcal{I} {\underbrace{\mathcal{A}\otimes_{\mathcal{I}}\cdots \otimes_{\mathcal{I}}\mathcal{A}}_{k-1\text{ times}}}\to M,\qquad k\geq 1\]
which satisfy the \({A}_\infty\) relations (which we do not spell out here). In addition \(M\) is said to be \textit{strictly unital} if 
\[m_2(\mathbf{x},\mathbf{1})=\mathbf{x},\qquad m_k(\mathbf{x},\ldots,\mathbf{1},\ldots)=0\qquad\text{for all }\mathbf{x}\in M,\, k\geq 3.\]\par
One of the \({A}_\infty\) relations ensures that \(m_1\colon M\to M\) is a differential, i.e., \(m_1^2=0\). Indeed, if \(m_k=0\) for \(k\geq 3\) then one recovers the usual notion of a right differential module over \(\mathcal{A}\) with \(m_1\) the differential and \(m_2\colon M\otimes_{\mathcal{I}} \mathcal{A}\to M\) the right action of \(\mathcal{A}\). 

A Type A module is said to be \textit{bounded} if for all \(\mathbf{x}\in M\) there exists some \(n\) so that for all 
\[k_1+\cdots+k_i>n,\qquad a_{1,1},\ldots, a_{i,k_i}\in \mathcal{A}\]
one has 
\[m_{k_i+1}\circ (m_{k_{i-1}+1}\otimes \mathbb{I}_{\mathcal{A}^{\otimes k_i}})\circ \cdots \circ (m_{k_1+1}\otimes \mathbb{I}_{\mathcal{A}^{\otimes k_2+\cdots + k_i}})(\mathbf{x}\otimes a_{1,1}\otimes \cdots \otimes a_{i,k_i}) =0. \]
This boundedness condition will be needed when discussing box tensor products. 

\subsection*{Type D modules} A left Type D module over \(\mathcal{A}\) is an \(\mathcal{I}\)-module \(N\) equipped with a ``differential''
\[\delta^1\colon N\to \mathcal{A}\otimes_{\mathcal{I}} N.\]
In the case of the torus algebra, \(\delta^1\) should satisfy the condition that
\[(\mu_2\otimes \mathbb{I}_N)\circ (\mathbb{I}_{\mathcal{A}}\otimes \delta^1)\circ \delta^1\colon N\to \mathcal{A}\otimes_{\mathcal{I}} N\]
vanishes, where \(\mu_2\colon \mathcal{A}\otimes \mathcal{A}\to \mathcal{A}\) is the algebra multiplication. (In the general case one speaks of Type D modules over a differential algebra, in which case the above equation includes an additional term involving the algebra differential.)\par 
Given such a type D module \(N\), one can make \(\mathcal{A}\otimes_{\mathcal{I}} N\) into a left differential module over \(\mathcal{A}\) by setting the differential to be
\[(\mu_2\otimes \mathbb{I}_{N})\circ (\mathbb{I}_A\otimes \delta^1)\colon \mathcal{A}\otimes_{\mathcal{I}}N\to \mathcal{A}\otimes_{\mathcal{I}} N.\]
(In general there is an additional term involving the algebra differential.) We often confuse a Type D module with its associated differential module. Given a Type D module as above, we inductively define maps 
\[\delta^k\colon N\to {\underbrace{\mathcal{A}\otimes_{\mathcal{I}}\cdots \otimes_{\mathcal{I}} \mathcal{A}}_{k\text{ times}}}\otimes_{\mathcal{I}} N,\qquad k\geq 0\]
by setting \(\delta^0 = \mathbb{I}_N\) and 
\[\delta^k = (\mathbb{I}_{\mathcal{A}^{\otimes k-1}} \otimes \delta^1)\circ \delta^{k-1},\qquad k\geq 2.\]\par
We say that the Type D module is \textit{bounded} if for all \(\mathbf{x}\in N\) there exists some \(n\) so that \(\delta^k(\mathbf{x})=0\) for all \(k> n\). 
\medskip

Before moving on to bimodules, we remark that the notions of boundedness for bimodules are more subtle. In particular, bimodules can be bounded on the left, bounded on the right, or bounded (which is stronger than being both left and right bounded). We avoid this discussion and direct the reader toward \cite[Section 2.2.4]{bimodules}. 

\subsection*{Type DD bimodules}
A Type DD bimodule over \(\mathcal{A}\) is an \(\mathcal{I}\)-module \(N\) equipped with a structure map
\[\delta^1\colon N\to \mathcal{A}\otimes_{\mathcal{I}} N \otimes_{\mathcal{I}}\mathcal{A}.\]
In the case of the torus algebra, \(\delta^1\) should satisfy the condition that
\[(\mu_2\otimes \mathbb{I}_N\otimes \mu_2)\circ (\mathbb{I}_{\mathcal{A}}\otimes \delta^1\otimes \mathbb{I}_{\mathcal{A}})\circ \delta^1\colon N\to \mathcal{A}\otimes_{\mathcal{I}} N\otimes_{\mathcal{I}} \mathcal{A}\]
vanishes. (In general there is an additional term involving the algebra differential.) As above, this structure can be used to give \(\mathcal{A}\otimes_{\mathcal{I}} N\otimes_{\mathcal{I}} \mathcal{A}\) the structure of a differential bimodule. We again iterate to define maps 
\[\delta^k\colon N\to {\underbrace{\mathcal{A}\otimes_{\mathcal{I}}\cdots \otimes_{\mathcal{I}} \mathcal{A}}_{k\text{ times}}}\otimes_{\mathcal{I}} N\otimes_{\mathcal{I}}{\underbrace{\mathcal{A}\otimes_{\mathcal{I}}\cdots \otimes_{\mathcal{I}} \mathcal{A}}_{k\text{ times}}},\qquad k\geq 0\]
in a manner similar to above. 
\subsection*{Type DA bimodules} A Type DA bimodule over \(\mathcal{A}\) is an \(\mathcal{I}\)-module \(M\) together with structure maps
\[m_{0,1,k}\colon M\otimes_{\mathcal{I}}{\underbrace{\mathcal{A}\otimes_{\mathcal{I}}\cdots\otimes_{\mathcal{I}}\mathcal{A}}_{k-1\text{ times}}}\to \mathcal{A}\otimes_{\mathcal{I}}\otimes M,\qquad k\geq 1\]
which satisfy appropriate \(\mathcal{A}_\infty\) relations. (In some places \(m_{0,1,k}\) is instead denoted by \(\delta_k^1\).) We say \(M\) is \textit{strictly unital} if
\[m_{0,1,2}(\mathbf{x},\mathbf{1})=\mathbf{1}\otimes\mathbf{x},\qquad m_{0,1,k}(\mathbf{x},\ldots,\mathbf{1},\ldots)=0\qquad\text{for all }\mathbf{x}\in M,\,k\geq 3.\]
For more details, see \cite[Section 2.2.4]{bimodules}.
\subsection*{Pairing a Type A module with a Type D module}
Let \((M,\{m_k\}_{k\geq 1})\) be a strictly unital right Type A module and \((N,\delta^1)\) a left Type D module, with at least one of \(M\) and \(N\) bounded. We define a differential \(\mathbb{F}_2\)-vector space \((M\boxtimes N,\partial^{\boxtimes})\) whose underlying vector space is \(M\otimes_{\mathcal{I}}N\) and whose differential is given by
\[\partial^{\boxtimes}(\mathbf{x}\otimes \mathbf{y}) = \sum_{k\geq 0}(m_{k+1}\otimes \mathbb{I}_N)\circ (\mathbb{I}_M\otimes \delta^k)(\mathbf{x}\otimes \mathbf{y}),\qquad \mathbf{x}\in M,\,\mathbf{y}\in N.\]
The boundedness assumption guarantees that the above sum is finite.
\subsection*{Pairing a Type DA bimodule with a Type D module} Let \((M,\{m_{0,1,k}\}_{k\geq 1})\) be a strictly unital Type DA module and \((N,\delta^1)\) a Type D module, with either \(N\) bounded or \(M\) bounded on the right. We define a Type D module \((M\boxtimes N,\delta^1_{\boxtimes})\) whose underlying \(\mathcal{I}\)-module is \(M\otimes_{\mathcal{I}}N\) and whose stricture map is given by 
\[\delta^1_{\boxtimes}(\mathbf{x}\otimes \mathbf{y}) = \sum_{k\geq 0}(m_{0,1,k+1}\otimes \mathbb{I}_N)\circ (\mathbb{I}_M\otimes \delta^k)(\mathbf{x}\otimes \mathbf{y}),\qquad \mathbf{x}\in M,\,\mathbf{y}\in N.\]
\subsection*{Pairing a Type A module with a Type DD bimodule} For each \(k\geq 0\), define \[\pi_k\colon\underbrace{\mathcal{A}\otimes_{\mathcal{I}} \cdots \otimes_{\mathcal{I}} \mathcal{A}}_{k\text{ times}}\to \mathcal{A}\]
by setting 
\[\pi_k(a_1\otimes \cdots \otimes a_k) = a_1\cdot\ldots\cdot a_k,\qquad a_1,\ldots,a_k\in \mathcal{A}.\]
When \(k=0\), interpret \(\pi_0\colon \mathcal{I}\to\mathcal{A}\) to be inclusion.
Let \((M,\{m_k\}_{k\geq 1})\) be a strictly unital right Type A module and \(N\) a Type DD module, with \(M\) bounded or \(N\) bounded on the left. We define a \textit{right} Type D module \((M\boxtimes N,\delta^1_{\boxtimes})\) whose underlying \(\mathcal{I}\)-module is \(M\otimes_\mathcal{I} N\) and whose structure map is given by 
\[\delta^1_{\boxtimes}(\mathbf{x}\otimes \mathbf{y}) = \sum_{k\geq 0}(\mathbb{I}_M\otimes \mathbb{I}_N\otimes \pi_k)\circ(m_{k+1}\otimes \mathbb{I}_N\otimes \mathbb{I}_{\mathcal{A}^{\otimes k}})\circ (\mathbb{I}_M\otimes \delta^k)(\mathbf{x}\otimes \mathbf{y}),\qquad \mathbf{x}\in M,\,\mathbf{y}\in N.\]
In the case of the torus algebra, we can re-interpret \((M\boxtimes N,\delta^1_{\boxtimes})\) as a left Type D module over \(\mathcal{A}\) using the fact that \(\mathcal{A}^{\mathrm{op}}\cong \mathcal{A}\).

\section{Calculating bordered Floer homology} \label{sec: calculating}
In this appendix, we describe the calculation of the modules \[\widehat{\mathit{CFA}}\big({-Y(n,m),\mathcal{F}}\big)\qquad\text{and}\qquad \widehat{\mathit{CFD}}\big({-Y(n,m),\mathcal{F}}\big)\] where \(\mathcal{F}\) is one of the framings from Section \ref{sec: identifying}. We begin by identifying a lower genus bordered diagram for \(-Y(n,m)\). This diagram arises from the following link surgery diagram in \(S^3\) for \[\big(Y(g_{n,m})_0(K),K^*\big),\] where \(Y(g_{n,m})_0(K)\) denotes 0-surgery on \(K\) and \(K^*\) is the dual knot of the surgery, i.e., the core of the glued-in solid torus. This is shown in Figure \ref{fig: surgery two}. 
\begin{figure}[ht]
\centering
    \includegraphics[scale=0.15]{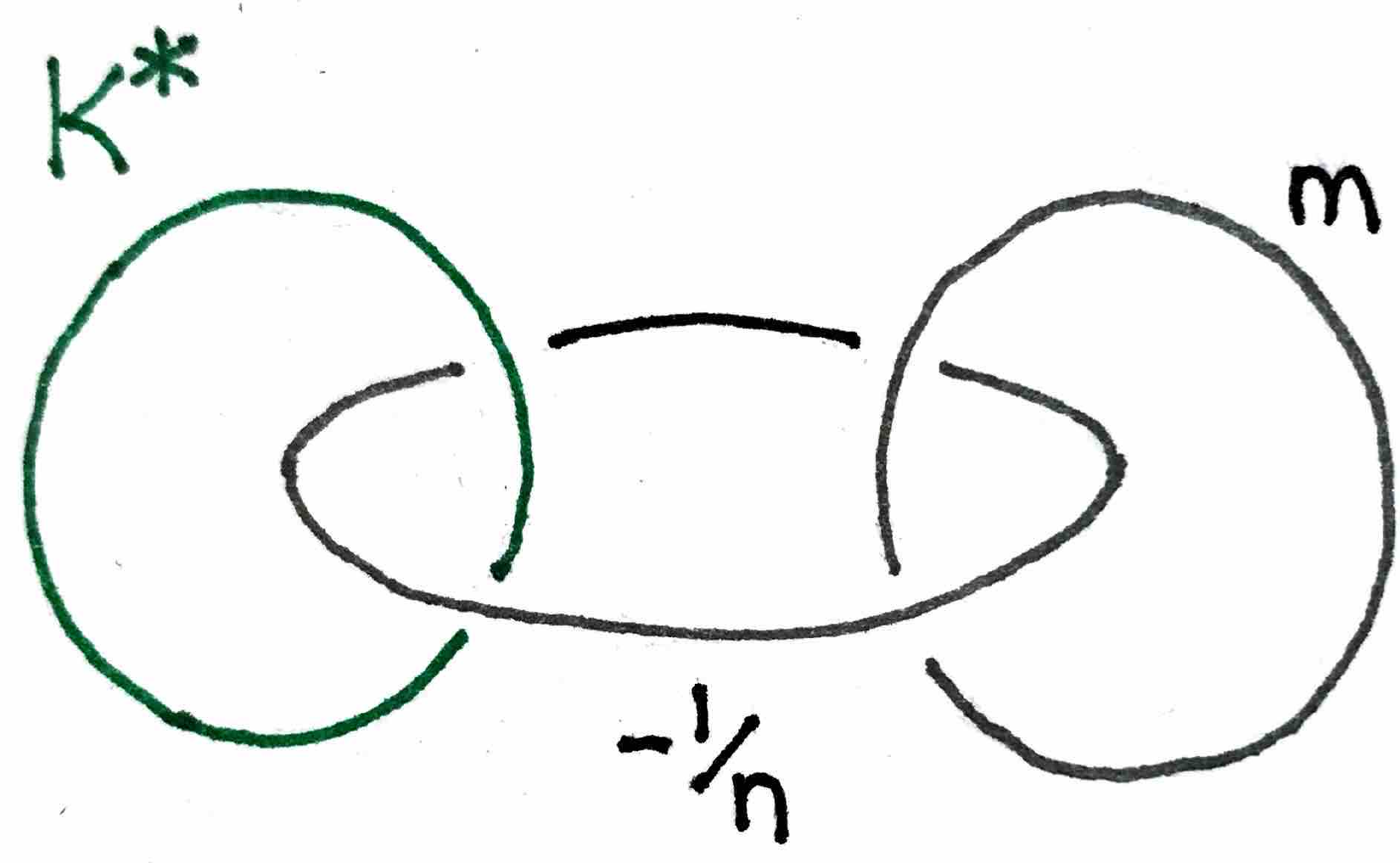}
    \captionof{figure}{}
    \label{fig: surgery two}
\end{figure}

(To get this link surgery diagram, one can perform \(0\)-surgery along \(K\) in Figure \ref{fig: surgery one} and then slam-dunk the \(n\)-framed unknot. Alternatively, one can think about capping off the binding component parallel to \(K\) in \(Y(g_{n.m})\).)

Removing a tubular neighborhood \(N(K^*)\) of \(K^*\) from this picture yields \(Y(n,m)\). The pages of \(Y(n,m)\) meet \(\partial N(K^*)\) in meridians for \(K^*\). This gives the following bordered Heegaard diagram for \(-Y(n,m)\). 
\begin{figure}[ht]
\centering
    \includegraphics[scale=0.9]{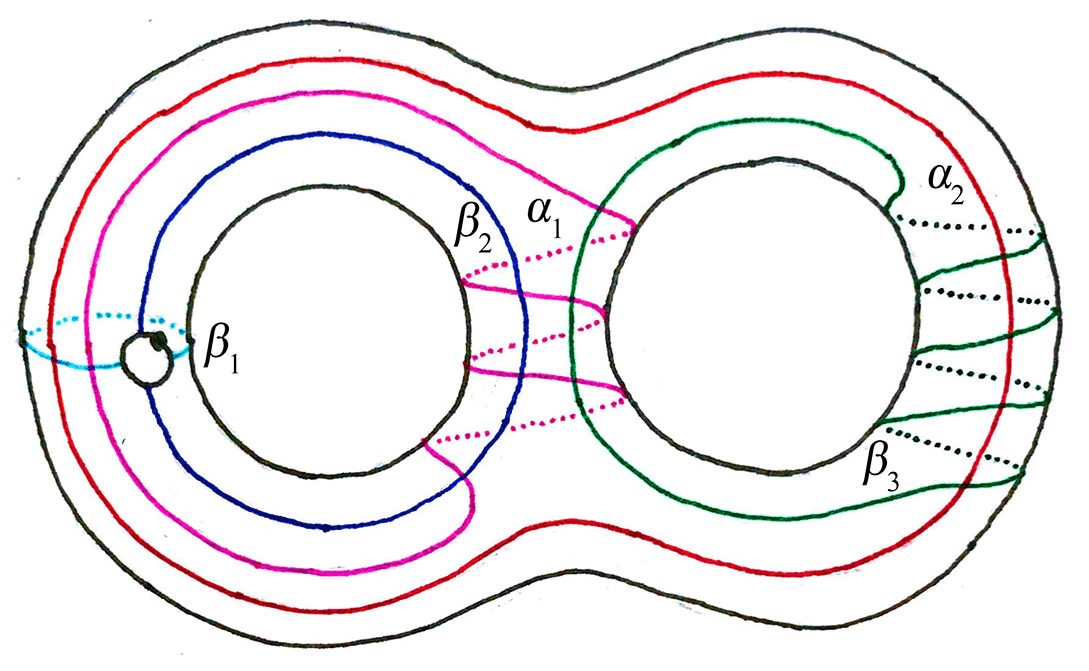}
    \caption{The case \(n=3\), \(m=4\).}
    \label{fig: simpler heegaard}
\end{figure}

In Figure \ref{fig: simpler heegaard}, there are two \(\alpha\) curves, \(\alpha_1,\alpha_2\) shown in pink and red, respectively. There are three \(\beta\) curves, \(\beta_1,\beta_2,\beta_3\) shown in cyan, blue, and green, respectively. One can see that the above basepoint choice corresponds to the I/III position from Section \ref{sec: identifying}. (From the discussion at the beginning of the same section, there is no way to resolve the I/III ambiguity, but the resulting modules are the same.) In Figure \ref{fig: plane heegaard} is the same diagram, but drawn in \(S^2=\mathbb{R}^2\cup \{\infty\}\) with surgery done twice to get a genus two surface. 
\begin{figure}[ht]
\centering
    \includegraphics[scale=0.7]{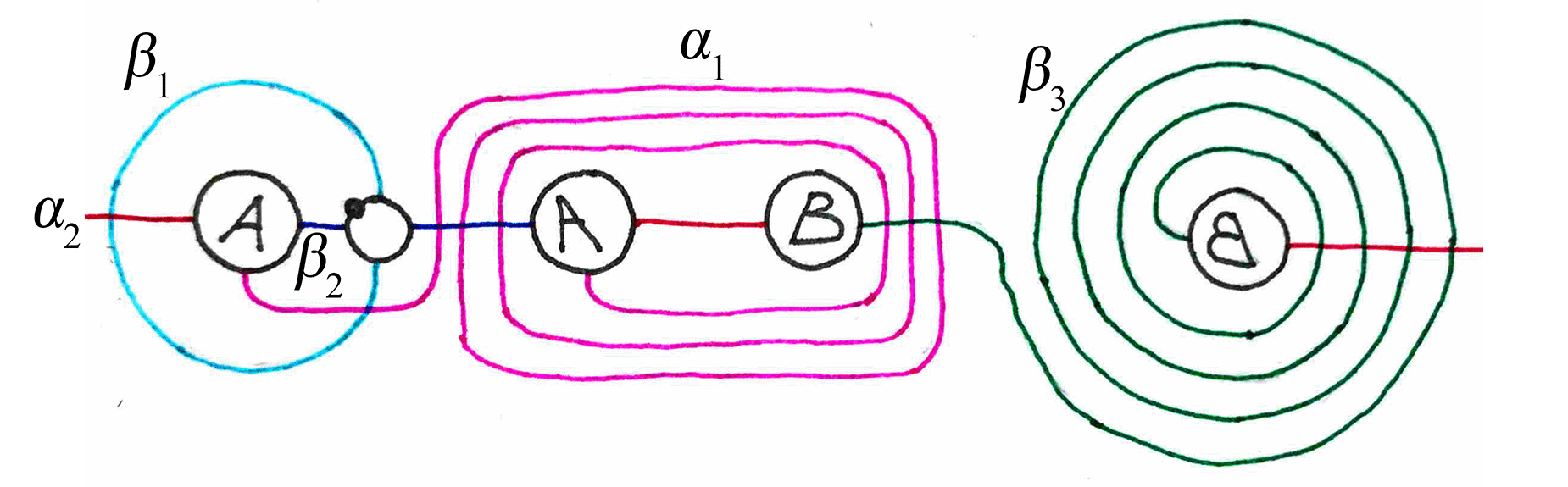}
    \caption{The red \(\alpha\) curve should be interpreted as passing through the point at infinity.}
    \label{fig: plane heegaard}
\end{figure}

We now perform the Sarkar--Wang algorithm \cite{SarkarWang} to obtain the nice Heegaard diagram in Figure \ref{fig: nice heegaard}. In other words, we apply isotopies (i.e., ``finger-moves'') the \(\beta\)-curves above to ensure that every region not adjacent to the basepoint is either a square or a bigon. The three non-basepoint regions adjacent to the pointed matched circle are ``squares'' with one of the sides lying on the pointed matched circle. Computing Heegaard Floer homology from a nice Heegaard diagram is purely combinatorial.

 \begin{figure}[ht]
 \centering
  \includegraphics[scale = 0.8]{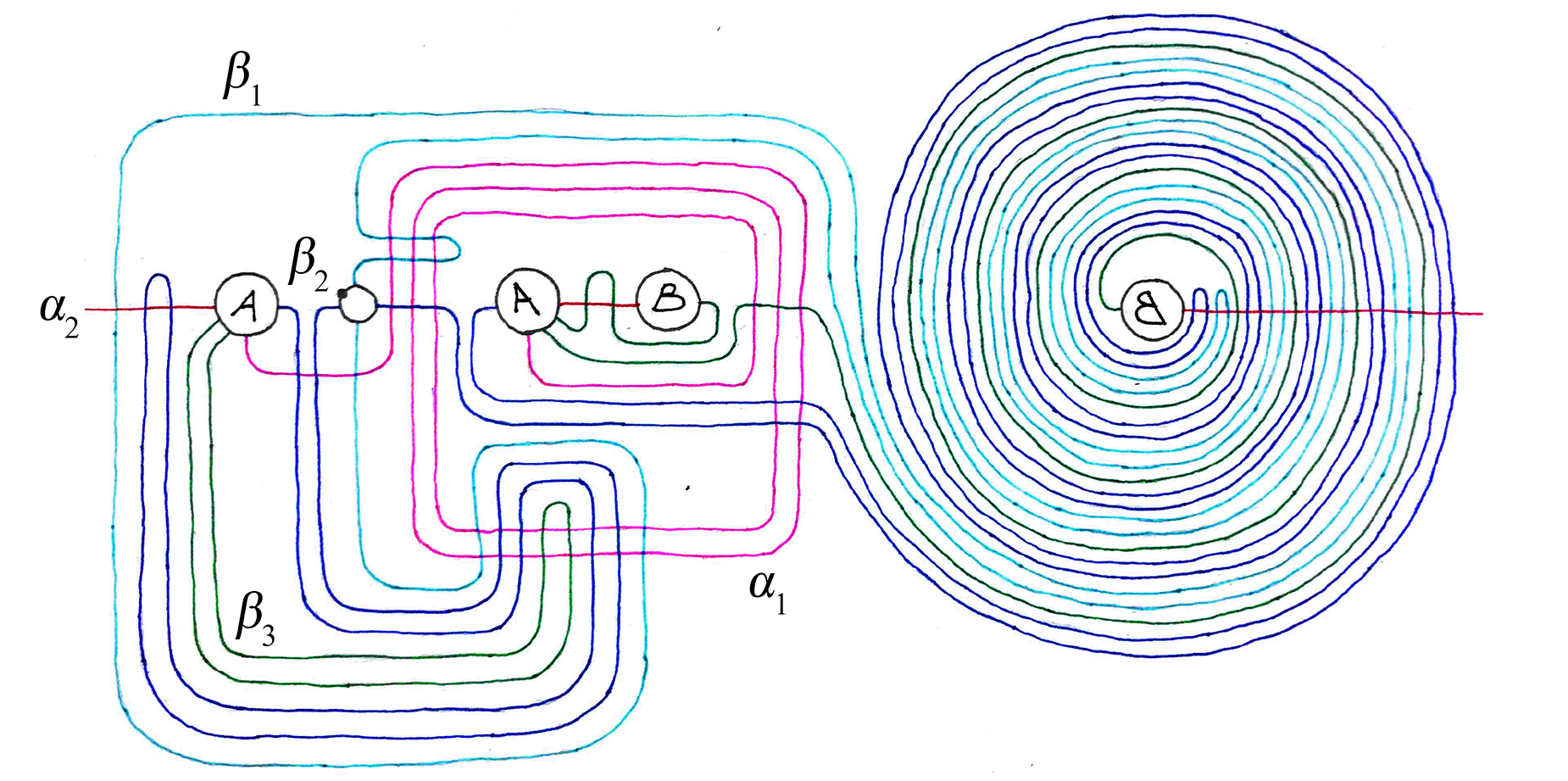}
  \caption{}
  \label{fig: nice heegaard}
 \end{figure}
 
In Figures \ref{fig: names_one}, \ref{fig: names_two}, and \ref{fig: names_three}, we zoom in to parts of the diagram to show how we label points. For clarity, we stick to the following convention for labeling points based on the curves they lie on. 
\[\begin{tikzpicture}
    \def\gap{2.5}
    \def\smallgap{1}
    \def\smallergap{0.5}
    \def\vertgap{0.5}
    \def\epsilon{0.25}
    \node[anchor = mid] at (-\smallgap,0) {\(a:{}\)};
    \node[anchor = mid] at (-\smallergap,0) {\(\alpha_1\)};
    \node[anchor = mid] at (0,\vertgap) {\(\beta_2\)};
    \draw[ultra thick, blue] (0,-\epsilon)--(0,\epsilon);
    \draw[ultra thick, Rhodamine] (-\epsilon,0)--(\epsilon,0);
    \filldraw[black] (0,0) circle (2pt);
    \node[anchor = mid] at ({\gap-\smallgap},0) {\(b:{}\)};
    \node[anchor = mid] at ({\gap-\smallergap},0) {\(\alpha_1\)};
    \node[anchor = mid] at (\gap,\vertgap) {\(\beta_1\)};
    \draw[ultra thick, cyan] (\gap,-\epsilon)--(\gap,\epsilon);
    \draw[ultra thick, Rhodamine] ({\gap-\epsilon},0)--({\gap+\epsilon},0);
    \filldraw[black] (\gap,0) circle (2pt);
    \node[anchor = mid] at ({2*\gap-\smallgap},0) {\(c:{}\)};
    \node[anchor = mid] at ({2*\gap-\smallergap},0) {\(\alpha_2\)};
    \node[anchor = mid] at ({2*\gap},\vertgap) {\(\beta_2\)};
    \draw[ultra thick, blue] ({2*\gap},-\epsilon)--({2*\gap},\epsilon);
    \draw[ultra thick, red] ({2*\gap-\epsilon},0)--({2*\gap+\epsilon},0);
    \filldraw[black] ({2*\gap},0) circle (2pt);
    \node[anchor = mid] at ({3*\gap-\smallgap},0) {\(d:{}\)};
    \node[anchor = mid] at ({3*\gap-\smallergap},0) {\(\alpha_2\)};
    \node[anchor = mid] at ({3*\gap},\vertgap) {\(\beta_1\)};
    \draw[ultra thick, cyan] ({3*\gap},-\epsilon)--({3*\gap},\epsilon);
    \draw[ultra thick, red] ({3*\gap-\epsilon},0)--({3*\gap+\epsilon},0);
    \filldraw[black] ({3*\gap},0) circle (2pt);
    \node[anchor = mid] at ({4*\gap-\smallgap},0) {\(x:{}\)};
    \node[anchor = mid] at ({4*\gap-\smallergap},0) {\(\alpha_2\)};
    \node[anchor = mid] at ({4*\gap},\vertgap) {\(\beta_3\)};
    \draw[ultra thick, Green] ({4*\gap},-\epsilon)--({4*\gap},\epsilon);
    \draw[ultra thick, red] ({4*\gap-\epsilon},0)--({4*\gap+\epsilon},0);
    \filldraw[black] ({4*\gap},0) circle (2pt);
    \node[anchor = mid] at ({5*\gap-\smallgap},0) {\(y:{}\)};
    \node[anchor = mid] at ({5*\gap-\smallergap},0) {\(\alpha_1\)};
    \node[anchor = mid] at ({5*\gap},\vertgap) {\(\beta_3\)};
    \draw[ultra thick, Green] ({5*\gap},-\epsilon)--({5*\gap},\epsilon);
    \draw[ultra thick, Rhodamine] ({5*\gap-\epsilon},0)--({5*\gap+\epsilon},0);
    \filldraw[black] ({5*\gap},0) circle (2pt);
\end{tikzpicture}\]
We depict the case \(n=3\) and \(m=4\), but try to indicate what the labels are in the general case.
\begin{figure}[ht]
  \centering
  \includegraphics[scale = 0.8]{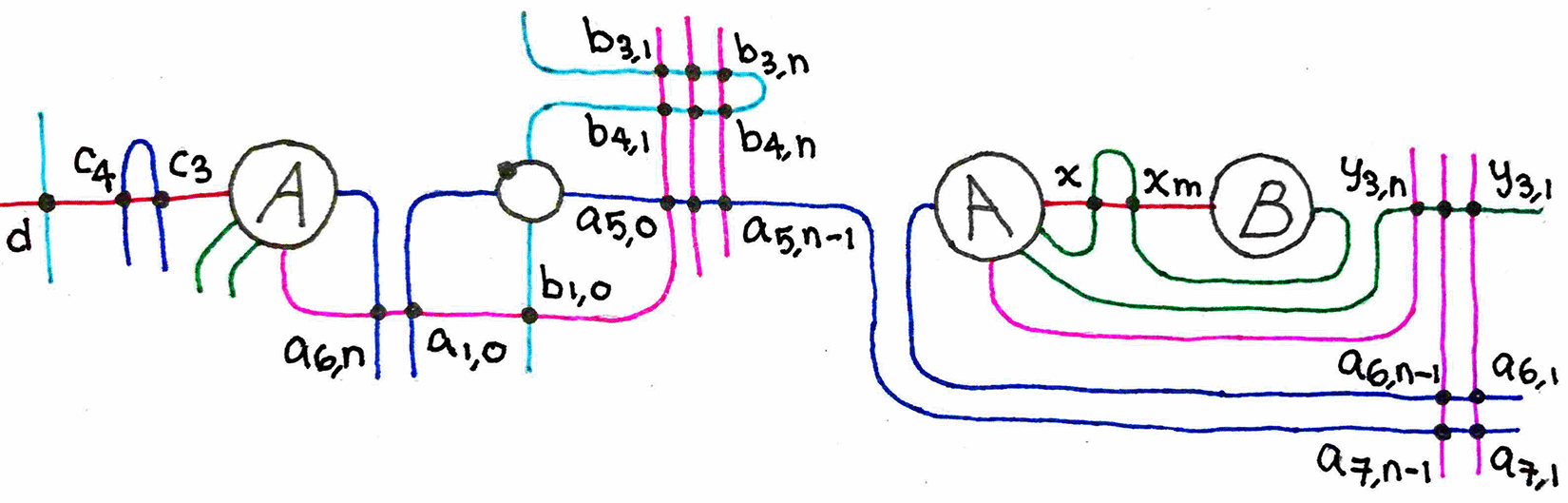}
  \caption{}
  \label{fig: names_one}
\end{figure}
\begin{figure}[ht]
\centering
  \includegraphics[scale = 0.9]{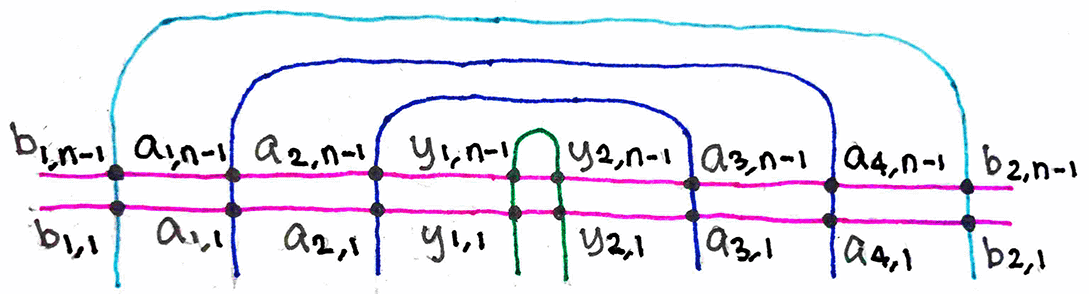}
  \caption{}
  \label{fig: names_two}
\end{figure}
\begin{figure}[ht]
\centering
  \includegraphics[scale = 0.8]{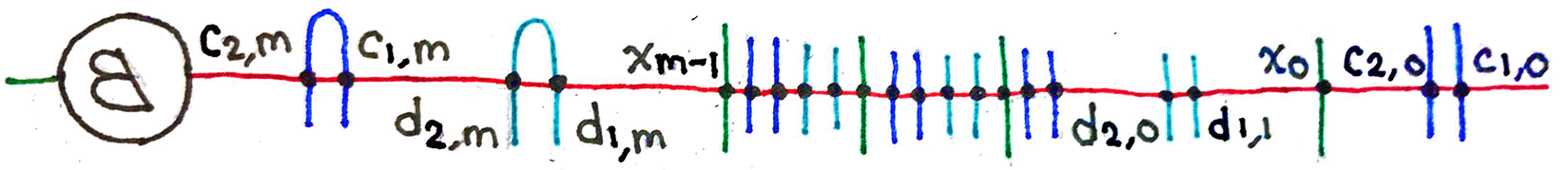}
  \caption{}
  \label{fig: names_three}
 \end{figure}
\begin{remark}
To safely apply the Sarkar-Wang algorithm to a \(\beta\)-type bordered Heegaard diagram one should apply finger moves to the \(\alpha\) curves, first applying a ``protective finger move'' near the pointed matched circle as explained in \cite[Chapter 8]{bordered}. Instead of doing this, we searched for a sequence of \(\beta\) finger moves which did not result in moves colliding with the pointed matched circle. We did this so that the number of finger moves would not depend on \(n\) and \(m\). 
\end{remark}

\begin{proposition}
\label{prop: type D I/III}
The module \(\widehat{\mathit{CFD}}\big({-Y(n,m),\mathcal{F}_{\mathrm{I/III}}}\big)\) has a model given by the I/III decorated graph in Figure {\normalfont \ref{fig: initial type D}}.
\end{proposition}
\begin{proof}
We count squares and bigons in Figure \ref{fig: nice heegaard} to identify a model of \({\widehat{\mathit{CFD}}\big({-Y(n,m),\mathcal{F}_{\mathrm{I/III}}}}\big)\). We then simplify the resulting Type D module by reducing pure differentials as explained in, e.g.\ \cite[Section 2.6]{Levine}.\par 

Recall that a \textit{pure differential} is one of the form \(\delta^1(\mathbf{x})=\mathbf{1}\otimes \mathbf{y}\). To remove a pure differential from \(\mathbf{x}\) to \(\mathbf{y}\) in a Type D decorated graph, one looks for all zigzags of the form 
\[\mathbf{z}\xrightarrow{\,\rho_I\,}\mathbf{y}\leftarrow \mathbf{x}\xrightarrow{\,\rho_J\,}\mathbf{w},\]
and replaces such a zigzag by 
\[\mathbf{z}\xrightarrow{\,\rho_I\rho_J\,}\mathbf{w}.\]
Here we allow one or both of the labels \(\rho_I,\rho_J\) to be empty. The new label \(\rho_I\rho_J\) denotes the product of \(\rho_I\) and \(\rho_J\) in the torus algebra \(\mathcal{A}\). The new decorated graph represents an equivalent Type D module. If multiple pure differentials are canceled, the result may depend on the specific order in which they are canceled. \par 
We now outline a systematic way to simplify the type D decorated graph resulting from the nice Heegaard diagram. Consider the subgraph of the Type D decorated graph consisting of all vertices (generators) with only the unlabeled arrows (pure differentials). Within each connected component of this subgraph, we try to simplify as much as possible without introducing any additional non-pure differentials. (We may however remove non-pure differentials or shift the source/target of existing non-pure differentials.) This allows us to perform these simplifications without specifying which component we simplify first. We illustrate this process in Figures \ref{fig:start} through \ref{fig:end}. 

Note there is not a unique way to carry out this simplification. Implicit in this figure is, for each component, a sequence of pure differentials to cancel in order. However, we omit the explicit sequence as to not take up too much space. (It seems to the author that many choices of such sequences yield the same final answer.) For each component, we show the following:
\begin{itemize}
\item We give the initial un-simplified component, together with all incoming/outgoing non-pure differentials from/to other components. 
\item If the component can be completely removed, we indicate so. If not, we give a capital letter to label the generators in the new simplified component. 
\item If the component cannot be completely removed, we show the simplified component. In this simplified component, some incoming/outgoing non-pure differentials may be removed due to simplification in the original component; however some removals may also occur due to simplification in other components. 
\end{itemize}
Figures start on the next page. 
\clearpage
\begin{figure}[ht]
\centering
\includegraphics[scale=0.38]{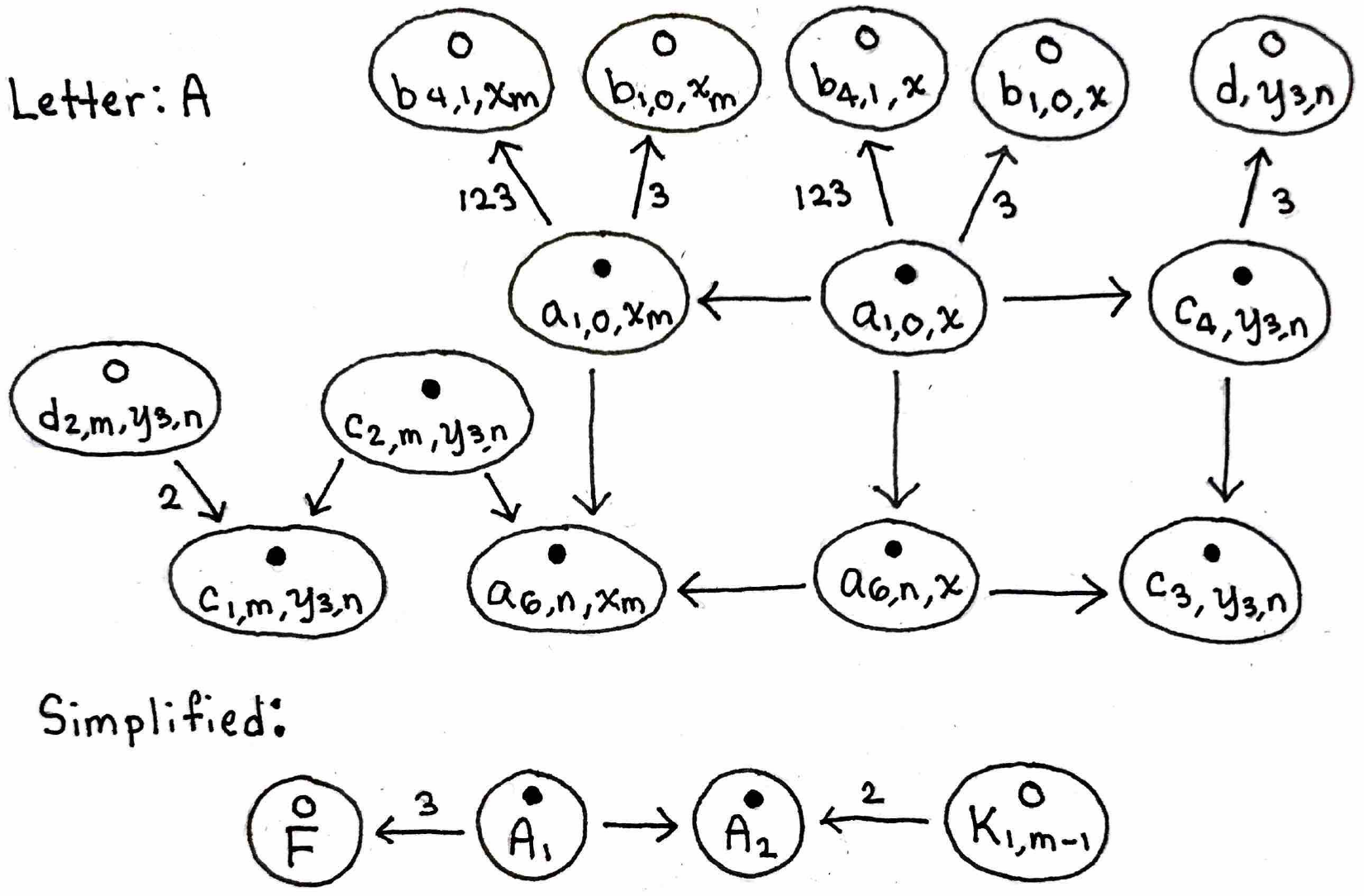}
\caption{}\label{fig:start}
\end{figure}
\begin{figure}[ht]
\centering
\includegraphics[scale=0.45]{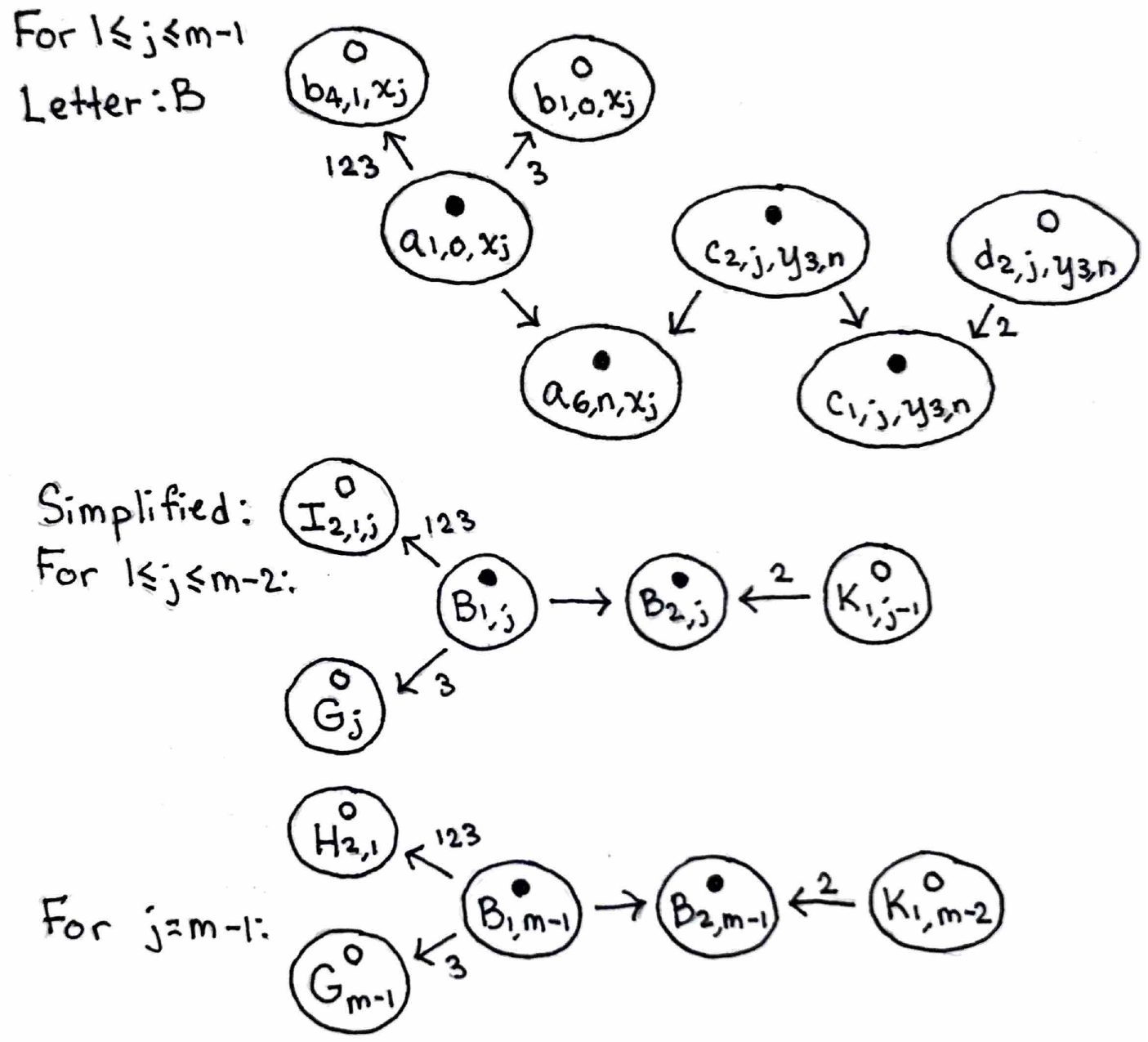}
\caption{}
\end{figure}
\begin{figure}[ht]
\centering
\includegraphics[scale=0.35]{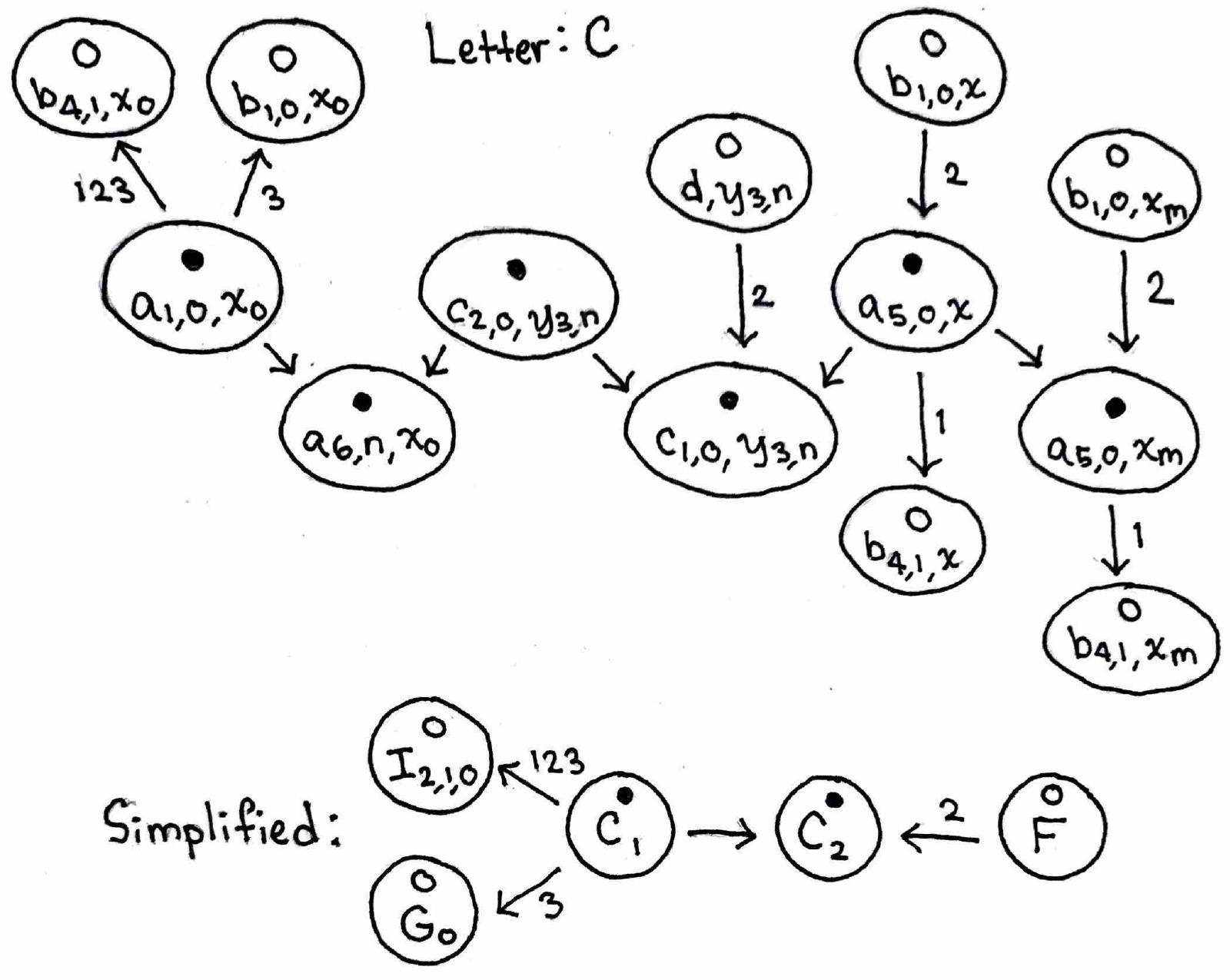}
\caption{}
\end{figure}
\begin{figure}[ht]
\centering
\includegraphics[width = 13cm, height = 2.6cm]{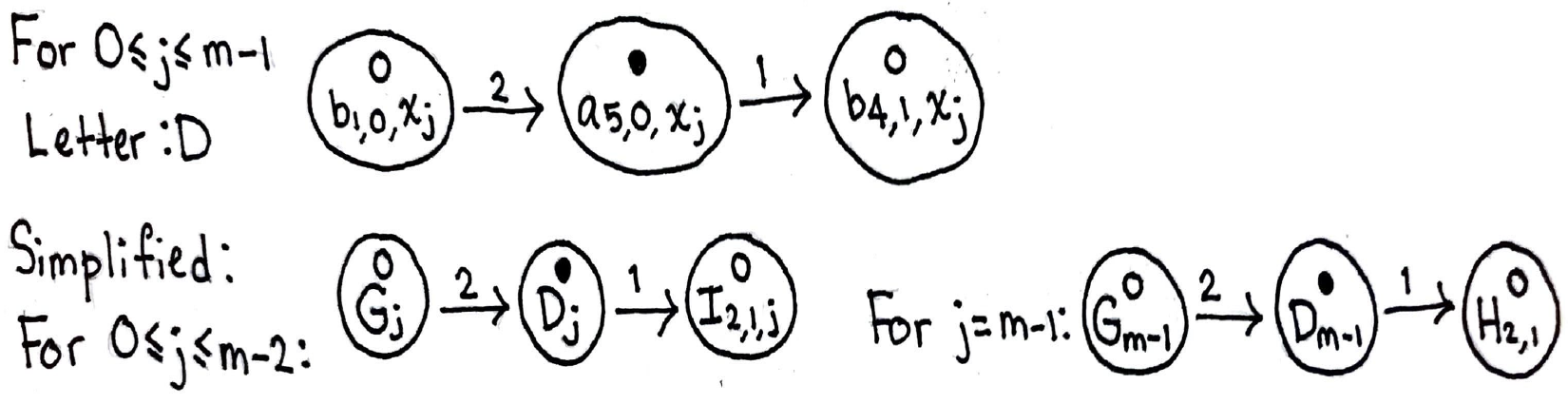}
\caption{}
\end{figure}
\begin{figure}[ht]
\centering
\includegraphics[scale=0.4]{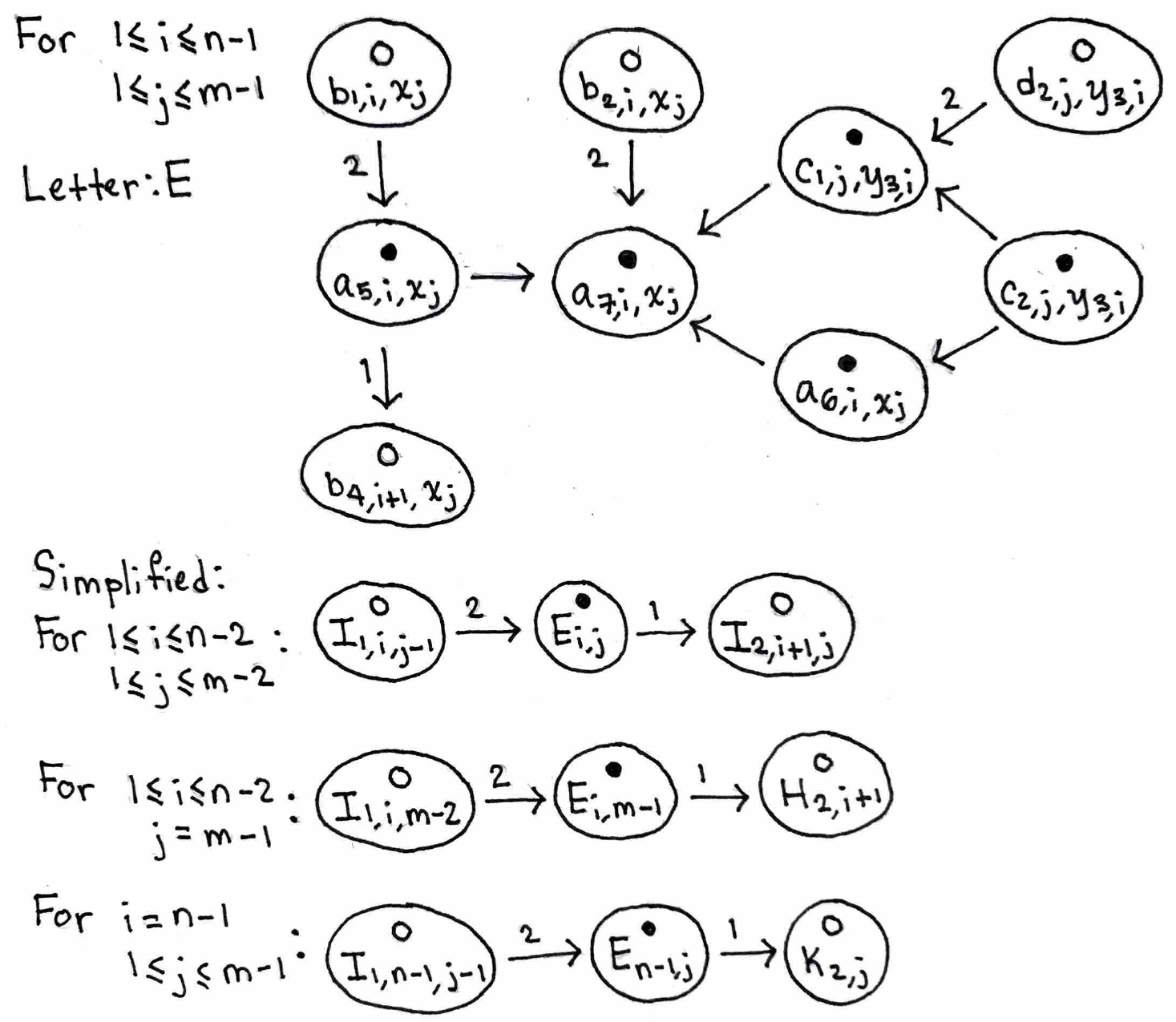}
\caption{}
\end{figure}
\begin{figure}
\centering
\includegraphics[scale=0.32]{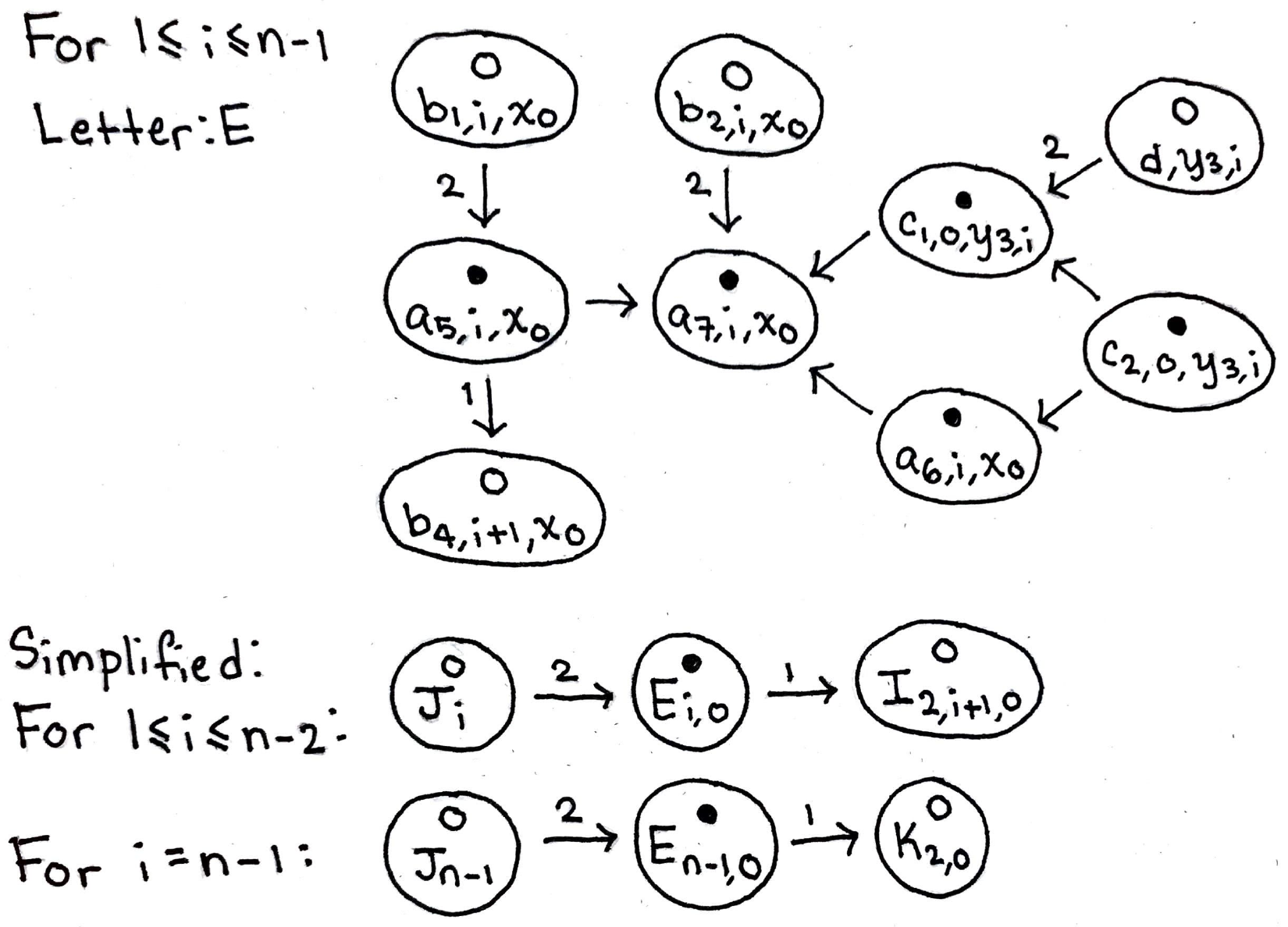}
\caption{}
\end{figure}
\begin{figure}[ht]
\centering
\includegraphics[scale=1]{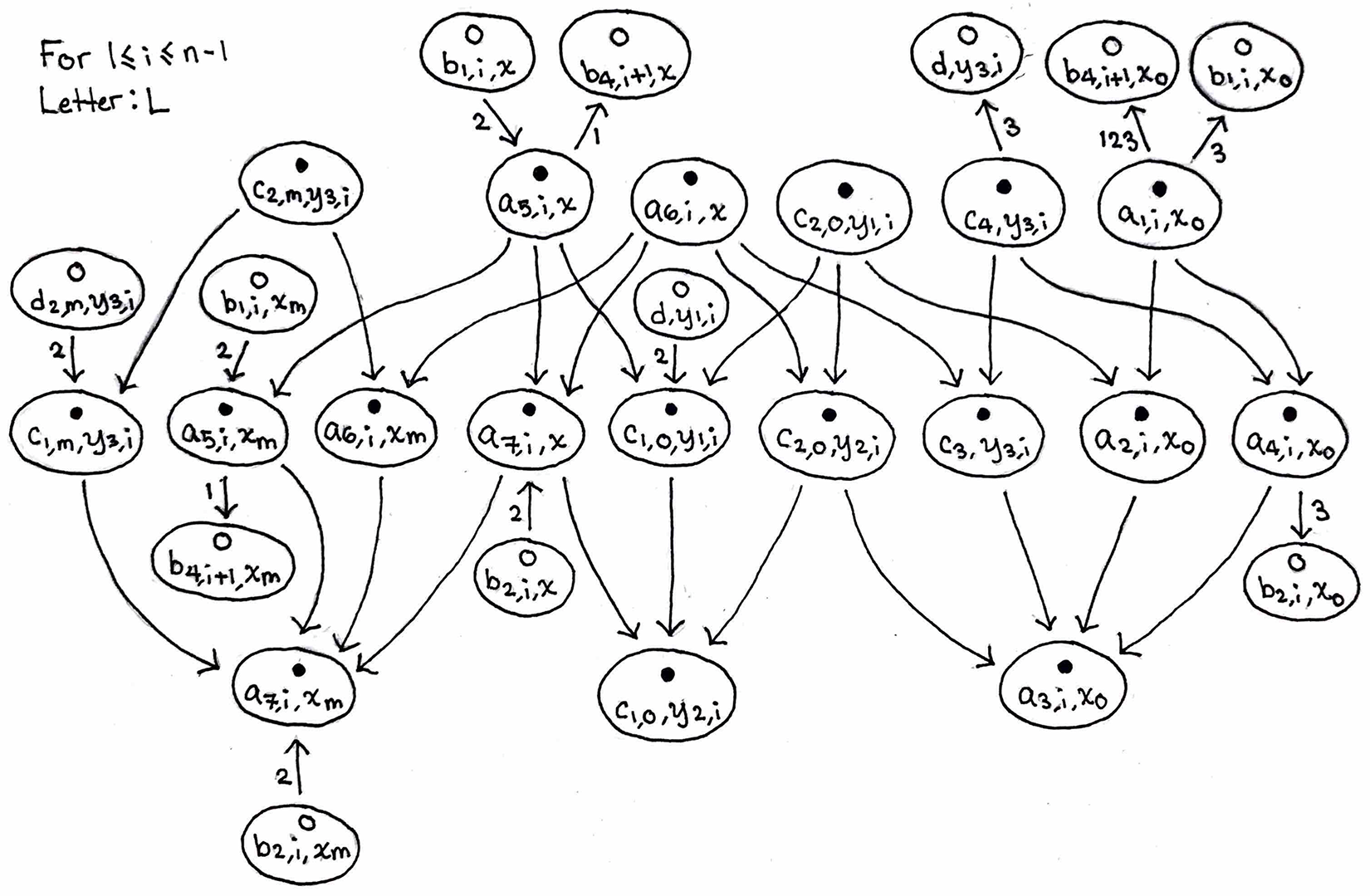}\\
\includegraphics[width = 9cm, height = 4.5cm]{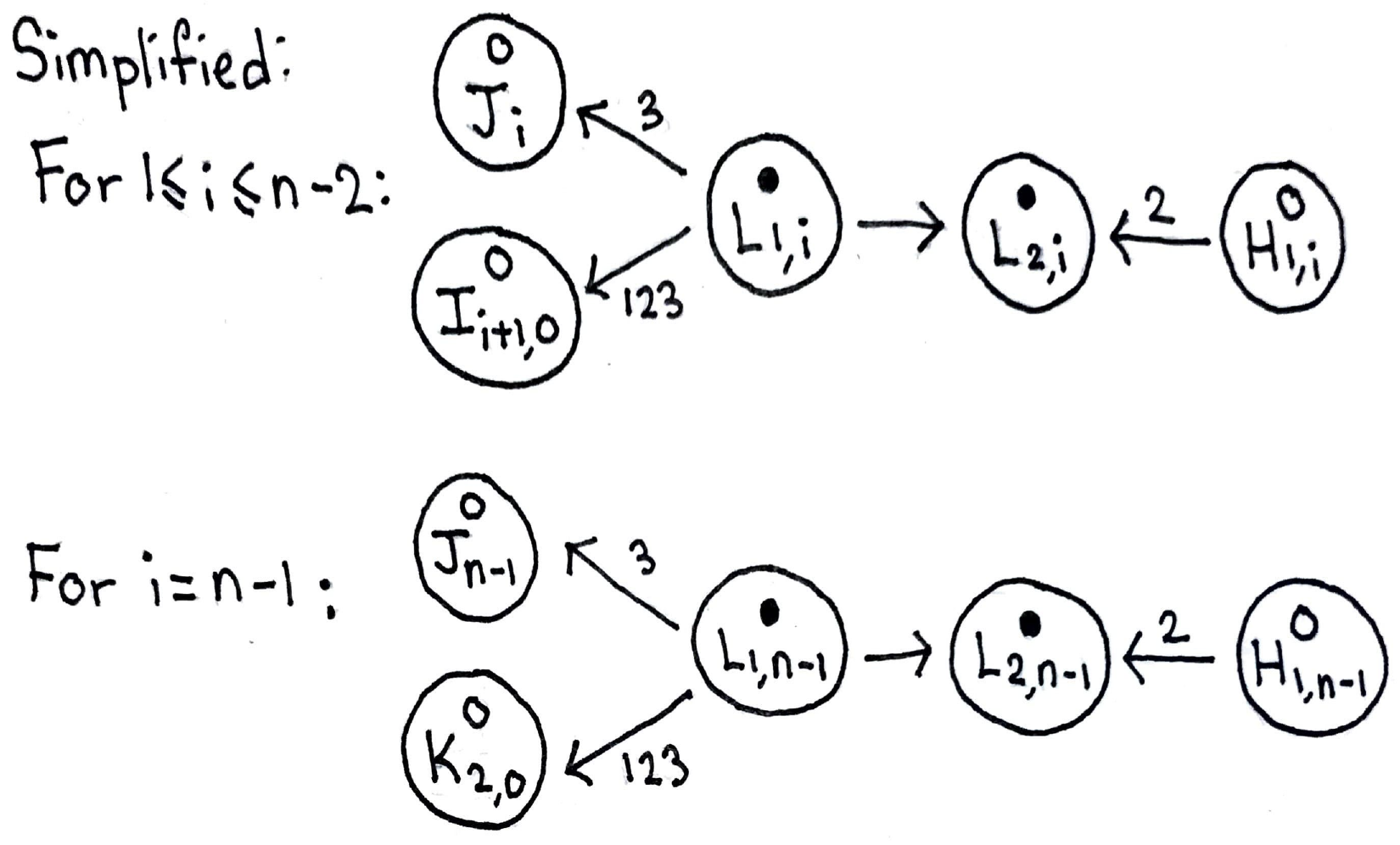}
\caption{}
\end{figure}
\begin{figure}[ht]
\centering
\includegraphics[scale=0.38]{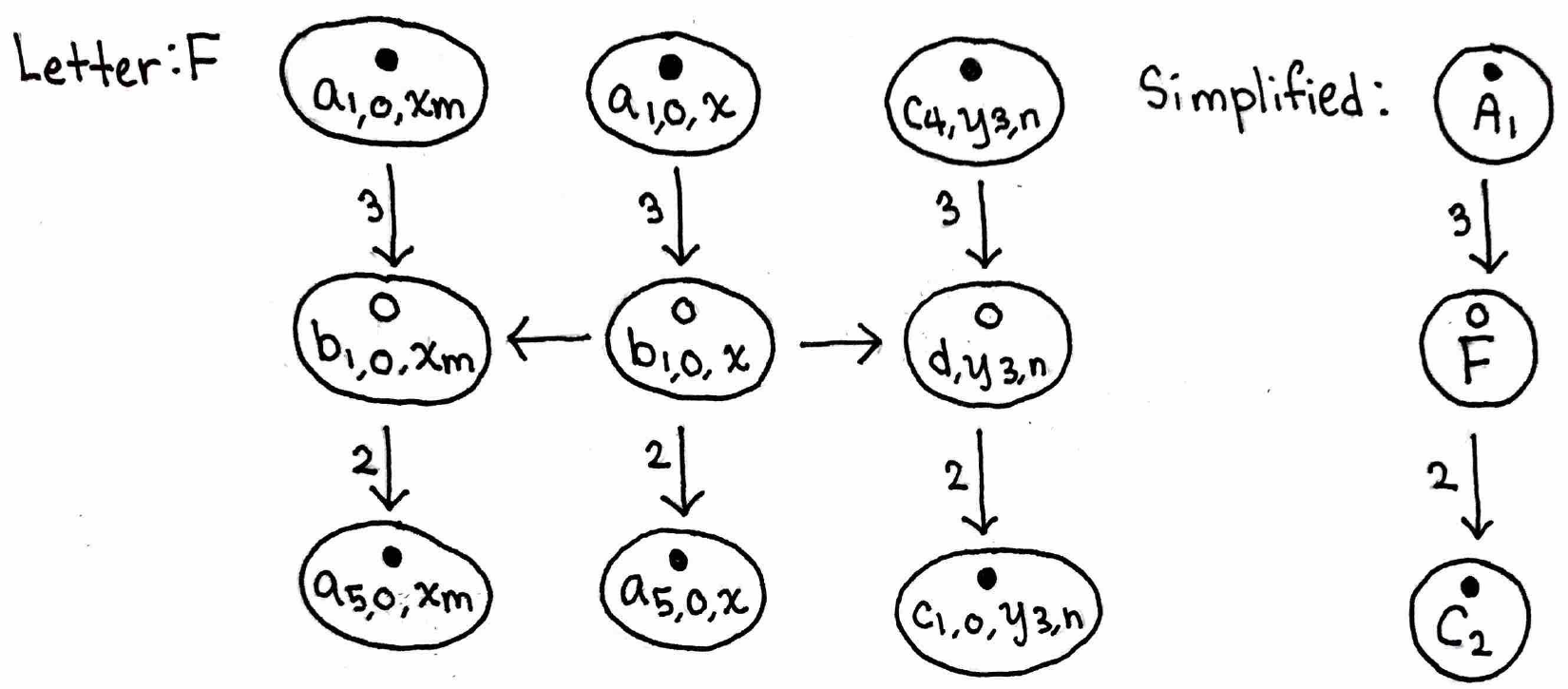}
\caption{}
\end{figure}
\begin{figure}[ht]
\centering
\includegraphics[width = 13cm, height = 2.6cm]{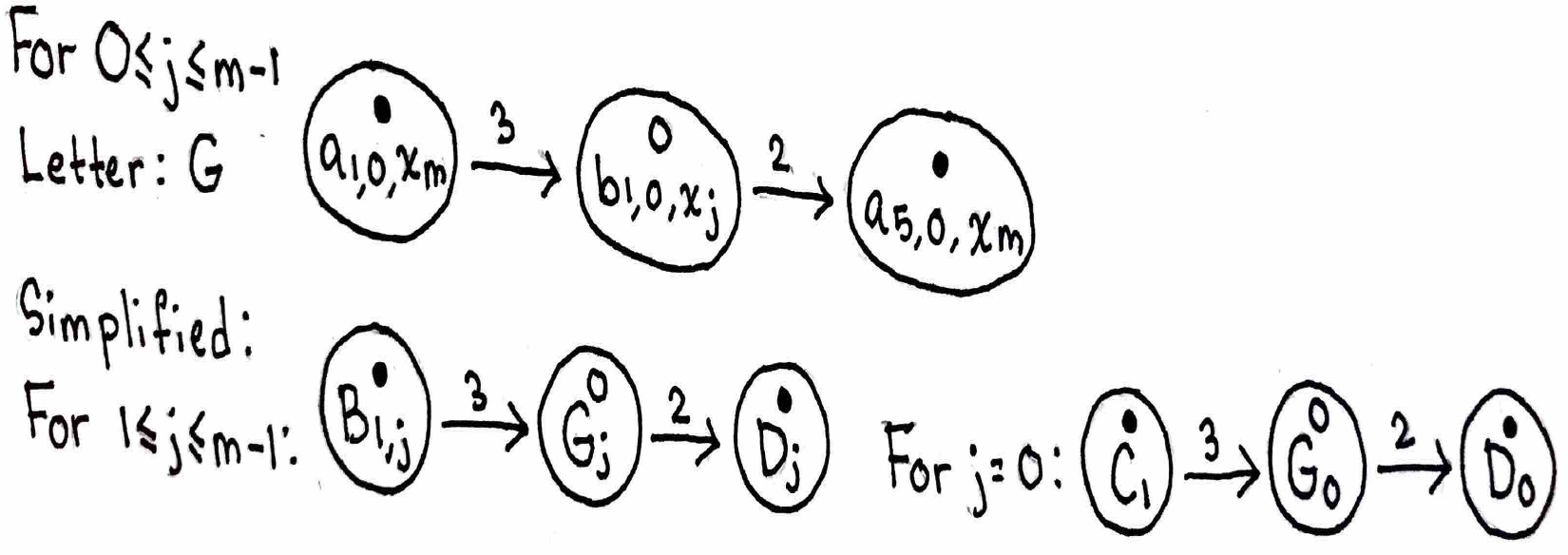}
\caption{}
\end{figure}
\begin{figure}[ht]
\centering
\includegraphics[scale = 1]{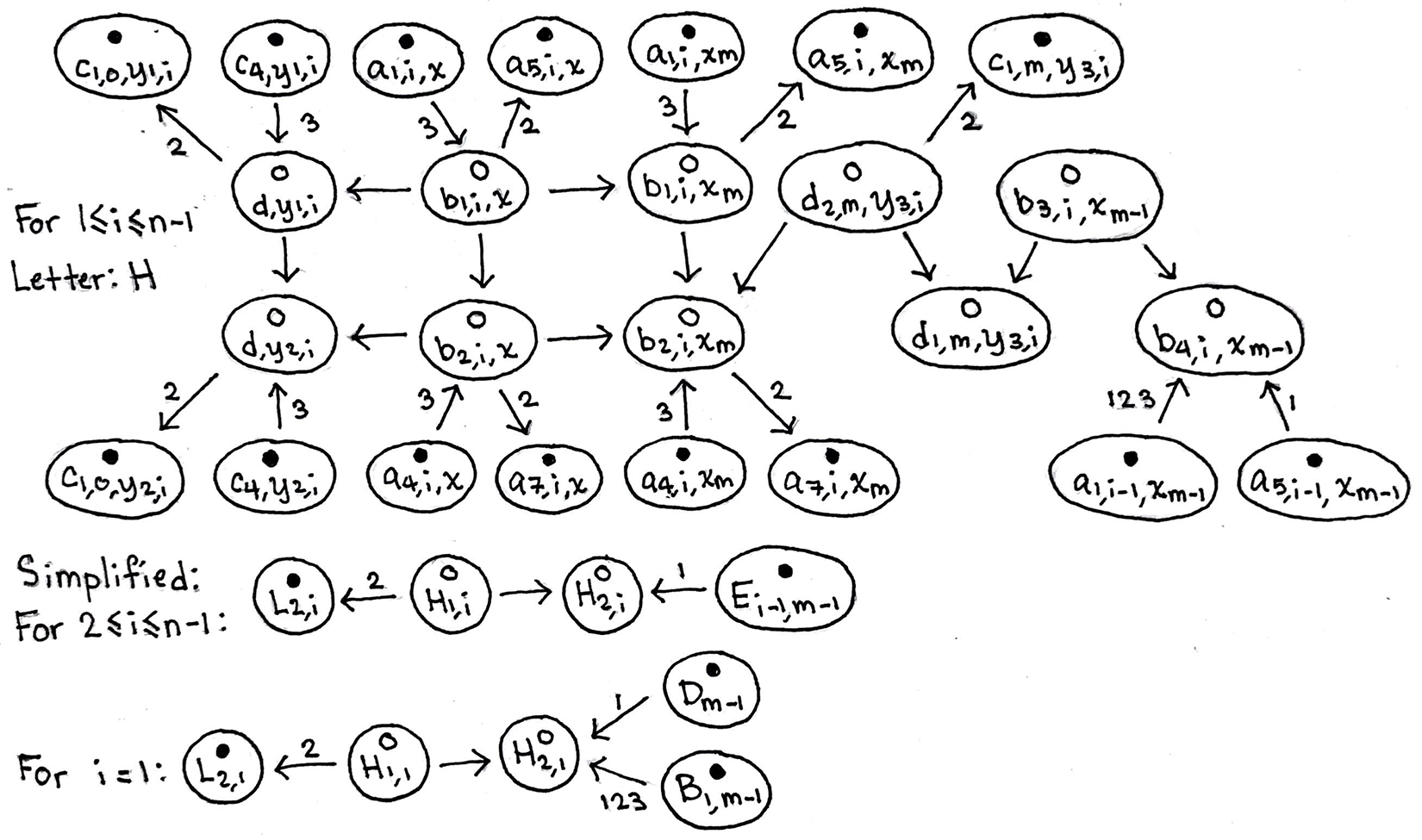}
\caption{}
\end{figure}
\begin{figure}[ht]
\centering
\includegraphics[scale = 0.32]{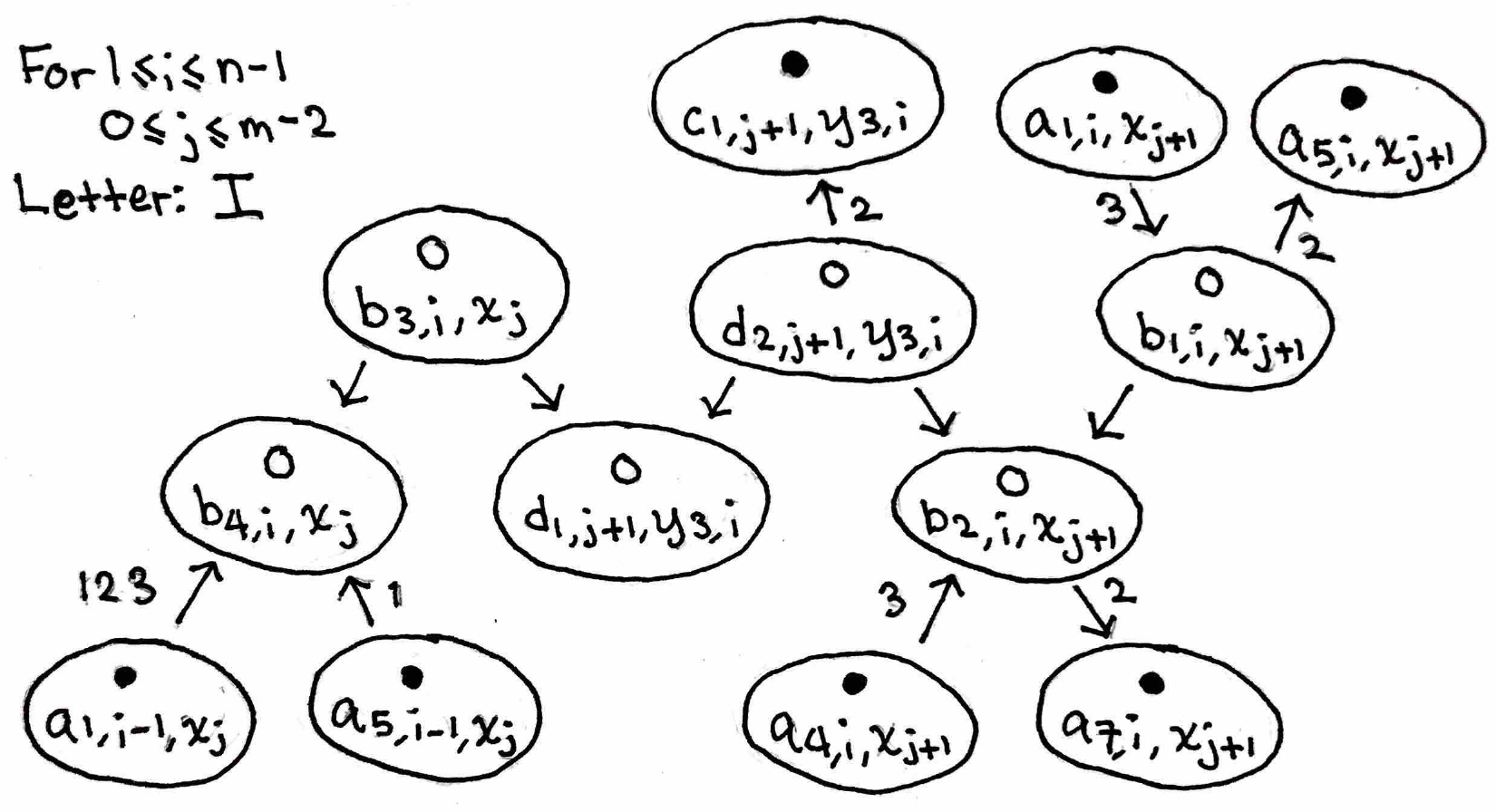}\\
\includegraphics[scale = 0.32]{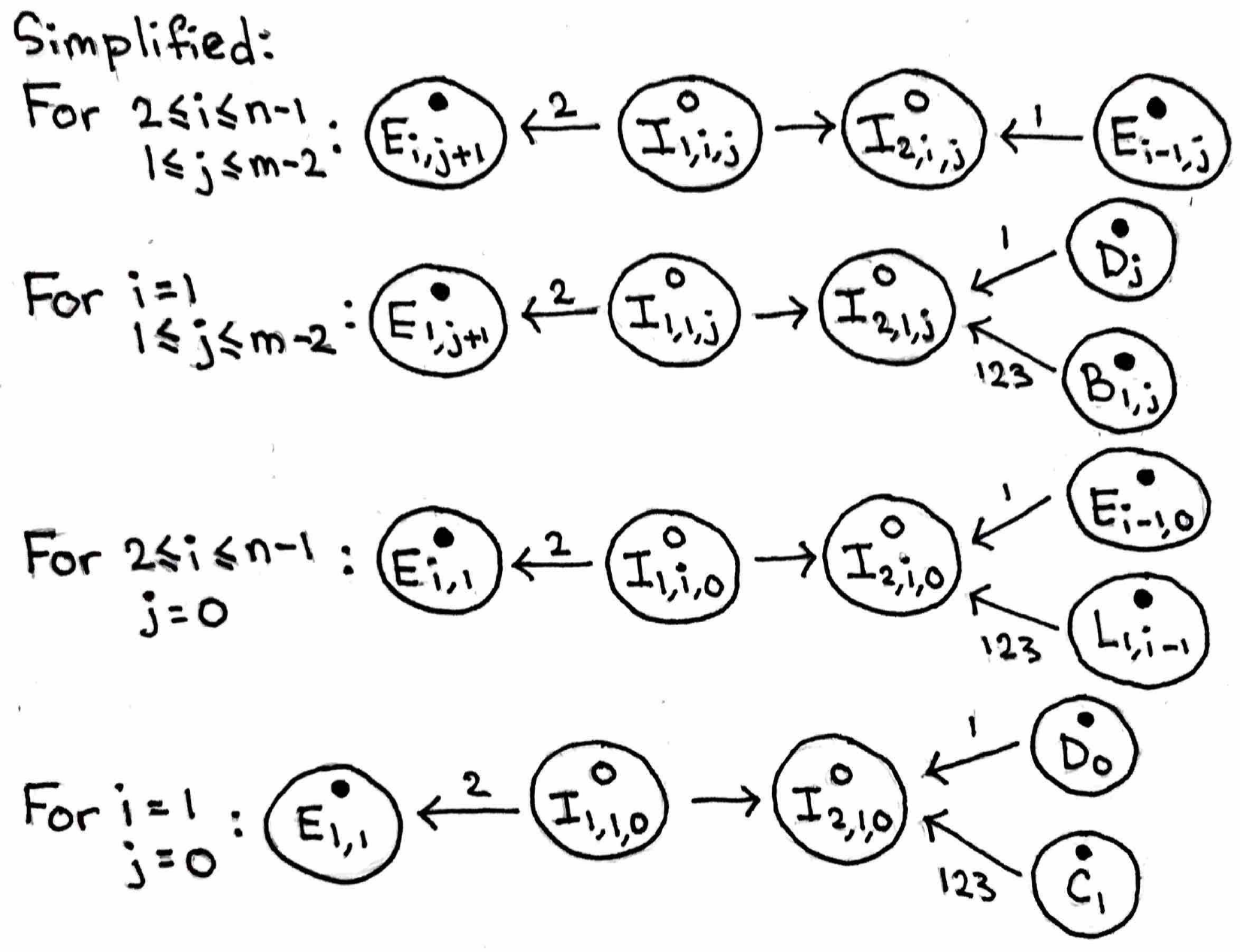}
\caption{}
\end{figure}
\begin{figure}[ht]
\centering
\includegraphics[width = 10cm, height = 4cm]{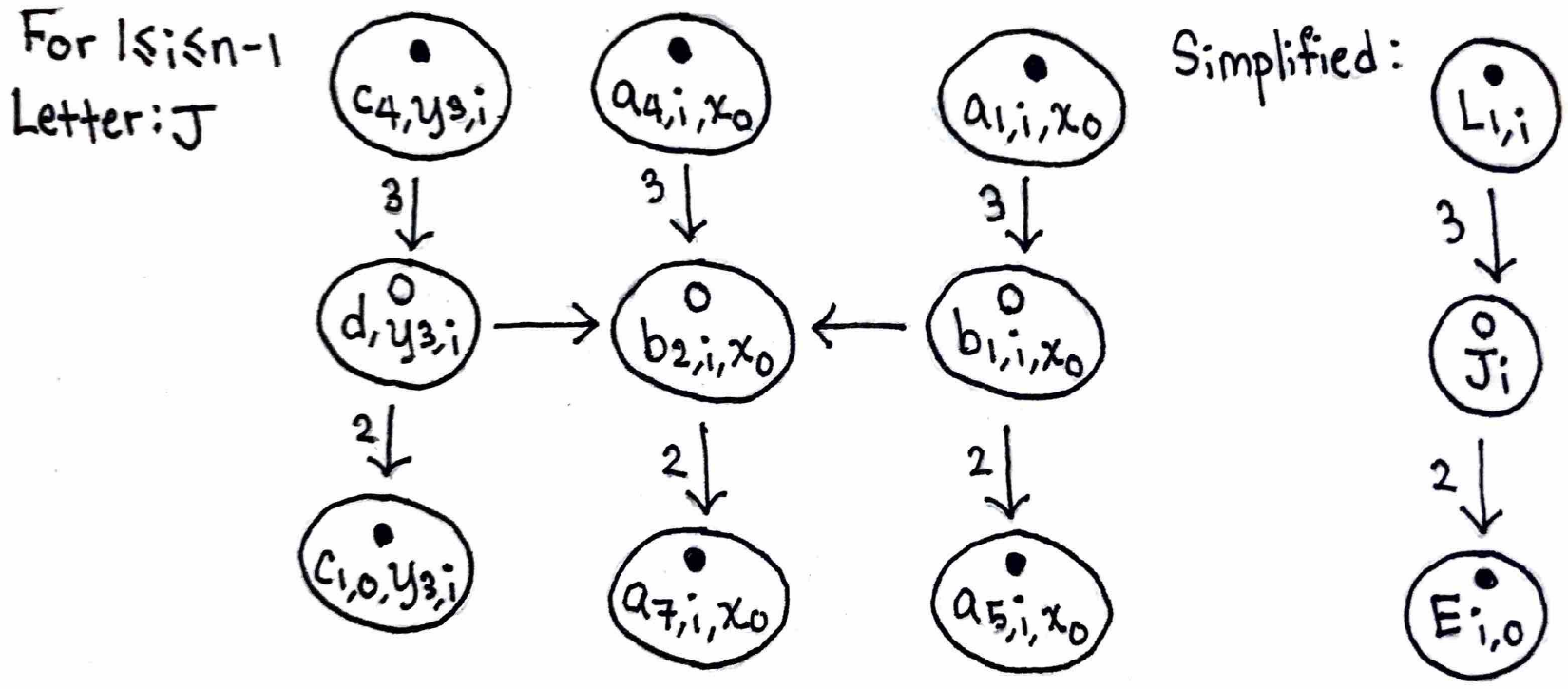}
\caption{}
\end{figure}
\begin{figure}[ht]
\centering
\includegraphics[scale = 0.35]{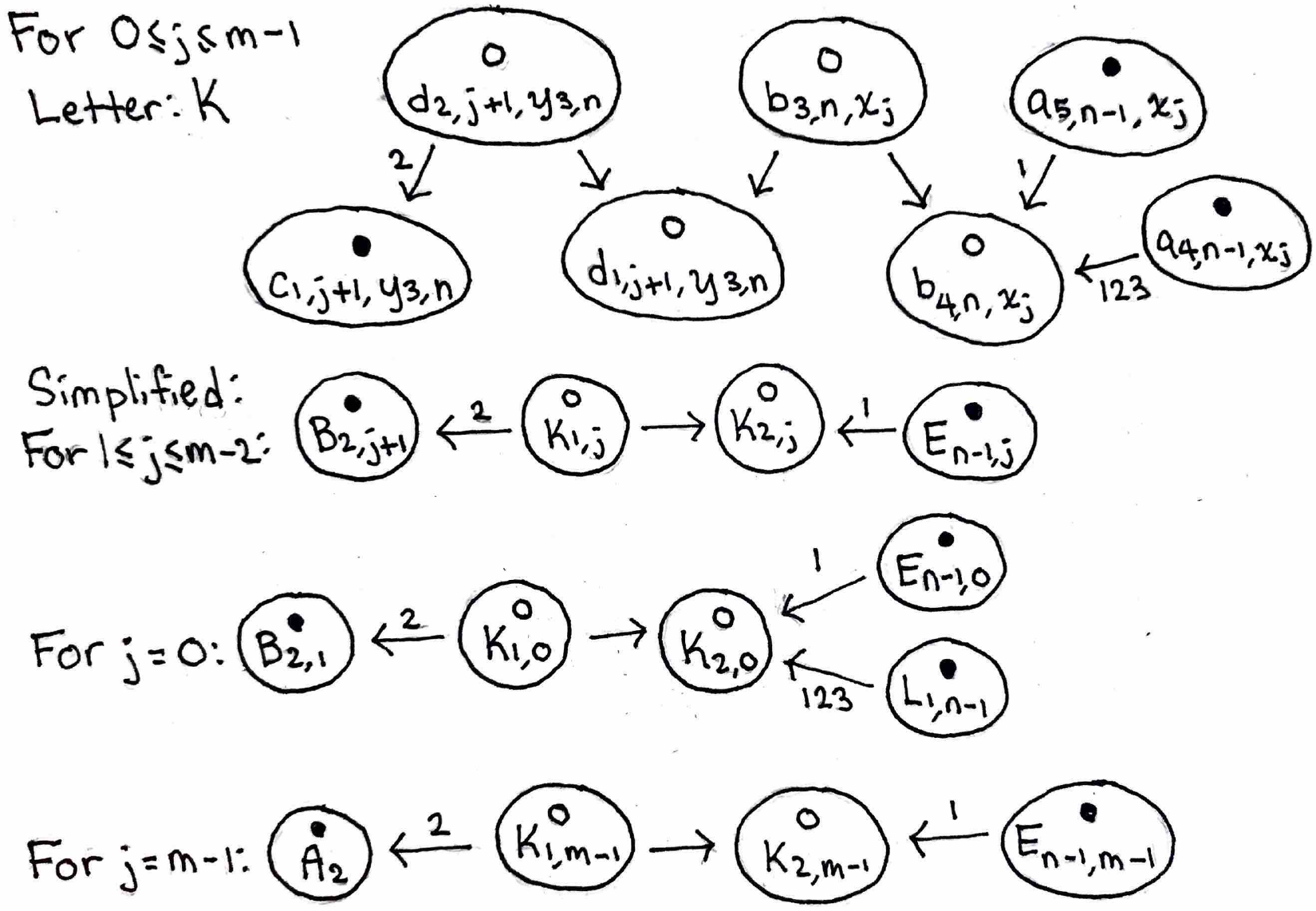}
\caption{}
\end{figure}
\begin{figure}[ht]
\centering
\includegraphics[scale = 1]{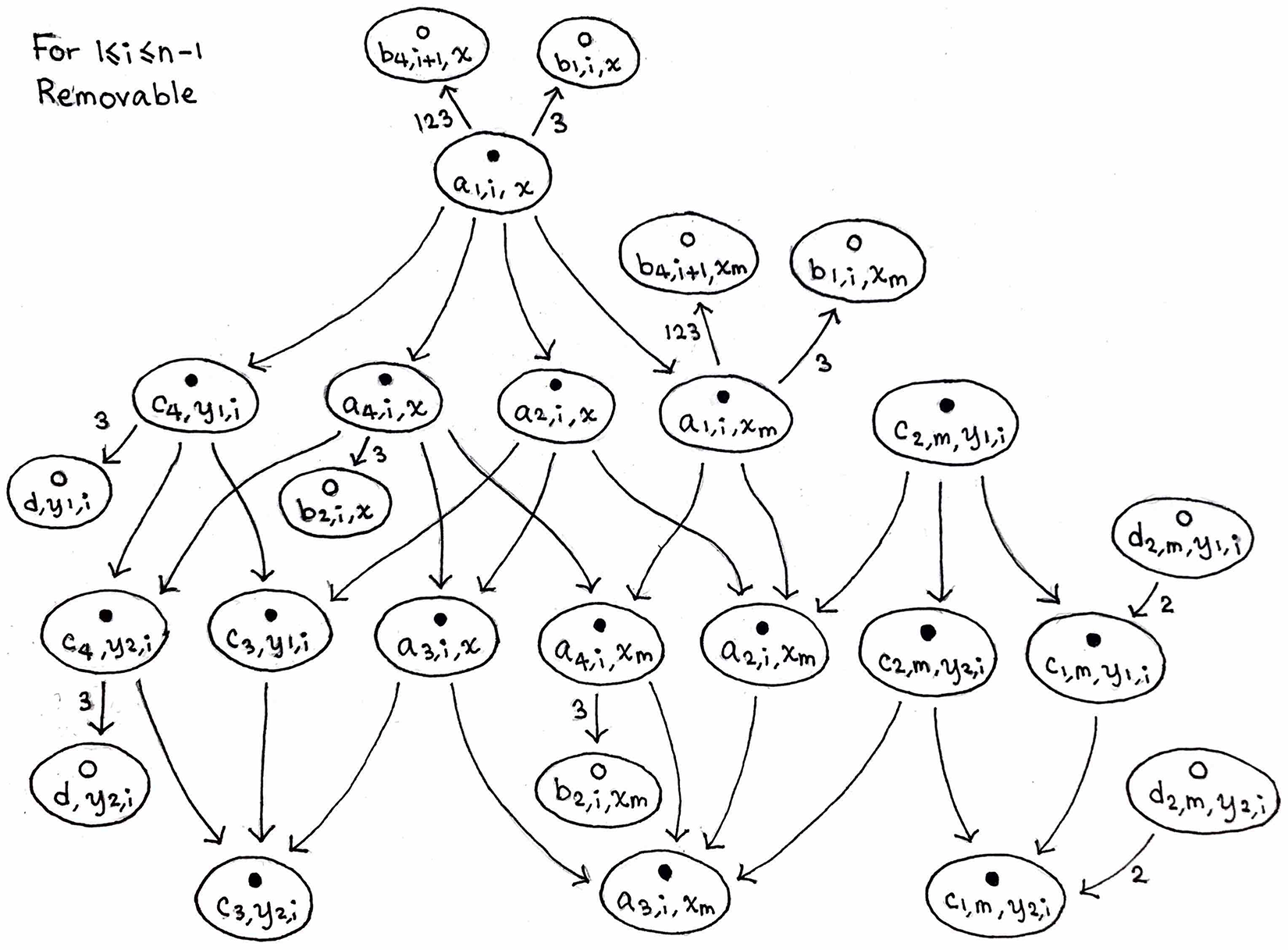}
\caption{}
\end{figure}
\begin{figure}[ht]
\centering
\includegraphics[scale = 0.5]{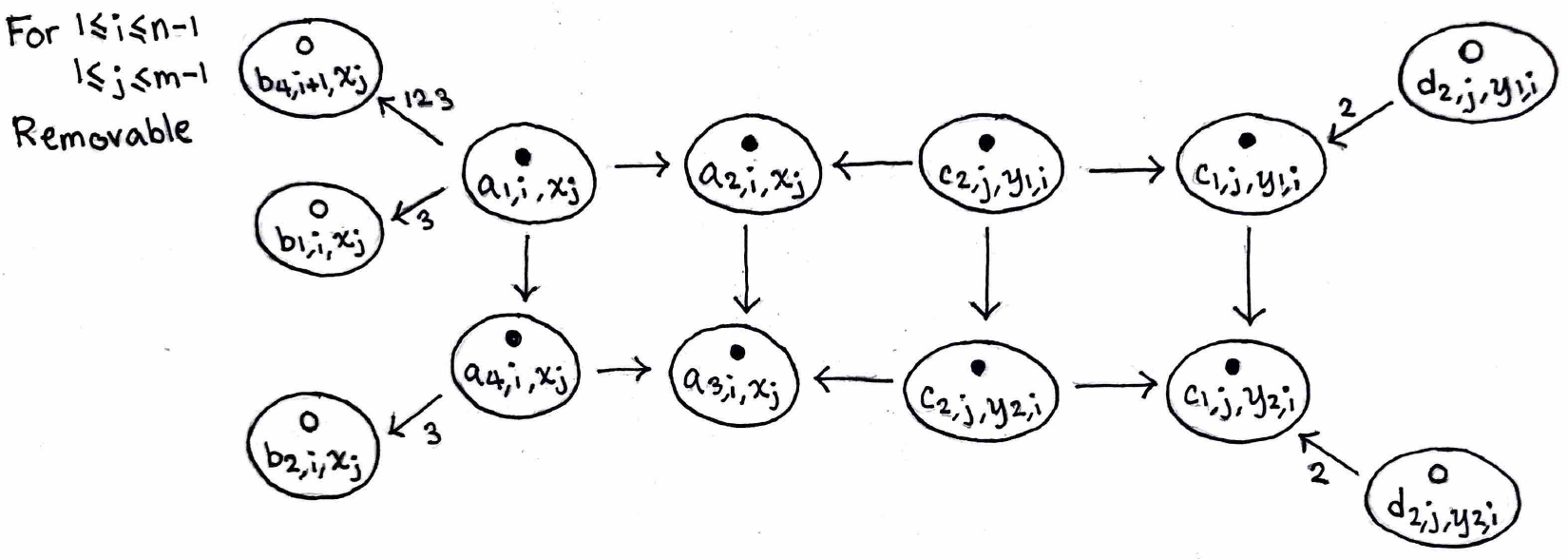}
\caption{}
\end{figure}
\begin{figure}[ht]
\centering
\includegraphics[scale = 1]{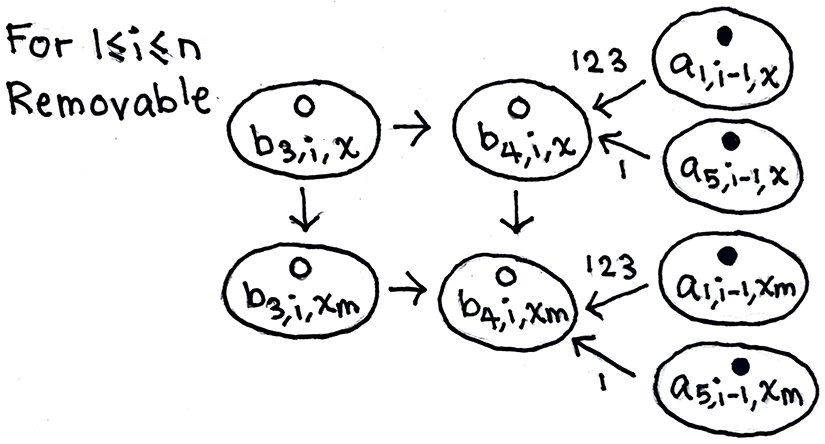}
\caption{}
\end{figure}
\begin{figure}[ht]
\centering
\includegraphics[scale = 0.25]{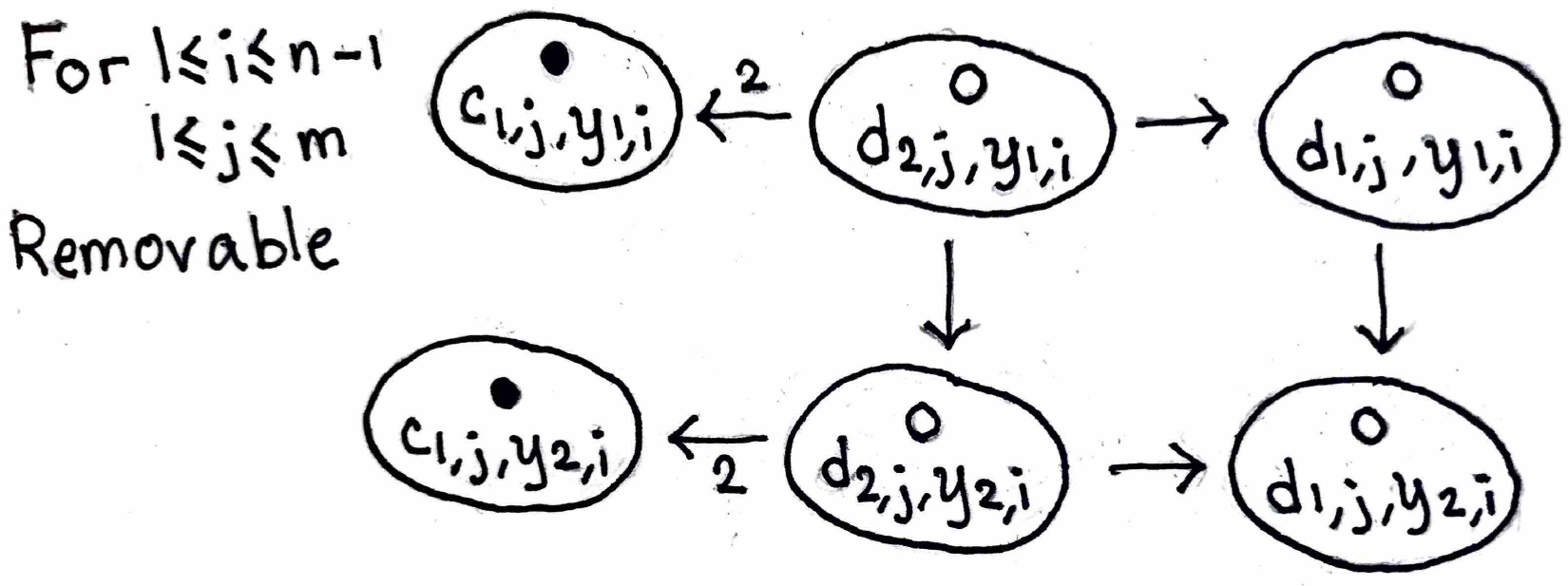}
\caption{}\label{fig:end}
\end{figure}
\clearpage 
Gathering all the simplified components yields the decorated graph in Figure \ref{fig: large type D}. 
\begin{figure}[ht]
\centering
\includegraphics[scale=1]{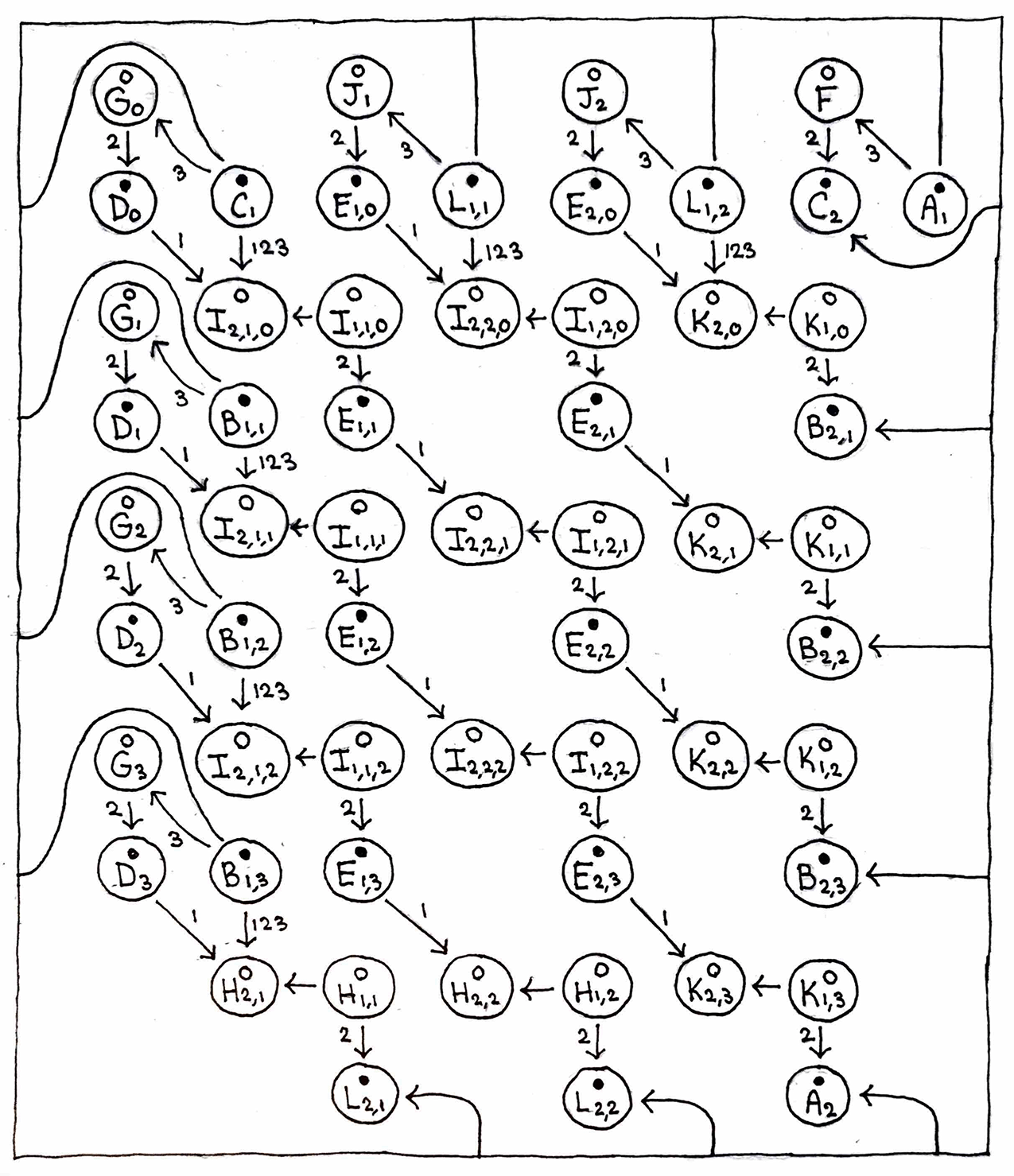}
\caption{The case \(n=3\), \(m=4\). The decorated graph should be interpreted as lying on a torus.}
\label{fig: large type D}
\end{figure}

Canceling the pure differentials in this decorated graph in any order then gives the I/III graph in Figure {\normalfont\ref{fig: initial type D}}. 
\end{proof}
In \cite{HeddenLevine} it is explained how to turn a Type D decorated graph for \({\widehat{\mathit{CFD}}}\) with no pure-differentials into a Type A decorated graph for \(\widehat{\mathit{CFA}}\) (with no \(m_1\) action). In brief, one makes the label changes \(1\to 3\) and \(3\to 1\). Performing this procedure gives the following corollary. 
\begin{corollary}
\label{corollary: type A I/III}
The module \(\widehat{\mathit{CFA}}\big({-Y(n,m),\mathcal{F}_{\mathrm{I/III}}}\big)\) has a model given by the I/III decorated graph in Figure {\normalfont\ref{fig: initial type A}}.
\end{corollary}
Tensoring the Type A module from the above corollary with the bimodules from \cite[Section 10.2]{bimodules} and \cite[Section 5]{factoring}, one can change the position of the basepoint. (Although the above Type A module is not bounded, the bimodules are bounded on the left and right for the same reason as explained in the beginning of Section \ref{sec: gluing}.) This gives the following.
\begin{corollary}
The module \(\widehat{\mathit{CFD}}\big({-Y(n,m),\mathcal{F}_{\mathrm{II/IV}}}\big)\) has a model given by the II/IV decorated graph in Figure {\normalfont\ref{fig: initial type D}}.
\label{corollary: type D II/IV}
\end{corollary}
Finally, we once again apply the procedure in \cite{HeddenLevine} to obtain the following.
\begin{corollary}
The module \(\widehat{\mathit{CFA}}\big({-Y(n,m),\mathcal{F}_{\mathrm{II/IV}}}\big)\) has a model given by the I/III decorated graph in Figure {\normalfont\ref{fig: initial type A}}.
\label{corollary: type A II/IV}
\end{corollary}

\section{Data from \texttt{hf-hat-obd-nice}} \label{sec:data}
In this appendix, we give some data obtained from \texttt{hf-hat-obd-nice} \cite{computing}. We first give a Python program which outputs a region list for a Heegaard diagram corresponding to an open book with monodromy \(f\in \mathrm{Mod}(\Sigma_{0,4},\partial \Sigma_{0,4})\) whose FDTCs are all \(1\) with 
\begin{equation}\pi(f) = \pm\begin{bmatrix}r & s\\ p & q\end{bmatrix},\qquad r,s,p,q>0,\, p>q.\label{eq:matrix}\end{equation}
\begin{lstlisting}[language=Python, caption = Code for generating a region list]
def region_list(r, s, p, q): 
    a = 1 + q + ((r * q) // p) // 2 
    b = 1 + (p - q) + ((r * (p - q)) // p) // 2 
    n = 1 + p + r // 2 
    last = 2 * n + q + s//2 + 2
    rlist = [
        [3, 2, 2 * n + 1, n + a + 1, n + a + 2, 4],
        [n + 2, 2, last, a + 2, a + 3, n + 3],
        [2 * n + 1, 0, 3, 2 * n + 2],
        [last, 1, n + 2, last - 1],
        [b + n + 1, 0, n + 1, b + n + 2],
        [b + 2, 1, 2 * n, b + 3]]
    for i in range(a - 1):
        rlist.append([i + 4, 2 * n + 3 + i, 2 * n + 2 + i, i + 3])
        rlist.append([n + 3 + i, last - 2 - i, last - 1 - i, n + 2 + i])
    for i in range(b - 3):
        rlist.append([i + 4, n + a + 2 + i, n + a + 3 + i, i + 5])
        rlist.append([n + 3 + i, a + 3 + i, a + 4 + i, n + 4 + i])
    for i in range(a - 3):
        rlist.append([b + 3 + i, 2 * n - i, 2 * n - 1 - i, b + 4 + i])
    rlist.append([0, 2 * n + 1, 2, n + 2, 1, b + 2, b + 1, 2 * n, 1, last, 2, 3, 0, n + b + 1, n + b, n + 1])
    return rlist
\end{lstlisting}

More specifically, consider the page \(\Sigma_{0,4}\) as shown in Figure \ref{fig:square_holes}, with the binding component parallel to \(a_i\) labeled \(B_i\). We take a basis of arcs for the page consisting of the following: 
\begin{itemize}
\item an arc \(\beta_0\) of slope \(0\) from \(B_2\) to \(B_1\), 
\item an arc \(\beta_1\) of slope \(0\) from \(B_3\) to \(B_4\), 
\item an arc \(\beta_2\) of slope \(\infty\) from \(B_2\) to \(B_3\).
\end{itemize}
We then set \(\alpha_i=f^{-1}(\beta_i)\) for \(i=0,1,2\). The contact intersection between \(\alpha_i\) and \(\beta_i\) is given the label \(i\). Subsequently, the remaining intersection points are ordered as follows: 
\begin{enumerate}[label=(\roman*)]
\item Intersection points on \(\beta_0\) are given increasing labels (starting at 3), going from \(B_2\) to \(B_1\). 
\item Next, intersection points on \(\beta_1\) are given increasing labels going from \(B_3\) to \(B_4\).
\item Finally, intersection points on \(\beta_2\) are given increasing labels going from \(B_2\) to \(B_3\). 
\end{enumerate}
From these labels, a region list is generated according to the conventions in \cite{computing}. The resulting region list can be entered into \texttt{hf-hat-obd} and \texttt{hf-hat-obd-nice}, as explained in the same paper. In the case of \texttt{hf-hat-obd-nice}, the region list must first be run through \texttt{makenice}.
\begin{remark}
The Heegaard diagram described above is almost nice, except for 2 hexagonal regions. These can be dealt with using two short finger moves. It might be more efficient to directly enter in a nice diagram into \texttt{hf-hat-obd-nice}. Unfortunately, the author has not done the necessary work to generate a region list for such a nice diagram. 
\end{remark}
\par 

We give a table of results obtained from \texttt{hf-hat-obd-nice} on whether the Heegaard Floer contact invariant \(c(f)\) vanishes, where \(f\) is as above. We emphasize again that \(f\) has all FDTCs equal to 1. Among such monodromies, we only consider those not covered by Theorems \ref{thm:main_one}, \ref{thm:main_two} and Propositions \ref{prop:tight_all_twos}, \ref{prop:tight_two_threes}. In addition, if \(p,q,r,s\) are as in (\ref{eq:matrix}), then for a given \(p,q\) we only consider the smallest possible values for \(r\) and \(s\). (Recall \(\pi(f)\) is constrained to lie in the subgroup (\ref{eq:subgroup}) of \(\mathrm{PSL}(2,\mathbb{Z})\).) 
\begin{center}
\captionsetup{type = table}
\begin{longtblr}{Q[l]|Q[l]|Q[l]|}
\(\pi(f)\) &  Continued fraction for \(-p/q\) & \(c(f)=0\)? \\ 
\hline
\(\pm\begin{bmatrix}13 & 8 \\ 8 & 5\end{bmatrix}\) & \([-2,-3,-2]\) & No \\ \hline \(\pm\begin{bmatrix}23 & 14 \\ 18 & 11\end{bmatrix}\) & \([-2,-3,-4]\) & No\\ 
\hline 
\(\pm\begin{bmatrix}33 & 20\\ 28 & 17\end{bmatrix}\) & \([-2,-3,-6]\) & Yes \\ \hline
\(\pm\begin{bmatrix}7 & 4 \\ 12 & 7\end{bmatrix}\) & \([-2,-4,-2]\) & No\\
\hline
\(\pm\begin{bmatrix}7 & 4 \\ 26 & 15\end{bmatrix}\) & \([-2,-4,-4]\) & Yes \\ \hline 
\(\pm\begin{bmatrix}25 & 14 \\ 16 & 9\end{bmatrix}\) & \([-2,-5,-2]\) & No\\
\hline
\(\pm\begin{bmatrix}43 & 24 \\ 34 & 19\end{bmatrix}\) & \([-2,-5,-4]\) & No \\ \hline
\(\pm\begin{bmatrix}11 & 4 \\ 30 & 11\end{bmatrix}\) & \([-3,-4,-3]\) & Yes\\
\hline
\(\pm\begin{bmatrix}11 & 4 \\ 52 & 19\end{bmatrix}\) & \([-3,-4,-5]\) & Yes \\ \hline 
\(\pm\begin{bmatrix}29 & 8 \\ 18 & 5\end{bmatrix}\) & \([-4,-3,-2]\) & No\\
\hline
\(\pm\begin{bmatrix}51 & 14 \\ 40 & 11\end{bmatrix}\) & \([-4,-3,-4]\) & No \\ \hline
\(\pm\begin{bmatrix}15 & 4 \\ 26 & 7\end{bmatrix}\) & \([-4,-4,-2]\) & Yes\\
\hline
\(\pm\begin{bmatrix}15 & 4 \\ 56 & 15\end{bmatrix}\) & \([-4,-4,-4]\) & Yes \\ \hline
\(\pm\begin{bmatrix}53 & 14 \\ 34 & 9\end{bmatrix}\) & \([-4,-5,-2]\) & No\\
\hline
\(\pm\begin{bmatrix}25 & 18 \\ 18 & 13\end{bmatrix}\) & \([-2,-2,-3,-3]\) & No \\ \hline 
\(\pm\begin{bmatrix}39 & 28 \\ 32 & 23\end{bmatrix}\) & \([-2,-2,-3, -5]\) & No\\
\hline
\(\pm\begin{bmatrix}49 & 34 \\ 36 & 25\end{bmatrix}\) & \([-2,-2,-5,-3]\) & Yes \\ \hline
\(\pm\begin{bmatrix}13 & 8 \\ 34 & 21\end{bmatrix}\) & \([-2,-3,-3, -3]\) & No\\
\hline
\(\pm\begin{bmatrix}23 & 10 \\ 16 & 7\end{bmatrix}\) & \([-3,-2,-2,-3]\) & No \\ \hline
\(\pm\begin{bmatrix}37 & 16 \\ 30 & 13\end{bmatrix}\) & \([-3,-2,-2, -5]\) & Yes\\
\hline
\(\pm\begin{bmatrix}31 & 12 \\ 18 & 7\end{bmatrix}\) & \([-3,-3,-2,-2]\) & No \\ \hline
\(\pm\begin{bmatrix}21 & 8 \\ 34 & 13\end{bmatrix}\) & \([-3,-3,-3, -2]\) & No\\
\hline
\(\pm\begin{bmatrix}43 & 10 \\ 30 & 7\end{bmatrix}\) & \([-5,-2,-2,-3]\) & Yes \\ \hline 
\(\pm\begin{bmatrix}23 & 18 \\ 14 & 11\end{bmatrix}\) & \([-2,-2,-2, -3,-2]\) & No\\
\hline
\(\pm\begin{bmatrix}41 & 32 \\ 32 & 25\end{bmatrix}\) & \([-2,-2,-2,-3,-4]\) & No \\ \hline 
\(\pm\begin{bmatrix}13 & 10 \\ 22 & 17\end{bmatrix}\) & \([-2,-2,-2, -4,-2]\) & No\\
\hline
\(\pm\begin{bmatrix}13 & 10 \\ 48 & 37\end{bmatrix}\) & \([-2,-2,-2,-4,-4]\) & Yes \\ \hline
\(\pm\begin{bmatrix}47 & 36 \\ 30 & 23\end{bmatrix}\) & \([-2,-2,-2, -5,-2]\) & No\\
\hline
\(\pm\begin{bmatrix}25 & 16 \\ 14 & 9\end{bmatrix}\) & \([-2,-3,-2,-2,-2]\) & No \\ \hline
\(\pm\begin{bmatrix}47 & 30 \\ 36 & 23\end{bmatrix}\) & \([-2,-3,-2, -2,-4]\) & No\\
\hline
\(\pm\begin{bmatrix}19 & 12 \\ 30 & 19\end{bmatrix}\) & \([-2,-3,-2,-3,-2]\) & No \\ \hline
\(\pm\begin{bmatrix}17 & 10 \\ 22 & 13\end{bmatrix}\) & \([-2,-4,-2, -2,-2]\) & No\\
\hline
\(\pm\begin{bmatrix}37 & 10 \\ 48 & 13\end{bmatrix}\) & \([-4,-4,-2,-2,-2]\) & Yes \\ \hline
\(\pm\begin{bmatrix}39 & 32 \\ 28 & 23\end{bmatrix}\) & \([-2,-2,-2, -2,-3,-3]\) & No\\
\hline
\(\pm\begin{bmatrix}49 & 36\\ 34 & 25\end{bmatrix}\) & \([-2,-2,-3,-2,-2,-3]\) & No \\ \hline 
\(\pm\begin{bmatrix}35 & 16 \\ 24 & 11\end{bmatrix}\) & \([-3,-2,-2, -2,-2,-3]\) & No
\end{longtblr}
\captionof{table}{Data obtained from \texttt{hf-hat-obd-nice}. All of the FDTCs of \(f\) are equal to \(1\).}
\end{center}
\bigskip

\printbibliography
\end{document}